\documentclass{amsart}

\usepackage{amssymb}
\usepackage{amsthm}
\usepackage[brazil,spanish,english]{babel}
\usepackage{graphicx}

\usepackage{hyperref}
\usepackage[latin1]{inputenc}
\usepackage{epstopdf}
\usepackage{srcltx}
\usepackage[normalem]{ulem}
\usepackage{vruler}
\usepackage{color}

\usepackage[hmargin=30mm,top=30mm,bottom=35mm]{geometry}

\theoremstyle{plain}
\newtheorem{theorem}{Theorem}
\newtheorem{lemma}[theorem]{Lemma}
\newtheorem{sub-lemma}[theorem]{Sub-lemma}
\newtheorem{corollary}[theorem]{Corollary}

\newtheorem{proposition}[theorem]{Proposition}
\newtheorem*{proposition*}{Proposition}
\newtheorem*{theorem*}{Theorem}

\newtheorem{theoremain}{Theorem}

\newtheorem{propositionmain}[theoremain]{Proposition}
\newtheorem{corollarymain}[theoremain]{Corollary}

\theoremstyle{definition}

\newtheorem*{definition*}{Definition}
\theoremstyle{remark}

\newtheorem*{claim*}{Claim}

\newtheorem{remark}[theorem]{Remark}
\newtheorem*{remark*}{Remark}

\newcommand{\N}{\mathbb{N}}

\newcommand{\R}{\mathbb{R}}
\newcommand{\D}{\mathbb{D}}
\newcommand{\Z}{\mathbb{Z}}

\renewcommand{\S}{\mathbb{S}}
\newcommand{\T}{\mathbb{T}}
\newcommand{\A}{\mathbb{A}}
\newcommand{\Q}{\mathbb{Q}}

\renewcommand{\ge}{\geqslant}
\renewcommand{\le}{\leqslant}
\renewcommand{\geq}{\geqslant}
\renewcommand{\leq}{\leqslant}

\newcommand{\abs}[1]{\left\lvert{#1}\right\rvert}

\begin{document}

% \setvruler

\title{ Forcing theory for transverse trajectories of surface homeomorphisms}
\author{P. Le Calvez}
\address{Sorbonne Universit\'es, UPMC Univ Paris 06, Institut de Math\'ematiques de Jussieu-Paris Rive Gauche, 
UMR 7586, CNRS, Univ Paris Diderot, Sorbonne Paris Cit\'e, F-75005, Paris, France}
\email{patrice.le-calvez@imj-prg.fr}
\author{F. A. Tal}
\address{Instituto de Matem\'atica e Estat\'\i stica, Universidade de S\~ao Paulo, Rua do Mat\~ao 1010, Cidade Universit\'aria, 05508-090 S\~ao Paulo, SP, Brazil}
\email{fabiotal@ime.usp.br}
\thanks{F. A. Tal was partially supported by CAPES, FAPESP and CNPq-Brasil}

 \date{}
\begin{abstract}
This paper studies homeomorphisms of surfaces isotopic to the identity by means of purely topological methods and Brouwer theory. The main development is a novel theory of orbit forcing using maximal isotopies and transverse foliations. This allows us to derive new proofs for some known results as well as some new applications, among which we note  the following: we extend Franks and Handel's classification of zero entropy maps of $\S^2$  for non-wandering homeomorphisms; we show that if $f$ is a Hamiltonian homeomorphism of the annulus, then the rotation set of $f$ is either a singleton or it contains zero in the interior, proving a conjecture posed by Boyland; we show that there exist compact convex sets of the plane that are not the rotation set of some torus homeomorphisms, proving a first case of the Franks-Misiurewicz Conjecture; we extend a bounded deviation result relative to the rotation set to the general case of torus homeomorphisms.

 \end{abstract}
\maketitle

%&LaTeX
%format latex

%&LaTeX%st
%format latex
%&LaTeX
%format latex

%\documentclass[11pt]{article}
%\usepackage{epsfig}
%\usepackage{latexsym,graphicx,color}
%\input macintosh.tex
%\hsize 16cm
%\vsize 24.6cm
%\textheight 23cm \textwidth 16cm
%\headheight -1cm
%\oddsidemargin 0pt \evensidemargin 0pt
%\frenchspacing

%format latex

%-----------------Ins\'erer des figures-------------------------
%\def\picture #1 by #2 (#3){
  %\vbox to #2{
    %\hrule width #1 height 0pt depth 0pt
    %\vfill
    %\special{picture #3} % this is the low-level interface
   % }
  %}

%\def\scaledpicture #1 by #2 (#3 scaled #4){{
  %\dimen0=#1 \dimen1=#2
%  \divide\dimen0 by 1000 \multiply\dimen0 by #4
%  \divide\dimen1 by 1000 \multiply\dimen1 by #4
%  \picture \dimen0 by \dimen1 (#3 scaled #4)}
%  }

%\def\cqfd{\hfil CQFD}
%\def\R{{\bf R}}
%\def\Z{{\bf Z}}
%\def\N{{\bf N}}
%\def\Q{{\bf Q}}
%\def\T{{\bf T}}
%\def\A{{\bf A}}
%\def\C{{\bf C}}
%\def\D{{\bf D}}
%\def\S{{\bf S}}
%\def\a{\alpha}
%\def\b{\beta}
%\def\s{\sigma}
%\def\e{\varepsilon}
%\def\g{\gamma}\begin{document}

\def\eqalign#1{\null\,\vcenter{\openup\jot
\ialign{\strut\hfil$\displaystyle{##}$&
$\displaystyle{{}##}$\hfil \crcr #1\crcr }}\,}

\def\eqalignno#1{\displ@y \tabskip=\@centering
\halign to\displaywidth{\hfil$\@lign\displaystyle{##}$
\tabskip=0pt &$\@lign\displaystyle{{}##}$

\hfil\tabskip=\@centering
$\llap{$\@lign##$}\tabskip=Opt\crcr #1\crcr}}

%\thispagestyle{empty}
%\maketitle

\section{Introduction}

Let us begin by recalling some facts about Sharkovski's theorem, which can be seen as a typical example of an orbit forcing theory in dynamical systems. In this theorem, an explicit total order $\preceq$ on the set of natural integers is given that satisfies the following: every continuous transformation $f$ on $[0,1]$ that contains a periodic orbit of period $m$ contains a periodic orbit of period $n$ if $n\preceq m$. Much more can be said. If $f$ admits a periodic orbit of period different from a power of $2$,  one can construct a Markov partition and codes orbits with the help of the associated finite subshift. In particular one can prove that the topological entropy of $f$ is positive.  There exists a forcing theory about periodic orbits for surface homeomorphisms related to Nielsen-Thurston classification of surface homeomorphisms, with many interesting dynamical applications (see for example \cite{Bo} or \cite{Mo} for survey articles). In case of homeomorphisms isotopic to the identity, this theory deals with the braid types associated to the periodic orbits.  A more subtle theory (homotopic Brouwer theory) was introduced by M. Handel for surface homeomorphisms and developed by J. Franks and Handel to become a very efficient tool in  two-dimensional dynamics. 

The goal of the article is to give a new orbit forcing theory for surface homeomorphisms that are isotopic to the identity, theory that will be expressed in terms of {\it maximal isotopy}, {\it transverse foliations} and {\it transverse trajectories}. Note first that the class of surface homeomorphisms isotopic to the identity contains the time one maps of time dependent vector fields. Consequently, what is proved in this article can be applied to the dynamical study of a time dependent vector field on a surface, periodic in time.
In what follows, a surface $M$ is orientable and furnished with an orientation. If $f$ is a homeomorphism of $M$ isotopic to the identity, the choice of an isotopy $I=(f_t)_{t\in[0,1]}$ from the identity to $M$ should not be very important, as we are looking at the iterates of $f$. What looks like the {\it trajectory }of a point $z$, that means the path $I(z):t\mapsto f_t(z)$ \textcolor{black}{seems useless}. It appears that this is not the case: there are isotopies that are better than the other ones. This is clear if $f$ is the time one map of a complete time independent vector field $\xi$. The isotopy $(f_t)_{[0,1]}$ defined by the restriction of the flow $(f_t)_{t\in\R}$ is clearly better than any other choice of an isotopy, in the sense that it will be useful while studying the dynamics of $f$. It is easy to see that in this last case, there is no fixed point of $f$ \textcolor{black}{in the complement of the singular set of the vector field}  whose trajectory is contractible relative to \textcolor{black}{this same singular set.} In this situation, the singular points correspond to the fixed points of $I$, which means the points whose trajectory is constant.

 In general, let us say that an isotopy $I=(f_t)_{t\in[0,1]}$, that joins the identity to a homeomorphism $f$, is a maximal isotopy if there is no fixed point of $f$ whose trajectory is contractible relative to the fixed point set of $I$. A very recent result of F. B\'eguin, S. Crovisier and F. Le Roux \cite{BCL} asserts that such an isotopy always exists if $f$ is isotopic to the identity (a slightly weaker result was previously proved by O. Jaulent \cite{J}). A fundamental result \cite{Lec1} asserts that a maximal isotopy always admits a transverse foliation. This is a singular oriented foliation $\mathcal F$ whose singular set coincides with the fixed point set of $I$ and such that every non trivial trajectory is homotopic (relative to the endpoints) to a path that is transverse to the foliation (which means that it locally crosses every leaf from the right to the left). This path $I_{\mathcal F}(z)$, the transverse trajectory, is uniquely defined up to a natural equivalence relation (meaning that the induced path in the space of leaves is unique). In the case where $f$ is the time one map of a complete time independent vector field $\xi$, it is very easy to construct a transverse foliation by taking the integral curves of any vector field $\eta$ that is transverse to $\xi$, and in that case the trajectories $I(z)$ are transverse. In a certain sense, maximal isotopies are isotopies that are \textcolor{black}{ as close as} possible to isotopies induced by flows.

Maximal isotopies and transverse foliations are known to be efficient tools for the dynamical study of surface homeomorphisms (see \cite{D1}, \cite{D2}, \cite{KT2}, \cite{Lec1}, \cite{Lec2}\cite{Ler}, \cite{Mm}, \cite{T} for example). Usually they are used in the following way. Properties of $f$ are transposed ``by duality'' to properties of $\mathcal F$, then one studies the dynamics of the foliation and comes back to $f$. Roughly speaking, the leaves of the foliation are pushed along the dynamics. This property is cleverly used in the articles of P. D\'avalos (\cite{D1}, \cite{D2}). Our original goal was a boundedness displacement result (Theorem \ref{th:recurrent_on_the_lift_intro} of this introduction) which needed a formalization of the ideas of D\'avalos. This was nothing but a forcing theory for transverse trajectories. For every integer $n\geq 1$, let us define  by concatenation the paths $I^n(z)=\prod_{0\leq k<n} I(f^k(z))$ and $I_{\mathcal F}^n(z)=\prod_{0\leq k<n} I_{\mathcal F}(f^k(z))$. The basic question can be formulated as follows: from the knowledge of a finite family $(I_{\mathcal F}^{n_i}(z_i))_{1\leq i\leq p}$ of transverse trajectories, can we deduce the existence of other transverse trajectories $I^n_{\mathcal F}(z)$? The key result (Proposition \ref{pr: fundamental}), which is new and whose proof is very simple, can be stated as follows: if two paths $I_{\mathcal F}^{n_1}(z_1)$ and $I_{\mathcal F}^{n_2}(z_2)$ {\it intersect transversally relative to $\mathcal F$} (the precise definition will be given later in the article) then one can construct two other paths $I_{\mathcal F}^{n_1+n_2}(z_3)$ and $I_{\mathcal F}^{n_1+n_2}(z_4)$ by a natural change of direction at the intersection point. It becomes possible, in many situations to code transverse trajectories with the help of a Bernouilli shift or in other situations to construct transverse trajectories that are multiples of the same  loop.
%resemble the make a direct study of transverse trajectories, with applications to different problems of two dimensional dynamics. Many of the applications are about rotation sets of torus or annulus homeomorphisms, but they are some that are related to the entropy.

In order to obtain applications of this forcing theory, we need to relate the information obtained by the knowledge of these new sets of transverse trajectories to other properties of the dynamics.  To do so, one can define the {\it whole trajectory} $I^{\Z}(z)=\prod_{k\in\Z} I(f^k(z))$ and the {\it whole transverse trajectory} $I_{\mathcal F}^{\Z}(z)=\prod_{k\in\Z} I_{\mathcal F}(f^k(z))$ of a point $z$. The properties of the dynamics are recovered by three structural results that, together with Proposition \ref{pr: fundamental}, form the core of the theory. The first of these results is a realization result, Proposition \ref{pr: realization},  showing that in many cases, the existence of finite transverse trajectories that are equivalent to multiples of a given transverse loop implies the existence of a periodic point whose transverse trajectory for one period is equivalent to the  same transverse loop.  The second of these results, Theorem \ref{th:transverse_imply_periodic points}, shows that if there exist two recurrent points $z$ and $z'$ such that $I_{\mathcal F}^{\Z}(z)$ and $I_{\mathcal F}^{\Z}(z')$ intersect transversally relative to $\mathcal F$ (with a self intersection if $z=z'$) the number of periodic points of period $n$ for some iterate of $f$ grows exponentially in $n$. The third result, Theorem \ref{th:transverse_imply_entropy}, shows that if in the previous result we assume that the surface is closed, then the topological entropy of $f$ is strictly positive. This final result presents, to our knowledge, an entirely new mechanism to detect positive entropy, one that bypasses any requirement of smoothness of the map. Consequently, our applications are for general homeomorphisms isotopic to the identity, and include both new entropy theorems for maps of the annulus and generalizations of results known only for $C^1$-diffeomorphisms (sometimes for $C^{1+\varepsilon}$-diffeomorphisms, sometimes for $C^{\infty}$-diffeomorphisms).
There is no doubt that they are many similarities with Franks-Handel methods. Looking more carefully at the links between the two methods should be a project of high interest.

%One can define the {\it whole trajectory} $I^{\Z}(z)=\prod_{k\in\Z} I(f^k(z))$ and the {\it whole transverse trajectory} $I_{\mathcal F}^{\Z}(z)=\prod_{k\in\Z} I_{\mathcal F}(f^k(z))$ of a point $z$. An important fact (see Theorem \ref{th:transverse_imply_periodic points} and Theorem \ref{th:transverse_imply_entropy}) is that if there exist two recurrent points $z$ and $z'$ such that $I_{\mathcal F}^{\Z}(z)$ and $I_{\mathcal F}^{\Z}(z')$ intersect transversally relative to $\mathcal F$ (with a self intersection if $z=z'$) then the topological entropy of $f$ is positive and the number of periodic points of period $n$ for some iterate of $f$ grows exponentially in $n$.  A point that can be noticed is the fact that the field of applications of our method is the whole space of surface homeomorphisms isotopic to the identity. Consequently, we will be able to generalize some results already known for diffeomorphisms. 

Let us display now more precisely the main applications, beginning with the case of annulus homeomorphisms. Here, 
 $\mathcal{M}(f)$ is the set of invariant Borel probability measures  $\mu$ of $f$, the set $\mathrm{supp}(\mu)$ the support of $\mu$, the rotation number $\mathrm{rot}(\mu)$ the integral  $\int_{\A} \varphi\, d\mu$, where $\varphi:\A\to\R$ is the map lifted by $\pi_1\circ\check f- \pi_1$ (the map $\pi_1:(x,y)\mapsto x$ being the first projection), the segment $\mathrm{rot}(\check f)$ the set of rotation numbers of invariant measures.

\begin{theoremain}\label{th : annulus_intro}
Let $f$ be a homeomorphism of $\A=\T^1\times[0,1]$ that is isotopic to the identity and  $\check f$ a lift to $\R\times[0,1]$. Suppose that $\mathrm{rot}(\check f)$ is a non trivial segment and that $\rho$ is an endpoint of $\mathrm{rot}(\check f)$ that is rational. Define
$${\mathcal M}_{\rho}=\left\{ \mu\in {\mathcal M}(f)\,,\, \mathrm{rot}(\mu)=\rho\right\}, \enskip {X}_{\rho}= 
\overline{\bigcup_{\mu\in {\mathcal M}_{\rho}} \mathrm{supp}(\mu)}.$$
Then every invariant measure supported on ${X}_{\rho}$ belongs to ${\mathcal M}_{\rho}.$
\end{theoremain}

 We deduce immediately the following positive answer to a question of P. Boyland:

\begin{corollarymain}\label{cr:boylandannulus_intro}
Let $f$ be a homeomorphism of $\A$ that is isotopic to the identity and preserves a probability measure $\mu$ with full support. Let us fix a lift $\check f$. Suppose that $\mathrm{rot}(\check f)$ is a non trivial segment. The rotation number $\mathrm{rot}(\mu)$ cannot be an endpoint of $\mathrm{rot}(\check f)$ if this endpoint is rational.
\end{corollarymain}  

Let us explain what happens for torus homeomorphisms.  Here again
 $\mathcal{M}(f)$ is the set of invariant Borel probability measures $\mu$ of $f$, the set $\mathrm{supp}(\mu)$ the support of $\mu$ and the rotation vector $\mathrm{rot}\mu)$ the integral
 $\int_{\T^2} \varphi\, d\mu$, where $\varphi:\T^2\to\R^2$ is the map lifted by $\check f-\mathrm{Id}$. The set of rotation  vectors  of invariant measures $\mathrm{rot}(\check f)$ is a compact and convex subset of $\R^2$. Nothing is known about the plane subsets that can be written as such a rotation set. The following result gives the first obstruction:
\bigskip

\begin{theoremain}\label{th:impossible_rotation_set_intro}
 Let $f$ be a homeomorphism of $\T^2$ that is isotopic to the identity and $\check f$ a lift of $f$ to  $\R^2$. The frontier of $\mathrm{rot}(\check f)$ does not contain a segment with irrational slope that contains a rational point in its interior.
\end{theoremain}

It was previously conjectured by Franks and Misiurewicz in \cite{FM} that a line segment $L$  could not be realized as a rotation set of a torus homeomorphism in the following conditions: (i) $L$ has irrational slope and a rational point in its interior, (ii) $L$ has rational slope but no rational points and (iii) $L$ has irrational slope and no rational points. While Theorem \ref{th:impossible_rotation_set_intro}  implies the conjecture for case (i),  A. \'Avila has given a counter-example for case (iii).

The second result is a boundedness result:

\begin{theoremain}\label{th:bounded_deviation_intro}
 Let $f$ be a homeomorphism of $\T^2$ that is isotopic to the identity and $\check f$ a lift of $f$ to  $\R^2$. If $\mathrm{rot}(\check f)$ has a non empty interior, then there exist a constant $L$ such that for every $z\in\R^2$ and every $n\geq 1$, one has $d(\check f^n(z)-z, n\mathrm{rot}(\check f))\leq L$.
\end{theoremain}

Note that by definition of the rotation set one knows that
$$\lim_{n\to+\infty} {1\over n}\left(\max_{z\in\R^2} d(\check f^n(z)-z, n\mathrm{rot}(\check f))\right)=0$$Theorem \ref{th:bounded_deviation_intro}   clarifies the speed of convergence. It was already known for homeomorphisms in the special case of a polygon with rational vertices (see D\'avalos \cite{D2}) and for $\mathcal{C}^{1+\epsilon}$ diffeomorphisms (see Addas-Zanata \cite{AZ}). As already noted in \cite{AZ}, we can deduce an interesting result about {\it maximizing measures}, which means measure $\mu\in\mathcal M(f)$ whose rotation  vector  belongs to the frontier of $\mathrm{rot}(\check f)$.  The rotation number of such a measure belongs to at least one supporting line of $\mathrm{rot}(\check f)$. Such a line admits the equation $\psi(z)=\alpha(\psi)$ where 
$\psi$ is a non trivial linear form on $\R^2$ and
$$\alpha(\psi)=\max_{\mu\in {\mathcal M}(f)} \psi(\mathrm{rot}(\mu))=\max_{\mu\in {\mathcal M}(f)}\int_{\T^2}\psi\circ\varphi \, d\mu.$$ 
Set 
$$ {\mathcal M}_{\psi}=\left\{ \mu\in {\mathcal M}(f)\,,\, \psi(\mathrm{rot}(\mu))=\alpha(\psi)\right\}, \enskip {X}_{\psi}= 
\overline{\bigcup_{\mu\in {\mathcal M}_{\psi}} \mathrm{supp}(\mu)}.$$

The following result, that can be easily deduced from Theorem \ref{th:bounded_deviation} and Atkinson's Lemma in Ergodic Theory (see \cite{A}), tells us that the sets $X_{\psi}$ behave like the Mather sets of the Tonelli Lagrangian systems.

\begin{propositionmain} \label{pr: rotation number set_intro} Let $f$ be a homeomorphism of $\T^2$ that is isotopic to the identity and $\check f$ a lift of $f$ to  $\R^2$. Assume that $\mathrm{rot}(\check f)$ has a non empty interior. Then, every measure $\mu$ supported on ${X}_{\psi}$ belongs to  ${\mathcal M}_{\psi}$. Moreover,  if $z$ lifts a point of $X_{\psi}$, then for every  $n\geq 1$, one has $\vert\psi(\check f^n(z))-\psi(z)-n\beta(\psi)\vert \leq L\Vert \psi\Vert $, where 
$L$ is the constant given by  Theorem \ref{th:bounded_deviation_intro}.
\end{propositionmain}

It admits as an immediate corollary the torus version of Boyland's question:

\begin{corollarymain}\label{co: Boyland_intro}
Let $f$ be a homeomorphism of $\T^2$ that is isotopic to the identity, preserving a measure $\mu$ of full support, and $\check f$ a lift of $f$ to  $\R^2$. Assume that $\mathrm{rot}(\check f)$ has a non empty interior. Then $\mathrm{rot}(\mu)$ belongs to the interior of $\mathrm{rot}(\check f)$.
\end{corollarymain}
This result was known for $\mathcal{C}^{1+\epsilon}$ diffeomorphisms (see \cite {AZ}).

 The next resut is due to Llibre and MacKay, see \cite{LlM}. Its original proof uses Thurston-Nielsen theory of surface homeomorphisms, more precisely the authors prove that there exists a finite invariant set $X$ such that $f\vert_{\T^2\setminus X}$ is isotopic to a pseudo-Anosov map. We will give here an alternative proof by exhibiting $(n,\varepsilon)$ separated sets constructed with the help of transverse trajectories.

\begin{theoremain}\label{th:Llibre_MacKay_intro}
Let $f$ be a homeomorphism of $\T^2$ that is isotopic to the identity and $\check f$ a lift of $f$ to  $\R^2$. If $\mathrm{rot}(\check f)$ has a non empty interior, then the topological entropy of
$f$ is positive.
\end{theoremain}

Our original goal, while writing this article, was to prove the following boundedness displacement result:

\begin{theoremain}\label{th:recurrent_on_the_lift_intro} We suppose that $M$ is a compact orientable surface furnished with a Riemannian structure. We endow the universal covering space $\check M$ with the lifted structure and denote by $d$ the induced distance. Let $f$ be a homeomorphism of $M$ isotopic to the identity and $\check f$ a lift to $\check M$ naturally defined by the isotopy. Assume that there exists an open topological disk $U\subset M$ such that the fixed points set of $\check f$ projects into $U$. Then;

\smallskip
\noindent-\enskip either there exists $K>0$ such that $d(\check f^n(\check z), \check z)\leq  K$, for all $n\geq 0$ and all \textcolor{black}{bi-recurrent} point $\check z$ of $\check f$;

\smallskip
\noindent-\enskip or there exists a nontrivial covering automorphism $T$ and $q>0$ such that, for all $r/s\in (-1/q,1/q)$, the map $\check f^{q}\circ T^{-p}$ has a fixed point. In particular, $f$ has non-contractible periodic points of arbitrarily large prime period.
\end{theoremain}

Theorem \ref{th:recurrent_on_the_lift_intro} has an interesting consequence for torus homeomorphisms. Say a homeomorphism $f$ of $\T^2$ is {\it Hamiltonian} if it preserves a measure $\mu$ with full support and it has a lift $\check f$ (called the Hamiltonian lift of $f$) such that the rotation vector of $\mu$ is null. 

\begin{corollarymain}\label{co:hamiltonian_bounded_intro}
Let $f$ be a Hamiltonian homeomorphism of $\T^2$ such that all its periodic points are contractible, and such that it fixed point set is contained in a  topological disk. Then there exists $K>0$ such that if $\check f$ is the Hamiltonian lift of $f$, then for every $z$ and every $n\geq 1$, one has $\Vert \check f^n(z)-z\Vert\leq K$.
\end{corollarymain}

The study of non-contractible periodic orbits for Hamiltonian maps of sympletic manifolds has been receiving increased attention (see for instance \cite{GG}). A natural question in the area, posed by V. Ginzburg, is to determine if the existence of non-contractible periodic points is generic for smooth Hamiltonians.  A consequence of Corollary \ref{co:hamiltonian_bounded_intro}  is an affirmative answer for the case of the torus:

{\begin{propositionmain}\label{pr:Ginzburg_intro}
Let $\mathrm{Ham}_{\infty}(\T^2)$ be the set of Hamiltonian $C^{\infty}$ diffeomorphisms of $\T^2$ endowed with the Whitney $C^{\infty}$- topology. There exists a residual subset $\mathcal A$ of $\mathrm{Ham}_{\infty}(\T^2)$ such that every $f$ in $\mathcal A$ has non-contractible periodic points.
\end{propositionmain}

Let us explain now the results related to the entropy. For example we can give a short proof of the following improvement of a result due to Handel \cite{H1}.

\begin{theoremain} \label{th: H_intro} Let $f:\S^2\to\S^2$ be an orientation preserving homeomorphism such that the complement of the fixed point set is not an annulus. If $f$ is topologically transitive then the number of periodic points of period $n$ for some iterate of $f$ grows exponentially in $n$.  Moreover, the entropy of $f$ is positive.
\end{theoremain}

Another entropy result we obtain is related to the existence and continuous variation of rotation numbers for homeomorphisms of the open annulus. A stronger version of this result for diffeomorphisms was already proved in an unpublished paper of Handel \cite{H2}.
Given a homeomorphism of $\T^1\times\R$ and a lift $\check f$ to $\R^2$, we say that a point $z\in\T^1\times\R$ {\it has a rotation number} $\mathrm{rot}(z)$ if the $\omega$-limit of its orbit is not empty, and if for any compact set $K\subset\T^1\times\R$ and every increasing sequence of integers $n_k$ such that $f^{n_k}(z)\in K$ and any $\check z\in\pi^{-1}(z)$, 
$$\lim_{k\to\infty}\frac{1}{n_k}\left( \pi_1(\check f^{n_k}(\check z)- \pi_1(\check z)\right)=\mathrm{rot}(z),$$
where $\pi$ is the covering projection from $\R^2$ to $\T^1\times\R$ and $\pi_1:\R^2\to \R$ is the projection on the first coordinate. 

\begin{theoremain}\label{th:continuous_rotation_intro}
 Let $f$ be a homeomorphism of the open annulus $\T^1\times \R$ isotopic to the identity, $\check f$ a lift of $f$ to the universal covering and $f_{\mathrm{sphere}}$ be the natural extension of $f$ to the sphere obtained by compactifying each end with a point. If the topological entropy of $f_{\mathrm{sphere}}$ is zero, then each bi-recurrent point (meaning forward and backward recurrent) has a rotation number, 
and the function $z\mapsto \mathrm{rot}(z)$ is continuous on the set of bi-recurrent points.
\end{theoremain}

Let us finish with a last application. J. Franks and M. Handel recently gave a classification result for area preserving diffeomorphisms of $\S^2$ with entropy $0$ (see \cite{FH}). Their proofs are purely topological but the $C^1$ assumption is needed to use a Thurston-Nielsen type classification result relative to the fixed point set (existence of a normal form) and the $C^{\infty}$ assumption to use Yomdin results on arcs whose length growth exponentially by iterates. We will give a new proof of the fundamental decomposition result (Theorem 1.2 of \cite{FH}) which is the main building block in their structure theorem. In fact we will extend their result to the case of homeomorphisms and replace the area preserving assumption by the fact that every point is non wandering.

\begin{theoremain}\label{th: FH_intro} Let $f:\S^2\to\S^2$ be an orientation preserving homeomorphism such that $\Omega(f)=\S^2$ and $h(f)=0$. Then there exists a family of pairwise disjoint invariant open sets $(A_{\alpha})_{\alpha\in\mathcal A}$ whose union is dense such that:

\smallskip
\noindent{\bf i)}\enskip each $A_{\alpha}$ is an open annulus;

\smallskip
\noindent{\bf ii)}\enskip the sets $A_{\alpha}$ are the maximal fixed point free invariant open annuli;

\smallskip
\noindent{\bf iii)}\enskip the $\alpha$-limit \textcolor{black}{set of a point $z\not\in\bigcup_{\alpha\in\mathcal A}A_{\alpha}$ is included in a single connected component of the fixed point set $\mathrm{fix}(f)$ of $f$, and the same holds for the $\omega$-limit set of $z$};

\smallskip
\noindent{\bf iv)}\enskip let $C$ be a connected component of the frontier of $A_{\alpha}$ in $\S^2\setminus\mathrm{fix}(f)$, then the connected components of $\mathrm{fix}(f)$ that contain $\alpha(z)$ and $\omega(z)$ are independent of $z\in C$.

\end{theoremain}

\bigskip
Let us explain now the plan of the article. In the \textcolor{black}{second} section we will introduce the definitions of many mathematical objects, including precise definitions of rotation vectors and rotation sets.
The \textcolor{black}{third} section will be devoted to the study of transverse paths to a surface foliation. We will introduce the notion of a pair of equivalent paths, of a recurrent transverse path and of $\mathcal{F}$-transverse intersection between two transverse paths. An important result, which will be very useful in the proofs of Theorems  \ref{th: H_intro}  and \ref{th: FH_intro}  is Proposition \ref{pr:transverse_on_sphere} which asserts that a transverse recurrent path to
a singular foliation on $\S^2$ that has no $\mathcal{F}$-transverse self-intersection is equivalent to the natural lift of a transverse simple loop (i.e. an adapted version of Poincar\'e-Bendixson theorem).
We will recall the definition of maximal isotopies, transverse foliations and transverse trajectories in Section 4. We will state the fundamental result about $\mathcal{F}$-transverse intersections of transverse trajectories (Proposition \ref{pr: fundamental}) and its immediate consequences. An important notion that will be introduced is the notion of linearly admissible transverse loop. 
To any periodic orbit is naturally associated such a loop. A realization result (Proposition  \ref{pr: realization}) will give us sufficient conditions for a linearly admissible transverse loop to be associated to a periodic orbit. Section 5 will be devoted to the proofs of Theorem \ref{th:transverse_imply_periodic points} (about exponential growth of periodic orbits) and  Theorem \ref{th:transverse_imply_entropy} (about positiveness of the entropy). We will give the proofs of Theorem   \ref{th:recurrent_on_the_lift_intro}, \ref{th : annulus_intro} and  \ref{th: H_intro} in Section 6 while Section 7 will be almost entirely devoted to the proof of Theorem \ref{th: FH_intro} (we will prove Theorem \ref {th:continuous_rotation_intro} at the end of it).We will begin by stating a ``local version'' relative to a given maximal isotopy (Theorem \ref{th: FH local}).  We will study torus homeomorphisms in Section 8 and will give there the proofs of Theorems \ref{th:impossible_rotation_set_intro},  \ref{th:bounded_deviation_intro} and \ref{th:Llibre_MacKay_intro}.

\bigskip
We would like to thank Fr\'ed\'eric Le Roux for informing us of some important gaps in the original proofs of Theorem \ref{th:transverse_imply_periodic points}, and Theorem \ref{th:transverse_imply_entropy}.  We would also like to thank Andr\'es Koropecki for his useful comments and for discussions regarding Proposition \ref{pr:Ginzburg_intro}, and to Victor Ginzburg for presenting us the question on the genericity of non-contractible periodic points for Hamiltonian diffeomorphisms. Finally, we would like to thank the anonymous referee for the careful work and suggestions which greatly improved our text.

\section{Notations}

We will endow $\R^2$ with its usual scalar product $\langle\enskip\rangle$ and its usual orientation. We will write $\Vert\enskip\Vert$ for the associated norm. For every point $z\in \R^2$ and every set $X\subset \R^2$ we write
$d(z,X)=\inf_{z'\in X} \Vert z-z'\Vert$.  We denote by $\pi_1:(x,y)\mapsto x$ and $\pi_2:(x,y)\mapsto y$ the two projections. If $z=(x,y)$, we write $z^{\perp}=(-y,x)$.

The $r$-dimensional torus $\R^r/\Z^r$ will be denoted $\T^r$, the $2$-dimensional sphere will be denoted $\S^2$. A subset $X$ of a surface $M$ is called an {\it open disk} if it is homeomorphic to $\D=\{z\in\R^2\, ,\, \Vert z\Vert  <1\}$ and a closed disk if it is homeomorphic to $\overline{\D}=\{z\in\R^2\, ,\, \Vert z\Vert  \leq 1\}$. It is called an {\it annulus} if it homeomorphic to $\T^1\times J$, where $J$ is a non trivial interval of $\R$. In case where $J=[0,1]$, $J=(0,1)$, $J=[0,1)$, we will say that $X$ is a {\it closed annulus}, an {\it open annulus}, a {\it semi-closed annulus} respectively.

Given a homeomorphism $f$ of a surface $M$ and a point $z\in M$ we define the $\alpha$-limit set of $z$ by $\bigcap_{n\geq 0}\overline{\bigcup_{k\geq n} f^{-k}(z)}$ and we denote it $\alpha(z)$. We also define the $\omega$-limit set of $z$ by $\bigcap_{n\geq 0}\overline{\bigcup_{k\geq n} f^{k}(z)}$ and we denote it $\omega(z)$.

\subsection{ Paths, lines, loops}

 A {\it path} on a surface $M$ is a continuous map $\gamma:J\to M$ defined on an interval $J\subset\R$. In absence of ambiguity its image will also be called a path and denoted by $\gamma$. We will denote $\gamma^{-1}:-J\to M$ the path defined by  $\gamma^{-1}(t)=\gamma(-t)$.  If $X$ and $Y$ are two disjoint subsets of $M$, we will say that  a path  $\gamma:[a,b]\to M$ {\it joins} $X$ to $Y$ if $\gamma(a)\in X$ and $\gamma(b)\in Y$.  A path $\gamma:J\to M$ is {\it proper} if $J$ is open and the preimage of every compact subset of $M$ is compact. A {\it line} is an injective and proper path $\lambda:J\to M$, it inherits a natural orientation induced by the usual orientation of $\R$. If $M=\R^2$, the complement  of $\lambda$ has two connected components, $R(\lambda)$ which is on the right of $\lambda$ and $L(\lambda)$ which is on its left. More generally, if $M$ is a non connected surface with connected components homeomorphic to $\R^2$,  and if $M'$ is the connected component of $M$ containing $\lambda$, the two connected components of $M'\setminus\lambda$ will similarly be denoted $R(\lambda)$ and $L(\lambda)$.

 \medskip
 
Let us suppose that  $\lambda_0$ and $\lambda_1$ are two disjoint lines of $\R^2$. We will say that they are {\it comparable} if their right components are comparable for the inclusion. Note that $\lambda_0$ and $\lambda_1$ are not comparable if and only if  $\lambda_0$ and $(\lambda_1)^{-1}$ are comparable.

\medskip

Let us consider three lines $\lambda_0$, $\lambda_1$, $\lambda_2$ in $\R^2$. We will say that $\lambda_2$ is {\it above $\lambda_1$ relative to $\lambda_0$}  (and $\lambda_1$ is {\it below $\lambda_2$ relative to $\lambda_0$}) if:

\smallskip
\noindent-\enskip the three lines are pairwise disjoint;

\smallskip
\noindent-\enskip none of the lines separates the two others;

\smallskip

\noindent-\enskip if $\gamma_1$, $\gamma_2$ are  two disjoints paths that join $z_1=\lambda_0(t_1)$, $z_2=\lambda_0(t_2)$ to $z'_1\in\lambda_1$, $z'_2=\lambda_2$ respectively, and that do not meet the three lines but at the ends, then $t_2>t_1$.

\begin{figure}[ht!]
\hfill
\includegraphics [height=48mm]{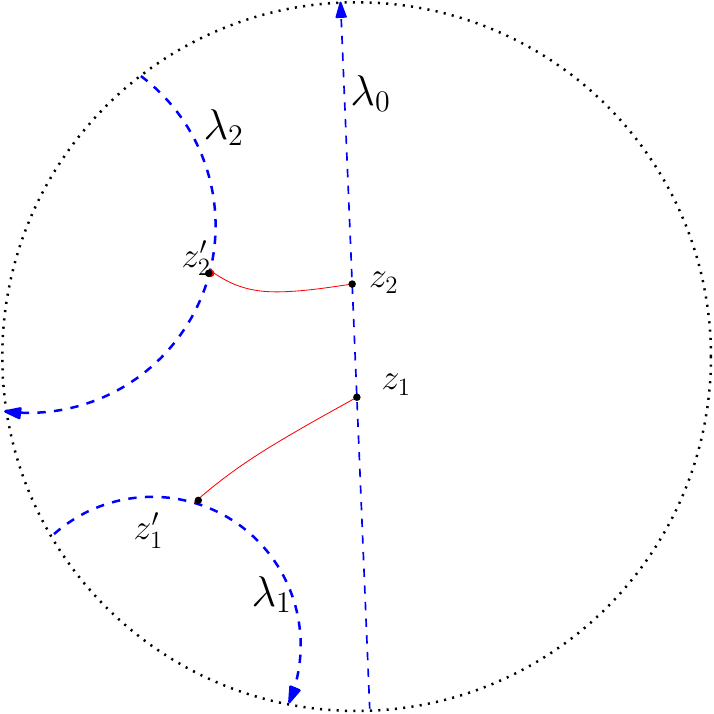}
\hfill{}
\caption{\small Order of lines relative to $\lambda_0$.}
\label{figure_order_lines}
\end{figure}

\smallskip
This notion does not depend on the orientation of $\lambda_1$ and $\lambda_2$ but depends of the orientation of $\lambda_0$ (see Figure \ref{figure_order_lines})\footnote{In all figures in the text, we will represent the plane $\R^2$ as the open disk. The reason being that in many cases we are dealing with the universal covering space of an a hyperbolic surface.}. If $\lambda_0$ is fixed, note that we get in that way an anti-symmetric and  transitive relation on every set of pairwise disjoint lines that are disjoint from $\lambda_0$.

\medskip

A proper path $\gamma$ of $\R^2$  induces a {\it dual function} $\delta$ on its complement, defined up to an additive constant as follows: for every $z$ and $z'$ in $\R^2\setminus\gamma$, the difference $\delta(z')-\delta(z)$ is the algebraic intersection number $\gamma\wedge \gamma'$ where $\gamma'$ is any path from $z$ to $z'$. If $\gamma$ is a line, there is a unique dual function $\delta_{\gamma}$ that is equal to $0$ on $R(\gamma)$ and to $1$ on $L(\gamma)$.

\medskip

Consider a unit vector $\rho \in\R^2,\, \Vert \rho\Vert=1$. 
Say that a proper path $\gamma:\R\to\R^2$ is {\it directed by} $\rho$ if
$$\lim_{t\to\pm\infty}\Vert\gamma(t)\Vert=+\infty, \enskip \lim_{t\to+\infty}\gamma(t)/\Vert\gamma(t)\Vert=\rho, \enskip \lim_{t\to-\infty}\gamma(t)/\Vert\gamma(t)\Vert=-\rho.$$
Observe that if $\gamma$ is directed by $\rho$, then  $\gamma^{-1}$ is directed by  $-\rho$ and that for every $z\in\R^2$, the translated path $\gamma+z: t\mapsto \gamma(t)+z$ is directed by $\rho$. Among the connected components of $\R^2\setminus\gamma$, two of them $R(\gamma)$ and $L(\gamma)$ are uniquely determined by the following: for every $z\in\R^2$,  one has $z-s\rho^{\perp}\in R(\gamma)$ and $z+s\rho^{\perp}\in L(\gamma)$ if $s$ is large enough. In the case where $\gamma$ is a line, the definitions agree with the former ones. Note that two disjoint lines directed by  $\rho$ are comparable.

\medskip

Instead of looking at paths defined on a real interval we can look at paths defined on an abstract interval $J$, which means a one dimensional oriented manifold homeomorphic to a real interval. If $\gamma: J\to M$ and $\gamma': J'\to M$ are two paths, if $J$ has a right end $b$ and $J'$ a left end $a'$  (in the natural sense),  and if $\gamma(b)=\gamma'(a')$, we can concatenate the two paths and define the path $\gamma\gamma'$ defined on the interval $J''=J\sqcup J'/b\sim a'$ coinciding with $\gamma$ on $J$ and $\gamma'$ on $J'$. One can define in a same way the concatenation $\prod_{l\in L} \gamma_{l}$ of paths indexed by a finite or infinite interval of $\Z$.

\medskip
A path $\gamma:\R\to M$ such that $\gamma(t+1)=\gamma(t)$ for every $t\in\R$ lifts a continuous map $\Gamma:\T^1\to M$. We will say that $\Gamma$ is a {\it loop} and $\gamma$ its {\it natural lift}. If $n\geq 1$, we denote $\Gamma^n$ the loop lifted by the path $t\mapsto\gamma(nt)$. Here again, if $M$ is oriented and $\Gamma$ homologous to zero, one can define a dual function $\delta$ defined up to an additive constant on $M\setminus\Gamma$ as follows: for every $z$ and $z'$ in $\R^2\setminus\Gamma$, the difference $\delta(z')-\delta(z)$ is the algebraic intersection number $\Gamma\wedge \gamma'$ where $\gamma'$ is any path from $z$ to $z'$.

\medskip
% A path $\gamma:\R\to M$ is called {\it recurrent} if, for every interval $I\subset \R$ there exists a disjoint interval $J$ and an increasing homeomorphisms $h:I\to J$ such that $\gamma(x)= \gamma(h(x))$.

\bigskip

\subsection{ Rotations vectors}

Let us recall the notion of rotation vector and rotation set for a homeomorphism of a closed manifold, introduced by Schwartzman \cite{Sc} (see also Pollicott \cite{P}). Let $M$ be an oriented closed connected manifold and $I$ an {\it identity isotopy} on $M$, which means an isotopy $(f_t)_{t\in[0,1]}$ such that $f_0$ is the identity. The {\it trajectory} of a point $z\in M$ is the path $I(z): z\mapsto f_t(z)$. If $\omega$ is a closed $1$-form on $M$, one can define the integral $\int_{I(z)}\omega$ on every trajectory $I(z)$. Write $f_1=f$ and denote $\mathcal{M}(f)$ the set of invariant Borel probability measures. For every $\mu\in \mathcal{M}(f)$, the integral $\int_M \left(\int_{I(z)}\omega \right)\, d\mu(z)$ vanishes when $\omega$ is exact. One deduces that $\omega\to \int_M\left( \int_{I(z)}\omega\right) \, d\mu(z)$ defines a natural linear form on the first cohomology group $H^1(M,\R)$, and by duality an element of the first homology group $H_1(M,\R)$, which is called the {\it rotation vector} of $\mu$ and denoted $\mathrm{rot}(\mu)$. The set $\mathcal{M}(f)$, endowed with the weak$^*$ topology, being convex and compact and the map $\mu\mapsto\mathrm{rot}(\mu)$ being affine, one deduces that the set $\mathrm{rot}(I)=\{\mathrm{rot}(\mu)\, ,\, \mu\in\mathcal{M}(f)\}$ is a convex compact subset of $H_1(M,\R)$. If $M$ is a surface of genus greater than $1$ and $I'$ is a different identity isotopy given by $(f'_t)_{t\in[0,1]}$ such that $f'_1=f$, then for all $z\in M$ the trajectories $I(z)$ and $I'(z)$ are homotopic with fixed endpoints. Therefore the  rotation vectors (and the rotation set) are independent  of the isotopy, depending only on $f$. If $M$ is a torus, it depends on a given lift of $f$. Let us clarify this case (see Misiurewicz-Zieman \cite{MZ}).
Let $f$ be a homeomorphism of $\T^2$ that is isotopic to the identity and $\widetilde f$ a lift of $f$ to the universal covering space $\R^2$. The map $\widetilde f-\mathrm{Id}$ is invariant by the integer translations $z\mapsto z+p$, $p\in\Z^2$, and lifts a continuous map $\varphi:\T^2\to\R$. The rotation vector of a Borel probability measure invariant by $f$ is the integral $\int_{\T^2} \varphi\, d\mu$. If $\mu$ is ergodic, then for $\mu$-almost every point $z$, the Birkhoff means converge to $\mathrm{rot}(\mu)$. If $\widetilde z\in\R^2$ is a lift of $z$, one has
$$\lim_{n\to +\infty} {\widetilde f^n(\widetilde z)- \widetilde z\over n}= \lim_{n\to +\infty}{1\over n}\sum_{k=0}^{n-1}\varphi(f^k(z))=\mathrm{rot}(\mu).$$
We will say that $z$ (or $\widetilde z$) has a rotation vector $\mathrm{rot}(\mu)$. The rotation set $\mathrm{rot}(\widetilde f)$ is a non empty compact convex subset of $\R^2$. It is easy to prove that every extremal point of $\mathrm{rot}(\widetilde f)$ is the rotation vector of an ergodic measure. Indeed the set of Borel probability measures of rotation vector $\rho\in \mathrm{rot}(f)$ is convex and compact, moreover its extremal points are extremal in  ${\mathcal M}(f)$ if $\rho$ is extremal in $\mathrm{rot}(f)$.  Observe also that for every $p\in\Z^2$ and every $q\in \Z$, the map $\widetilde f^q+p$ is a lift of $f^q$ and one has $\mathrm{rot}(\widetilde f^q+p)=q\mathrm{rot}(\widetilde f)+p$.

We will also be concerned with annulus homeomorphisms. Let $f$ be a homeomorphism of $\A=\T^1\times[0,1]$ that is isotopic to the identity and $\widetilde f$ a lift of $f$ to the universal covering space $\R\times[0,1]$. The map $\pi_1\circ f-\pi_1$ is invariant by the translation $T: z\mapsto z+(1,0)$ and lifts a continuous map $\varphi:\A\to\R$. The rotation number $\mathrm{rot}(\mu)$ of a Borel probability measure invariant by $f$ is the integral $\int_{\A} \varphi\, d\mu$. If $\mu$ is ergodic, then for $\mu$-almost every point $z$, the Birkhoff means converge to $\mathrm{rot}(\mu)$. If $\widetilde z\in\R\times[0,1]$ is a lift of $z$, one has
$$\lim_{n\to +\infty} {\pi_1\circ \widetilde f^n(\widetilde z)- \pi_1(\widetilde z)\over n}= \lim_{n\to +\infty}{1\over n}\sum_{k=0}^{n-1}\varphi(f^k(z))=\mathrm{rot}(\mu).$$
Here again we will say that $\widetilde z$ (or $z$) has a rotation number $\mathrm{rot}(\mu)$. The rotation set $\mathrm{rot}(\widetilde f)$ is a non empty compact real segment and every endpoint of $\mathrm{rot}(\widetilde f)$ is the rotation number of an ergodic measure. Here again, for every $p\in\Z$ and every $q\in \Z$, the map $\widetilde f^q\circ T^p$ is a lift of $f^q$ and one has $\mathrm{rot}(\widetilde f^q\circ T^p)=q\mathrm{rot}(\widetilde f)+p$.

Note that if $J$ is a real interval, one can also define the rotation number of an invariant probability measure of a homeomorphism of $\T^1\times J$ isotopic to the identity, for a given lift to $\R\times J$, provided the support of the measure is compact.

\section{Transverse paths  to surface foliations}

\subsection{General definitions}\label{subsection:Generaldefinitions}

\bigskip
Let us begin by introducing some notations that will be used throughout the whole text. A {\it singular oriented foliation} on an oriented surface $M$ is an oriented topological foliation $\mathcal F$ defined on an open set of $M$. We will call this set the {\it domain} of $\mathcal F$ and denote it $\mathrm{dom}(\mathcal F)$, its complement will be called the singular set (or set of singularities) and denoted $\mathrm{sing}(\mathcal F)$.  If the singular set is empty, we will say that $\mathcal F$  is {\it non singular}. A subset of $M$ is {\it saturated} if it is the union of singular points and leaves. A {\it trivialization neighborhood} is an open set $W\subset \mathrm{dom}(\mathcal F)$ endowed with a homeomorphism $h:W\to(0,1)^2$ that sends the restricted foliation $\mathcal F\vert_W$ onto the vertical foliation. If $\check M$ is a covering space of $M$ and $\check \pi: \check M\to M$ the covering projection,  $\mathcal F$ can be naturally lifted to a singular foliation $\check{\mathcal F}$ of $\check M$ such that $\mathrm{dom}(\check{\mathcal F})=\check\pi^{-1}(\mathrm{dom}(\mathcal F))$. If $\check N$ is a covering space of $\mathrm{dom}(\mathcal F)$, then the restriction of $\mathcal F$ to $\mathrm{dom}(\mathcal F)$ can also be naturally lifted to a non singular foliation of $\check N$. We will denote $\widetilde{\mathrm{dom}}(\mathcal F)$the universal covering space of $\mathrm{dom}(\mathcal F)$ and  $\widetilde{\mathcal F}$  the foliation lifted from $\mathcal{F}\mid_{\mathrm{dom}(\mathcal F)}$. For every $z\in\mathrm{dom}(\mathcal F)$ we will write $\phi_z$ for the leaf that contains $z$, $\phi^+_z$ for the positive half-leaf and $\phi_z^-$ for the negative one.  One can define the $\alpha$-limit and $\omega$-limit sets of $\phi$ as follows:
 $$\alpha(\phi)=\bigcap_{z\in\phi} \overline{\phi_z^-},\enskip \omega(\phi)=\bigcap_{z\in\phi} \overline{\phi_z^+}.$$
Suppose that a point $z\in\phi$ has a trivialization neighborhood $W$ such that each leaf of $\mathcal F$ contains no more than one leaf of $\mathcal F\vert_W$. In that case every point of $\phi$ satisfies the same property. If furthermore no closed leaf of $\mathcal F$ meets $W$, we will say that $\phi$ is {\it wandering}. Recall the following facts, in the case where $M=\R^2$ and $\mathcal F$ is non singular  (see Haefliger-Reeb \cite{HR}):

\smallskip
\noindent -\enskip  every leaf of $\mathcal F$ is a wandering line;

\smallskip
\noindent -\enskip  the space of leaves $\Sigma$,  furnished with the quotient topology, inherits a structure of  connected and simply connected one-dimensional manifold;

\smallskip
\noindent  -\enskip   $\Sigma$ is Hausdorff if and only if $\mathcal F$ is trivial (which means that it is the image of the vertical foliation by a plane homeomorphism) or equivalently if all the leaves are comparable.

\medskip
A path $\gamma:J\to M$ is {\it positively transverse}\footnote{in the whole text ``transverse'' will mean ``positively transverse''} to $\mathcal F$ if its image does not meet the singular set and if, for every $t_0\in J$, there exists a (continuous) chart $h:W\to (0,1)^2$ at $\gamma(t_0)$ compatible with the orientation and sending the restricted foliation ${\mathcal F}_W$ onto the vertical foliation oriented downward such that the map $\pi_1\circ h\circ \gamma$ is increasing in a neighborhood of $t_0$. Let $\check M$ be a covering space of $M$ and $\check \pi: \check M\to M$ the covering projection.  If  $\gamma:J\to \mathrm{dom}({\mathcal F})$ is  positively transverse to $\mathcal F$, every lift $\check\gamma:J\to \check M$ is transverse to the lifted foliation $\check{\mathcal F}$. Moreover, every lift $\widetilde\gamma:J\to \widetilde{\mathrm{dom}}({\mathcal F}) $ to the universal covering space $\widetilde{\mathrm{dom}}({\mathcal F})$ is transverse to the lifted non singular  foliation $\widetilde {\mathcal F}$.

Suppose first that $M=\R^2$ and that $\mathcal F$ is non singular. Say that two transverse paths $\gamma:J\to \R^2$ and $\gamma':J'\to \R^2$ are {\it equivalent for $\mathcal F$ or $\mathcal F$-equivalent} if they satisfy the three following equivalent conditions:

\smallskip
\noindent -\enskip   there exists an increasing homeomorphism $h:J\to J'$ such that $\phi_{\gamma'(h(t))}=\phi_{\gamma(t)}$, for every $t\in J$; 

\smallskip
\noindent -\enskip  the paths $\gamma$ and $\gamma'$ meet the same leaves;

\smallskip
\noindent -\enskip   the paths $\gamma$ and $\gamma'$ project onto the same path of $\Sigma$. 

\smallskip

Moreover, if $J=[a,b]$ and $J'=[a',b']$ are two segments, these conditions are equivalent to this last one:

\smallskip
\noindent  -\enskip   one has $\phi_{\gamma(a)}=\phi_{\gamma'(a')}$ and $\phi_{\gamma(b)}=\phi_{\gamma'(b')}$. 

\smallskip
 In that case, note that the leaves met by $\gamma$ are the leaves $\phi$ such that $R(\phi_{\gamma(a)})\subset R(\phi)\subset R(\phi_{\gamma(b)}).$ If the context is clear, we just say that the paths are {\it equivalent} and omit the dependence on $\mathcal F$.

If $\gamma:J\to \R^2$ is a transverse path, then for every $a<b$ in $J$, the set $L(\phi_{
 \gamma(a)})\cap R(\phi_{\gamma(b)})$ is a topological plane and $\gamma\vert_{(a,b)}$ a line of this plane. Let us say that $\gamma$ {\it has a leaf on its right} if there exists $a<b$ in $J$ and a leaf $\phi$ in $L(\phi_{
 \gamma(a)})\cap R(\phi_{\gamma(b)})$ that lies in the right of $\gamma\vert_{(a,b)}$. Similarly, one can define the notion of {\it having a leaf on its left }.

\begin{figure}[ht!]
\hfill
\includegraphics [height=48mm]{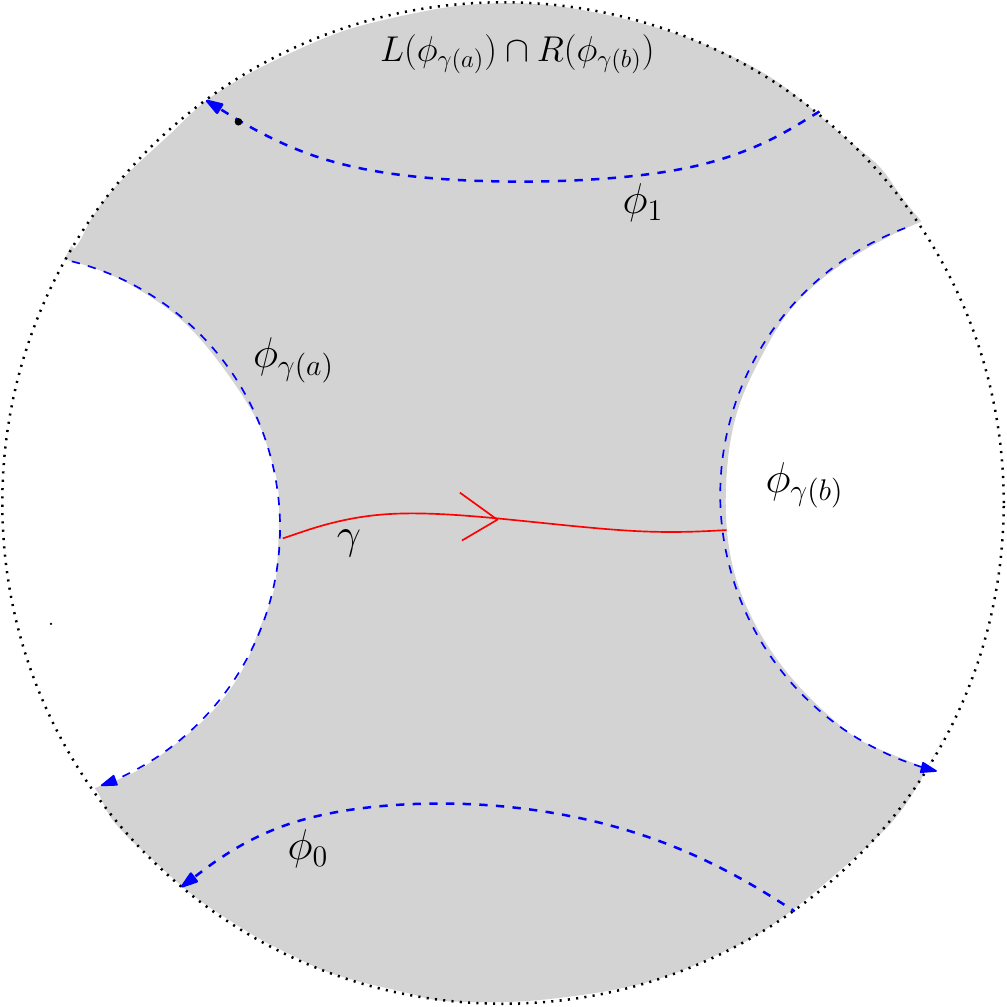}
\hfill{}
\caption{\small $\gamma:[a,b]\to\R^2$ has both a leaf on its right ($\phi_0$) and a leaf on its left ($\phi_1)$. $\gamma\vert_{(a,b)}$ is also a line in $L(\phi_{\gamma(a)})\cap R(\phi_{\gamma(b)})$. }
\label{Figure_Leaf_on_the_right}
\end{figure}

All previous definitions can be naturally extended in case every connected component of $M$ is a plane and $\mathcal F$ is not singular. Let us return to the general case. Two transverse paths $\gamma:J\to \mathrm{dom}(\mathcal F)$ and $\gamma':J'\to \mathrm{dom}(\mathcal F)$ are {\it equivalent for $\mathcal F$ or $\mathcal F$-equivalent} if they can be lifted to the universal covering space $ \widetilde{\mathrm{dom}}(\mathcal F)$ of $ \mathrm{dom}(\mathcal F)$ as paths that are equivalent for the lifted foliation $\widetilde{\mathcal F}$.  This implies that there exists an increasing homeomorphism $h:J\to J'$ such that, for every $t\in J$, one has $\phi_{\gamma'(h(t))}=\phi_{\gamma(t)}$. Nevertheless these two conditions are not equivalent. In Figure \ref{Figure_non_equivalent}, such a homeomorphism can be constructed but the two loops are not equivalent. Nonetheless, one can show that $\gamma$ and $\gamma'$ are equivalent for $\mathcal{F}$ if, and only if, there exists a {\it holonomic homotopy} between $\gamma$ and $\gamma'$, that is, if there exists a continuous transformation $H:J\times[0,1]\to \mathrm{dom}(\mathcal F)$ and an increasing homeomorphism $h:J\to J'$ satisfying:

\smallskip
\noindent -\enskip $H(t,0) = \gamma(t), \, H(t,1)=\gamma'(h(t))$;

\smallskip
\noindent -\enskip for all $t\in J$ and $s_1, s_2\in [0,1]$, $\phi_{H(t, s_1)}=\phi_{H(t, s_2)}$.  

\smallskip

\begin{figure}[ht!]
\hfill
\includegraphics [height=48mm]{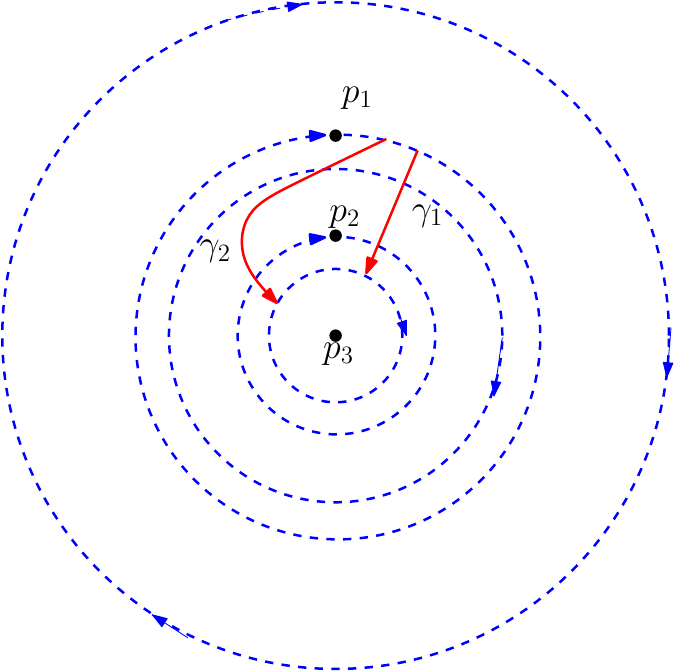}
\hfill{}
\caption{\small The paths $\gamma_1$ and $\gamma_2$ are not equivalent for $\mathcal{F}$, even though they cross the same leafs.}
\label{Figure_non_equivalent}
\end{figure}

 By definition, a transverse path  {\it has a leaf on its right} if it can be lifted to $ \widetilde{\mathrm{dom}}(\mathcal F)$  as a path with a leaf of $\widetilde{\mathcal F}$ on its right (in that case every lift has a leaf on its right)  and {\it has a leaf on its left}  if it can be lifted as a path with a leaf on its left.  Note that if $\gamma$ and $\gamma'$ have no leaf on their right and $\gamma\gamma'$ is well defined, then $\gamma\gamma'$ has no leaf on its right. Note also that if $\gamma$ and $\gamma'$ are $\mathcal{F}$-equivalent, and if $\gamma$ has a leaf on its right, then $\gamma'$ has a leaf on its right. We say that an $\mathcal{F}$-equivalence class has a leaf on its right (on its left) if some representative of the class has a leaf on its right (on its left).

Similarly, a loop $\Gamma:\T^1\to \mathrm{dom}(\mathcal F)$ is called {\it positively transverse} to $\mathcal F$ if it is the case for its natural lift $\gamma:\R\to \mathrm{dom}(\mathcal F)$.  It {\it has a leaf on its right} or {\it its left} if it is the case for $\gamma$. Two transverse loops $\Gamma:\T^1\to \mathrm{dom}(\mathcal F)$ and $\Gamma':\T^1\to \mathrm{dom}(\mathcal F)$ are {\it equivalent}  if there exists two lifts $\widetilde\gamma:\R\to \widetilde{\mathrm{dom}}(\mathcal F)$ and $\widetilde\gamma':\R\to \widetilde{\mathrm{dom}}(\mathcal F)$ of $\Gamma$ and $\Gamma'$ respectively, a covering automorphism $T$ and an orientation preserving homeomorphism $h:\R\to \R$, such that, for every $t\in \R$, one has 
$$ \widetilde\gamma(t+1)=T(\widetilde\gamma(t)), \enskip \widetilde\gamma'(t+1)=T(\widetilde\gamma'(t)),\enskip h(t+1)=h(t)+1, \enskip\phi_{\widetilde\gamma'(h(t))}=\phi_{\widetilde\gamma(t)}.$$ Of course $\Gamma^n$ and $\Gamma'^n$ are equivalent transverse loops, for every $n\geq 1$, if it is the case for $\Gamma$ and $\Gamma'$.  A transverse loop $\Gamma$ will be called {\it prime} if there is no transverse loop $\Gamma'$ and integer $n\geq 2$ such that $\Gamma$ is equivalent to $\Gamma'{}^n$.  

If two transverse loops $\Gamma$ and $\Gamma'$ are equivalent, there exists a holonomic homotopy between them and therefore they are freely homotopic in $\mathrm{dom}(\mathcal F)$, but the converse does not need to hold, as Figure \ref{Figure_Homotopic_notequivalent} shows.

\begin{figure}[ht!]
\hfill
\includegraphics [height=48mm]{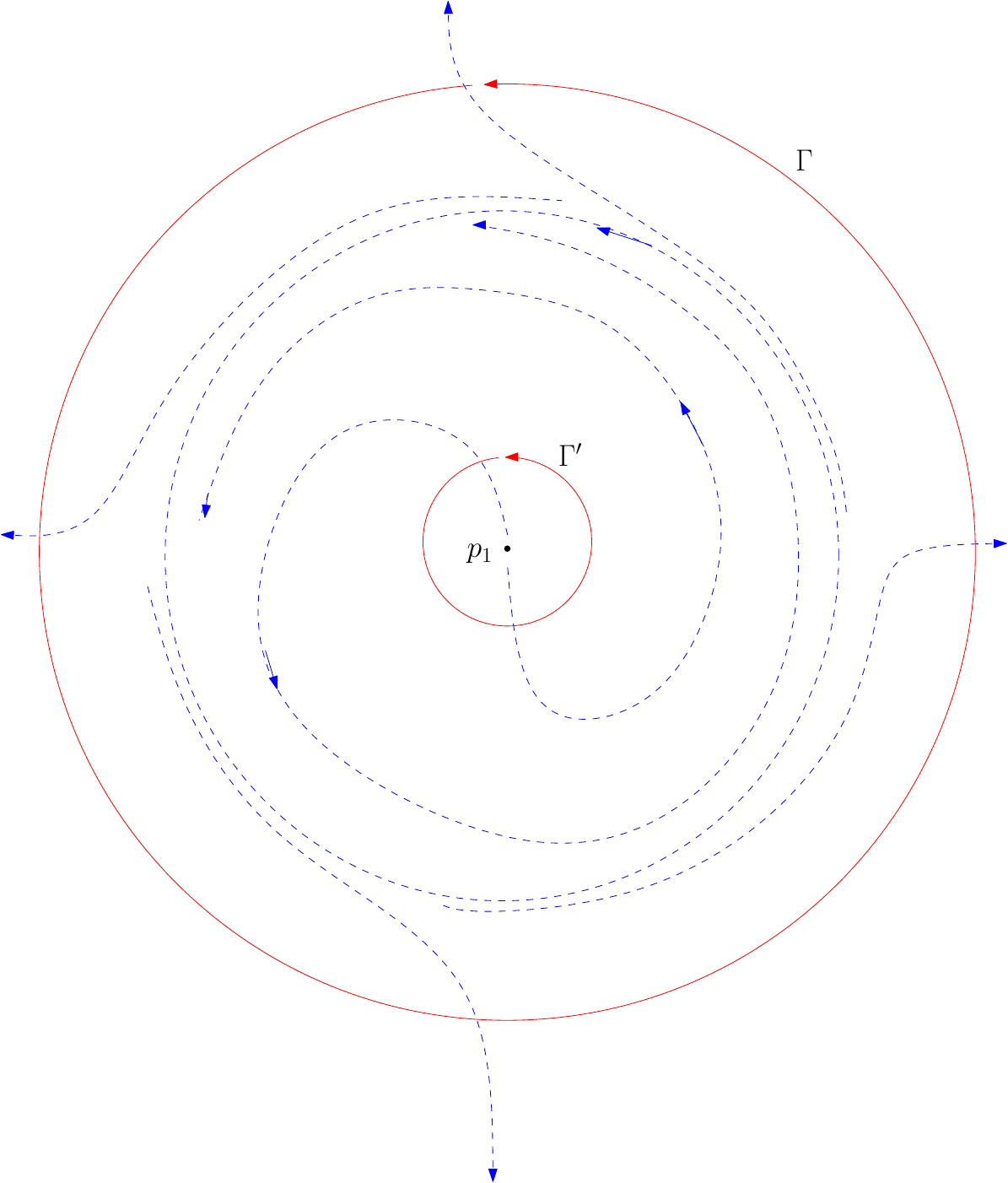}
\hfill{}
\caption{\small The transversal loops $\Gamma$ and $\Gamma'$ are not equivalent for $\mathcal{F}$, even though they are freely homotopic.}
\label{Figure_Homotopic_notequivalent}
\end{figure}

 A transverse path  $\gamma:\R\to M$ will be called $\mathcal F$-positively recurrent if for every segment $J\subset \R$ and every $t\in\R$ there exists a segment
$J'\subset [t,+\infty)$ such that $\gamma\vert_{J'}$ is equivalent to $\gamma\vert_J$.  It will be called $\mathcal F$-negatively recurrent if for every segment $J\subset \R$ and every $t\in\R$ there exists a segment $J'\subset (-\infty, t]$ such that $\gamma\vert_{J'}$ is equivalent to $\gamma\vert_J$. It is $\mathcal F$-bi-recurrent if it is both $\mathcal F$-positively and $\mathcal F$-negatively recurrent. Note  that, if  $\gamma:\R\to M$  and  $\gamma':\R\to M$ are $\mathcal F$-equivalent and if $\gamma$ is $\mathcal{F}$-positively recurrent (or $\mathcal F$-negatively recurrent), then so is $\gamma'$. We say that an $\mathcal{F}$-equivalence class is  positively recurrent (negatively recurrent, bi-recurrent) if some representative of the class is $\mathcal{F}$-positively recurrent (resp. $\mathcal{F}$-negatively recurrent, $\mathcal{F}$-bi-recurrent).

\medskip  We will very often use the following remarks. Suppose that $\Gamma$ is a transverse loop homologous to zero and $\delta$ a dual function. Then $\delta$ decreases along each leaf with  a jump at every intersection point. One deduces that every leaf met by $\Gamma$ is wandering. In particular,  $\Gamma$ does not meet any set $\alpha(\phi)$ or $\omega(\phi)$, which implies that for every leaf $\phi$, there exist $z_-$ and $z_+$ on $\phi$ such that $\Gamma$ does not meet neither $\phi_{z_-}^-$ nor $\phi_{z_+}^+$. Writing $n_+$  and $n_-$ for the value taken by $\delta$ on $\phi_{z_-}^-$  and  $\phi_{z_+}^+$ respectively, one deduces that $n_+-n_-$ is the number of times that $\Gamma$ intersect $\phi$. Note that $n_+-n_-$ is uniformly bounded. Indeed, the fact that every leaf that meets $\Gamma$ is wandering implies that $\T^1$ can be covered by open intervals where $\Gamma$ is injective and does not meet any leaf more than once.  By compactness, $\T^1$ can be covered by finitely many such intervals, which implies that there exists $N$ such that $\Gamma$ meets each leaf at most $N$ times. We have similar results for a multi-loop $\Gamma=\sum_{1\leq i\leq p}\Gamma_i$ homologous to zero.  In case where $M=\R^2$, we have similar results for a proper transverse path with finite valued dual function. In case of an infinite valued dual function, everything is true but the existence of $z_-$, $z_+$, $n_-$, $n_+$ and the finiteness condition about intersection with a given leaf. In particular a transverse line $\lambda$ meets every leaf at most once (because the dual function takes only two values) and one can define the sets $r(\lambda)$ and $l(\lambda)$, union of leaves included in $R(\lambda)$ and $L(\lambda)$ respectively. They do not depend on the choice of $\lambda$ in the equivalence class. Note that if the diameter of the leaves of $\mathcal F$ are uniformly bounded, every path equivalent to $\lambda$ is still a line. We have similar results for directed proper paths. If $\gamma$ is a proper path directed by a unit vector $\rho$, one can define the sets $r(\gamma)$ and $l(\gamma)$, union of leaves included in $R(\gamma)$ and $L(\gamma)$ respectively. They do not depend on the choice of $\gamma$ in the equivalence class. Moreover, if the leaves of $\mathcal F$ are uniformly bounded, every path equivalent to $\gamma$ is still a path directed by $\rho$. 
 
\bigskip

\subsection{$\mathcal{F}$-transverse intersection for non singular plane foliations}

\bigskip
We suppose here that  $M=\R^2$ and that $\mathcal F$ is non singular. 

\medskip Let $\gamma_1:J_1\to \R^2$ and $\gamma_2:J_2\to \R^2$ be two transverse paths. The set
$$X=\{(t_1,t_2)\in J_1\times J_2\enskip\vert \phi_{\gamma_1(t_1)}=\phi_{\gamma_2(t_2)}\},$$ if not empty,  is an interval that projects injectively on $J_1$ and $J_2$ as does its closure. Moreover, for every $(t_1,t_2) \in \overline X\setminus X$, the leaves $\phi_{\gamma_1(t_1)}$ 	and $\phi_{\gamma_2(t_2)}$ are not separated in $\Sigma$. To be more precise, suppose that $J_1$ and $J_2$ are real intervals  and that $\phi_{\gamma_1(t_1)}=\phi_{\gamma_2(t_2)}$. Set $J_1^-=J_1\cap (-\infty, t_1]$ and $J_2^-=J_2\cap (-\infty, t_2]$.  Then either one of the paths $\gamma_1\vert _{J_1^-}$, $\gamma_2\vert _{J_2^-}$ is equivalent to a subpath of the other one, or there exist $a_1<t_1$ and $a_2<t_2$ such that: 

\smallskip
\noindent -\enskip  $\gamma_1\vert _{(a_1,t_1]}$ and $\gamma_2\vert _{(a_2,t_2]}$ are equivalent;

\smallskip
\noindent -\enskip $\phi_{\gamma_1(a_1)}\subset L(\phi_{\gamma_2(a_2)}),\enskip \phi_{\gamma_2(a_2)}\subset L(\phi_{\gamma_1(a_1)})$

\smallskip
\noindent -\enskip $\phi_{\gamma_1(a_1)}$ and $\phi_{\gamma_2(a_2)}$ are not separated in $\Sigma$.

\medskip
Observe that the second property (but not the two other ones) is still satisfied when $a_1$, $a_2$ are replaced by smaller parameters. Note also that $\phi_{\gamma_2(a_2)}$ is either above or below $\phi_{\gamma_1(a_1)}$ relative to $\phi_{\gamma_1(t_1)}$ and that this property remains satisfied when $a_1$, $a_2$ are replaced by smaller parameters  and $t_1$ by any parameter in $(a_1,t_1]$. We have a similar situation on the \textcolor{black}{possible} right end of $X$.

\medskip

Let $\gamma_1:J_1\to \R^2$ and $\gamma_2:J_2\to \R^2$ be two transverse paths such that $\phi_{\gamma_1(t_1)}=\phi_{\gamma_2(t_2)}=\phi$. We will say that $\gamma_1$ and $\gamma_2$ {\it intersect $\mathcal{F}$-transversally and positively at $\phi$} (and $\gamma_2$ and $\gamma_1$ {\it intersect $\mathcal{F}$-transversally and negatively} at $\phi$) if there exist $a_1$, $b_1$ in $J_1$ satisfying $a_1<t_1<b_1$, and $a_2$, $b_2$ in $J_2$ satisfying $a_2<t_2<b_2$, such that:

\smallskip
\noindent -\enskip $\phi_{\gamma_2(a_2)}$ is below $\phi_{\gamma_1(a_1)}$  relative to $\phi$;

\smallskip
\noindent -\enskip $\phi_{\gamma_2(b_2)}$ is above $\phi_{\gamma_1(b_1)}$  relative to $\phi$.

See Figure \ref{figure1}.

\medskip

Note that, if $\gamma_1$ intersects $\mathcal{F}$-transversally $\gamma_2$,  if $\gamma_1'$ is equivalent to $\gamma_1$ and $\gamma_2'$ is equivalent to $\gamma_2$, then $\gamma_1'$ intersects $\mathcal{F}$-transversally $\gamma_2'$, and we say that the equivalence class of $\gamma_1$ intersect transversally the equivalence class of $\gamma_2$.

\smallskip
As none of the leaves $\phi$, $\phi_{\gamma_1(a_1)}$, $\phi_{\gamma_2(a_2)}$ separates the two others, one deduces that
$$\phi_{\gamma_1(a_1)}\subset L(\phi_{\gamma_2(a_2)}),\enskip \phi_{\gamma_2(a_2)}\subset L(\phi_{\gamma_1(a_1)})$$ and similarly that
$$\phi_{\gamma_1(b_1)}\subset R(\phi_{\gamma_2(b_2)}),\enskip \phi_{\gamma_2(b_2)}\subset R(\phi_{\gamma_1(b_1)}).$$

\smallskip

\begin{figure}[ht!]
\hfill
\includegraphics [height=48mm]{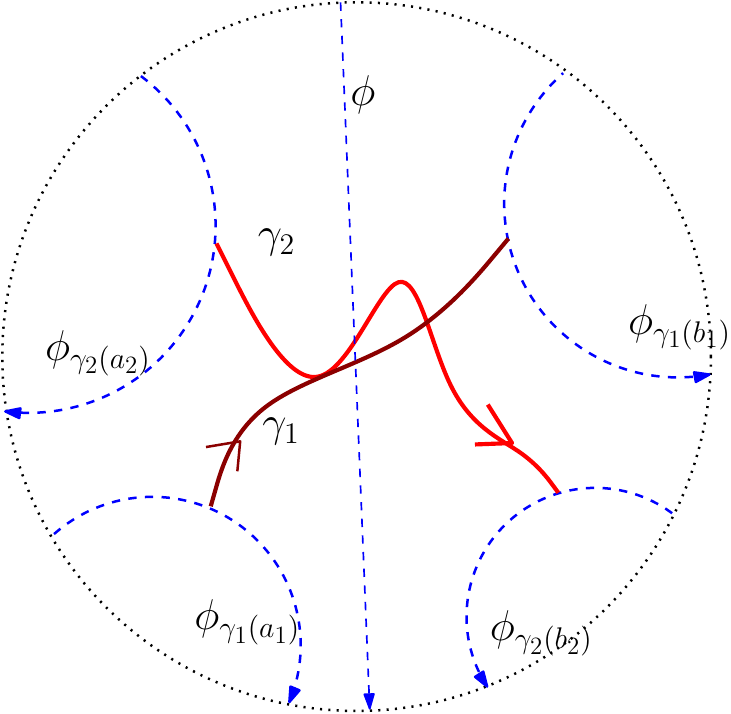}
\hfill{}
\caption{\small $\mathcal{F}$-transverse intersection. The tangency point is also a point of $\mathcal{F}$-transverse intersection.}
\label{figure1}
\end{figure}

As explained above, these properties remain true when $a_1$, $a_2$ are replaced by smaller parameters, $b_1$, $b_2$ by larger parameters and $\phi$ by any other leaf met by $\gamma_1$ and $\gamma_2$. Note that $\gamma_1$ and $\gamma_2$ have at least one intersection point and that one  can find two transverse paths $\gamma'_1$, $\gamma'_2$ equivalent to $\gamma_1$, $\gamma_2$ respectively, such that $\gamma'_1$ and $\gamma'_2$ have a unique intersection point, located on $\phi$, with a topologically transverse intersection. 
Note that, if $\gamma_1$ and $\gamma_2$ are two paths that meet the same leaf $\phi$, then either they intersect $\mathcal{F}$-transversally, or one can find two transverse paths $\gamma'_1$, $\gamma'_2$ equivalent to $\gamma_1$, $\gamma_2$, respectively, with no intersection point.  \bigskip

\subsection{ $\mathcal{F}$-transverse intersection in the general case}

Here again, the notion of $\mathcal{F}$-transverse intersection can be naturally extended in case every connected component of $M$ is a plane and $\mathcal F$ is not singular. Let us return now to the general case of a singular foliation $\mathcal F$ on a surface $M$.  Let $\gamma_1:J_1\to M$ and $\gamma_2:J_2\to M$ be two transverse paths that meet a common leaf $\phi=\phi_{\gamma_1(t_1)}=\phi_{\gamma_2(t_2)}$. We will say that $\gamma_1$ and $\gamma_2$ {\it intersect $\mathcal{F}$-transversally at} $\phi$ if there exist paths $\widetilde \gamma_1:J_1\to \widetilde{\mathrm{dom}}(\mathcal F)$ and $\widetilde \gamma_2:J_2\to  \widetilde{\mathrm{dom}}(\mathcal F)$, lifting $\gamma_1$ and $\gamma_2$, with a common leaf $\widetilde\phi=\phi_{\widetilde \gamma_1(t_1)}=\phi_{\widetilde \gamma_2(t_2)}$ that lifts $\phi$, and intersecting $\widetilde{\mathcal{F}}$-transversally at $\widetilde \phi$. If $\phi$ is closed the choices of $\widetilde \gamma_1$ and $\widetilde \gamma_2$ do not need to be unique, see Figure \ref{Figure_closed_leaf}. 

\begin{figure}[ht!]
\hfill
\includegraphics [height=48mm]{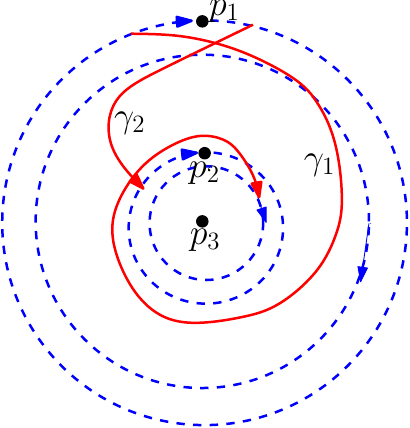}
\hfill{}
\caption{\small Given a lift $\widetilde \gamma_1$ of $\gamma_1$, there are two different lifts of $\gamma_2$  intersecting $\widetilde{\mathcal{F}}$-transversally $\widetilde \gamma_1$.}
\label{Figure_closed_leaf}
\end{figure}

Here again, we can give a sign to the intersection. As explained in the last subsection, there exist $t'_1$ and $t'_2$ such that $\gamma_1(t'_1)=
\gamma_2(t'_2)$ and such that  $\gamma_1$ and $\gamma_2$ intersect $\mathcal{F}$-transversally at $\phi_{\gamma_1(t'_1)}=\phi_{\gamma_2(t'_2)}$. In this case  we will say that  
$\gamma_1$ and $\gamma_2$ intersect $\mathcal{F}$-transversally at $\gamma_1(t'_1)=\gamma_2(t'_2)$.
In the case where $\gamma_1=\gamma_2$ we will talk of an $\mathcal{F}$-transverse self-intersection. A transverse path $\gamma$ has an $\mathcal{F}$-transverse self-intersection if for every lift $\widetilde\gamma$ to the universal covering space of the domain, there exists a non trivial covering automorphism $T$ such that  $\widetilde\gamma$ and $T(\widetilde\gamma)$ have a $\widetilde{\mathcal{F}}$-transverse intersection.  We will often use the following fact. Let $\gamma_1:J_1\to M$ and $\gamma_2:J_2\to M$ be two transverse paths that meet a common leaf $\phi=\phi_{\gamma_1(t_1)}=\phi_{\gamma_2(t_2)}$. If $J'_1$, $J'_2$ are two sub-intervals of $J_1$, $J_2$ that contain $ t_1$, $t_2$ respectively and if $\gamma_1\vert_{J'_1}$ and $\gamma_2\vert_{J'_2}$ intersect $\mathcal{F}$-transversally at $\phi$, then $\gamma_1$ and $\gamma_2$ intersect $\mathcal{F}$-transversally at $\phi$.

\medskip
Similarly, let $\Gamma$ be a loop positively transverse to $\mathcal F$ and $\gamma$ its natural lift. If $\gamma$ intersects $\mathcal{F}$-transversally a transverse path $\gamma'$ at a leaf $\phi$, we will say that $\Gamma$ and $\gamma'$ intersect $\mathcal{F}$-transversally at $\phi$. Moreover if $\gamma'$ is the natural lift of a transverse loop $\Gamma'$ we will say that $\Gamma$ and $\Gamma'$ intersect $\mathcal{F}$-transversally at $\phi$. Here again we can talk of self-intersection.

\medskip
As a conclusion, note that if two transverse paths have an $\mathcal{F}$-transverse intersection, they both have a leaf on their right and a leaf on their left.

\bigskip

\subsection{ Some useful results}

\
In this section, we will state different results that will be useful in the rest of the article. Observe that the finiteness condition for the next proposition is satisfied if every leaf of $\mathcal F$ is wandering, or when $M$ has genus $0$.

\begin{proposition}\label{pr:transverse_on_surface} Let $\mathcal F$ be an oriented singular foliation on a surface and $(\Gamma_i)_{1\leq i\leq m}$ a family of prime transverse loops that are not pairwise equivalent. We suppose that the leaves met by the loops $\Gamma_i$ are never closed and that there exists an integer $N$ such that no loop $\Gamma_i$ meets a leaf more than $N$ times. Then, for every $i\in\{1,\dots, m\}$, there exists a transverse loop $\Gamma'_i$ equivalent to $\Gamma_i$ such that:

\smallskip
\noindent{\bf i)}\enskip \enskip $\Gamma'_i$ and $\Gamma'_j$ do not intersect if
$\Gamma_i$ and $\Gamma_j$ have no $\mathcal{F}$-transverse intersection;

\smallskip
\noindent{\bf ii)}\enskip \enskip   $\Gamma'_i$ is simple if $\Gamma_i$ has no $\mathcal{F}$-transverse self-intersection.
\end{proposition}

\begin{proof}

There is  a natural partial order on  $ \mathrm{dom}({\mathcal F})$ defined as follows: write $z\leq z'$ if $\phi_{z}$ is not closed and $z'\in\phi_{z}^+$. One can suppose, without loss of generality, that the loops $\Gamma_i$ are included in the same connected component $W$ of  $ \mathrm{dom}({\mathcal F})$. One can lift ${\mathcal F}\vert_{W}$ to an oriented foliation $\widetilde {\mathcal F}$ on the universal covering space $\widetilde W$ of $W$. We will parameterize $\Gamma_i$ by a copy $\T^1_i$ of $\T^1$ and consider the $\T^1_i$ as disjoint circles. We will  endow the set $\T_*=\sqcup_{1\leq i\leq m} \T^1_i$ with the natural topology generated by the open sets of the  $\T^1_i$. We get a continuous map $\Gamma:\T_*\to W$ (a multi-loop) by setting $\Gamma(t)=\Gamma_{i(t)}(t)$, where $t\in \T^1_{i(t)}$. Suppose that $t\not=t'$  and $\phi_{\Gamma(t)}=\phi_{\Gamma(t')}$. One can lift the loops $\Gamma_{i(t)}$ and $\Gamma_{i(t')}$ to lines $\widetilde \gamma_{i(t)}:\R\to \widetilde W$ and $\widetilde \gamma_{i(t')}:\R\to \widetilde W$ transverse to $ \widetilde {\mathcal  F}$ such that $\widetilde\phi_{\widetilde\gamma_{i(t)}(\widetilde t)}=\widetilde\phi_{\widetilde\gamma_{i(t')}(\widetilde t')}=\widetilde\phi$, where $\widetilde t$ and $\widetilde t'$ lift $t$ and $t'$ respectively. The fact that the loops are prime and   not equivalent implies that $\widetilde\gamma_{i(t)}\vert_{[\widetilde t,+\infty)}$ and $\widetilde\gamma_{i(t')}\vert_{[\widetilde t',+\infty)}$ are not equivalent and similarly that $\widetilde\gamma_{i(t)}\vert_{(-\infty, \widetilde t]}$ and $\widetilde\gamma_{i(t')}\vert_{(-\infty, \widetilde t']}$ are not equivalent.  So, $\widetilde\phi_{\widetilde\gamma_{i(t')}}(\widetilde t'')$ is above or below $\widetilde\phi_{\widetilde\gamma_{i(t)}}(\widetilde t'')$ relative to $\widetilde\phi$ if $\abs{\widetilde t''}$ is sufficiently large.  Moreover the option does not depend on the choice of the lifts. We will write $t\prec t'$ in the case where  $\widetilde\phi_{\widetilde\gamma_{i(t')}}(\widetilde t'')$ is above $\widetilde\phi_{\widetilde\gamma_{i(t)}}(\widetilde t'')$ and $\widetilde\phi_{\widetilde\gamma_{i(t')}}(-\widetilde t'')$ is above $\widetilde\phi_{\widetilde\gamma_{i(t)}}(-\widetilde t'')$ for $\widetilde t''$ sufficiently large. Observe that one has $t\prec t'$ or $t'\prec t$ in the two following cases:
 
\smallskip
\noindent-\enskip\enskip   $i(t)\not=i(t')$ and $\Gamma_{i(t)}$ and $\Gamma_{i(t')}$ have no $\mathcal{F}$-transverse intersection;

\smallskip
\noindent-\enskip\enskip   $i(t)=i(t')$ and $\Gamma_{i(t)}$ has no $\mathcal{F}$-transverse self-intersection.

\medskip

We will say that $t\in \T_*$ is a {\it good parameter} of $\Gamma$, if for every $t'\in \T_*$, one has
$$t\prec t'\Rightarrow \Gamma(t)<\Gamma(t').$$To get the proposition it is sufficient to construct, for every $i\in\{1,\dots, m\}$, a transverse loop $\Gamma'_i$  equivalent to $\Gamma_i$ such that the induced multi-loop $\Gamma'$ has only good parameters. Let us define the order $o(t)$ of $t\in \T_*$ to be the number of $t'\in \T_*$ such that $t\prec t'$. Note that every parameter of order $0$ is a good parameter. We will construct $\Gamma'$ by induction, supposing that every parameter of order $\leq r$ is good and constructing $\Gamma'$ such that every parameter of order $\leq r+1$ is good.  Note that for every $s$, the set $\T_{\leq s}$ of parameters of order $\leq s$ is closed and the set $\T_{\mathrm {good}}$ of good parameters is open.  The set $\T_{\mathrm {bad}} =\T_{\leq r+1}\setminus \T_{\mathrm {good}} $ is closed and disjoint from $\T_{\leq r}$: it contains only parameters of order $r+1$. Let us fix an open neighborhood $O$ of $\T_{\mathrm {bad}} $ disjoint from $\T_{\leq r}$. By hypothesis, for every $t\in \T_{\mathrm {bad}} $, one can find \textcolor{black}{$r+1$ points $\theta_0(t)$,  \dots, $\theta_{r}(t)$} in $\T_{*} $ such that
$t\prec \theta_{i}(t)$ for every \textcolor{black}{$i\in\{0,\dots, r\}$} and among the $\Gamma(\theta_i(t))$ a smallest one $\Gamma(\theta(t))$  (for the order $\leq$). Each $\theta_i(t)$ belongs to $\T_{\leq r}$ and therefore is disjoint from $O$. Note that each function $\theta_i$ can be chosen continuous in a neighborhood of a point $t$, which implies that $t\mapsto \Gamma(\theta(t))$ is continuous on $\T_{\mathrm {bad}} $. It is possible to make a perturbation of $\Gamma$ supported on $O$ by sliding continuously each point $\Gamma(t)$ on $\phi_{\Gamma(t)}^-$ to obtain a transverse multi-loop $\Gamma'$  such that $\Gamma'(t)<\Gamma(\theta(t))$. Since the perturbation is a holonomic homotopy, $\Gamma'$ must be equivalent to $\Gamma$.

Since $\theta_{i}(t) \in \T_{\leq r}$ for every \textcolor{black}{$i\in\{0,\dots, r\}$ }, we have $\Gamma(\theta_i (t))=\Gamma'( \theta_{i}(t))$ and so $\Gamma'(t)< \Gamma'( \theta(t))$.
\end{proof}

Let us continue with the following adapted version of Poincar\'e-Bendixson  Theorem.

\begin{proposition}\label{pr:transverse_on_sphere}
 Let $\mathcal F$ be an oriented singular foliation on $\S^2$ and $\gamma:\R\to \S^2$  an $\mathcal{F}$-bi-recurrent transverse path. The following properties are equivalent:

\smallskip
\noindent{\bf i)}\enskip \enskip $\gamma$ has no $\mathcal{F}$-transverse self-intersection;

\smallskip
\noindent{\bf ii)}\enskip \enskip  there exists a transverse simple loop $\Gamma'$ such that $\gamma$ is equivalent to the natural lift $\gamma'$ of $\Gamma'$;

\smallskip
\noindent{\bf iii)} \enskip \enskip the set $U=\bigcup_{t\in\R}\phi_{\gamma(t)}$ is an open annulus.
\end{proposition}

\begin{proof}\enskip  
To prove that {\bf ii)} implies {\bf iii)}, just note that a dual function of $\Gamma'$ takes only two consecutive values, which implies that every leaf of $\mathcal F$ meets $\Gamma'$ at most once.

To prove that {\bf iii)} implies {\bf i)} it is sufficient to note that if $\bigcup_{t\in\R}\phi_{\gamma(t)}$ is an annulus, each connected component of its preimage in the universal covering space of  $ \mathrm{dom}({\mathcal F})$ is an open set, union of leaves, where the lifted foliation $\widetilde{\mathcal F}$ is trivial. This implies that $\gamma$ has no $\mathcal{F}$-transverse self-intersection.

It remains to prove that {\bf i)} implies {\bf ii)}. The path $\gamma$ being $\mathcal{F}$-bi-recurrent,  one can find  $a<b$ such that $\phi_{\gamma(a)}=\phi_{\gamma(b)}$. Replacing $\gamma$ by an equivalent transverse path, one can suppose that $\gamma(a)=\gamma(b)$. Let $\Gamma$ be the loop naturally defined by the closed path $\gamma\vert_{[a,b]}$. As explained previously, every leaf that meets $\Gamma$ is wandering and consequently, if $t$ and $t'$ are sufficiently close, one has $\phi_{\Gamma(t)}\not=\phi_{\Gamma(t')}$. Moreover, because $\Gamma$ is positively transverse to $\mathcal F$, one cannot find an increasing sequence $(a_n)_{n\geq 0}$ and a decreasing sequence $(b_n)_{n\geq 0}$, such that $\phi_{\gamma(a_n)}=\phi_{\gamma(b_n)}$. So, there exist $a\leq a'<b'\leq b$ such that $t\mapsto \phi_{\gamma(t)}$ is injective on $[a',b')$ and satisfies $\phi_{\gamma(a')}=\phi_{\gamma(b')}$. Replacing $\gamma$ by an equivalent transverse path, one can suppose that $\gamma(a')=\gamma(b')$. The set $U=\bigcup_{t\in[a',b']}\phi_{\gamma(t)}$ is an open annulus and the loop $\Gamma'$ naturally defined by the closed path $\gamma\vert_{[a',b']}$  is a simple loop. 

Let us prove now that $\gamma$ is equivalent to the natural lift $\gamma'$ of $\Gamma'$. Being $\mathcal{F}$-bi-recurrent it cannot be equivalent to a strict subpath of $\gamma'$. So it is sufficient to prove that it is included in $U$. We will give a proof by contradiction. We denote the two connected components of the complement of $U$ as $X_1,\, X_2$. Suppose that there exists $t\in\R$ such that $\gamma(t)\not\in U$. The path $\gamma$ being $\mathcal{F}$-bi-recurrent and the sets $X_i$ saturated,  there exists $t'\in\R$ separated from $t$ by $[a',b']$ such that $\gamma(t')$ is in the same component $X_i$ than $\gamma(t)$. More precisely, one can find real numbers
$$t_1<a''\leq a'<b'\leq b''<t_2$$ and an integer $k\geq 1$, uniquely determined such that

\smallskip
\noindent-\enskip $\gamma\vert_{[a'',b'']}$ is equivalent to $\gamma\vert_{[a',b']}^k$;

\smallskip
\noindent-\enskip $\gamma\vert_{(t_1,a'')}$ and $\gamma\vert_{(b'',t_2)}$ are included in $U$ but do not meet $\phi_{\gamma(a')}$;

\smallskip

\noindent-\enskip $\gamma(t_1)$ and $\gamma(t_2)$ do not belong to $U$.

Moreover, if $\gamma(t_2)$ does not belong to the same component $X_i$ than $\gamma(t_1)$, one can find real numbers $t_2\leq t_3<t_4$ uniquely determined such that

\smallskip
\noindent- \enskip$\gamma(t_4)$ belongs to the same component $X_i$ than $\gamma(t_1)$;

\smallskip
\noindent-\enskip$\gamma\vert_{[t_2, t_4)}$ does not meet this component, 

\smallskip
\noindent-\enskip $\gamma\vert_{(t_3, t_4)}$ is included in $U$;

\smallskip
\noindent-\enskip $\gamma(t_3)$ does not belong to $U$.

\medskip
Observe now that if $\gamma(t_1)$ and $\gamma(t_2)$ belong to the same component $X_i$, then $\gamma_{[t_1, b'']}$ and $\gamma_{[a'', t_2]}$ intersect $\mathcal{F}$-transversally at $\phi_{\gamma(a'')}=\phi_{\gamma(b'')}$. Suppose now that $\gamma(t_1)$ and $\gamma(t_2)$ do not belong to the same component $X_i$. Fix $t\in(t_3,t_4)$. There exists $t'\in[a',b']$ such that $\phi_{\gamma(t')}=\phi_{\gamma(t)}$. Observe that  $\gamma\vert_{[t_1, t_2]}$ and $\gamma\vert_{[t_3, t_4]}$ intersect $\mathcal{F}$-transversally at $\phi_{\gamma(t')}=\phi_{\gamma(t)}$. 
\end{proof}

\begin{remark} Note that the proof above tells us that if $\gamma$ is $\mathcal{F}$-positively or $\mathcal{F}$-negatively recurrent, there exists a transverse simple loop $\Gamma'$ such that $\gamma$ is equivalent to a subpath of the natural lift $\gamma'$ of $\Gamma'$. 
\end{remark}

\begin{figure}[ht!]
\hfill
\includegraphics[height=48mm]{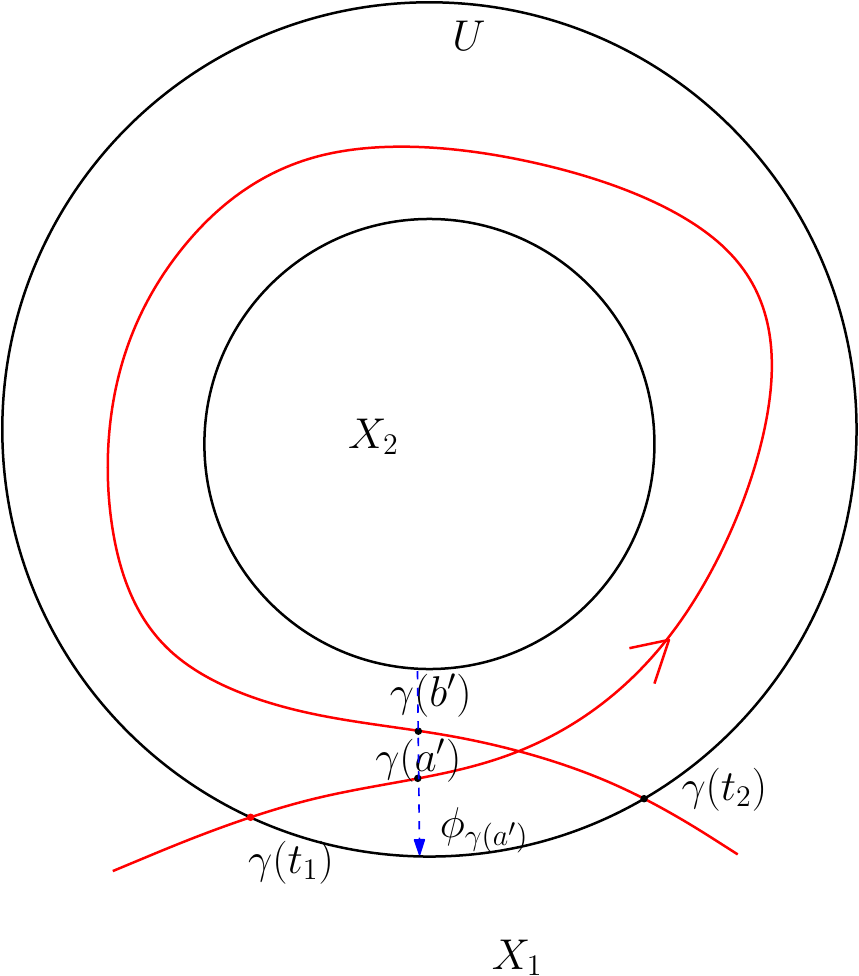}
\hfill
\includegraphics[height=48mm]{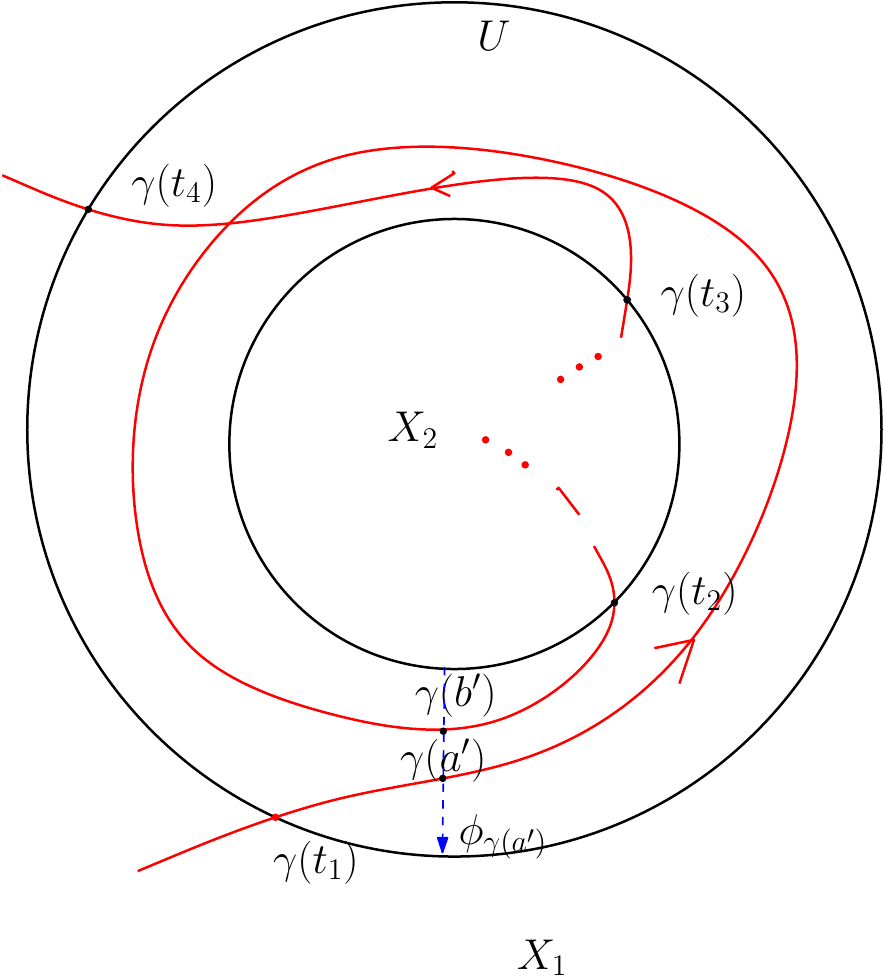}
%[height=28mm]{Figura1.pdf}
\hfill{}
\caption{\small Proof of Proposition \ref{pr:transverse_on_sphere}.}
\label{figuretransverseonsphere}
\end{figure}

The next result is a slight modification.

\begin{proposition} \label{pr:transverse_on_plane}
 Let $\mathcal F$ be an oriented singular foliation on $\R^2$ with leaves of uniformly bounded diameter and $\gamma$ be a transverse proper path. The following properties are equivalent:

\smallskip
\noindent{\bf i)}\enskip \enskip $\gamma$ has no $\mathcal{F}$-transverse self-intersection;

\smallskip
\noindent{\bf ii)}\enskip \enskip  $\gamma$ meets every leaf at most once;

\smallskip
\noindent{\bf iii)}\enskip \enskip  $\gamma$ is a line.
\end{proposition}

\begin{proof}
\enskip The fact that {\bf ii)} implies {\bf iii)} is obvious, as is the fact that {\bf iii)} implies {\bf i)}. It remains to prove that {\bf i)} implies {\bf ii)}. Let us suppose that  $\phi_{\gamma(a)}=\phi_{\gamma(b)}$, where $a<b$. We will prove that $\gamma$ has a transverse self-intersection. Like in the proof of the previous proposition, replacing $\gamma$ by an equivalent transverse path, one can find $a\leq a'<b'\leq b$ such that $\gamma(a')=\gamma(b')$, such that $U=\bigcup_{t\in[a',b']}\phi_{\gamma(t)}$ is an open annulus and such that the loop $\Gamma'$ naturally defined by the closed path $\gamma\vert_{[a',b']}$  is a simple loop. Write $X_1$ for the unbounded connected component of $\R^2\setminus\Gamma'$ and $X_2$ for the bounded one. The path $\gamma$ being proper, one can find real numbers
$$t_1<a''\leq a'<b'\leq b''<t_2$$ and an integer $k\geq 1$, uniquely determined such that

\smallskip
\noindent-\enskip $\gamma\vert_{[a'',b'']}$ is equivalent to $\gamma\vert_{[a',b']}^k$;

\smallskip
\noindent-\enskip $\gamma\vert_{(t_1,a'')}$ and $\gamma\vert_{(b'',t_2)}$ are included in $U$ but do not meet $\phi_{\gamma(a')}$;

\smallskip

\noindent-\enskip $\gamma(t_1)$ and $\gamma(t_2)$ do not belong to $U$.

As seen in the proof of the previous proposition, if $\gamma(t_1)$ and $\gamma(t_2)$ belong to the same component $X_i$, then $\gamma_{[t_1, b'']}$ and $\gamma_{[a'', t_2]}$ intersect $\mathcal{F}$-transversally at $\phi_{\gamma(a'')}=\phi_{\gamma(b'')}$. If $\gamma(t_1)\in X_1$ and $\gamma(t_2)\in X_2$, using the fact that $\gamma$ is proper, one can find real numbers $t_2\leq t_3<t_4$ uniquely determined such that

\smallskip
\noindent- \enskip$\gamma(t_4)$ belongs to $X_1$;

\smallskip
\noindent-\enskip$\gamma\vert_{[t_2, t_4)}$ does not meet $X_1$, 

\smallskip
\noindent-\enskip $\gamma\vert_{(t_3, t_4)}$ is included in $U$;

\smallskip
\noindent-\enskip $\gamma(t_3)$ belongs to $X_2$.

\medskip
As seen in the proof of the previous proposition, $\gamma\vert_{[t_1, t_2]}$ and $\gamma\vert_{[t_3, t_4]}$ intersect $\mathcal{F}$-transversally. The case where $\gamma(t_1)\in X_2$ and $\gamma(t_2)\in X_1$  can be treated analogously.
\end{proof}

Let us add another result describing paths  with no $\mathcal{F}$-transverse self-intersection:

\begin{proposition}
Let $\mathcal F$ be an oriented singular foliation on $\R^2$,  $\gamma$ a transverse proper path and $\delta$ a dual function of $\gamma$. If $\gamma'$ is a transverse path that does not intersect $\mathcal{F}$-transversally $\gamma$, then $\delta$ takes a constant value on the union of the leaves met by $\gamma'$ but not by $\gamma$.
\end{proposition}

\begin{proof}
\enskip Let us suppose that $\gamma'$ meets two leaves $\phi_0$ and $\phi_1$, disjoint from $\gamma$ and such that $\delta$ does not take the same value on $\phi_0$ and on $\phi_1$.
One can suppose that $\gamma'$ joins $\phi_0$ to $\phi_1$. 
Let $W$ be the connected component of  $ \mathrm{dom}({\mathcal F})$ that contains $\gamma$. Write $\widetilde W$ for the universal covering space of $W$ and $\widetilde{\mathcal F}$ for the lifted foliation. Every lift of $\gamma$ is a line. Fix a lift $\widetilde \gamma'$, it joins a leaf $\widetilde \phi_0$ that lifts $\phi_0$ to a leaf $\widetilde \phi_1$ that lifts $\phi_1$. By hypothesis, there exists a lift $\widetilde\gamma$ of $\gamma$ such that the dual function $\delta_{\widetilde\gamma}$ do not take the same value on $\widetilde \phi_0$ and $\widetilde \phi_1$. One can suppose that $\widetilde \phi_0\subset r(\widetilde\gamma)$ and $\widetilde \phi_1\subset l(\widetilde\gamma)$ for instance  (recall that $r(\widetilde \gamma)$ is the union of leaves included in the connected component of $\widetilde W\setminus\widetilde\gamma$ on the right of $\widetilde \gamma$ and $l(\widetilde \gamma)$ the union of leaves included in the other component). The foliation $\widetilde{\mathcal F}$ being non singular, the sets $r(\widetilde\gamma)$ and $l(\widetilde\gamma)$ are closed. Consequently, there exists a subpath $\widetilde\gamma''$ of $\widetilde\gamma'$ that joins a leaf  of $r(\widetilde\gamma)$ to a leaf of $l(\widetilde\gamma)$ and that is contained but the ends in the open set $\widetilde U$, union of leaves met by $\widetilde \gamma$. Observe now that $\widetilde \gamma$ and $\widetilde \gamma''$ intersect $\widetilde{\mathcal{F}}$-transversally and positively. 
\end{proof}

\bigskip
We deduce immediately

\begin{corollary}\label{co: no_intersection_two_sets}
Let $\mathcal F$ be an oriented singular foliation on $\R^2$, $\gamma$ a transverse path that is either a line or a proper path directed by a unit vector $\rho$ and $\gamma'$ a transverse path. If $\gamma$ and $\gamma'$ do not intersect $\mathcal{F}$-transversally, then $\gamma'$ cannot meet both sets $r(\gamma)$ and $l(\gamma)$. \end{corollary}

\bigskip
Given a transverse loop $\Gamma$ with a $\mathcal F$-transverse self-intersection and its natural lift $\gamma$, there exists some integer $K$ for which $\gamma\vert_{[0, K]}$ also has an $\mathcal{F}$-transverse self-intersection. Let us continue this section with an estimate of the minimal such $K$ when $\Gamma$ is homologous to zero. 

\begin{proposition}\label{pr:shorttransversalitypforloops}
Let $\mathcal F$ be an oriented singular foliation on $M$ and $\Gamma:\T^1\to M$ a transverse loop homologous to zero in $M$ with an $\mathcal{F}$-transverse self-intersection. If $\gamma:\R\to M$ is the natural lift of $\Gamma$, then $\gamma\vert_{[0,2]}$ has an $\mathcal{F}$-transverse self-intersection.
\end{proposition}

\begin{proof}

 Write $\widetilde{\mathrm{dom}}(\mathcal F)$ for the universal covering space of $\mathrm{dom}(\mathcal F)$. If $\widetilde \gamma:J\to  \widetilde{\mathrm{dom}}(\mathcal F)$ is a path and $T$ a covering automorphism, write $T(\widetilde\gamma):J\to  \widetilde{\mathrm{dom}}(\mathcal F)$ for the path satisfying $T(\widetilde\gamma)(t)=T(\widetilde\gamma(t))$ for every $t\in J$. Choose a lift $\widetilde\gamma$  of $\gamma$ to $\widetilde{\mathrm{dom}}(\mathcal F)$  and write $T$ for the covering automorphism such that $\widetilde\gamma(t+1)=T(\widetilde \gamma)(t)$,  for every $t\in\R$. 
Since  $\gamma$ has an $\mathcal{F}$-transverse self-intersection and is periodic of period $1$, there  exist a covering automorphism $S$ and
$$  a_1<t_1<b_1, \, a_2<t_2<b_2,$$ 
  such that

\smallskip
\noindent -\enskip $\widetilde\gamma\vert_{(a_1,b_1)}$ is equivalent to  $S(\widetilde\gamma)\vert_{(a_2,b_2)}$;

\smallskip
\noindent -\enskip $\widetilde\gamma\vert_{[a_1,b_1]}$ and  $S(\widetilde\gamma)\vert_{[a_2,b_2]}$ have a $\widetilde{\mathcal{F}}$-transverse intersection at $\widetilde \gamma(t_1)=S(\widetilde \gamma)(t_2)$,

 \smallskip
\noindent -\enskip both $a_1, a_2$ belong to $[0, 1)$.

We will show that $b_1\le a_1+1$ and $b_2 \le a_2+1$ , which implies that  $\gamma\vert_{[0, 2]}$ has an $\mathcal{F}$-transverse self-intersection.
Assume for a contradiction that $b_1>a_1+1$ (the case where $b_2> a_2+1$ is treated similarly). Then we can find $a_1', a_2', b_2'$ with 
$$ a_1<a_1'<t_1, \, a_2<a_2'<t_2<b_2'<b_2$$  such that $\widetilde\gamma\vert_{[a_1', a_1'+1]}$ is equivalent to $S(\widetilde\gamma)\vert_{[a_2', b_2']}$. 

 Consider first the case where $b_2'=a_2'+1$. In that situation there exists an increasing homeomorphism
$ h:[a_1', a_1'+1] \to [a_2', a_2'+1], $ such that $h(t_1)=t_2$ and $\phi_{\widetilde\gamma(t)}=\phi_{S(\widetilde\gamma)(h(t))}$. This implies that 
$$T(\phi_{\widetilde\gamma(a'_1)})=\phi_{\widetilde\gamma(a'_1+1)}=\phi_{S(\widetilde\gamma)(a'_2+1)}=STS^{-1}\phi_{\widetilde\gamma(a'_2)}=STS^{-1}\phi_{\widetilde\gamma(a'_1)}.$$  In case $STS^{-1}=T$, 
one can extend $h$  to a homeomorphism of the real line that commutes with the translation $t\mapsto t+1$ such that  $\phi_{\widetilde\gamma(t)}=\phi_{S(\widetilde\gamma)(h(t))}$, for every $t\in\R$. If $K$ is large enough, then $[-K,K]$ contains $[a_1,b_1]$ and $h([-K,K])$ contains $[a_2,b_2]$. This contradicts the fact that $\widetilde\gamma\vert_{[a_1,b_1]}$ and  $S(\widetilde\gamma)\vert_{[a_2,b_2]}$ have a $\widetilde{\mathcal{F}}$-transverse intersection at $\widetilde\gamma(t_1)= S(\widetilde\gamma)(t_2)$. 
 In case $STS^{-1}\not =T$, the leaf $\phi_{\widetilde\gamma(a'_1)}$ is invariant by the commutator $T^{-1}STS^{-1}$ and so projects into a closed leaf of $\mathcal F$ that is homological to zero in $\mathrm{dom}(\mathcal F)$, which means that it bounds a closed surface \textcolor{black}{in this domain.  This closed surface, being a subsurface of $\mathrm{dom}(\mathcal F)$, is naturally foliated by $\mathcal{F}$, a non singular foliation, and one gets a contradiction by Poincar\'e-Hopf formula. One also gets a contradiction since this closed leaf has a non zero intersection number with the loop $\Gamma$.}
 % and since both the leaf and $\Gamma$ are homologous to zero.} 

Now assume that $b_2'< a_2'+1$. Let $s \in (b_2', a_2'+1)$ and consider $\phi_{\gamma(s)}$. As noted \textcolor{black}{in the last paragraph of Subsection \ref{subsection:Generaldefinitions}}, since $\Gamma$ is homologous to zero, it intersects every given leaf a finite number of times. Let $n$ be the number of times it intersects $\phi_{\gamma(s)}$. It is equal to the number of times $\gamma\vert_{[a_1', a_1'+1)}$ or $\gamma\vert_{[a_2', a_2'+1)}$ intersect $\phi_{\gamma(s)}$. On the other hand, since  $\gamma\vert_{[a_2', b_2')}$ is equivalent to $\gamma\vert_{[a_1', a_1'+1)}$, it must also intersect  $\phi_{\gamma(s)}$ exactly $n$ times, and since $s \in (b_2', a_2'+1)$,  $\gamma\vert_{[a_2', a_2'+1)}$ needs to intersect $\phi_{\gamma(s)}$ at least $n+1$ times, a contradiction (see Figure \ref{Figure_Minimal_transversality}).

Finally, if $b_2'>a_2'+1$, then $\gamma\vert_{[a_2', a_2'+1]}$ is equivalent to $\gamma\vert_{[a_1', b_1']}$ for some $b_1'<a_1'+1$ and the same reasoning as above may be applied.
\end{proof}

\begin{figure}[ht!]
\hfill
\includegraphics [height=48mm]{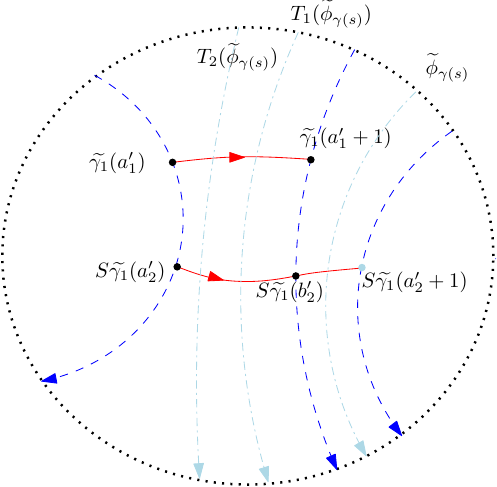}
\hfill{}
\caption{\small Contradiction from Proposition \ref{pr:shorttransversalitypforloops}. $\widetilde{\gamma}\mid_{[a_1', a_1'+1]}$ and $S\widetilde{\gamma}\mid_{[a_2', a_2'+1]}$ must cross the same finite number of lifts of  $\phi_{\gamma(s)}$.}
\label{Figure_Minimal_transversality}
\end{figure}

We will finish this subsection with a result (Proposition \ref{pr: finite_homotopy-classes}) that will be useful later. Let us begin with this simple lemma.

\begin{lemma}\label{le: finite_transverse_intersection}
Let $\mathcal F$ be an oriented singular foliation on $M$ and $\widetilde{\mathcal F}$ the lifted foliation on the universal covering space $\widetilde{\mathrm{dom}}(\mathcal F)$ of $\mathrm{dom}(\mathcal F)$.  Let $\widetilde\gamma$ and $\widetilde\gamma'$ be two lines of $\widetilde{\mathrm{dom}}(\mathcal F)$ transverse to $\widetilde{\mathcal F}$,  invariant by $T$ and $T'$ respectively, where $T$ and $T'$ are non trivial covering automorphisms. Up to a right composition by a power of $T$ and a left composition by a power of $T'$  there are finitely many covering automorphisms $S$ such that $\widetilde\gamma$ and $S(\widetilde\gamma')$ intersect $\widetilde{\mathcal F}$-transversally.

\end{lemma}

\begin{proof} Suppose $\widetilde\gamma:\R\to \widetilde{\mathrm{dom}}(\mathcal F)$ and  $\widetilde\gamma':\R\to \widetilde{\mathrm{dom}}(\mathcal F)$ parameterized such that $\widetilde\gamma(t+1)=T(\widetilde\gamma(t))$ and $\widetilde\gamma'(t+1)=T'(\widetilde\gamma(t))$, for every $t\in\R$. The group of covering automorphisms acts freely and properly. So there exists $ L<+\infty$ automorphisms $S$ such that $\widetilde\gamma\vert_{[0,1]}\cap S(\widetilde\gamma'\vert_{[0,1]})\not=\emptyset$. If $\widetilde\gamma$ and $S(\widetilde\gamma')$ intersect $\widetilde{\mathcal F}$-transversally, there exist $t$ and $t'$ such that $\widetilde\gamma(t)=S(\widetilde\gamma')(t')$. Write $[x]$ for the integer part of a real number $x$. One has $\widetilde\gamma(t-[t])=T^{-[t]}ST'{}^{[t']}(\widetilde\gamma)(t'-[t'])$, which implies that $T^{-[s]}ST'{}^{[t]}$ is one of the $L$ previous automorphisms.
\end{proof}

\bigskip
Let $\mathcal F$ be an oriented singular foliation on $M$ and $\widetilde{\mathcal F}$ the lifted foliation on the universal covering space $\widetilde{\mathrm{dom}}(\mathcal F)$ of $\mathrm{dom}(\mathcal F)$. Let $\Gamma$ be a loop on $M$ transverse to  $\mathcal F$ and 
 $\widetilde\gamma$ a lift of $\Gamma$ to  $\widetilde{\mathrm{dom}}(\mathcal F)$. Write $T$ for the covering automorphism such that $\widetilde\gamma(t+1)=T(\widetilde\gamma(t))$ for every $t\in\R$. If $\delta:J\to \widetilde{\mathrm{dom}}(\mathcal F)$ is a transverse path equivalent to a subpath of $\widetilde\gamma$ we define its {\it width} (relative to $\widetilde \gamma$) to be the largest integer $l$ \textcolor{black}{(possibly infinite)} such that $\delta$ meets $l$ translates of a leaf by a power of $T$. More precisely, \textcolor{black}{$\mathrm{width}_{\widetilde \gamma}(\delta)=\infty$} if there exists a leaf $\phi$ such that $\delta$ meets infinitely many translates of $\phi$ by a power of $T$, and  $\mathrm{width}(\delta)=l<+\infty$ if there exists a leaf $\phi$ such that $\delta$ meets every leaf $T^k(\phi)$, $0\leq k<l$, and if $l+1$ does not satisfy this property. By Lemma \ref{le: finite_transverse_intersection}, up to a left composition by a power of $T$ there are finitely many lifts $S(\widetilde\gamma)$ such that $\widetilde\gamma$ and $S(\widetilde\gamma)$ intersect $\widetilde{\mathcal F}$-transversally. This number is clearly independent of the chosen lift  $\widetilde \gamma$, we denote it $\mathrm{self}(\Gamma)$. Saying that $\Gamma$ has a $\mathcal F$-transverse self-intersection
 means that $\mathrm{self}(\Gamma)\not=0$. If $\widetilde\gamma$ and $S(\widetilde\gamma)$ intersect $\widetilde{\mathcal F}$-transversally, one can consider the maximal subpath of  $S(\widetilde\gamma)$ 
 that is equivalent to a subpath of $\widetilde\gamma$. Note that its width (relative to $\widetilde\gamma$) is finite. Looking at all the $S(\widetilde\gamma)$ that intersect  $\widetilde{\mathcal F}$-transversally $\widetilde\gamma$ and taking the supremum, one gets a finite number because the $S(\widetilde\gamma)$ are finite up to a composition by a power of $T$. This number is independent of the choice of $\widetilde\gamma$, we denote it $\mathrm{width}(\Gamma)$.  One gets a total order $\preceq_{\widetilde \gamma}$ on the set of leaves met by a lift $\widetilde\gamma$, where
 $$\widetilde\phi\preceq_{\widetilde \gamma}\widetilde\phi'\enskip\Leftrightarrow \enskip R(\widetilde\phi)\subset R(\widetilde\phi').$$
 
Note that if $\widetilde\gamma$ and $S(\widetilde\gamma)$ intersect transversally and $\widetilde\phi$ is a leaf met by $\widetilde\gamma$, then \textcolor{black}{there exist} at most $\mathrm{width}(\Gamma)$ translates of $\widetilde\phi$ by a power of $T$ that meet $S(\widetilde\gamma)$. Moreover, there exists $k\in\Z$ such that every leaf $\widetilde\phi'$ met by $\widetilde\gamma$ and $S(\widetilde\gamma)$ satisfies \textcolor{black}{$T^{k}(\widetilde\phi)\prec_{\widetilde\gamma} \widetilde\phi'\prec_{\widetilde\gamma}T^{\mathrm{width}(\Gamma)+1+k}(\widetilde\phi)$} and consequently that $\widetilde\phi$ and $T^{\mathrm{width}(\Gamma)+1} (\widetilde\phi)$ are separated by $T^{-k}S(\widetilde\gamma)$.

\begin{proposition}\label{pr: finite_homotopy-classes} Let $\mathcal F$ be an oriented singular foliation on $M$, let $\Gamma$ be a transverse loop with a $\mathcal F$-self-transverse intersection and $\gamma$ its natural lift. Write
 $$M(\Gamma)=\,\mathrm{self}(\Gamma)\mathrm{width}(\Gamma)(\mathrm{width}(\Gamma)+1)+1.$$
 \textcolor{black}{Consider} two points $z$ and $z'$ disjoint from $\Gamma$ and look at the set of homotopy classes, with fixed endpoints, of paths starting at $z$ and ending at $z'$. \textcolor{black}{There are} at most $2M(\Gamma)$ classes which are represented both by a path disjoint from $\Gamma$ and by a (possibly different) transverse path equivalent to a subpath of $\gamma$.
\end{proposition}

 \begin{proof}
 Fix a lift $\widetilde z$ of $z$ in $\widetilde{\mathrm{dom}}(\mathcal F)$ and denote $\widetilde X$ the set of lifts $\widetilde z'$ of $z'$ such that there exists a path from $\widetilde z$ to $\widetilde z'$ disjoint from all lifts of $\gamma$ and such that there exists a transverse path from $\widetilde z$ to $\widetilde z'$ that is $\widetilde{\mathcal F}$-equivalent to a subpath of at least one lift of $\gamma$. The proposition is equivalent to showing that $\widetilde X$ does not contain more than $2M(\Gamma)$ points.

By definition, for every $\widetilde z'\in \widetilde X$, there exists a transverse path $\widetilde\delta_{\widetilde z'}$ from $\widetilde z$ to $\widetilde z'$, unique up to equivalence. Moreover the set $\widetilde X\cup\{\widetilde z\}$ is included in a connected component $\widetilde W$ of the complement of the union of lifts of $\gamma$. There is  no lift $\widetilde\gamma$ of $\gamma$ that separates points of $\widetilde X$: for each lift $\widetilde \gamma$,  the set $\widetilde X$ is included in $R(\widetilde\gamma)$ or in $ L(\widetilde\gamma)$. One can write $\widetilde X=\widetilde X_r\cup \widetilde X_l$, where $\widetilde z'\in \widetilde X_r$ if there exists a lift $\widetilde \gamma$ of $\gamma$ satisfying $\widetilde X\subset R(\widetilde\gamma)$ such that $\widetilde\delta_{\widetilde z'}$ is a subpath of \textcolor{black}{$\widetilde\gamma$} (up to equivalence). One define similarly $\widetilde X_l$ replacing the condition $\widetilde X\subset R(\widetilde\gamma)$  by $\widetilde X\subset L(\widetilde\gamma)$.  We will prove that 
$\widetilde X_r$ and $\widetilde X_l$ do not contain more than $M(\Gamma)$ points. The two situations being similar, we will study the first one.

\begin{lemma}\label{le: distinct leaves} If $\widetilde z'_1$ and $\widetilde z'_2$ are two different points in $\widetilde X$, then $\phi_{{\widetilde z'_1}}\not=\phi_{{\widetilde z'_2}}$.
\end{lemma} 
\begin{proof} \textcolor{black}{See Figure \ref{Figure_MGamma} for the following construction.} Suppose for example that $\widetilde z'_1\in\phi^+_{{\widetilde z'_2}}$ and denote by $S$ the covering automorphism such that $\widetilde z'_1=S(\widetilde z'_2)$. There exists a lift $\widetilde \gamma$ of $\gamma$ such that $\widetilde\delta_{\widetilde z'_1}$ is a subpath of $\widetilde\gamma$ (up to equivalence). Note that there exists $k\in\Z$ such that $\widetilde z'_1\in R(S^k(\widetilde\gamma))$ and $\widetilde z'_2\in L(S^k(\widetilde\gamma))$. The lift $S^k(\widetilde\gamma)$ separates $\widetilde z'_1$ and $\widetilde z'_2$, we have a contradiction.
\end{proof}

\begin{figure}[t!]
\hfill
\includegraphics[height=45mm]{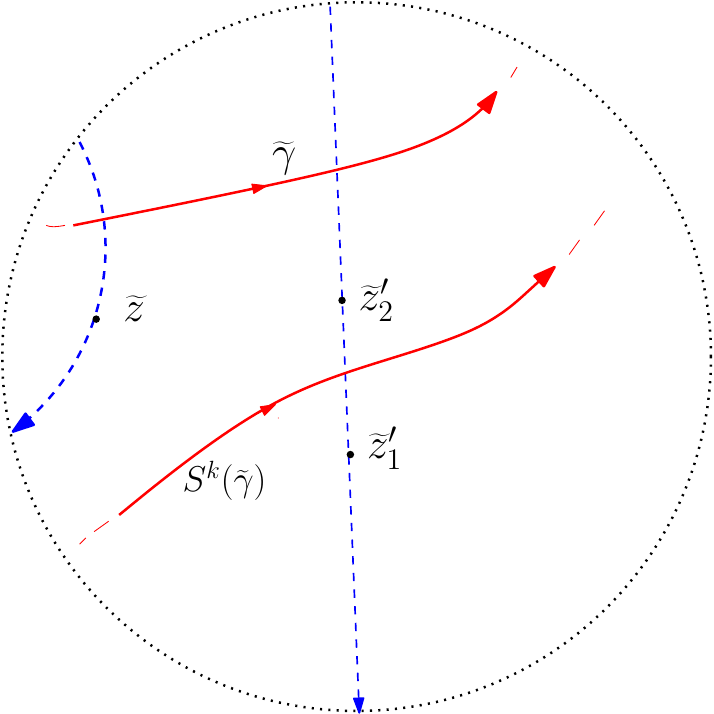}
%[height=55mm]
\hfill
\includegraphics[height=45mm]{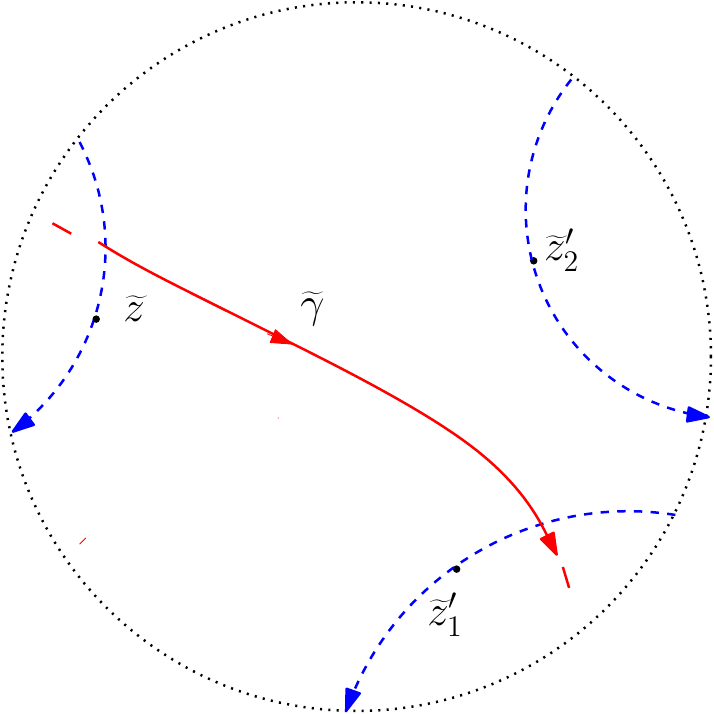}
\hfill{}
\caption{\small Cases described in Lemma \ref{le: distinct leaves} (left) and Lemma  \ref{le: comparable}(right).}
\label{Figure_MGamma}
\end{figure}

\begin{lemma}\label{le: comparable} If $\widetilde z'_1$ and $\widetilde z'_2$ are two different points in $\widetilde X_r$, then up to equivalence, one of the paths $\widetilde\delta_{\widetilde z'_1}$, $\widetilde\delta_{\widetilde z'_2}$ is a subpath of the other one.

\end{lemma} 
\begin{proof} Each path $\widetilde\delta_{\widetilde z'_i}$, $i\in\{1,2\}$, joins $\phi_{\widetilde z}$ to $\phi_{{\widetilde z'_i}}$. We claim that \textcolor{black}{ either $L(\phi_{{\widetilde z'_1}})\subset L(\phi_{{\widetilde z'_2}})$ or $L(\phi_{{\widetilde z'_2}})\subset L(\phi_{{\widetilde z'_1}})$.} If this is not the case, one of the leaves
$\phi_{{\widetilde z'_1}}$, $\phi_{{\widetilde z'_2}}$ is above the other one relative to  $\phi_{\widetilde z}$. Suppose that it is $\phi_{{\widetilde z'_1}}$. By definition of $\widetilde X_r$, there exists a lift $\widetilde \gamma$ of $\gamma$ satisfying $ \widetilde X_r\subset R(\widetilde\gamma)$ such that $\widetilde\delta_{\widetilde z'_1}$ is a subpath of $\widetilde\gamma$ (up to equivalence). Note now that $\phi_{{\widetilde z'_2}}\subset L(\widetilde\gamma)$ which contradicts the fact that $\widetilde z'_2\in R(\widetilde\gamma)$.

Since \textcolor{black}{$L(\phi_{\widetilde z})$ contains both $L(\phi_{{\widetilde z'_1}})$ and $L(\phi_{{\widetilde z'_2}})$, and since, by the previous lemma,  $\phi_{\widetilde z}, \, \phi_{{\widetilde z'_1}}$, and $\phi_{{\widetilde z'_2}}$ are all  distinct,} either $\phi_{{\widetilde z'_1}}$ separates $\phi_{\widetilde z}$ from $\phi_{{\widetilde z'_2}}$, or $\phi_{{\widetilde z'_2}}$ separates $\phi_{\widetilde z}$ from $\phi_{{\widetilde z'_1}}$. In the first case, $\widetilde\delta_{\widetilde z'_1}$ is equivalent to a subpath of $\widetilde\delta_{\widetilde z'_2}$, and in the second case $\widetilde\delta_{\widetilde z'_2}$ is equivalent to a subpath of $\widetilde\delta_{\widetilde z'_1}$.
\end{proof}

 \bigskip
We will \textcolor{black}{suppose that $\widetilde X_r$ has at least $K$ points, and we will show that $K\le M(\Gamma)$, which proves the proposition.} 
%points and will find a contradiction.
Using Lemmas \ref{le: distinct leaves} and \ref{le: comparable}, one can find a family $(\widetilde z'_i)_{0\leq i\leq \textcolor{black}{K-1}}$ of points of $\widetilde X_r$ such that $\widetilde\delta_{\widetilde z'_i}$ is a strict subpath of $\widetilde\delta_{\widetilde z_j}$ (up to equivalence) if $i<j$. One knows that there exists a lift $\widetilde\gamma$ of $\gamma$ such that $\widetilde\delta_{\widetilde z'_{M(\Gamma)}}$ is equivalent to a subpath of $\widetilde\gamma$. One deduces that every leaf \textcolor{black}{$\phi_{\widetilde z'_i}$, $0\leq i\leq \textcolor{black}{K-1}$, is met by $\widetilde \gamma$ and that $\phi_{\widetilde z'_i} \prec_{\widetilde\gamma}\phi_{\widetilde z'_j}$} if $i<j$. Write $\widetilde z'_i=T_i(\widetilde z'_0)$ and note that $T_i$ belongs to $\mathrm{stab}(\widetilde W)$, the stabilizer of $\widetilde W$ in the group of covering automorphisms.

\begin{figure}[ht!]
\hfill
\includegraphics [height=48mm]{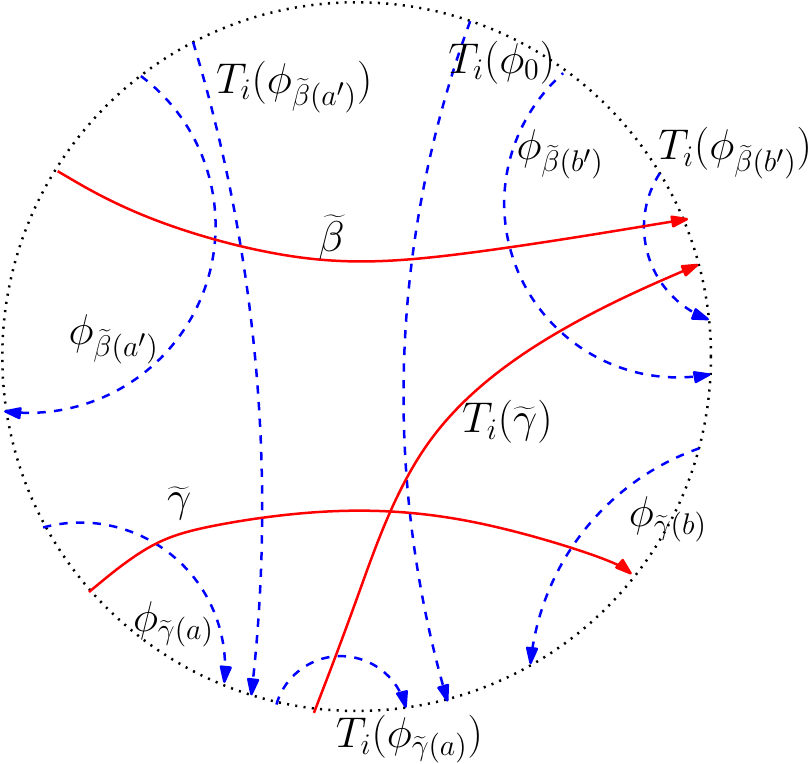}
\hfill{}
\caption{\small Final case of Lemma \ref{le: transverse}.}
\label{Figure_Proposition9_entropy}
\end{figure}

\begin{lemma}\label{le: transverse} The lifts $\widetilde\gamma$ and $T_i(\widetilde\gamma)$ intersect transversally for every $i\in\{1,\dots, \textcolor{black} {K-1}\}$.

\end{lemma} 
\begin{proof} 
\textcolor{black}{ Fix $i\in\{1,\dots, \textcolor{black}{K-1}\}$.} One can find a transverse line $\widetilde\beta$ invariant by $T_i$ passing through \textcolor{black}{$\widetilde z'_0$ and $\widetilde z'_i$}. Write $\widetilde\delta$ for the maximal subpath of  $\widetilde\gamma$ 
 that is equivalent to a subpath of $\widetilde\beta$. If $\widetilde\gamma$ and $\widetilde\beta$ intersect $\widetilde{\mathcal F}$-transversally, then $\widetilde\gamma$ separates \textcolor{black}{$T_i^{-k}(\phi_{\widetilde z'_0})$ and $T_i^{k}(\phi_{\widetilde z'_0})$} if $k$ is large enough. So, it separates \textcolor{black}{$T_i^{-k}(\widetilde z'_0)$ and $T_i^{k}(\widetilde z'_0)$}. This contradicts the fact that $T_i\in\mathrm{stab}(\widetilde W)$. If $\widetilde\delta$ is unbounded or equivalently if $\mathrm{width}_{\beta}(\gamma)=+\infty$, then $\widetilde\delta$ is forward or backward invariant by $T_i$. Look at the first case.  Since $\gamma$ has a $\mathcal F$-transverse self-intersection, there exists  a covering automorphism $S$ such that $\widetilde \gamma$ and $S(\widetilde \gamma)$ intersect $\widetilde{\mathcal F}$-transversally. For every integer $n$,  the lines $T^nS(\widetilde\gamma)$ and $\widetilde \gamma$  intersect $\widetilde{\mathcal F}$-transversally. One deduces that if  $n$ is large enough, then $T^nS(\widetilde\gamma)$ and $\widetilde \delta$  intersect $\widetilde{\mathcal F}$-transversally and consequently  $T^nS(\widetilde\gamma)$ and $\widetilde \beta$  intersect $\widetilde{\mathcal F}$-transversally.  So, if $k$ is large enough, $T^nS(\widetilde\gamma)$ separates \textcolor{black}{$T_i^{-k}(\phi_{\widetilde z'_0})$ and $T_i^{k}(\phi_{\widetilde z'_0})$}. This again contradicts the fact that $T_i\in\mathrm{stab}(\widetilde W)$. The case where $\widetilde\delta$ is backward invariant can be treated analogously. It remains to study the case where $\delta$ is bounded and  $\widetilde\gamma$ and $\widetilde\beta$ do not intersect $\widetilde{\mathcal F}$-transversally.

\textcolor{black}{See Figure \ref{Figure_Proposition9_entropy} for the following construction.} Write $\widetilde\delta'$ for the maximal subpath of  $\widetilde\beta$ 
 that is equivalent to \textcolor{black}{$\widetilde\delta$. There exist} $a<b$ such that $\widetilde\delta=\widetilde\gamma\vert_{(a,b)}$ and $a'<b'$ such that $\widetilde\delta'=\widetilde\beta\vert_{(a',b')}$.  The sets $R(\widetilde\gamma(a))$ and $R(\widetilde\beta(a'))$ are disjoint, as are the sets $L(\widetilde\gamma(b))$ and $L(\widetilde\beta(b'))$. There exists a leaf \textcolor{black}{$\phi_0$} such that $\widetilde\delta$ and $\widetilde\delta'$ meet \textcolor{black}{$\phi_0$ and $T_i(\phi_0)$}. In particular \textcolor{black}{$T_i(\phi_0)$} is met by $\widetilde\gamma\vert_{[a,b]}$ and $T_i(\widetilde\gamma)\vert_{[a,b]}$. By assumption, $\widetilde\gamma$ and $\widetilde \beta$ do not intersect $\widetilde{\mathcal F}$-transversally. There is no loss of generality by supposing that $\phi_{\widetilde\beta(a')}$ is \textcolor{black}{below} $\phi_{\widetilde\gamma(a)}$ and $\phi_{\widetilde\beta(b')}$ \textcolor{black}{below} $\phi_{\widetilde\gamma(b)}$ relative to \textcolor{black}{$\phi_0$ or $T_i(\phi_0)$}. The leaf $T_i(\phi_{\widetilde\beta(a')})$ is above $T_i(\phi_{\widetilde\gamma(a)})= \phi_{T_i(\widetilde\gamma)(a)}$ relative to \textcolor{black}{ $T_i(\phi_0)$}. Moreover\textcolor{black}{, since $T_i$ preserves $\widetilde \beta$, one has that $T_i(\phi_{\widetilde\beta(a')})$ separates $\phi_{\widetilde\beta(a')}$ and $T_i(\phi_0)$, and therefore it is crossed by both $\widetilde\delta$ and $\widetilde\delta'$.} This  implies that there exists a transverse path that joins $\phi_{\widetilde\gamma(a)}$ to $T_i(\phi_{\widetilde\beta(a')})$. Consequently, $\phi_{\widetilde\gamma(a)}$ belongs to $R(T_i(\phi_{\widetilde\beta(a')}))$ and is above $\phi_{T_i(\widetilde\gamma)(a)}$ relative to \textcolor{black}{$T_i(\phi_0)$}. Similarly, there exists a transverse path that joins $\phi_{\widetilde\beta(b')}$ to $\phi_{T_i(\widetilde\gamma)(b)}$.   Consequently, $\phi_{T_i(\widetilde\gamma)(b)}$ belongs to $L(\phi_{\widetilde\beta(b')})$ and is above $\phi_{\widetilde\gamma(b)}$ relative to \textcolor{black}{$T_i(\phi_0)$}. We have proved that the paths $\widetilde\gamma_{[a,b]}$ and $T_i(\widetilde\gamma)_{[a,b]}$ intersect $\widetilde{\mathcal F}$-transversally.\end{proof}

%In this situation, let $L=\mathrm{width}_{\beta}(\gamma)$, which is larger than $2$ since $\widetilde \gamma$ intersects both $\phi_{\widetilde z_0'}$ and $T_i(\phi_{\widetilde z_0'})$. Let $\phi_0$ be a leaf intersected by $\widetilde \beta$ and such that $\widetilde \gamma$ intersects both $\phi_0$ and $T_i^{L-1}(\phi_0)$. In particular, $T_i^{-1}(\phi_0)$ and $T_i^{L}(\phi_0)$ are disjoint from $\widetilde \gamma$. Since $\widetilde \gamma$ and $\widetilde \beta$ do not intersect transversally, either $T_i^{-1}(\phi_0)\cup T_i^{L}(\phi_0)\subset l(\widetilde \gamma)$, or $T_i^{-1}(\phi_0)\cup T_i^{L}(\phi_0)\subset r(\widetilde \gamma)$.  The cases being similar, we assume the former happens, which implies that $\phi_0\subset l(T_i(\widetilde \gamma))$  and so $\widetilde \gamma$ intersects $l(T_i(\widetilde \gamma))$. Also, since $T_i(\widetilde \gamma)$ intersects $T_i^{L}(\phi_0)$, we have that $T_i(\widetilde \gamma)\cap l(\widetilde \gamma)\not=\emptyset$. Therefore $\widetilde\gamma$ and $T_i(\widetilde\gamma)$ intersect $\widetilde{\mathcal F}$-transversally.} \end{proof}

By the definition of $\mathrm{self}(\Gamma)$ and Lemma \ref{le: transverse}, there exists covering automorphisms $(S_l)_{0\le l <\mathrm{self}(\Gamma)}$  such that  $\widetilde\gamma$ and $S_l(\widetilde\gamma)$ intersect $\widetilde{\mathcal F}$-transversally, a family $(l_i)_{1\leq i\leq \textcolor{black}{K-1}}$ in $\{0, \dots, \mathrm{self}(\Gamma)-1\}$ and families of relative integers $(n_i)_{1\leq i\leq \textcolor{black}{K-1}}$, $(m_i)_{1\leq i\leq \textcolor{black}{K-1}}$, such that $T_{i}=T^{n_i}S_{l_i}T^{m_i}$. \textcolor{black}{ Note that, if $1\le i, j \le   \textcolor{black}{K-1}$, and $i\not=j$, then $T_i\not=T_j$, and the function that assigns for each $i$ the triple $(n_i, l_i, m_i)$ is injective.}

Define, for $0\le l <\mathrm{self}(\Gamma)$ the set $ I_l=\{1\le i\le \textcolor{black}{K-1}, \, l_i=l \}$ and fix some $l$ in $\{1, ..., \mathrm{self}(\Gamma)\}$. If $i\in I_l$, each leaf $T^{m_i}(\phi_{\widetilde z'_{0}})= S_{l}^{-1}T^{-n_i}(\phi_{\widetilde z'_{i}})$ is met both by $\widetilde\gamma$ and by $S_{l}^{-1}(\widetilde\gamma)$, and since $\widetilde\gamma$ and $S_{l}^{-1}(\widetilde\gamma)$  also  intersect $\widetilde{\mathcal F}$-transversally, we deduce that  there are at most $\mathrm{width}(\Gamma)$ possible values for $m_{i}$ with $i\in I_l$. Fix such a value $m$ and \textcolor{black}{ consider the set $ I_{l, m}=\{1\le i\le \textcolor{black}{K-1}, \, l_i=l,\, m_i=m \}$. Note that, if $i\in I_{l, m}$, then the leaf $\phi_{\widetilde z'_{i}}= T^{n_i}S_{l}T^{m}(\phi_{\widetilde z'_{0}})$ is met by both $\widetilde \gamma$ and $S^{l}(\widetilde \gamma)$.}  The fact that no line \textcolor{black}{$T^kS\widetilde\gamma$, $k\in\Z$} separates the leaves \textcolor{black}{$\phi_{\widetilde z'_{i}}, i\in I_{l, m}$ implies, as noted just before Proposition \ref{pr: finite_homotopy-classes},} that there at most $\mathrm{width}(\Gamma)+1$ such leaves, that means at most $\mathrm{width}(\Gamma)+1$ possible values of $n_i$ and \textcolor{black}{elements in $I_{l, m}$.} One deduces that 
$$\textcolor{black}{K\leq \mathrm{self}(\Gamma)\mathrm{width}(\Gamma)(\mathrm{width}(\Gamma)+1)=M(\Gamma),}$$
%$$KM(\Gamma)=\mathrm{self}(\Gamma)\mathrm{width}(\Gamma)(\mathrm{width}(\Gamma)+1)+1\leq \mathrm{self}(\Gamma)\mathrm{width}(\Gamma)(\mathrm{width}(\Gamma)+1),$$
\textcolor{black}{as desired}.
%also intersect $\widetilde{\mathcal{F}}$-transversally, each $m_i$ must belong to  a set with at most $\mathrm{width}(\Gamma)+1$ elements. 
% . So it is met by the maximal subpath of  $S_{l_i}^{-1}(\widetilde\gamma)$  that is equivalent to a subpath of $\widetilde\gamma$. There exists $M_l\in\Z$ such that  $$T^{M_l}(\phi_{\widetilde z_{0}})\prec T^{m_i}(\phi_{\widetilde z_{0}})\prec T^{M_l+\mathrm{width}(\Gamma)+2}(\phi_{\widetilde z_{0}})$$ and
%$$T^{N+n_{i_l}}(\phi_{\widetilde z_{0}})\prec \phi_{\widetilde z_{i_l}}\prec T^{N+\mathrm{width}(\Gamma)+2+n_{i_l}}(\phi_{\widetilde z_{0}}).$$
%again we deduce that, for a given $0\le \bar l < \mathrm{self}(\Gamma)$, if $i\in I_{\bar l}$, there are at most $\mathrm{width}(\Gamma)$ possible values for $m_i$.
\end{proof}
\subsection{Transverse homology set}

\

For any loop $\Gamma$ on $M$, let us denote $[\Gamma]\in H_1(M,\Z)$ its singular homology class. The {\it transverse homology set} of ${\mathcal F}$ is the smallest set $\mathrm{THS}({\mathcal F})$  of $H_1(M,\Z)$, that is stable by addition and contains all classes of loops positively transverse to $\mathcal F$. The following result will also be useful:

\begin{proposition}\label{pr: THS}
 Let $\mathcal F$ be a singular oriented foliation on $\T^2$ and $\widetilde{\mathcal F}$ its lift to $\R^2$. If one can find finitely many classes $\kappa_i\in \mathrm{THS}({\mathcal F})$, $1\leq i\leq r$, that linearly generate the whole homology of the torus and satisfy 
 $\sum_{1\leq i\leq r}\kappa_i=0$, then the diameters of the leaves of $\widetilde{\mathcal F}$ are uniformly bounded.
\end{proposition}

\begin{proof} Decomposing each class $\kappa_i$ and taking out all the loops homologous to zero, one can suppose (changing $r$ is necessary) that for every $i\in\{1,\dots,r\}$, there exists a transverse loop $\Gamma_i$ such that $[\Gamma_i]=\kappa_i$. The fact that the $\kappa_i$ linearly generate the whole homology of the torus implies that the multi-loop $\Gamma=\sum_{1\leq i\leq p}\Gamma_i$ is connected (as a set) and that the connected components of its complement are simply connected. Moreover, these components are lifted  in uniformly bounded simply connected domains of $\R^2$, let us say by a constant $K$. The multi-loop $\Gamma$ being homologous to zero induces a dual function $\delta$ on its complement. It has been explained before that $\delta$ decreases on each leaf of $\mathcal F$ and is bounded. Consequently, there exists an integer $N$ such that every leaf meets at most $N$ components. If one lifts it to $\R^2$, one find a path of diameter  bounded by $NK$.
\end{proof}
%\hfill$\Box$ 

\bigskip

\section{Maximal isotopies, transverse foliations, admissible paths}

\bigskip

\subsection{Singular isotopies, maximal isotopies}

\

Let us begin by introducing mathematical objects related to isotopies. Let $f$ be a homeomorphism of an oriented surface $M$. An {\it identity isotopy of $f$} is a path joining the identity to $f$ in the space of homeomorphisms of $M$, furnished with the $C^0$ topology (defined by the uniform convergence of maps and their inverse on every compact set).  We will write $\mathcal I$ for the set of identity isotopies of $f$ and will say that $f$ is {\it isotopic to the identity} if this set is not empty.  If $I=(f_t)_{t\in[0,1]}\in\mathcal I$ is such an isotopy, we can define the {\it trajectory} of a point $z\in M$, which is the path $I(z): t\mapsto f_t(z)$. More generally, for every $n\geq 1$ we define $I^n(z)=\prod_{0\leq k<n} I(f^k(z))$ by concatenation. We will also use the following notations  $$I^{\N}
 (z)=\prod_{0\leq k<+\infty} I (f^k(z)),\enskip  I^{-\N}
 (z)=\prod_{-\infty<k<0} I(f^k(z)), \enskip I^{\Z}
 (z)=\prod_{-\infty<k<+\infty} I(f^k(z)).$$
The last path will be called the {\it whole trajectory} of $z$.

One can define the {\it fixed point set}  $\mathrm{fix}(I)=\bigcap_{t\in[0,1]}\mathrm{fix}(f_t)$  of $I$, which is the set of points with trivial trajectory.

A wider class of isotopies is the class of {\it singular isotopies}.  Such an object $I$ is an identity isotopy defined on an open set invariant by $f$, the {\it domain} of $I$, whose complement, the {\it singular set},  is included in the fixed point set of $f$. We will write $\mathrm{dom}(I)$  for the domain and $\mathrm{sing}(I)$  for the singular set. Like in the case of a global isotopy, one can define the trajectory $I(z)$ of a point $z\in \mathrm{dom}(I)$ and the fixed point set, which is included in the domain. Note that any isotopy $I\in\mathcal I$ is itself a singular isotopy with empty singular set and induces by restriction to the complement of the fixed point set a singular isotopy such that $\mathrm{sing}(I)=\mathrm{fix}(I)$.

If $\check M$ is a covering space of $M$ and $\check \pi:\check M\to M$ the covering projection, every identity isotopy $I$ can be lifted to $\check M$ as an identity isotopy $\check I=(\check f_t)_{t\in[0,1]}$. The homeomorphism $\check f=\check f_1$ is {\it the lift of $f$ associated to $I$} or {\it induced by $I$}. Similarly, every singular isotopy can be lifted as a singular isotopy $\check I$ of $\check f$ such that $\mathrm{dom}(\check I)=\check\pi^{-1}(\mathrm{dom}(I))$.

Let us recall now some results due to O. Jaulent \cite{J}. Denote the set of singular isotopies by $\mathcal{I}_{\mathrm{sing}}$. It is not difficult to show that one gets a preorder $\preceq$ on $\mathcal{I}_{\mathrm{sing}}$, writing
 $I\preceq I'$ if: 

\smallskip
\noindent {\bf i)}\enskip\enskip$\mathrm{dom}(I')\subset\mathrm{dom}(I)$; 

\smallskip
\noindent {\bf ii)}\enskip\enskip
for every $z\in \mathrm{dom}(I')$, the trajectories $I'(z)$ and $I(z)$ are homotopic in $\mathrm{dom}(I)$;

 \smallskip

\noindent {\bf iii)}\enskip\enskip for every $z\in \mathrm{dom}(I)\setminus  \mathrm{dom}(I')$, the trajectory $I(z)$ is homotopic to zero in $\mathrm{dom}(I)$.

\smallskip
We will say that $I$ and $I'$ are {\it equivalent} if $I\preceq I'$ and $I'\preceq I$. We will say that $I$ is {\it maximal} if there is no singular isotopy $I'$ such that $I\preceq I'$ and $\mathrm{dom}(I')\not=\mathrm{dom}(I)$. If $I$ is a singular isotopy and if there is no point $z\in\mathrm{fix}(f)\cap\mathrm{dom}(I)$ such that $I(z)$ is homotopic to zero in $\mathrm{dom}(I)$, then $I$ is maximal by {\bf iii)}. The converse is true. If there is a point $z\in\mathrm{fix}(f)\cap\mathrm{dom}(I)$ such that $I(z)$ is homotopic to zero in $\mathrm{dom}(I)$, one can find an isotopy $I'$ on $\mathrm{dom}(I)$ that is homotopic to $I$ (as a path in the space of homeomorphisms of $M$ and relative to the ends) and that fixes $z$. Taking the restriction of $I'$ on $\mathrm{dom}(I)\setminus\{z\}$, one finds a singular isotopy strictly larger than $I$. The main result of \cite{J} is the fact that every singular isotopy is smaller than a maximal singular isotopy. In fact, the result is more precise and can be stated for {\it hereditary singular isotopies} (we will explain later the interest of looking at this class of singular isotopies). By definition, such an isotopy $I$ satisfies the following condition: for every open set $U$ containing $\mathrm{dom}(I)$, there exists $I'\in \mathcal{I}_{\mathrm{sing}}$ such that $I'\preceq I$ and $\mathrm{dom}(I')=U$.  Writing $\mathcal{I}_{\mathrm{her}}$ for the set of  hereditary singular isotopies, we have the following result due to Jaulent \cite{J}:

\begin{theorem}\label{th: maximal}For every $I\in \mathcal{I}_{\mathrm{her}}$ there exists  $I'\in \mathcal{I}_{\mathrm{her}}$, maximal in $\mathcal{I}_{\mathrm{her}}$, satisfying $I\preceq I'$. Such an isotopy $I'$ is maximal in $ \mathcal{I}_{\mathrm{sing}}$ and so there is no point $z\in\mathrm{fix}(f)\cap\mathrm{dom}(I')$ such that $I'(z)$ is homotopic to zero in $\mathrm{dom}(I')$
\end{theorem}

Note that if $\check M$ is a covering space of $M$ and $\check \pi:\check M\to M$ the covering projection, then for every singular isotopies $I$, $I'$  satisfying $I\preceq I'$, the respective lifts $\check I$, $\check I'$ satisfy $\check I\preceq \check I'$. Note also that a singular isotopy $I$ is maximal if and only if its lift $\check I$ is maximal.

Let us explain the reason why hereditary singular isotopies are important. It is related to the following problem. If $I$ is a 
singular isotopy, \textcolor{black}{does} there exists a global  isotopy $I'\in\mathcal I$ such that \textcolor{black}{$\mathrm{fix}(I')= \mathrm{sing}(I)\cup\mathrm{fix}(I)$} and $I'\vert _{M\setminus \mathrm{fix}(I')}$ equivalent to $I$ ?
Such an isotopy $I'$ always exists in the case where $\mathrm{fix}(f)$ is totally disconnected. Indeed, in that case, $I$ naturally extends to an isotopy on $M$ that fixes the ends of $\mathrm{dom}(I)$ corresponding to points of $M$. The problem is much more difficult in the case where $\mathrm{fix}(f)$ is not totally disconnected. The fact that $I$ is a hereditary singular isotopy is necessary because the restriction of a global isotopy to the complement of its fixed point set is obviously hereditary. It appears that this condition is sufficient. This
is the purpose of a recent work by B\'eguin-Crovisier-Le Roux \cite{BCL}. Following \cite{BCL}, Jaulent's theorem about existence of maximal isotopies can be stated in the following  much more natural form: for every $I\in \mathcal I$, there exists $I'\in \mathcal I$ such that:

\smallskip
\noindent {\bf i)}\enskip\enskip$\mathrm{fix}(I)\subset\mathrm{fix}(I')$;

\smallskip
\noindent {\bf ii)} \enskip\enskip  $I'$ is homotopic to $I$ relative to $\mathrm{fix}(I)$;

\smallskip
\noindent {\bf iii)}\enskip\enskip  there is no point $z\in\mathrm{fix}(f)\setminus\mathrm{fix}(I')$ whose trajectory $I'(z)$ is homotopic to zero in $M\setminus\mathrm{fix}(I')$.

The last condition can be stated in the following equivalent form:

\smallskip
\noindent -\enskip\enskip  if $\widetilde I'=(\widetilde f'_t)_{t\in[0,1]}$ is the identity isotopy that lifts $I'\vert_{M\setminus\mathrm{fix}(I')}$ to the universal covering space of $M\setminus\mathrm{fix}(I')$, then $\widetilde f'_1$ is fixed point free.

The typical example of an isotopy $I\in{\mathcal I}$ verifying {\bf iii)} is the restricted family $I=(f_t)_{t\in[0,1]}$ of a topological flow $(f_t)_{t\in\R}$ on $M$. Indeed, one can lift the flow $(f_t\vert_{M\setminus\mathrm{fix}(I)})_{t\in\R}$ as a flow $(\widetilde f_t)_{t\in\R}$ on the universal covering space of $M\setminus\mathrm{fix}(I)$. This flow has no fixed point and consequently no periodic point. So $\widetilde f_1$ is fixed point free, which exactly means that the condition {\bf iii)} is fulfilled. In particular, the restriction of $f$ to $M\setminus\mathrm{fix}(I)$ is a hereditary maximal isotopy. To construct a maximal singular isotopy that is not  hereditary, let us consider the flow $(f_t)_{t\in\R}$ on $\R^2$ defined as follows in polar coordinates
$$f_t(r,\theta)=( r, \theta + 2\pi t h(r))$$ 
where 
$$\textcolor{black}{h(r) = \begin{cases} r(1-r), &\mbox{if } r\in[0,1], \\ 
r^{-1}(1-r^{-1}), & \mbox{if } r\in[1,+\infty). \end{cases} }$$

and set $f=f_1$. The fixed point set of the flow is the union of the origin and the unit circle $\S^1$ and so, the restriction of the isotopy $(f_t)_{t\in[0,1]}$ to the complement of this fixed point set is a maximal hereditary singular isotopy. The isotopy $I'=(f'_t)_{t\in[0,1]}$, whose domain is the complement of $\{(0,0)\}\cup\S^1$, that coincides with $(f_t)_{t\in[0,1]}$ on the set of points such that $0<r<1$ and defined on the set of points such that $r>1$ by:

$$f'_t(r,\theta) = \begin{cases} ( r, \theta +4\pi tr), &\mbox{if } t\in[0,1/2], \\ 
( r, \theta +4\pi (t-1/2)h(r)), & \mbox{if } t\in[1/2,1]. \end{cases} $$

is not a hereditary singular isotopy. There is no singular isotopy $I''$ of $f$ whose domain is the complement of $\{(0,0), (1,0)\}$ such that $I''\preceq I'$ \textcolor{black}{ for the following reason. If $\widetilde I''$ is the lift of $I''$ to the universal covering of $\mathrm{dom}(I'')$, and if $\ z$ is a lift of $(0, -1)$, and $\widetilde z_n$ are lifts of $(0, -1-1/n)$ such that $\widetilde z_n$ converges to $\widetilde z$, then the trajectory of $\widetilde I''(\widetilde z)$ would be a closed loop, but the endpoints of the trajectories of $\widetilde I''(\widetilde z_n)$ do not converge to $\widetilde z$, since the trajectories of $I'(z_n)$ and $I''(z_n)$ are homotopic in $\mathrm{dom}(I)$.}

\medskip Since the proof of \cite{BCL} is not published yet, we will use the formalism of singular isotopies in the article.

\

\bigskip

\subsection{Transverse foliations}

\

Let $f$ be an orientation preserving plane homeomorphism. By definition, a {\it Brouwer line} of $f$ is a topological line $\lambda$ such that $f(\overline{L(\lambda)})\subset L(\lambda)$ (or equivalently a line $\lambda$ such that $f(\lambda)\subset L(\lambda)$ and $f^{-1}(\lambda)\subset R(\lambda))$. The classical Brouwer Plane Translation Theorem asserts that $\R^2$ can be covered by Brouwer lines in case $f$ is fixed point free (see \cite{Br}). Let us recall now the equivariant foliated version of this theorem (see \cite{Lec2}).  Suppose that $f$ is a homeomorphism  isotopic to the identity on an oriented surface $M$. Let $I$ be a maximal singular isotopy and write  $\widetilde I=(\widetilde f_t)_{t\in[0,1]}$ for the lifted identity defined on the universal covering space $\widetilde{\mathrm{dom}}(I)$ of $\mathrm{dom}(I)$. Recall that $\widetilde f=\widetilde f_1$ is fixed point free. Suppose first that $\mathrm{dom}(I)$ is connected. In that case, $\widetilde{\mathrm{dom}}(I)$ is a plane and we have \cite{Lec2}:

\begin{theorem}\label{th: BrEquiv}There exists a non singular topological oriented foliation $\widetilde {\mathcal F}$ on $\widetilde{\mathrm{dom}}(I)$, invariant by the covering automorphisms, whose leaves are Brouwer lines of $\widetilde f$.  
\end{theorem}

Consequently, for every point $\widetilde z\in \widetilde{\mathrm{dom}}(I)$, one has
$$\widetilde f(\widetilde z)\in L(\phi_{\widetilde z}), \enskip\widetilde z\in R(\phi_{\widetilde f(\widetilde z)}).$$ This implies that there exists a path $\widetilde\gamma$ positively transverse to $\widetilde{\mathcal F}$ that joins $\widetilde z$ to $\widetilde f(\widetilde z)$.  As noted in section 3.1, this path is uniquely defined up to $\widetilde{\mathcal{F}}$-equivalence, provided the endpoints remain the same. The leaves of the lifted foliation $\widetilde {\mathcal F}$ met by $\widetilde \gamma$ are the leaves $\phi$ such that $R(\phi_{\widetilde z})\subset R(\phi)\subset R(\phi_{\widetilde f(\widetilde z)})$. In particular,  every leaf met by  $\widetilde \gamma$ is met by $\widetilde I(\widetilde z)$.
Of course, $\widetilde{\mathcal F}$ lifts a singular foliation ${\mathcal F}$ such that $\mathrm{dom}(\mathcal F)=\mathrm{dom}(I)$. We immediately get the following result, still true in case $\mathrm{dom}(I)$ is not connected:

\begin{corollary}\label{co: BrEquiv}There exists a singular topological oriented foliation $ {\mathcal F}$ satisfying $\mathrm{dom}({\mathcal F})=\mathrm{dom}(I)$ such that for every $z\in\mathrm{dom}(I)$ the trajectory $I(z)$  is homotopic, relative to the endpoints, to a path $\gamma$ positively transverse to ${\mathcal F}$ and this path is uniquely defined up to equivalence. 
\end{corollary}

We will say that a foliation $\mathcal F$ satisfying the conclusion of Corollary \ref{co: BrEquiv} is {\it transverse to} $I$.  Observe that if $\check M$ is a covering space of $M$ and $\check \pi:\check M\to M$ the covering projection, a foliation $\mathcal F$ transverse to a maximal singular isotopy $I$ lifts to a foliation $\check{\mathcal F}$ transverse to the lifted isotopy $\check I$.

%\textcolor{black}{ The following lemma will be used in Lemma \ref{le: two cases}}
%\begin{lemma}\label{lm:perturbation_foliation}
%\textcolor{black}
%{Given an $\mathcal{F}$ transverse to $I$ and a point $z\in\mathrm{dom}(I)$, there exists $U$ a neighborhood of $z$ in $\mathrm{dom}(I)$ such that if $\mathcal{F}'$ is another oriented foliation such that $\mathcal{F}$ and $\mathcal{F}'$ are equal in $\mathrm{dom}(I)\setminus U$, then $\mathcal{F}'$ is also transverse to $I$.}
%\end{lemma}
%\begin{proof}
%\textcolor{black}{Quote result from references.}
%\end{proof}

We will write $I_{\mathcal{F}} (z)$ for the class of paths that are positively transverse to ${\mathcal F}$, that  join $z$ to $f(z)$ and that homotopic in $\mathrm{dom}(I)$ to $I(z)$, relative to the endpoints. We will also use the notation $I_{\mathcal{F}} (z)$ for every path in this class and called it the {\it transverse trajectory} of $z$. 
Similarly, for every $n\geq1$, one can define $I_{\mathcal{F}}^n
 (z)=\prod_{0\leq k<n} I_{\mathcal{F}} (f^k(z))$, that is either a transverse path passing through the points $z$, $f(z)$, \dots, $f^n(z)$, or a set of such paths. Similarly we will define  $$I_{\mathcal{F}}^{\N}
 (z)=\prod_{0\leq k<+\infty} I_{\mathcal{F}} (f^k(z)),\enskip  I_{\mathcal{F}}^{-\N}
 (z)=\prod_{-\infty<k<0} I_{\mathcal{F}} (f^k(z)), \enskip I_{\mathcal{F}}^{\Z}
 (z)=\prod_{-\infty<k<+\infty} I_{\mathcal{F}} (f^k(z)).$$
The last object will be called the {\it whole transverse trajectory} of $z$.

If $z$ is a periodic point of period $q$, there exists a transverse loop $\Gamma$ whose natural lift $\gamma$ satisfies $\gamma\vert_{[0,1]}= I_{\mathcal{F}}^q
 (z)$. By definition a transverse loop is {\it associate} to $z$ if it is $\mathcal{F}$-equivalent to $\Gamma$  (the definition was given in subsection 3.1). Of course this does not depend on the choices of the $I_{\mathcal{F}} (f^k(z)), 0\leq k<q$. 
 
 Let us state two results that will be useful \textcolor{black}{in Subsection 5.2}.

\begin{lemma}\label{le: continuity}
Fix $z \in{\mathrm{dom}(I)}$,  $n\geq 1$, and parameterize $I_{\mathcal{F}}^n
 (z)$ by $[0,1]$. For every $0<a<b<1$, there exists a neighborhood $V$ of $z$ such that, for every $z'\in V$, the path $I_{\mathcal{F}}^n(z)\vert_{[a,b]}$ is equivalent to a subpath of $I_{\mathcal{F}}^n
 (z')$. Moreover, there exists a neighborhood $W$ of $z$ such that, for every $z'$ and $z''$ in $W$, the path $I_{\mathcal{F}}^n(z')$ is equivalent to a subpath of $I_{\mathcal{F}}^{n+2} (f^{-1}(z''))$

\end{lemma}

\begin{proof} Keep the notations introduced above. Fix a lift $\widetilde z \in \widetilde{\mathrm{dom}(I)}$ of $z$ and denote by $\phi$ and $\phi'$ the leaves of $\widetilde {\mathcal F}$ containing $\widetilde I_{\widetilde{\mathcal{F}}}^n(\widetilde z)(a)$ and $\widetilde I_{\widetilde{\mathcal{F}}}^n(\widetilde z)(b)$ respectively. 
One has
$$\overline {R(\phi_{\widetilde z })}\subset R(\phi)\subset \overline {R(\phi)}\subset R(\phi')\subset \overline {R(\phi')}\subset R(\phi_{\widetilde f^n(\widetilde z) }).$$ If $V\subset\mathrm{dom}(I)$ is a topological disk, small neighborhood of $z$, the lift $\widetilde V$ that contains $\widetilde z$ satisfies
$$\widetilde V \subset R(\phi), \,\widetilde f^n(\widetilde V) \subset L(\phi').$$
 Consequently, for every $z'\in V$, the path
 $I_{\mathcal{F}}^n(z)\vert_{[a,b]}$ is equivalent to a subpath of $I_{\mathcal{F}}^n
 (z')$.
 
Let us prove the second assertion.
 One can find a leaf $\phi$ of the lifted foliation such that
$$\overline {R(\phi_{\widetilde f^{-1}( \widetilde z )})}\subset R(\phi)\subset \overline {R(\phi)}\subset R(\phi_{\widetilde z })$$
and a leaf $\phi'$  such that
$$\overline {R(\phi_{\widetilde f^n( \widetilde z )})}\subset R(\phi')\subset \overline {R(\phi')}\subset R(\phi_{\widetilde f^{n+1}(\widetilde z )}).$$
If $W\subset\mathrm{dom}(I)$ is a topological disk, small neighborhood of $z$, the lift $\widetilde W$ that contains $\widetilde z$ satisfies
$$\widetilde f^{-1}(\widetilde W) \subset R(\phi), \,\widetilde W\subset L(\phi), \,\widetilde f^n(\widetilde W) \subset R(\phi'), \widetilde f^{n+1}(\widetilde W)\subset L(\phi').$$
Consequently, for every $z'$ and $z''$ in $W$, the path $I^n_{\mathcal F}(z')$ is equivalent to a subpath of $I_{\mathcal F}^{n+2}(f^{-1}(z''))$. \end{proof}

Say that $z\in M$ is {\it positively recurrent} if $z\in\omega(z)$, which means that there is a subsequence of the sequence $(f^n(z))_{n\geq 0}$ that converges to $z$. Say that $z\in M$ is {\it negatively recurrent} if $z\in\alpha(z)$, which means that there is a subsequence of the sequence $(f^{-n}(z))_{n\geq 0}$ that converges to $z$.  Say that  $z\in M$ is {\it bi-recurrent} if it is positively and negatively recurrent. An immediate consequence of the previous lemma is the fact  that if $z\in\mathrm{dom}(I)$ is positively recurrent, negatively recurrent or bi-recurrent, then $I_{\mathcal{F}}^{\Z}
 (z)$ is $\mathcal F$-positively recurrent, $\mathcal F$-negatively recurrent or $\mathcal F$-bi-recurrent respectively.

\begin{figure}[ht!]
\includegraphics [height=48mm]{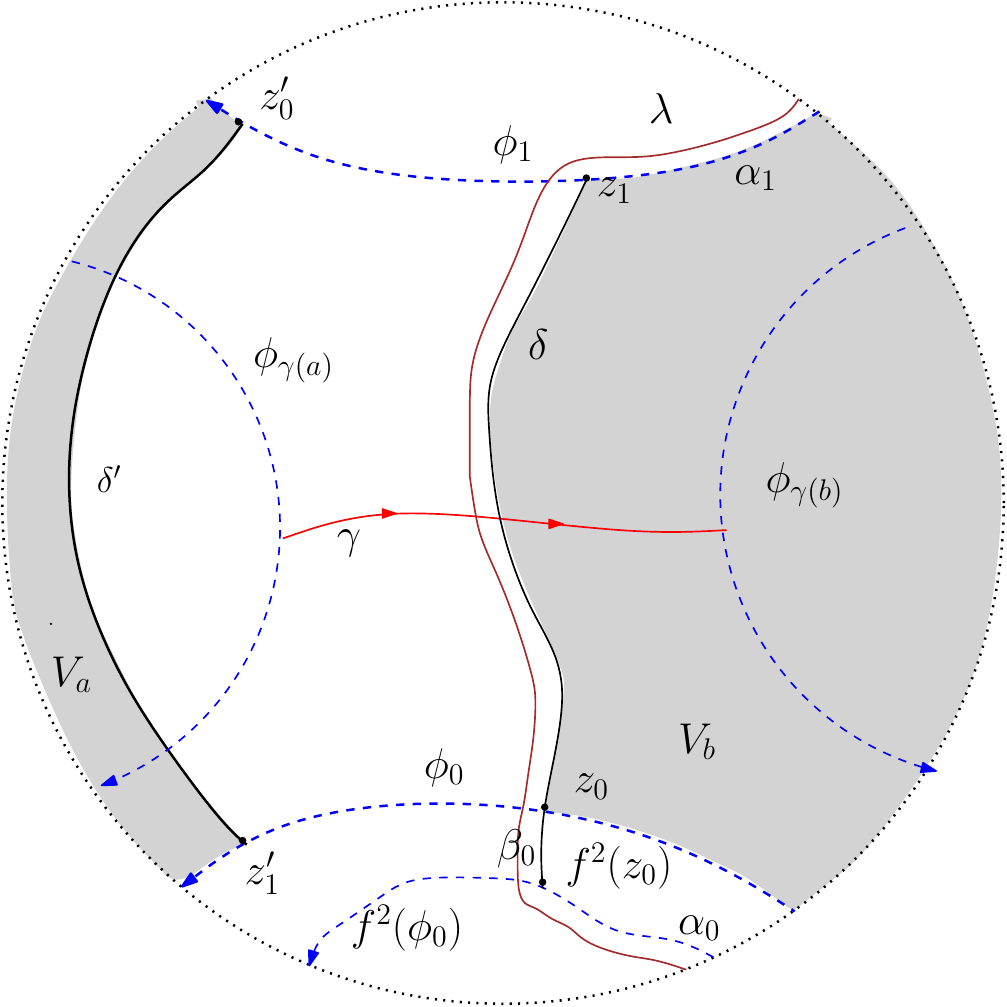}
%[height=28mm]{Figura1.pdf}
%\hfill{}
\caption{\small Lemma \ref{le: neighborhood of infinity}, construction of $V_a$ and $V_b$.}
\label{figurelemma22}
\end{figure}

\begin{lemma}\label{le: neighborhood of infinity}
Suppose that $\gamma:[a,b]\to \mathrm{dom}(I)$ is a transverse path that has a leaf on its right and a leaf on its left. Then, there exists a compact set $K\subset  \mathrm{dom}(I)$ such that \textcolor{black}{for every $n>0$ and } for every transverse trajectory $I^n_{\mathcal F}(z)$ that contains a subpath equivalent to $\gamma$, there exists $k\in\{0,\dots, n-1\}$ such that $f^k(z)$ belongs to $K$. \end{lemma}

\begin{proof}Lifting our path to the universal covering space of the domain, it is sufficient to prove the result in the case where $\mathrm{dom}(I)$ is a plane.

\textcolor{black}{ Figure \ref{figurelemma22} illustrates the following construction.} Suppose that $\phi_0$ is on the right of $\gamma$ and $\phi_1$ on its left and write $W$ for the connected component of the complement of $\phi_0\cup\phi_1$ that contains $\gamma$. Since  $\phi_0$ and $\phi_1$ are Brouwer lines, every orbit that goes from $\overline{R(\phi_{\gamma(a)})}$ to $\overline {L(\phi_{\gamma(b)})}$ is contained in $W$. Let $\delta$ be a simple path that joins a point $z_0$ of $\phi_0$ to a point $z_1$ of $\phi_1$, that is contained in $W$ but the endpoints and that does not meet neither $\overline{R(\phi_{\gamma(a)})}$ nor $\overline {L(\phi_{\gamma(b)})}$. Write $V_b$ for the connected component of $W\setminus\delta$ that contains $\overline {L(\phi_{\gamma(b)})}$. We will extend $\delta$ as a line $\alpha_0\beta_0\delta \beta_1\alpha_1$ as follows. 
If $W$ is contained in $R(\phi_0)$, set $\alpha_0=f^{2}(\phi_{z_0}^-) $ and choose for $\beta_0$ a simple path that joins $f^{2} (z_0)$ to $z_0$ and is contained in $R(f^{2}(\phi_0))\cap L(\phi_0)$ but the endpoints. If $W$ is contained in $L(\phi_0)$,  set $\alpha_0=(\phi_{z_0}^{+})^{-1}$ and $\beta_0=\{z_0\}$.
Similarly, if  $W$ is contained in $R(\phi_1)$, choose for $\beta_1$ a simple path that joins $z_1$ to $f^{2} (z_1)$  and is contained in $L(\phi_1)\cap R(f^{2}(\phi_1))$ but the endpoints and set $\alpha_1=f^{2}(\phi_{z_1}^{+})$. Otherwise, if $W$ is contained in $L(\phi_1)$,  set  $\beta=\{z_1\}$ and $\alpha_1=(\phi_{z_1}^{-})^{-1}$.
Note that $\lambda=\alpha_0\beta_0\delta \beta_1\alpha_1$ is a line.

The image  of  $\beta_0\delta\beta_1$ by $f^{-1}$ is compact and the images of $\alpha_0$ and $\alpha_1$ by $f^{-1}$ are disjoint from $W$.  So, one can find a simple path $\delta'$ that joins a point $z'_0$ of $\phi_0$ to a point $z'_1$ of $\phi_1$, that is contained in $W$ but the endpoints, that does not meet $\overline{V_b}$ and such that the connected component $V_a$ of $W\setminus\delta$ that does not contain $V_b$ (and that  meets $\overline{R(\phi_{\gamma(a)})}$), does not intersect  $f^{-1}(\lambda)$.  This implies that $f(V_a)$ and $V_b$ are separated by $\lambda$ and satisfy $f(V_a)\cap V_b=\emptyset$. So, every orbit that goes from $\overline{R(\phi_{\gamma(a)})}$ to $\overline {L(\phi_{\gamma(b)})}$ has to meet both sets $V_a$ and $V_b$ but is not included in the union of these sets. It must meet the compact set  $K=\overline{W\setminus (V_a\cup V_b)}$.
\end{proof}

\bigskip

\subsection{Admissible paths}

\

Until the end of the whole section, we suppose given a homeomorphism  $f$ isotopic to the identity on an oriented surface $M$ and a maximal singular isotopy $I$. We write  $\widetilde I=(\widetilde f_t)_{t\in[0,1]}$ for the lifted identity defined on the universal covering space $\widetilde{\mathrm{dom}}(I)$ of $\mathrm{dom}(I)$ and set $\widetilde f=\widetilde f_1$ for the lift of $f\vert_{\mathrm{dom}(I)}$ induced by the isotopy. We suppose that $\mathcal F$ is a foliation transverse to $I$ and write $\widetilde{\mathcal F}$ for the lifted foliation on $\widetilde{\mathrm{dom}}(I)$.

We will say that a path $\gamma:[a,b]\to {\mathrm{dom}}(I)$, positively transverse to $\mathcal F$, is {\it admissible of order $n$} if it is equivalent to a path $I_{\mathcal{F}}^n
 (z)$, $z\in\mathrm{dom}(I)$, in the sense defined in subsection 3.1. It means that if $\widetilde \gamma:[a,b]\to \widetilde {\mathrm{dom}}(I)$ is a lift of $\gamma$,  there exists a point $\widetilde z\in\widetilde{\mathrm{dom}}(I)$ such that $\widetilde z\in\phi_{\widetilde \gamma (a)}$ and $\widetilde f^n(\widetilde z)\in\phi_{\widetilde \gamma (b)}$, or equivalently, that 
 $$\widetilde f^n(\phi_{\widetilde \gamma (a)})\cap \phi_{\widetilde \gamma(b)}\not=\emptyset.$$
 
 We will say that $\gamma$ is {\it admissible of order $\leq n$} if it is a subpath of an admissible path of order $n$.
If $\widetilde \gamma:[a,b]\to \widetilde {\mathrm{dom}}(I)$ is a lift of $\gamma$, this means that 
$$\widetilde f^n(\overline{R(\phi_{\widetilde \gamma(a)})})\cap \overline{L(\phi_{\widetilde \gamma(b)})}\not=\emptyset.$$

More generally, we will say that a transverse path $\gamma:J\to{\mathrm{dom}}(I)$ defined on an interval is {\it admissible} if for every segment $[a,b]\subset J$, there exists $n\geq 1$ such that $\gamma\vert_{[a,b]}$ is admissible of order $\leq n$. If $\widetilde \gamma:J\to \widetilde {\mathrm{dom}}(I)$ is a lift of $\gamma$, this means that for every $a<b$ in $J$, there exists $n\geq 1$ such that $$\widetilde f_1^n(\overline{R(\phi_{\widetilde \gamma(a)})})\cap \overline{L(\phi_{\widetilde \gamma(b)})}\not=\emptyset.$$

Similarly, we will say that a transverse loop $\Gamma$ is admissible if its natural lift is admissible. If the context is clear, we will say that a path is of order $n$ (order $\leq n$) if it is admissible of order $n$ (resp. admissible of order $\leq n$).

\medskip
Let us finish this subsection with a useful result which says that except in some particular trivial cases, there is no difference between being of order $\leq n$ and being of order $n$ (and so of being of order $\leq n$ and being of order $m$ for every $m\geq n$).

\begin{proposition}\label{pr: order plane}
  Let $\gamma: [a,b]\to   {\mathrm{dom}}(I)$ be  a transverse path of order $\leq n$  but  not of order $n$, then  $\gamma$ has no leaf on its right and no leaf on its left.
\end{proposition}

\begin{proof}
 \enskip  Lifting the path to the universal covering space of the domain, it is sufficient to prove the result in case where $ {\mathrm{dom}}(I)$ is a plane. By hypothesis, one has:
   $$f^n(\phi_{ \gamma (a)})\cap \phi_{\gamma(b)}=\emptyset, \enskip f^n(\overline{R(\phi_{\gamma(a)})})\cap \overline{L(\phi_{\gamma(b)})}\not=\emptyset.$$ This implies that $f^n(\overline{L(\phi_{\gamma(a)})})\subset L(\phi_{\gamma(b)})$ and $f^{-n}(\overline{R(\phi_{\gamma(b)})})\subset R(\phi_{\gamma(a)})$. Suppose that there exists a leaf $\phi$ in $\overline{L(\phi_{ \gamma(a)})})\cap \overline{R(\phi_{\gamma(b)})}$ that does not meet $\gamma$.  Recall that $\phi$ is a Brouwer line. One of the sets $R(\phi)$ or $L(\phi)$ is included in $\overline{L(\phi_{ \gamma(a)})})\cap \overline{R(\phi_{ \gamma(b)})}$. It cannot be $R(\phi)$, because $f^{-n}(R(\phi))$ would be contained both in $R(\phi)$ and in $R(\phi_{\gamma(a)})$; it cannot be $L(\phi)$, because $f^{n}(L(\phi))$ would be contained both in $L(\phi)$ and in $L(\phi_{\gamma(b)})$. We have a contradiction. \end{proof}

\bigskip

\subsection{The fundamental proposition}

\

The next proposition is a new result about maximal isotopies and transverse foliations. It gives us an operation that permits to construct admissible paths from a pair of admissible paths and its proof is very simple. Nevertheless, this fundamental result will have many interesting consequences.

\begin{proposition}\label{pr: fundamental}
 Suppose that  $\gamma_1: [a_1,b_1]\to M$ and $\gamma_2: [a_2,b_2]\to M$ are transverse paths that intersect $\mathcal{F}$-transversally at $\gamma_1(t_1)=\gamma_2(t_2)$. If $\gamma_1$ is admissible of order $n_1$ and $\gamma_2$ is admissible of order $n_2$, then $\gamma_1\vert_{[a_1,t_1]}\gamma_2\vert_{[t_2,b_2]}$ and $\gamma_2\vert_{[a_2,t_2]}\gamma_1\vert_{[t_1,b_1]}$ are admissible of order $n_1+n_2$. Furthermore, either one of these paths is admissible of order $\min(n_1, n_2)$ or both paths are admissible of order $\max(n_1, n_2)$.
\end{proposition}

\begin{proof}
 \enskip By lifting to the universal covering space of the domain, it is sufficient to prove the result in the case where $M$ is a plane and $\mathcal F$ is non singular. 

\medskip

By Proposition \ref{pr: order plane}, each path 
$\gamma_1$, $\gamma_2$, $\gamma_1\vert_{[a_1,t_1]}\gamma_2\vert_{[t_2,b_2]}$ and $\gamma_2\vert_{[a_2,t_2]}\gamma_1\vert_{[t_1,b_1]}$, having a leaf on its right or on its left,
will be admissible of order $m$ if it is admissible of order $\leq m$. Note first that  for every integers $k_1$, $k_2$ in $\Z$, one has 
$$f^{k_1}(\overline{R(\phi_{ \gamma_1(a_1)})})\cap\ f^{k_2}(\overline{R(\phi_{ \gamma_2(a_2)})})=
f^{k_1}(\overline{L(\phi_{ \gamma_1(b_1)})})\cap f^{k_2}(\overline{L(\phi_{ \gamma_2(b_2)})})=\emptyset.$$ For every $i\in\{1,2\}$ define the sets
$$X_i= f^{n_i}(\overline{R(\phi_{ \gamma_i(a_i)})})\cup \overline{L(\phi_{\gamma_i(b_i)})}, \enskip Y_i= f^{-n_i}(\overline{L(\phi_{ \gamma_i(b_i)})})\cup \overline{R(\phi_{\gamma_i(a_i)})},$$
which are connected according to the admissibility hypothesis.

If $\gamma_1\vert_{[a_1,t_1]}\gamma_2\vert_{[t_2,b_2]}$ is not admissible of order $n_1$, then $X_1\cap \overline{L(\phi_{\gamma_2(b_2)})}=\emptyset$ and so  $X_1$ separates $\overline{R(\phi_{\gamma_2(a_2)})}$ and $\overline{L(\phi_{\gamma_2(b_2)})}$. This implies that none of the sets $X_1\cap X_2$ and $X_1\cap Y_2$ is empty. The first property implies that $f^{n_2}(\overline{R(\phi_{ \gamma_i(a_2)})})\cap \overline{L(\phi_{\gamma_1(b_1)})}\not=\emptyset$, which means that $\gamma_2\vert_{[a_2,t_2]}\gamma_1\vert_{[t_1,b_1]}$ is admissible of order $n_2$.  The second one implies that $ f^{-n_2}(\overline{L(\phi_{ \gamma_2(b_2)})})\cap f^{n_1}(\overline{R(\phi_{ \gamma_1(a_1)})})\not=\emptyset$, which means that $\gamma_1\vert_{[a_1,t_1]}\gamma_2\vert_{[t_2,b_2]}$ is admissible of order $n_1+n_2$.

If $\gamma_1\vert_{[a_1,t_1]}\gamma_2\vert_{[t_2,b_2]}$ is not admissible of order $n_2$, then $Y_2\cap \overline{R(\phi_{\gamma_1(b_1)})}=\emptyset$ and so  $Y_2$ separates $\overline{R(\phi_{\gamma_1(a_1)})}$ and $\overline{L(\phi_{\gamma_1(b_1)})}$. This implies that none of the sets $Y_2\cap Y_1$ and $Y_2\cap X_1$ is empty. The first property implies that  $\gamma_2\vert_{[a_2,t_2]}\gamma_1\vert_{[t_1,b_1]}$ is admissible of order $n_1$.  The second one implies that $\gamma_1\vert_{[a_1,t_1]}\gamma_2\vert_{[t_2,b_2]}$ is admissible of order $n_1+n_2$.

In conclusion, $\gamma_1\vert_{[a_1,t_1]}\gamma_2\vert_{[t_2,b_2]}$ is admissible of order $n_1+n_2$ . Moreover, if it is not admissible of order $\min(n_1,n_2)$ then $\gamma_2\vert_{[a_2,t_2]}\gamma_1\vert_{[t_1,b_1]}$ is admissible of order $\max(n_1,n_2)$. The paths $\gamma_1$ and $\gamma_2$ playing the same role, we get the proposition.\end{proof}
%\hfill$\Box$

\begin{figure}[ht!]
\hfill
\includegraphics[height=48mm]{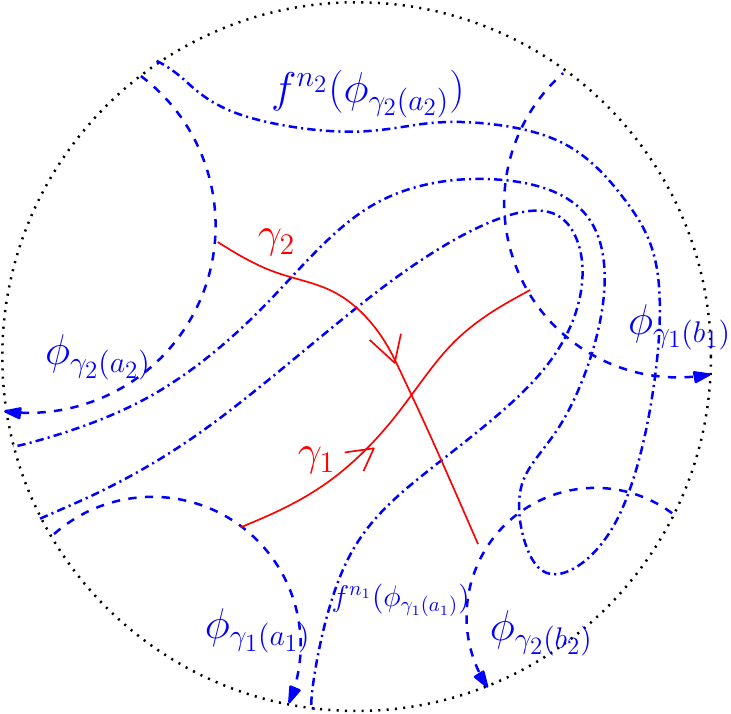}
\hfill
\includegraphics[height=48mm]{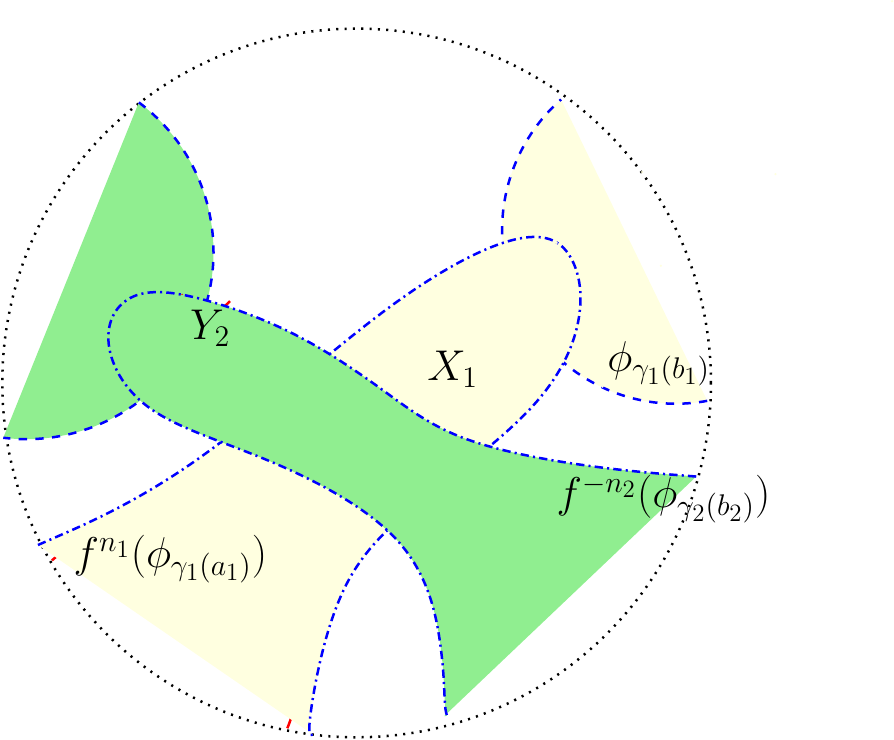}
%[height=28mm]{Figura1.pdf}
\hfill{}
\caption{\small Fundamental lemma, the case where $\gamma_1\vert_{[a_1,t_1]}\gamma_2\vert_{[t_2,b_2]}$ is not admissible of order $n_1$.}
\label{figure2}
\end{figure}

\bigskip
One deduces immediately the following:

\begin{corollary} \label{co: first induction transverse}
 Let  $\gamma_i: [a_i,b_i]\to M$, $1\leq i\leq r$, be a family of $r\geq 2$ transverse paths.  We suppose that for every $i\in\{1,\dots,r\}$ there exist $s_i\in[a_i,b_i]$ and $t_i\in[a_i,b_i]$, such that:

\smallskip
\noindent {\bf i)}\enskip   $\gamma_i\vert_{[s_i,b_i]}$ and $\gamma_{i+1}\vert_{[a_{i+1},t_{i+1}]}$ intersect $\mathcal{F}$-transversally at $\gamma_i(t_i)=\gamma_{i+1}(s_{i+1})$ if $i<r$;

\smallskip
\noindent {\bf ii)} \enskip  one has $s_1=a_1<t_1<b_1$, $a_r<s_r<t_r=b_r$ and  $a_i<s_i<t_i<b_i$ if $1<i<r$;

\smallskip
\noindent {\bf iii)} \enskip    $\gamma_i$ is admissible of order $n_i$.

\medskip Then  $\prod_{1\leq i\leq r} \gamma_i\vert_{[s_i,t_i]}$ is admissible of order $\sum_{1\leq i\leq r} n_i$.
\end{corollary}

\begin{proof}
 \enskip Here again, it is sufficient to prove the result when $M=\R^2$ and $\mathcal F$ is not singular. One must prove by induction on $q\in\{2,\dots, r\}$ that 
$$\left(\prod_{1\leq i<q} \gamma_i\vert_{[s_i,t_i]}\right)\gamma_q\vert_{[s_q,b_q]}$$ is admissible of order $\sum_{1\leq i\leq q} n_i$. The result for $q=2$ is nothing but Proposition \ref{pr: fundamental}. Suppose that it is true for $q<r$ and let us prove it for $q+1$. The paths 
$$\left(\prod_{1\leq i<q} \gamma_i\vert_{[s_i,t_i]}\right)\gamma_q\vert_{[s_q,b_q]}$$ 
and $\gamma_{q+1}$ intersect $\mathcal{F}$-transversally at $\gamma_q(t_q)=\gamma_{q+1}(s_{q+1})$ because this is the case for the subpaths $\gamma_q\vert_{[s_q,b_q]}$ and $\gamma_{q+1}\vert_{[a_{q+1},t_{q+1}]}$. One deduces that $$\left(\prod_{1\leq i\leq q} \gamma_i\vert_{[s_i,t_i]}\right)\gamma_{q+1}\vert_{[s_{q+1},b_{q+1}]}$$ is admissible of order $\sum_{1\leq i\leq q+1} n_i$.
\end{proof}

\bigskip
The following result is more subtle. The $\mathcal{F}$-transverse intersection property is stated on the paths $\gamma_i$ and not on subpaths but the signs of intersection are the same.

\bigskip
\begin{corollary}\label{co: induction transverse}
 Let  $\gamma_i: [a_i,b_i]\to M$, $1\leq i\leq r$, be a family of $r\geq 2$ transverse paths. We suppose that for every $i\in\{1,\dots,r\}$ there  exist $s_i\in[a_i,b_i]$ and $t_i\in[a_i,b_i]$, such that:

\smallskip
\noindent {\bf i)}\enskip   $\gamma_i$ and $\gamma_{i+1}$ intersect $\mathcal{F}$-transversally and positively at $\gamma_i(t_i)=\gamma_{i+1}(s_{i+1})$ if $i<r$;

\smallskip
\noindent {\bf ii)} \enskip  one has $s_1=a_1<t_1<b_1$, $a_r<s_r<t_r=b_r$ and  $a_i<s_i<t_i<b_i$ if $1<i<r$;

\smallskip
\noindent {\bf iii)} \enskip    $\gamma_i$ is admissible of order $n_i$.

\medskip Then $\prod_{1\leq i\leq r} \gamma_i\vert_{[s_i,t_i]}$ is admissible of order $\sum_{1\leq i\leq r} n_i$.
\end{corollary}

\begin{proof}
 \enskip Here again, it is sufficient to prove the result when $M=\R^2$ and $\mathcal F$ is not singular. Here again, one must prove by induction on $q\in\{2,\dots, r\}$ that 
$$\left(\prod_{1\leq i<q} \gamma_i\vert_{[s_i,t_i]}\right)\gamma_q\vert_{[s_q,b_q]}$$  is admissible of order $\sum_{1\leq i\leq q} n_i$ and here again, the case $q=2$ is nothing but  Proposition \ref{pr: fundamental}. Supposing that it is true for $q<r$,  one must prove that 
$$\left(\prod_{1\leq i<q} \gamma_i\vert_{[s_i,t_i]}\right)\gamma_q\vert_{[s_q,b_q]}$$ 
and $\gamma_{q+1}$ intersect $\mathcal{F}$-transversally and positively at $\gamma_q(t_q)=\gamma_{q+1}(s_{q+1})$. By hypothesis, one knows that $\phi_{\gamma_{q+1}(b_{q+1})}$ is above $\phi_{\gamma_q(b_q)}$ relative to $\phi_{\gamma_q(t_q)}$. It remains to prove that $\phi_{\gamma_{q+1}(a_{q+1})}$ is below $\phi_{\gamma_1(a_1)}$ relative to $\phi_{\gamma_q(t_q)}$. For every $i\in\{1,\dots, q-1\}$, the leaves $\phi_{\gamma_i(a_i)}$ and $\phi_{\gamma_{i+1}(a_{i+1})}$ belong to  $R(\phi_{\gamma_i(t_i)})$ and $\phi_{\gamma_{i+1}(a_{i+1})}$ is below $\phi_{\gamma_i(a_i)}$ relative to $\phi_{\gamma_i(t_i)}$. So, each $\phi_{\gamma_i(a_i)}$ belongs to  $R(\phi_{\gamma_q(t_q)})$ and $\phi_{\gamma_{i+1}(a_{i+1})}$ is below $\phi_{\gamma_i(a_i)}$ relative to $\phi_{\gamma_q(t_q)}$. One deduces that $\phi_{\gamma_{q+1}(a_{q+1})}$ is below $\phi_{\gamma_1(a_1)}$ relative to $\phi_{\gamma_q(t_q)}$. \end{proof}

\bigskip
Let us finish by explaining the interest of this result in the case where an admissible transverse path has an $\mathcal{F}$-transverse self-intersection.

\bigskip
\begin{proposition} \label{pr: self-intersection} Suppose that $\gamma:[a,b]\to M$ is a transverse path admissible of order $n$ and that $\gamma$ intersects itself $\mathcal{F}$-transversally at $\gamma(s)=\gamma(t)$ where $s<t$. Then 
$\gamma\vert_{[a,s]}\gamma\vert_{[t,b]}$
is admissible of order $n$  and
$\gamma\vert_{[a,s]}\left(\gamma\vert_{[s,t]}\right)^{q}\gamma\vert_{[t,b]}$
is admissible of order $qn$ for every $q\geq 1$.
\end{proposition}

\begin{proof}

\textcolor{black}{ See Figure \ref{figure3} illustrating the construction below.} Applying Corollary \ref{co: induction transverse} to the family
$$\gamma_i=\gamma,  \,s_i=s \enskip\mathrm{if}\enskip 1<i\leq q \, , \,t_i=t \enskip\mathrm{if}\enskip \enskip 1\leq i<q ,$$ one knows that
$$\gamma\vert_{[a,t]}\left(\gamma\vert_{[s,t]}\right)^{q-2}\gamma\vert_{[s,b]}=\gamma\vert_{[a,s]}\left(\gamma\vert_{[s,t]}\right)^{q}\gamma\vert_{[t,b]}$$
is admissible of order $qn$ for every $q\geq 2$. Moreover the induction argument and the last sentence of Proposition \ref{pr: fundamental} tell us either that $\gamma\vert_{[a,s]}\gamma\vert_{[t,b]}$ is admissible of order $n$, or that $\gamma\vert_{[a,s]}\left(\gamma\vert_{[s,t]}\right)^{q}\gamma\vert_{[t,b]}$
is admissible of order $n$ for every $q\geq 1$. To get the proposition, one must prove that the last case is impossible.

We do not lose any generality by supposing that  $\mathrm{dom}(I)$ is connected.  Fix a lift $\widetilde \gamma$ of $\gamma$ and denote $T$ the covering automorphism such that $\widetilde \gamma(t)=T(\widetilde \gamma(s))$.   The quotient space $\widehat{\mathrm{dom}}({I})=\widetilde{\mathrm{dom}}({I})/T$ is an annulus and one gets an identity isotopy $\widehat I=(\widehat f_t)_{\in[0,1]}$ on  $\widehat{\mathrm{dom}}(I)$ by projection, as a homeomorphism $\widehat f=\widehat f_1$ and a transverse foliation $\widehat{\mathcal F}$. The path $\widetilde\gamma$ projects onto a transverse path $\widehat \gamma$. The path $\widetilde \gamma'=\prod_{k\in \Z} T^k(\widetilde \gamma\vert_{[s,t]})$ is a line that lifts a loop $\widehat \Gamma'$ of $\widehat{\mathrm{dom}}(I)$ transverse to $\widehat {\mathcal F}$. The union of leaves that meet $\widetilde\gamma'$ is a plane $\widetilde U$ that lifts an annulus $\widehat U$ of $\widehat{\mathrm{dom}}(I)$. The fact that $\gamma$ intersects itself $\mathcal{F}$-transversally at $\gamma(t)=\gamma(s)$ means that $\widetilde \gamma$ and $T(\widetilde \gamma)$ intersect $\widetilde{\mathcal{F}}$-transversally at $\widetilde \gamma(t)=T(\widetilde \gamma(s))$. One deduces the following:

\smallskip
\noindent-\enskip the paths $\widetilde\gamma_{[a,s]}$ and $\widetilde\gamma_{[t,b]}$ are not contained in $\widetilde U$;

\smallskip
\noindent-\enskip if $a'\in[a,s)$ is the largest value such that $\widetilde\gamma(a')\not\in \widetilde U$ and $b'\in(t,b]$  the smallest value such that $\widetilde\gamma(b')\not\in \widetilde U$, then $ \widetilde\gamma(a')$ and $\widetilde\gamma(b')$ are in the same connected component of \textcolor{black}{$\widetilde{\mathrm{dom}}(I)\setminus \widetilde \gamma'$.}

The fact that $\gamma\vert_{[a,s]}\left(\gamma\vert_{[s,t]}\right)^{q}\gamma\vert_{[t,b]}$
is admissible of order $n$ implies that $$\widetilde\gamma\vert_{[a,s]}\prod_{0\leq k<q} T^k(\widetilde \gamma\vert_{[s,t]})\,T^{q-1}(\widetilde\gamma\vert_{[t,b]})$$ is admissible of order $n$  or equivalently that 
$\widetilde f^{n}(\phi_{\widetilde \gamma(a')})\cap T^{q-1}(\phi_{\widetilde \gamma(b')})\not=\emptyset$. So one must prove that this cannot happen if $q$ is large enough.

There is no loss of generality by supposing that the leaves $\phi_{\widetilde \gamma(a')}$ and $\phi_{\widetilde \gamma(b')}$ are on the right of $\widetilde \gamma'$.  The projected leaves $\phi_{\widehat \gamma(a')}$ and $\phi_{\widehat \gamma(b')}$ are lines contained in the boundary of $\widehat U$. One can compactify the annulus $\widehat{\mathrm{dom}}(I)$ with a point $S$ at the end on the right of $\widehat \Gamma'$ and  a point $N$ at the end on the left of $\widehat \Gamma'$. We know that the $\alpha$-limit and $\omega$-limit sets of $\phi_{\widehat \gamma(a')}$ and $\phi_{\widehat \gamma(b')}$ are reduced to $S$.  Let us fix $\widetilde z_0\in  \phi_{\widetilde\gamma(a')}$.  
 The sets  $T^k(\widetilde f^n (\overline{R( \phi_{\widetilde\gamma(a')})}))$, $k\in\Z$, are pairwise disjoint and one can choose a simple path $\alpha$ joining $T(\widetilde f^n( \widetilde z_0))$ to $\widetilde f^n( \widetilde z_0)$ and disjoint from $\widetilde f^n(\overline{R(\phi_{\widetilde\gamma(a')})})$ and $T(\widetilde f^n(\overline{R(\phi_{\widetilde\gamma(a')})}))$ but at its ends. One can extend $\alpha$ in  $\overline {L(\phi_{\widetilde\gamma(a')})}\cap T(\overline{L(\phi_{\widetilde\gamma(a')})})$ to a simple path $\alpha'$ joining  $T(\widetilde z_0)$ to $\widetilde z_0$ and disjoint from   $\overline{R(\phi_{\widetilde\gamma(a')})}$ and $T(\overline{R( \phi_{\widetilde\gamma(a')})})$ but at its ends. The path $\alpha''=T(\phi_{\widetilde z_0}^-)\,\alpha'\, \phi_{\widetilde z_0}^+$ is a line and $L(\alpha'')$ contains $T(\widetilde f^n(\phi_{\widetilde z_0}^-))$ and $\widetilde f^n( \phi_{\widetilde z_0}^+)$. Let us choose a real parameterization $t\mapsto \phi_{\widetilde\gamma(b')}(t)$ of $ \phi_{\widetilde\gamma(b')}$. The fact that the $\alpha$-limit and $\omega$-limit sets of $\phi_{\widehat \gamma(b')}$ are reduced to $S$ implies that there exists $K>0$ such that  for every $q\geq 0$, $\alpha'$ does not meet neither $T^q(\phi_{\widetilde\gamma(b')}\vert_{(-\infty, -K]})$ nor $T^q( \phi_{\widetilde\gamma(b')}\vert_{[K, +\infty)})$. One deduces that there exists $q_0$ such that for every $q\geq q_0$, $\alpha'$ does not meet  $T^q(\phi_{\widetilde\gamma(b')})$.  This implies that if $q\geq q_0$, then $T^q(\phi_{\widetilde\gamma(b')})$ does not meet $\alpha''$ and so  is included in $R(\alpha'')$. In particular it cannot intersect neither $T(\widetilde f^n(\phi_{\widetilde z_0}^-))$ nor $\widetilde f^n( \phi_{\widetilde z_0}^+)$. Consequently,  this implies that $\widetilde f^n(\phi_{\widetilde\gamma(a')})$ does not meet $T^q(\widetilde \phi_{\widetilde\gamma(b')})$, if $q\geq q_0$.
\end{proof}

\begin{figure}[ht!]
\hfill
\includegraphics[height=50mm]{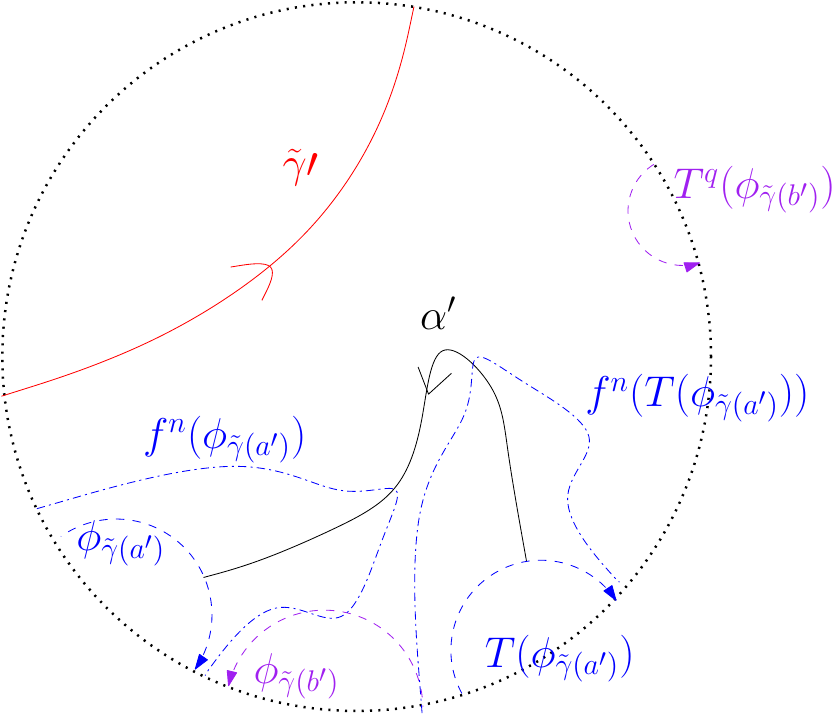}
%[height=28mm]{Figura1.pdf}
\hfill{}
\caption{\small Proof of Proposition \ref{pr: self-intersection}.}
\label{figure3}
\end{figure}

\begin{corollary}   \label{co: no intersection} 
Let $\gamma:[a,b]\to M$ be a transverse path admissible of order $n$. Then, there exists a transverse path of order $n$, $\gamma':[a,b]\to M$ such that $\gamma'$ has no $\mathcal{F}$-transverse self-intersections, and $\phi_{\gamma'(a)}=\phi_{\gamma(a)},\, \phi_{\gamma'(b)}=\phi_{\gamma(b)}$.
\end{corollary}

\begin{proof}
Note first that there exists  a transverse path $\gamma':[a,b]\to M$ equivalent to $\gamma$ with finitely many self-intersections (not necessarily $\mathcal{F}$-transverse). Indeed, choose for every $z$ on $\gamma$, a trivialization neighborhood $W_z$. Divide the interval in $n$ intervals $J_i=[a_i,b_i]$ of equal length and set  $\gamma_i=\gamma_{J_i}$, so that $\gamma=\prod_{1\leq i\leq n}\gamma_i$. If $n$ is large enough, then for every $i$, the union of $\gamma_i$ and all paths $\gamma_j$ that meet $\gamma_i$ is contained in a set $W_{z}$. Let us begin by perturbing each $\gamma_i$  to find an equivalent path $\gamma'_i$, such that $\gamma'_i(b_i)=\gamma'_{i+1}(a_i)$, if $i<n$, and such that the $\gamma'_i(b_i)$ are all distinct. One can also suppose that that for every $i$, the union of $\gamma'_i$ and all  $\gamma'_j$ that meet $\gamma'_i$ is contained in a set $W_{z}$. Suppose that for every $i<i_0$ and every $j\not=i$, the paths $\gamma'_i$ and $\gamma'_j$ have finitely many points of intersection. One can perturb in an equivalent way each $\gamma'_j$ on $(a_j,b_j)$, $j>i_0$,  such that it intersects $\gamma_{i_0}$ finitely many often, without changing the intersection points  with $\gamma_i$ if $i<i_0$ and such that condition concerning the trivialization neighborhoods is still satisfied. One knows that for every $i\leq i_0$ and every $j\not=i$, the new paths $\gamma'_i$ and $\gamma'_j$ have finitely  many points of intersection.

Let $\mathcal{G}$ be the collection of all transverse paths that are admissible of order $n$ whose initial leaf is $\phi_{\gamma(a)}$ and whose final leaf is $\phi_{\gamma(b)}$. Let   $\gamma':[a, b]\to M$ be a path in $\mathcal{G}$ that is minimal  with regards to the number of self-intersections. Then $\gamma'$ has no $\mathcal{F}$-transverse self-intersections. Indeed, if $\gamma'$ had an $\mathcal{F}$-transverse self-intersection at $\gamma'(t)=\gamma'(s)$ where $s<t$, by the Proposition \ref{pr: self-intersection} the path $\gamma'\mid_{[a, s]}\gamma'\mid_{[t, b]}$ would also be also contained in $\mathcal{G}$ and it would have a strictly smaller number of self-intersections.
\end{proof}

\bigskip

\subsection{Realizability of transverse loops}

\medskip
Let $\Gamma$ be a  transverse loop associated to a \textcolor{black}{ periodic point $z$} of period $q$. Recall that it means that $\Gamma$ is equivalent to a transverse loop $\Gamma'$ whose natural lift $\gamma'$ is equivalent to the whole transverse trajectory of $z$. In particular, if $\gamma$ is the natural lift of $\Gamma$, there exists $t\in(-1,0]$ such that  $\phi_{\gamma(t)}=\phi_z$ and such that for every $n\geq 1$, $\gamma_{[t,t+n]}$ is equivalent to
 \textcolor{black}{$I_{\mathcal{F}} ^{nq}(z)$}. So, the loop satisfies the following:
$$(P_q): \enskip\enskip\mathrm{for \enskip every} \enskip n\geq 1,\gamma\vert_{[0,n-1]} \enskip\mathrm {is} \enskip\mathrm{admissible\enskip of \enskip order} \enskip \leq nq.$$ 
The following question is natural:

\medskip

{\it Let $\Gamma$ be a  transverse loop that satisfies $(P_q)$. Is $\Gamma$ associated to a periodic orbit of period $q$?} 

\bigskip

We will see that in many situations, it is the case. In such situations, $f$ will have infinitely many periodic orbits. More precisely, for every rational number $r/s\in(0,1/q]$ written in an irreducible way, the loop $\Gamma^r$ will be associated to a periodic orbit of period $s$. In fact the weaker following property will be sufficient:

\smallskip
\noindent $(Q_q):$ there exist two sequences $(r_k)_{k\geq 0}$ and $(s_k)_{k\geq 0}$ of natural integers satisfying
$$\lim_{k\to+\infty}r_k=\lim_{k\to+\infty}s_k=+\infty,\enskip \limsup_{k\to+\infty}{r_k/s_k}\geq 1/q$$   
such that $\gamma\vert_{[0, r_k]}$ is admissible of order $\leq s_k$. 

We will say that a transverse loop $\Gamma$ is {\it linearly admissible of order $q$} if it satisfies $(Q_q)$ (note that every equivalent loop will satisfy the same condition).

Let us define now the {\it natural covering associated} to $\Gamma$ (or to its natural lift $\gamma$) and introduce some useful notations.  Fix a lift $\widetilde \gamma$ of $\gamma$ and denote $T$ the covering automorphism such that $\widetilde \gamma(t+1)=T(\widetilde \gamma(t))$ for every $t\in\R$. The path $\widetilde \gamma$ is a line and the union of leaves that it crosses is a topological plane $\widetilde U$. 
Moreover it projects onto the natural lift of a loop $\widehat\Gamma$ in the quotient space $\widehat{\mathrm{dom}}(I)=\widetilde{\mathrm{dom}}(I)/T$. One gets an identity isotopy $\widehat I=(\widehat f_t)_{\in[0,1]}$ on $\widehat{\mathrm{dom}}(I)$ by projection, as a homeomorphism $\widehat f=\widehat f_1$ and a transverse foliation $\widehat{\mathcal F}$. The loop $\widehat\Gamma$ is transverse to $\widehat{\mathcal F}$ and the union of leaves that  it crosses is a topological annulus $\widehat U$. 

Before stating the realization result, let us recall the following lemma (for example, see \cite{Lec2}, Theorem 9.1,  for a proof that uses maximal isotopies and transverse foliations). A loop in an annulus will be called {\it essential} if it is not homotopic to zero.

\bigskip
\begin{lemma} \label{le: intersection}
Let $J$ be a real interval, $f$ a homeomorphism of $\T^1\times J$ isotopic to the identity and $\widetilde f$ a lift of $f$ to $\R\times J$. We suppose that:
 
 \medskip
\noindent-\enskip every essential simple loop $\Gamma\subset \T^1\times J$ meets its image by $f$;

\medskip
\noindent{-}\enskip there exist two probability measures $\mu_1$ and $\mu_2$ with compact support, invariant by $f$, such that their rotation numbers (for $\widetilde f$) satisfy $\mathrm{rot}(\mu_1)<\mathrm{rot}(\mu_2)$.

\medskip
Then, for every $r/s\in[\mathrm{rot}(\mu_1),\mathrm{rot}(\mu_2)]$ written in an irreducible way, there exists a point $z\in\R\times J$ such that $\widetilde f^s(z)=z+(r,0).$

\end{lemma}

Let us state now the principal result of this subsection.

\bigskip
\begin{proposition} \label{pr: realization}
 Let $\Gamma$ be a linearly admissible transverse loop of order $q$ that satisfies one of the three following conditions. Then for every rational number $r/s\in(0,1/q]$ written in an irreducible way,  $\Gamma^r$ is associated to a periodic orbit of period $s$. 

\medskip
\noindent{\bf i)}\enskip The loop $\Gamma$ has a leaf on its left and a leaf on its right,  and the annulus $\widehat U$
 does not contain a simple loop homotopic to $\widehat\Gamma$ disjoint from its image by $\widehat f$.

\medskip
\noindent{\bf ii)}\enskip There exists both an admissible transverse path that intersects $\Gamma$ $\mathcal{F}$-transversally and positively, and an admissible transverse path that intersects $\Gamma$ $\mathcal{F}$-transversally and negatively.

\medskip
\noindent{\bf iii)}\enskip The loop $\Gamma$ has an $\mathcal{F}$-transverse self-intersection.
\end{proposition}

\bigskip

\noindent{\it Proof}.  The condition {\bf iii)} is stronger than {\bf ii)} because $\Gamma$ intersects itself $\mathcal{F}$-transversally positively and negatively. The condition {\bf ii)} tells us that there is an admissible transverse path that intersects $\widehat\Gamma$ $\widehat{\mathcal{F}}$-transversally and positively, and an admissible transverse path that intersects $\widehat\Gamma$ $\widehat{\mathcal{F}}$-transversally and negatively. But this implies that {\bf i)} is satisfied because there exists orbits that cross $\widehat U$ in both ways. It remains to prove the result under the assumption   {\bf i)}.

We do not lose any generality by supposing that {$\mathrm{dom}(I)$} is connected, which means that $\widetilde{\mathrm{dom}}(I)$ is a plane and $\widehat{\mathrm{dom}}(I)$ an annulus. By assumption, we know that there exists a leaf on the left of $\widetilde\gamma$ and a leaf on its right. One can compactify $\widehat{\mathrm{dom}}(I)$ with a point $S$ at the end on the right of $\widehat \Gamma$ and a point $N$ at the end on the left of $\widehat \Gamma$. We will denote by $\widehat{\mathrm{dom}}(I)_{\mathrm {sph}}$ this compactification and still write $\widehat f$ for the extension that fixes the added points. The
$\omega$-limit set $\omega(\widehat \phi)$ in $\widehat{\mathrm{dom}}(I)_{\mathrm {sph}}$ of a leaf $\widehat \phi\subset \widehat U$ does not depend on $\widehat \phi$.

 \begin{figure}[t!]
\hfill
\includegraphics[height=45mm]{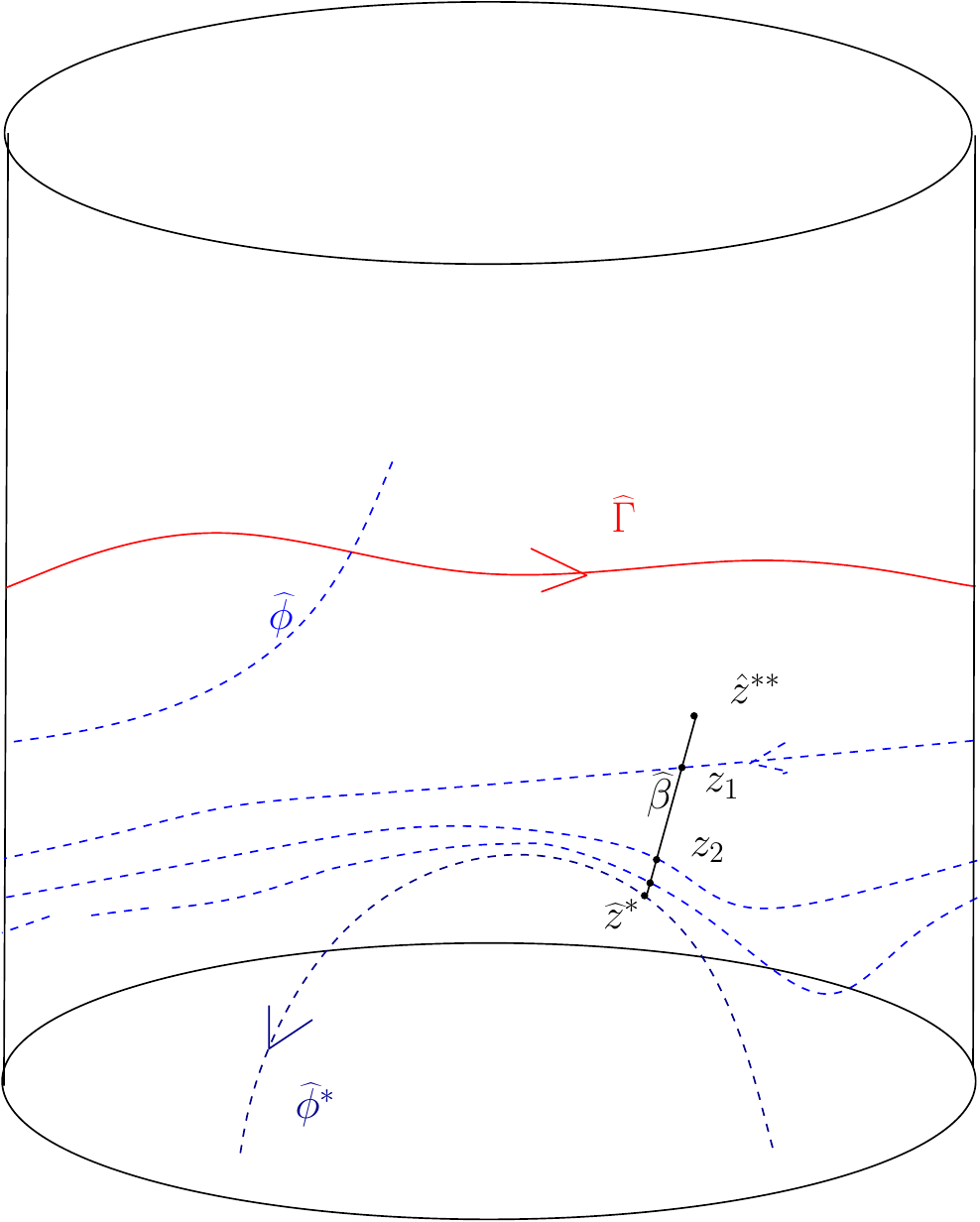}
\hfill
\includegraphics[height=55mm]{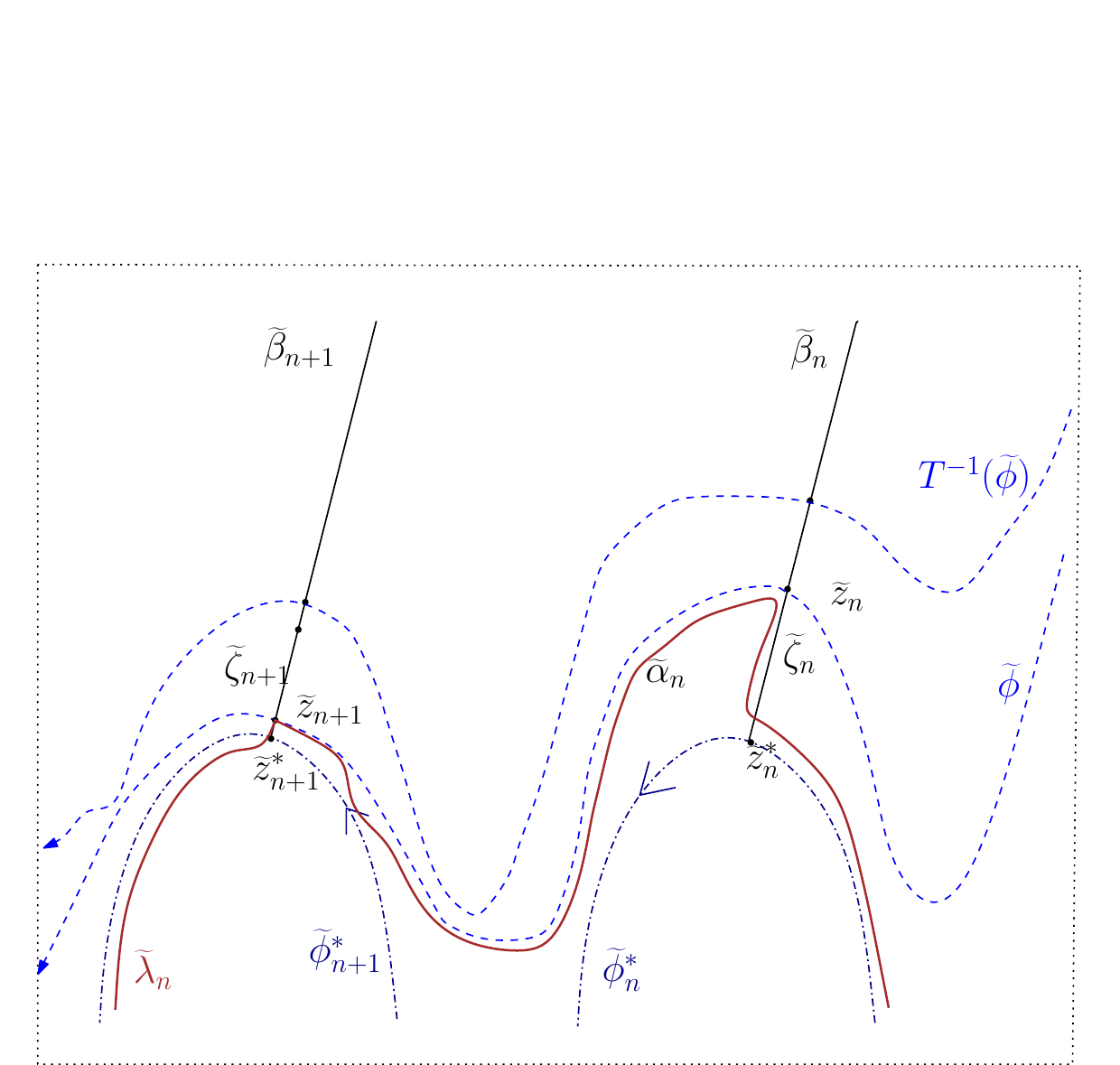}
\hfill{}
\caption{\small Construction of $\widetilde \gamma_n$ \textcolor{black}{ in Lemma \ref{le: two cases}.}}
\label{figure4}
\end{figure}

\bigskip

\begin{lemma} \label{le: two cases} The set $\omega(\widehat \phi)$ is reduced to $S$.
\end{lemma}

\begin{proof} If not,  $\omega(\widehat \phi)$ is either a closed leaf that bounds $\widehat U$ or the union of $S$ and of leaves homoclinic to $S$ (which means that the two limit sets are reduced to $S$). In the first case, the closed leaf that bounds $\widehat U$ is homotopic to  $\widehat \Gamma$ and disjoint from its image by $\widehat f$. A simple loop included in $\widehat U$ sufficiently close will satisfy the same properties. This contradicts the assumptions of the proposition.
 
Let us study now the second case. Choose a point $\widehat z^*\in \omega(\widehat \phi)\setminus\{S\}$ and denote by $\widehat \phi^*$ the leaf that contains $\widehat z^*$. One can suppose that $\widehat \Gamma$ is on the right of $\widehat \phi^*$. This is independent of the choice of $\widehat z^*$ and in that case, the leaf  $\widehat \phi$ is on the right of $\widehat \phi^*$. \textcolor{black}{ Let us present two arguments to deal with this situation. The first argument is the following: As $\widehat{\mathrm{dom}}(I)$ has only finitely many ends and $\widehat{\mathcal{F}}$ is transversal to $\widehat I$, one could adapt the proof of Lemma 3.3 of \cite{Ler} to show that, given any point $\widehat y$ in $\widehat{\mathrm{dom}}(I)$, there exists a neighborhood $V_{\widehat y}$ of $\widehat y$ such that, if $\widehat{\mathcal{F}}'$ is an oriented foliation of $\widehat{\mathrm{dom}}(I)$ that is equal to $\widehat{\mathcal{F}}$ in the complement of $V_{\widehat y}$, then $\widehat{\mathcal{F}}'$ is also transversal to $\widehat I$. Let $V'$ be a neighborhood of $\widehat z^*$ as given by this result, and we further assume that there exists a homeomorphism $h:V'\to [-1,1]\times[-1,1]$ sending the leafs of $\widehat{\mathcal{F}}\cap V'$ onto horizontal line segments oriented from right to left, and such that $h(\widehat z^*)=(0,0)$. Note that $h^{-1}([-1, 1]\times (0,1])\subset \widehat U$, and since $\widehat z^*\in \omega(\widehat \phi)$, there exist $t_1<t_2$ such that $\widehat \phi(t_1), \widehat \phi(t_2)$ both belong to $V'$, and such that $\widehat \phi(s)$ is disjoint from $V'$ if $t_1< s< t_2$. This implies that there exists $x_1, x_2$ in $(0,1)$ such that $h( \widehat \phi(t_1))=(-1, x_1)$ and $h( \widehat \phi(t_2))=(1, x_2)$.  One can find then an oriented foliation $\widehat{\mathcal{F}}'$ of  $\widehat{\mathrm{dom}(I)}$ that agrees with $\widehat{\mathcal{F}}$ in the complement of $h^{-1}([-1, 1]\times (0,1])$, and such that, if $\sigma$ is the line segment in $h(V')$ connecting $(-1, x_1)$ to $(1, x_2)$, then $h^{-1}(\sigma)$ is contained in a single leaf of the foliation. In particular, the leaf $\widehat \phi'$ of $\widehat{\mathcal{F}}'$ that contains $h^{-1}(\sigma)$ is a closed leaf contained in $\widehat U$. Since $\widehat{\mathcal{F}}'$ is also transversal to the isotopy $\widehat I$,  one obtains again a contradiction as in the first case.}

\textcolor{black}{ Since Lemma 3.3 of \cite{Ler} is not stated in the form we used above, let us present a complete argument for the second case:  We}
 have the following result; for every neighborhood $V$ of $S$, there exists a neighborhood $W$ of  $\omega(\widehat \phi)$ in $\widehat{\mathrm{dom}}(I)_{\mathrm {sph}}$ such that $\widehat f(W\setminus V)\cap \overline{\widehat U}=\emptyset$.  Let us consider a simple path $\widehat\beta$ joining a point $\widehat z^{**}\in \widehat U$ to $\widehat z^*$ positively transverse to $\widehat {\mathcal F}$, included (but the end $\widehat z^*$) in $\widehat U$ and sufficiently small that its image by $\widehat f$ will be included in the connected component of $\widehat{\mathrm{dom}}(I)\setminus \widehat \phi^*$ that is on the left of $\widehat \phi^*$.  The leaf $\widehat\phi$ meets $\widehat\beta$ in a``monotone'' sequence $(\widehat z_n)_{n\geq 0}$, where $\lim_{n\to+\infty} \widehat z_n=\widehat z^{*}.$ More precisely, for every real parameterization of $\widehat \phi$, one has $\widehat z_n=\widehat \phi (t_n)$, where $t_{n+1}>t_n$, and $\lim_{n\to+\infty} t_n=+\infty$. Moreover, $\widehat z_{n+1}$ is ``closer'' to $\widehat z^{*}$ than $\widehat z_{n}$ on $\widehat\beta$.  We will prove that if $n$ is large enough, the simple loop $\widehat \Gamma_n$ obtained  by concatenating the segment $\widehat\alpha_n\subset\widehat\phi$ joining $\widehat z_{n}$ to $\widehat z_{n+1}$ and the subpath $\widehat\xi_n$ of $\widehat \beta^{-1}$ joining $\widehat z_{n+1}$ to $\widehat z_{n}$ is disjoint from its image by $\widehat f$\textcolor{black}{, see Figure \ref{figure4} for the following construction.} 
 We will begin by extending $\widehat\beta$ in a simple proper path (with the same name) contained in $\widehat{\mathrm{dom}}(I)\setminus \omega(\widehat \phi)$ ``joining'' the end $N$ to $\widehat z^*$. One can find a neighborhood $W'$ of  $\omega(\widehat \phi)$ in $\widehat{\mathrm{dom}}(I)_{\mathrm {sph}}$ that intersects $\widehat\beta$ only between $\widehat z_0$ and $\widehat z^*$. If $n$ is large enough, $\widehat\alpha_n$ will be contained in $W'$ and so will intersect $\widehat \beta$ only at the points $\widehat z_{n}$ and $\widehat z_{n+1}$. We will suppose $n$ large enough to satisfy this property. Fix a lift $\widetilde z_0$ of $\widehat z_0$, write $\widetilde\beta_0$ for the lift of $\widehat\beta$ that contains $\widetilde z_0$, write $\widetilde z^*_0$ for its end and $\widetilde\phi^*_0$ for the lift of $\widehat\phi^*$ that contains $\widetilde z^*_0$. For every $n\geq 0$ define
$$\widetilde z^*_n=T^{-n}(\widetilde z^*_0), \enskip \widetilde\beta_n=T^{-n}(\widetilde\beta_0), \enskip \widetilde\phi^*_n=T^{-n}(\widetilde\phi^*_0).$$
Write $\widetilde z_n$ for the lift of $\widehat z_n$ that lies on $\widetilde \beta_n$, write $\widetilde\zeta_n$ for the segment of $\widetilde \beta_n$ that joins $\widetilde z_n$ to $\widetilde z^*_n$ and $\widetilde \alpha_n$ for the lift of $\widehat \alpha_n$ that joins $\widetilde z_n$ to  $\widetilde z_{n+1}$. Choose a parameterization $\widehat\phi^*:\R\to\widehat {\mathrm{dom}}(I)$ of $\widehat \phi^*$ sending $0$ onto $\widetilde z^*$ and lift it to parameterize the leaves $ \widetilde\phi^*_n$. We will prove that the line
$$\widetilde\lambda_n=\widetilde \phi^*_n\vert_{(-\infty,0]}\,\widetilde\zeta_n^{-1}\,\widetilde \alpha_n\,\widetilde \zeta_{n+1}\,\widetilde \phi^*_{n+1}\vert_{[0,+\infty]}$$ is a Brouwer line if $n$ is large enough.  Observe first that one has
$$L(\widetilde \phi^*_n)\cup L(\widetilde \phi^*_{n+1})\subset L(\widetilde\lambda_n)\subset L\left(\widetilde \phi^*_n\vert_{(-\infty,0]}\,\widetilde\beta_n^{-1}\right)\cap L\left(\widetilde\beta_{n+1}\, \widetilde\phi^*_{n+1}\vert_{[0,+\infty)}\right), $$ then note that if $K$ is large enough one has
$$\widetilde f^{-1}\left(\widetilde \phi^*_n\vert_{(-\infty,-K]}\right)\subset  R\left(\widetilde \phi^*_n\vert_{(-\infty,0]}\,\widetilde\beta_n^{-1}\right), \enskip \widetilde f^{-1}\left(\widetilde\phi^*_{n+1}\vert_{[K,+\infty]}\right)\subset R\left(\widetilde\beta_{n+1}\, \widetilde\phi^*_{n+1}\vert_{[0,+\infty)}\right).$$
Let $V$ be a neighborhood of $S$ such that $\widehat f(V)\cap \left(\widehat \phi^*([-K,K])\cup\widehat \beta\right)=\emptyset$ and $W$ a neighborhood of $\omega (\widehat \phi)$ such that $f(W\setminus V) \cap \overline {\widehat U}=\emptyset$. If $n$ is large enough, then $\widehat \Gamma_n$ is included in $W$. Let us prove that $\widetilde \lambda_n$ is a Brouwer line of $\widetilde f$ and then that $\widehat\Gamma_n$ is disjoint from its image by $\widehat f$. The leaves $\widetilde \phi^*_n$ and $\widetilde \phi^*_{n+1}$ being Brouwer lines of $\widetilde f$, one has
$$\widetilde f(\widetilde \phi^*_n\vert_{(-\infty,0]})\subset L(\widetilde \phi^*_n) \subset L(\widetilde\lambda_n),\enskip
\widetilde f(\widetilde \phi^*_{n+1}\vert_{[0,+\infty)})\subset L(\widetilde \phi^*_{n+1}) \subset L(\widetilde\lambda_n).$$
By hypothesis on $\widehat \beta$, one knows that
$$\widetilde f(\widetilde \zeta_n)\subset L(\widetilde \phi^*_n) \subset L(\widetilde\lambda_n), \enskip
\widetilde f(\widetilde \zeta_{n+1})\subset L(\widetilde \phi^*_{n+1}) \subset L(\widetilde\lambda_n).$$

%\textcolor{green}{
The path $\widetilde \alpha_n$ being included in a leaf of $\widetilde {\mathcal F}$ and each leaf being a Brouwer line of $\widetilde f$, one knows that 
$$\widetilde f(\widetilde \alpha_n)\cap \widetilde \alpha_n=\emptyset.$$ To prove that $\widetilde \lambda_n$ is a Brouwer line, it remains to prove that $\widetilde f(\widetilde \alpha_n)$ does not meet any of the paths
 $$\widetilde \phi^*_n\vert_{(-\infty,0]}, \enskip \widetilde \phi^*_{n+1}\vert_{[0, +\infty)},\enskip \widetilde \zeta_n,\enskip\widetilde \zeta_{n+1}.$$ By hypothesis, one knows that $\widehat f(\widehat\alpha_n)$ does not meet neither $\widehat \phi^*([-K,K])$, nor $\widehat\zeta_0$. Moreover, one knows that $\widetilde \alpha_n$ does not meet neither $\widetilde f^{-1}(\widetilde \phi^*_n\vert_{(-\infty,K]})$, nor $\widetilde f^{-1}(\widetilde \phi^*_{n+1}\vert_{[K, +\infty)})$. So, we are done.
 %}
 
%\textcolor{green}{
To prove that $\widehat\Gamma_n$ is disjoint from its image by $\widehat f$, one must prove that  $\widehat\Gamma_n$ is lifted to a path that is disjoint from its image by $\widetilde f$. This path will be included in the union of the images by the iterates of $T$ of the path $\widetilde\zeta_n^{-1}\,\widetilde \alpha_n\,\widetilde \zeta_{n+1}$.  So it is sufficient to prove that the union of these translates is disjoint from its image by $\widetilde f$. Observe now that every path $T^k( \widetilde\zeta_n^{-1}\,\widetilde \alpha_n\,\widetilde \zeta_{n+1})$, $k\in\Z$, is disjoint from $L(\widetilde\lambda_n)$, which implies that it is disjoint from $\widetilde f( \widetilde\zeta_n^{-1}\,\widetilde \alpha_n\,\widetilde \zeta_{n+1})$.
%}
\end{proof}

\begin{lemma}\label{le: free loop}
There is no simple loop included in $\widehat{\mathrm{dom}}(I)$ homotopic to  $\widehat \Gamma$ that is disjoint from its image by $\widehat f$.
\end{lemma}

\begin{proof}
 \enskip Suppose that there exists a simple loop $\widehat \Gamma_0$ included in $\widehat{\mathrm{dom}}(I)$ that is homotopic to  $\widehat \Gamma$ and disjoint from its image by $\widehat f$. One can suppose for instance that $\widehat f(\widehat \Gamma_0)$ is included in the component of $\widehat{\mathrm{dom}}(I)_{\mathrm {sph}}\setminus\widehat \Gamma_0$  that contains $N$, and orient $\widehat \Gamma_0$ in such a way that this component, denoted by $L(\widehat \Gamma_0)$, is on the left of $\widehat \Gamma_0$. The loop $\widehat\Gamma_0$ meets finitely many leaves  of $\widehat {\mathcal F}$ homoclinic to $S$ that are on the frontier of $\widehat U$. We denote them $\widehat\phi_i$, $1\leq i\leq p$. Let us prove first that $\widehat\Gamma$ is on the left side of each $\widehat\phi_i$. Indeed, if $\widehat\Gamma$ is on the right side of $\widehat\phi_i$, writing $\widetilde \Gamma_0$ for the lift of  $\widehat \Gamma_0$ and $\widetilde\phi_i$ for a lift of $\widehat\phi_i$, one finds  a non empty compact subset $\overline {L(\widetilde\Gamma_0))}\cap \overline {L(\widetilde\phi_i)}$ of $\widetilde{\mathrm{dom}}(I)$ that is forward invariant by $\widetilde f$. But such a set does not exist because $\widetilde f$ is fixed point free.

Each loop $\widehat\phi_i\cup \{S\}$ bounds a Jordan domain $\widehat L_i$ of $\widehat{\mathrm{dom}}(I)_{\mathrm {sph}}$ that contains $N$. By a classical result of Ker\'ekj\'art\'o \cite{Ke}, one knows that the connected component of $L(\widehat\Gamma_0)\cap (\bigcap_{1\leq i\leq p} \widehat L_i)$ that contains $N$ is a Jordan domain whose boundary $\widehat \Gamma_1$ is a simple loop homotopic to $\widehat \Gamma$ in $\widehat{\mathrm{dom}}(I)$, disjoint from its image by $\widehat f$,  and included in $\overline{\widehat U}\cup L(\Gamma)$ (see Figure \ref{figure_construction_widehatgamma1}). By intersecting $\widehat\Gamma_1$ with the leaves  of $\widehat {\mathcal F}$ homoclinic to $N$ that are on the frontier of $\widehat U$, one constructs similarly  a simple loop $\widehat \Gamma_2$ included in $\widehat{\mathrm{dom}}(I)\cap \overline{\widehat U}$ that is homotopic to  $\widehat \Gamma$ in $\widehat M$ and disjoint from its image by $\widehat f$. It remains to approximate $\widehat\Gamma_2$ by a simple loop included in $\widehat U$ \textcolor{black}{and we get a contradiction since we are assuming condition {\bf i)} in Proposition \ref{pr: realization}}.
\end{proof}

\begin{figure}[t!]

\includegraphics[height=45mm]{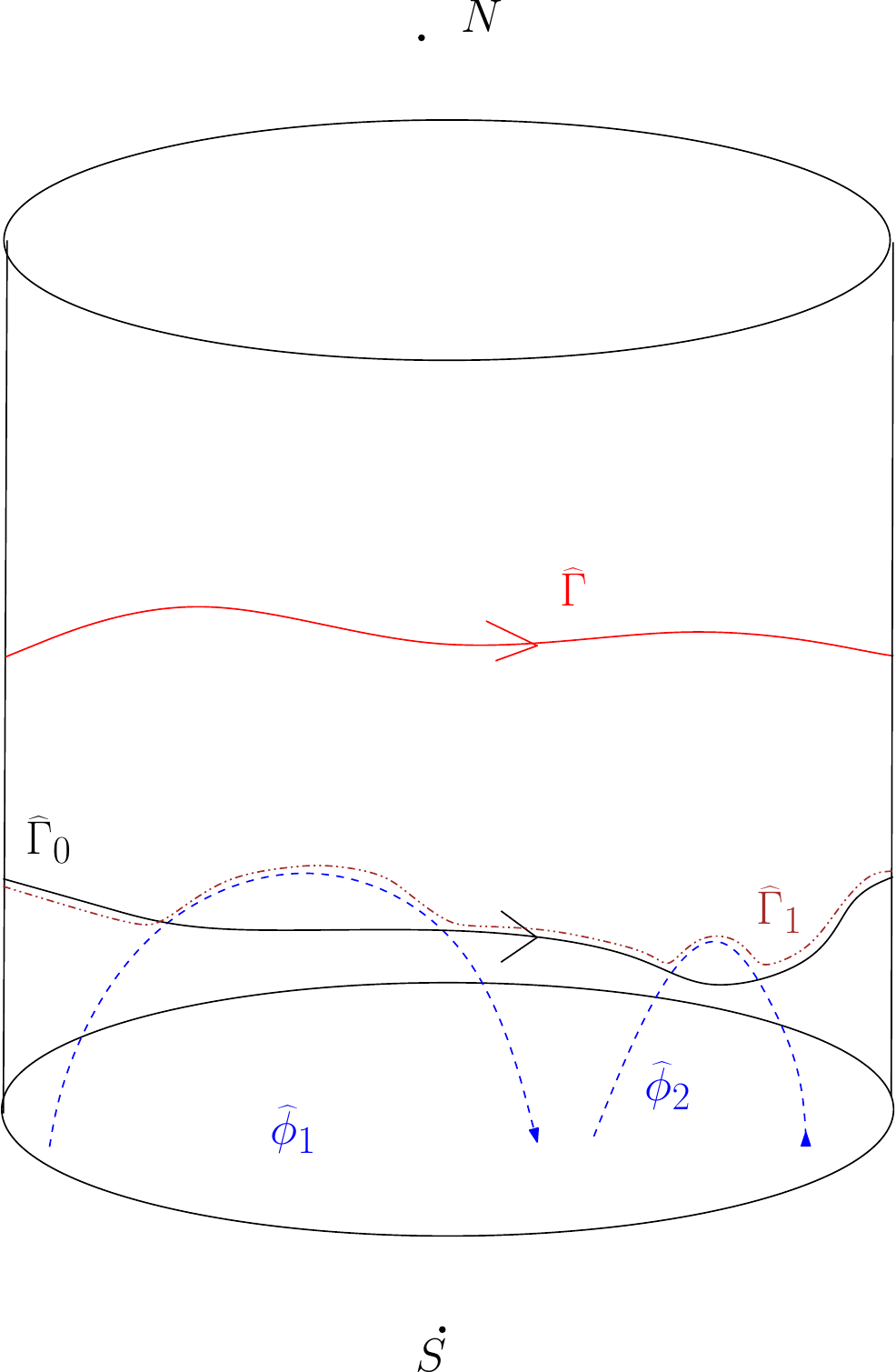}

\caption{\small Construction of $\widehat{\Gamma}_1$ \textcolor{black}{in Lemma \ref{le: free loop}}.}
\label{figure_construction_widehatgamma1}
\end{figure}

\bigskip

\noindent{\it End of the proof of Proposition \ref{pr: realization}}. \enskip One must prove that  for every rational number $r/s\in(0,1/q]$ written in an irreducible way, there exists a point $\widetilde z\in\widetilde{\mathrm{dom}}(I)$ such that $\widetilde f^s(\widetilde z)=T^r(\widetilde z)$. Indeed, the orbit of $\widetilde z$ should be contained in $\widetilde U$, the point $\widetilde z$ will project in ${\mathrm{dom}}(I)$ onto a periodic point $z$ of $f$ of period $s$, finally the loop $\Gamma^r$ will be  associated to $z$.

Write $\widetilde\phi_0$ for the leaf containing $\widetilde \gamma(0)$ and $\widehat\phi_0$ for its projection in $\widehat{\mathrm{dom}}(I)$. Using the analogous of Lemma \ref{le: two cases} for $\alpha$-limit sets, one can suppose that $\widehat\phi_0$ is a line. Let us fix a leaf $\widehat\phi_N$ homoclinic to $N$ and a leaf $\widehat\phi_S$ homoclinic to $S$\textcolor{black}{, which exists since we are assuming that the loop $\Gamma$ has a leaf on its right and a leaf on its left}. Each of them is disjoint from all its images by the (non trivial) iterates of $\widehat f$. By a result of B\'eguin, Crovisier, Le Roux (see \cite{Ler}, Proposition 2.3.3) one knows that there exists a compactification $\widehat{\mathrm{dom}}(I)_{\mathrm {ann}}$ obtained by blowing up the two ends $N$ and $S$ replaced by circles $\widehat\Sigma_N$ and $\widehat\Sigma_S$ such that $\widehat f$ extends to a homeomorphism $\widehat f_{\mathrm {ann}}$ that admits fixed points on each added circle  with a rotation number equal to zero for the lift that extends $\widetilde f$. Moreover, one can suppose that each set $\omega(\widehat\phi_N)$ and $\omega(\widehat\phi_S)$ is reduced to a unique point on $\widehat\Sigma_N$ and $\widehat\Sigma_S$ respectively. One can join a point of $\widehat\phi_N$ to a point of $\widehat\phi_S$ by a segment disjoint from $\widehat \phi_0$. Consequently, one can construct a line $\widehat\lambda$ in $\widehat{\mathrm{dom}}(I)$, disjoint from $\widehat\phi_0$, that admits a limit on each added circle.
Write $\widetilde{\mathrm{dom}}(I)_{\mathrm {ann}}=\widetilde{\mathrm{dom}}(I)\sqcup\widetilde\Sigma_N\sqcup\widetilde\Sigma_S$ for the universal covering space of $\widehat{\mathrm{dom}}(I)_{\mathrm {ann}}$ and keep the notation $T$ for the natural covering automorphism. Write $\widetilde \lambda$ for the  lift of $\widehat\lambda$ located between $\widetilde \phi_0$ and $T(\widetilde \phi_0)$. One can construct a continuous real function $\widetilde g$ on  $\widetilde{\mathrm{dom}}(I)_{\mathrm {ann}}$ that satisfies $\widetilde g(T(\widetilde z)) =\widetilde g(z)+1$ and vanishes on $\widetilde \lambda$. The function $\widetilde h=\widetilde g\circ \widetilde f- \widetilde g$ is invariant by $T$ and lifts a continuous function $\widehat h:\widehat{\mathrm{dom}}(I)_{\mathrm {ann}}\to\R$. If $\mu$ is a Borel probability measure invariant by $\widehat f$, the quantity $ \int_{\widehat{\mathrm{dom}}(I)_{\mathrm {ann}}} h\, d\mu$ is the rotation number of the measure $\mu$ for the lift $\widetilde f_{\mathrm {ann}}$. Let us consider now the real function $\widetilde g_0$ on $\widetilde{\mathrm{dom}}(I)_{\mathrm {ann}}$, that coincides with $\widetilde g$  on $\widetilde\Sigma_N\cup\widetilde\Sigma_S$, that satisfies $\widetilde g_0(T(\widetilde z)) =\widetilde g_0(\widetilde z)+1$ and that vanishes on  $\widetilde \phi_0$ and at every point located between $\widetilde \phi_0$ and $T(\widetilde \phi_0)$. Note that $\widetilde g-\widetilde g_0$ is uniformly bounded by a certain number $K$ and invariant by $T$.  
The property $(Q_q)$ satisfied by $\Gamma$ tells us that for every $k\geq 0$, one can find a point $\widetilde z_k\in R(\widetilde \phi_0)$ such that $\widetilde f^{s_k}(\widetilde z_k)\in L(T^{r_k}(\widetilde \phi_0))$. Write $\widehat z_k$ for its projection in $\widehat{\mathrm{dom}}(I)$. Observe that $$\widetilde g_0(f^{s_k}(\widetilde z_k))-\widetilde g_0(\widetilde z_k)\geq r_k.$$ By taking a subsequence, one can suppose that
$$\lim_{k\to +\infty} {1\over s_k}\left(\widetilde g_0(f^{s_k}(\widetilde z_k))-\widetilde g_0(\widetilde z_k)\right)=\rho\in[{1\over q}, +\infty]$$ and so that
$$\eqalign{\left\vert{1\over s_k} \sum_{i=0}^{s_k-1}\widehat h(\widehat f^{i}(\widehat z_{k}))-\rho\right\vert&= \left\vert{1\over s_k} (\widetilde g(f^{s_k}(\widetilde z_k))-\widetilde g(\widetilde z_k))-\rho\right \vert\cr
&\leq \left\vert{1\over s_k} (\widetilde g_0(f^{s_k}(\widetilde z_k))-\widetilde g_0(\widetilde z_k))-\rho\right \vert+{2K\over s_k}.\cr}$$
Write $\delta_{\widehat z}$ for the Dirac measure at a point $ \widehat z\in\widehat{\mathrm{dom}}(I)_{\mathrm {ann}}$ and choose a measure $\mu$ that is the limit of a subsequence of
$\left({1\over s_k} \sum_{i=0}^{s_k-1}\delta_{\widehat f^{i}_{\mathrm {ann}}(\widehat z_{k})}\right)_{k\geq 0}$ for the weak$^*$ topology. One knows that $\mu$ is an invariant measure of rotation number $\rho$.
As the rotation number induced on the boundary circles are equal to $0$, one deduces that the rotation set $\mathrm{rot}(\widetilde f_{\mathrm {ann}})$ contains $[0,\rho]$. The intersection property
 supposed in {\bf i)} implies by Lemma \ref{le: intersection} that for every rational number $r/s\in(0,1/q]$ written in an irreducible way, there exists a point $\widetilde z\in\widetilde{\mathrm{dom}}(I)_{\mathrm {ann}}$ such that $\widetilde f_{\mathrm {ann}}^s(\widetilde z)=T^r(\widetilde z)$. But this point does not belong to the boundary circles because the induced rotation numbers are equal to $0$. So its belongs to $\widetilde{\mathrm{dom}}({I})$. \hfill$\Box$

\section{\textcolor{black}{ Exponential growth of periodic points and entropy}}

In this section we give a sufficient condition for the exponential growth of periodic points of a surface homeomorphism. This condition will imply that the topological entropy is positive in the compact case. We will make use of these criteria later.

We assume here, as in the previous section, that $f$ is a homeomorphism isotopic to the identity on an oriented surface $M$ and that $I=(f_t)_{t\in[0,1]}$ is a maximal hereditary singular isotopy, which implies that $f_1=f\vert_{\mathrm{dom}(I)}$.  We write  $\widetilde I=(\widetilde f_t)_{t\in[0,1]}$ for the lifted identity defined on the universal covering space $\widetilde{\mathrm{dom}}(I)$ of $\mathrm{dom}(I)$ and set $\widetilde f=\widetilde f_1$ for the lift of $f\vert_{\mathrm{dom}(I)}$ induced by the isotopy. We suppose that $\mathcal F$ is a foliation transverse to $I$ and write $\widetilde{\mathcal F}$ for the lifted foliation on $\widetilde{\mathrm{dom}}(I)$. 

\subsection{Exponential growth of periodic points}

The main result of this section is

\begin{theorem}\label{th:transverse_imply_periodic points}
 Let $\gamma_1, \gamma_2 :\R\to M$ be two admissible positively recurrent transverse paths (possibly equal) with an $\mathcal{F}$-transverse intersection. Then  the number of periodic points of period $n$ for some iterate of $f$ grows exponentially in $n$.\end{theorem}

Theorem \ref{th:transverse_imply_periodic points} is a direct consequence of Lemma \ref{lm:recurrent_to_loop} and Proposition \ref{pr:transverse_loop_exponentialgrowth}.

\begin{lemma}\label{lm:recurrent_to_loop}
Let $\gamma_1,\, \gamma_2$ be two admissible $\mathcal{F}$-positively recurrent transverse paths  (possibly equal) with an $\mathcal{F}$-transverse intersection, and let $I_1$ and $I_2$ be two real segments. Then there exists a linearly admissible transverse loop $\Gamma$ with an $\mathcal{F}$-transverse self-intersection, such that $\gamma_1\vert_{I_1}$ and $\gamma_2\vert_{I_2}$ are equivalent to subpaths of the natural lift of $\Gamma$.
\end{lemma}

\begin{proof}
As explained at the end of subsection 3.3, we can find $a_1, b_1, t_1, a_2, b_2, t_2$ such that $I_1\subset [a_1, b_1]$, $I_2\subset [a_2,b_2]$ and such that $\gamma_1\vert_{[a_1,b_1]}$ intersects $\mathcal{F}$-transversally $\gamma_2\vert_{[a_2, b_2]}$ at $\gamma_1(t_1)=\gamma_2(t_2)$. Since $\gamma_1$ is $\mathcal{F}$-positively recurrent, we can find 
$$b_1<a'_1<b'_1<a''_1<b''_1 $$
such that $\gamma_1\vert_{[a_1,b_1]}$, $ \gamma_1\vert_{[a'_1,b'_1]}$ and $\gamma_1\vert_{[a''_1,b''_1]}$ are equivalent. In particular, there exists
$$a'_1<t'_1<b'_1<a''_1<t''_1<b''_1 $$
 such that

\smallskip
\noindent -\enskip\enskip $\gamma_1\vert_{[a_1,t_1]}$, $ \gamma_1\vert_{[a'_1,t'_1]}$ and $\gamma_1\vert_{[a''_1,t''_1]}$ are equivalent;

\smallskip
\noindent -\enskip\enskip 
$\gamma_1\vert_{[t_1,b_1]}$, $\gamma_1\vert_{[t'_1,b'_1]}$ and $\gamma_1\vert_{[t''_1,b''_1]}$ are equivalent.

Moreover, replacing $\gamma_1$ by an equivalent path, one can suppose that $\gamma_1(t_1)=\gamma_1(t'_1)=\gamma_1(t''_1)$. Since $\gamma_2$ is $\mathcal{F}$-positively recurrent, we can also find
$$b_2<a'_2<t'_2<b'_2<a''_2<t''_2<b''_2 $$
 and replace $\gamma_2$ by an equivalent path such that a similar statement holds with the necessary changes. Note that this implies that $\gamma_1$ is $\mathcal{F}$-transverse to $\gamma_2$ at both $\gamma_1(t''_1)=\gamma_2(t_2)$ and $\gamma_1(t_1)=\gamma_2(t''_2)$. Suppose that $\gamma_1\vert _{[a_1,b''_1]}$ and $\gamma_2\vert _{[a_2,b''_2]}$ are admissible of order $\leq q$ and apply Corollary  \ref{co: first induction transverse} to the families $(\gamma_i)_{1\leq i\leq 2n}$,  $(s_i)_{1\leq i\leq 2n}$, $(t_i)_{1\leq i\leq 2n}$ where
$$\gamma_{2j+1}=  \gamma_1\vert _{[a_1,b''_1]}, \,\gamma_{2j}=  \gamma_2\vert _{[a_2,b''_2]} $$
and
$$s_{2j+1}= t_1\enskip\mathrm{if}\enskip j>0 \,,\,s_{2j}= t_2\,, \,t_{2j+1}= t''_1 \,,\,t_{2j}= t''_2 \enskip\mathrm{if}\enskip j<n.$$
 One deduces that for every $n\ge 1$, 
$$\gamma_1\vert_{[a_1, t_1]}\left( \gamma_1\vert_{[t_1, t''_1]}\gamma_2\vert_{[t_2, t''_2]}\right)^n\gamma_2\vert_{[t''_2, b_2]}$$
is admissible of order $2nq$ and consequently that $\left( \gamma_1\vert_{[t_1, t''_1]}\gamma_2\vert_{[t_2, t''_2]}\right)^n$ is admissible of order $\leq 2nq$. So, the closed path $\gamma'=\gamma_1\vert_{[t_1, t''_1]}\gamma_2\vert_{[t_2, t''_2]}$ defines a loop that is linearly admissible: it satisfies the condition $(Q_{2q})$ stated in the previous section. Furthermore, since both $\gamma_1\vert_{[a'_1,b'_1]}$ and $\gamma_2\vert_{[a'_2, b'_2]}$ are subpaths of $\gamma'$, the induced loop has an $\mathcal{F}$-transverse self-intersection.\end{proof}

%\texcolor{blue}{For the remainder of this subsection, we assume that there exists a linearly transverse loop $\Gamma$ with a \textcolor{blue}{$\mathcal{F}$-transverse} self-intersection.}

\begin{proposition}\label{pr:transverse_loop_exponentialgrowth}
If there exists a linearly admissible transverse loop $\Gamma$ with an $\mathcal{F}$-transverse self-intersection, then the number of periodic points of period n for some iterate of $f$ grows exponentially in $n$. \end{proposition}

{\it Proof.}
The proof of Proposition \ref{pr:transverse_loop_exponentialgrowth} will last until the end of this subsection.  Suppose that $\Gamma$ satisfies the condition $(Q_{q_0})$ and denote $\gamma$ its natural lift. By assumption,  there exist  $s<t$ such that $\gamma$ has an $\mathcal{F}$-transverse self-intersection at $\gamma(s)=\gamma(t)$. So, one can apply the realization result (Proposition  \ref{pr: realization}) and deduce that $\Gamma$ is associated to a fixed point of $f^{q_0}$. Modifying $\Gamma$ in its equivalence class if necessary, one can suppose that for every $r\geq 1$, the path  $\gamma\vert_{[0,r]}$ is admissible of order $rq_0$. Adding the same positive integer to both $s$ and $t$, one can find  a positive integer  $K$ such that $\gamma\vert_{[0, K]}$ has an $\mathcal{F}$-transverse self-intersection at $\gamma(s)=\gamma(t)$ and one knows  that $\gamma\vert_{[0, mK]}$ is admissible of order \textcolor{black}{$mKq_0$} for every $m\geq 1$. \textcolor{black}{To get our proposition}, one needs a preliminary result. Set $$\gamma_1=\gamma\vert_{[s,t]},\,\gamma_2=\gamma_{[t, K+s]}.$$

\begin{lemma}\label{lm:sequence_is_admissible}
For every sequence $(\varepsilon_i)_{i\in\N}\in\{1,2\}^{\N}$, every $n\geq 1$, and every $m\geq 1$ the path
$$\gamma\vert_{[0, s]}\prod_{0\leq i<n} \gamma_{\varepsilon _i}\, \gamma_{[t,  mK]}$$ 
is admissible of order \textcolor{black}{$(n+m)Kq_0$}.
\end{lemma}

\begin{proof}
We will give a proof by induction on $n$. 

Let us begin with the case where $n=1$. If $\varepsilon_0=1$, we must prove that $\gamma\vert_{[0,s]}\gamma\vert_{[s,t]}\gamma\vert_{[t,mK]}=\gamma\vert_{[0,mK]}$ is admissible of order  $ (m+1)\textcolor{black}{Kq_0}$, which is true by \textcolor{black}{Proposition \ref{pr: order plane} as it} is admissible of order $ m\textcolor{black}{Kq_0}$). If $\varepsilon_0=2$, we must prove that $$\gamma\vert_{[0,s]}\gamma\vert_{[t,s+K]}\gamma\vert_{[t,mK]}=\gamma\vert_{[0,s]}\gamma\vert_{[t,s+K]}\gamma\vert_{[t+K,(m+1)K]}$$ is admissible of order  $(m+1)\textcolor{black}{Kq_0}$. The path $\gamma\vert_{[0,(m+1)K]}$ having an $\mathcal{F}$-transverse self-intersection at $\gamma(t)=\gamma(s)$ and being admissible of order  $\leq (m+1)\textcolor{black}{Kq_0}$, one deduces by Proposition  \ref{pr: self-intersection} that $\gamma\vert_{[0,s]}\gamma\vert_{[t,(m+1)K]}$ is admissible of order  $(m+1)\textcolor{black}{Kq_0}$. This last path has an $\mathcal{F}$-transverse self-intersection at $\gamma(t+K)=\gamma(s+K)$. Applying Proposition  \ref{pr: self-intersection} again, one deduces that $\gamma\vert_{[0,s]}\gamma\vert_{[t,s+K]}\gamma\vert_{[t+K,(m+1)K]}$ is admissible of order  $(m+1)\textcolor{black}{Kq_0}$.

\medskip
Suppose now the result proved for $n$.   There are three cases to consider.

\medskip

\textcolor{black}{The first case  is if} the sequence $(\varepsilon_i)_{0\leq i\leq n}$ is constant equal to $1$. We apply  Corollary \ref {co: induction transverse} to the families $(\gamma_i)_{1\leq i\leq n+1}$,  $(s_i)_{1\leq i\leq n+1}$, $(t_i)_{1\leq i \leq n+1}$, where
$$\gamma_{i}= \gamma\vert _{[0,K]}\enskip\mathrm{if }\enskip i\leq n, \enskip
\gamma_{n+1}=\gamma\vert _{[0,mK]}
 $$
and
$$s_{i}= s \enskip\mathrm{if}\enskip i>0 \,,\,t_{i}= t \enskip\mathrm{if}\enskip i\leq n \,$$
 One deduces that
 $$\gamma\vert_{[0,t]} \left(\gamma\vert_{[s,t]}\right)^{n-1}\gamma_{[s,mK]}=\gamma\vert_{[0,s]} \left(\gamma\vert_{[s,t]}\right)^{\textcolor{black}{n}}\gamma\vert_{[t,mK]}$$
is admissible of order $(m+n)\textcolor{black}{Kq_0}$ and so is admissible of order  $(m+n+1)\textcolor{black}{Kq_0}$.

\medskip

\textcolor{black}{ The second case to consider is if} there exists $n'<n$ such that $\varepsilon_{n'}=2$ and $\varepsilon _i=1$ if $i>n'$. We apply  Corollary \ref {co: induction transverse} to the families $(\gamma_i)_{1\leq i\leq n-n'+1}$,  $(s_i)_{1\leq i\leq n-n'+1}$, $(t_i)_{1\leq i \leq n-n'+1}$ where
$$\gamma_{0}= \gamma\vert_{[0,s]}\prod_{0\leq i<n'} \gamma_{\varepsilon_i} \gamma\vert_{[t,2K]}\,,\, \gamma_{i}= \gamma\vert _{[K,2K]}\enskip\mathrm{if }\enskip 1<i\leq n-n', \enskip
\gamma_{n-n'+1}=\gamma\vert _{[K,(m+1)K]}
 $$
and
$$s_{i}= s+K \enskip\mathrm{if}\enskip i>0 \,,\,t_{i}= t +K\enskip\mathrm{if}\enskip i \leq n-n' \,$$
The induction hypothesis tells us that $\gamma_{0}= \gamma\vert_{[0,s]}\prod_{0\leq i<n'} \gamma_{\varepsilon_i} \gamma_{[t,2K]}$ is  admissible of order $(n'+2)\textcolor{black}{Kq_0}$, so 
 $$\gamma\vert_{[0,s]}\prod_{0\leq i<n'} \gamma_{\varepsilon_i} \gamma\vert_{[t,s+K]}\left(\gamma\vert_{[s+K,t+K]}\right)^{n-n'+1}\gamma\vert_{[t+K,(m+1)K]}=\gamma\vert_{[0,s]}\prod_{0\leq i\leq n} \gamma_{\varepsilon_i} \gamma\vert_{[t,mK]}$$
is admissible of order $(n'+2)\textcolor{black}{Kq_0}+(n-n'-1)\textcolor{black}{Kq_0}
+m\textcolor{black}{Kq_0}=(m+n+1)\textcolor{black}{Kq_0} $.

\textcolor{black}{ The final case to consider is if}  $\varepsilon_n=2$. We must prove that $$\gamma\vert_{[0,s]}\prod_{0\leq i<n} \gamma_{\varepsilon_i} \gamma\vert_{[t,s+K]}\gamma\vert_{[t,mK]}=\gamma\vert_{[0,s]}\prod_{0\leq i<n} \gamma_{\varepsilon_i} \gamma\vert_{[t,s+K]}\gamma\vert_{[t+K,(m+1)K]}$$ is admissible of order $(m+n+1)\textcolor{black}{Kq_0}$. The path $\gamma\vert_{[K,(m+1)K]}$ having an $\mathcal{F}$-transverse self-intersection at $\gamma(t+K)=\gamma(s+K)$, the same is true when we extend this path on the left by adding $\gamma\vert_{[0,s]}\prod_{0\leq i<n} \gamma_{\varepsilon_i}\gamma\vert_{[t,K]}$. Moreover by induction hypothesis, one knows that
$$\gamma\vert_{[0,s]}\prod_{0\leq i<n} \gamma_{\varepsilon_i}\gamma\vert_{[t,K]}\gamma\vert_{[K,(m+1)K]}$$  is admissible of order $(m+n+1)\textcolor{black}{Kq_0}$.
Applying Proposition  \ref{pr: self-intersection}, one deduces that $$\gamma\vert_{[0,s]}\prod_{0\leq i<n} \gamma_{\varepsilon_i} \gamma\vert_{[t,s+K]}\gamma\vert_{[t+K,(m+1)K]}$$  is admissible of order $(m+n+1)\textcolor{black}{Kq_0}$. \end{proof}

\begin{figure}[ht!]
\hfill
\includegraphics [height=48mm]{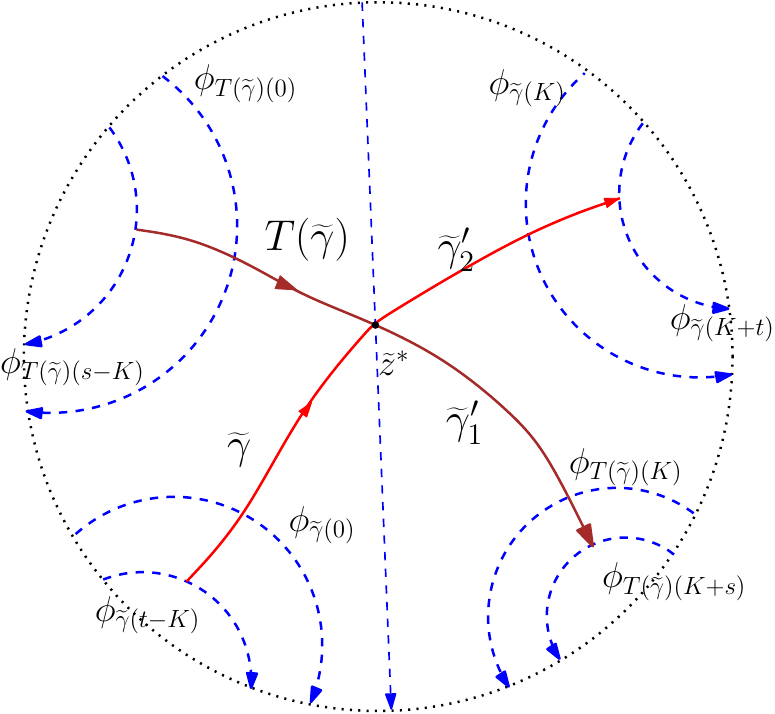}
\hfill{}
\caption{\small Relative position of the leafs of the endpoints of $\widetilde \gamma_1$ and $\widetilde \gamma_2$, when both start at $\widetilde z^* $\textcolor{black}{ in Lemma \ref{lm:associated_periodic_point}}.}
\label{Figure_exponentialgrowth}
\end{figure}

\begin{lemma}\label{lm:associated_periodic_point}
Let $\mathbf{e}= (\varepsilon_i)_{i\in\N}\in\{1,2\}^{\N}$ be a periodic word of period $q$ which is not periodic of period 1. Then the loop $\Gamma_{\mathbf e}$, defined by the closed path $\prod_{0\leq i<q} \gamma_{\varepsilon_i}$, is equivalent to a transverse loop associated to a fixed point of \textcolor{black}{$f^{qKq_0}$}.
\end{lemma}

\begin{proof}
Lemma \ref{lm:sequence_is_admissible} tells us that, for every $n$, the path
$$\gamma\vert_{[0, s]}\left(\prod_{0\leq i< q} \gamma_{\varepsilon _i}\right)^n\, \gamma_{[t, K]}$$ 
is admissible of order \textcolor{black}{$(1+qn)Kq_0$} and consequently that  $\left(\prod_{0\leq i<q} \gamma_{\varepsilon _i}\right)^n$ is admissible of order \textcolor{black}{$(1+qn)Kq_0$}. So $\Gamma_{\mathbf e}$ is linearly admissible: it satisfies the condition $\textcolor{black}{(Q_{qKq_0})}$. Note that, since $\Gamma_{\mathbf{e}}$ is not a constant sequence, it has a self-intersection. The \textcolor{black}{lemma} follows by applying the realization result (Proposition  \ref{pr: realization}).
\end{proof}

Consider now the paths 
$$\gamma_1'=\gamma_1\gamma_2=\gamma_{[s, K+s]}, \,\,\, \gamma_2'=\gamma_2\gamma_1=\gamma\vert_{[t, K+t]}.$$ 

Since $\gamma\vert_{[0,K]}$ has a $\mathcal F$-transverse self-intersection at $z^*=\gamma(s)=\gamma(t)$, then for every lift $\widetilde\gamma$ of $\gamma$, there exists a covering automorphism $T$ such that $\widetilde\gamma\vert_{[0,K]}$ and $T(\widetilde\gamma)\vert_{[0,K]}$ have a $\widetilde{\mathcal F}$-transverse intersection at \textcolor{black}{$\widetilde\gamma(t)=T(\widetilde\gamma)(s)$}. Consequently, $\widetilde\gamma\vert_{[0,K+t]}$ and $T(\widetilde\gamma)\vert_{[0,K+s]}$ have a $\widetilde{\mathcal F}$-transverse intersection at \textcolor{black}{$\widetilde \gamma(t)=T(\widetilde\gamma)(s)$}. This implies that among the leaves $\phi_{\widetilde\gamma(K+t)}$ and $\phi_{T(\widetilde\gamma)(K+s)}$, one is above the other one relative to $\phi_{\widetilde\gamma(t)}=\phi_{T(\widetilde\gamma)(s)}$. This means that if $\widetilde\gamma'_1$ and $\widetilde\gamma'_2$ lift $\gamma'_1$ and $\gamma'_2$ respectively and start from the same point $\widetilde z^*$, then the leaf containing the ending point of $\widetilde\gamma'_1$ is either above or below the leaf containing the ending point of $\widetilde\gamma'_2$  (relative to $\phi_{\widetilde z}$), \textcolor{black}{see Figure \ref{Figure_exponentialgrowth}}. We do not loose any generality by supposing it is the former, which means that $\widetilde\gamma\vert_{[0,K]}$ and $T(\widetilde\gamma)\vert_{[0,K]}$ have a positive $\widetilde{\mathcal F}$-transverse intersection at $\gamma(t)=T(\widetilde\gamma)(s)$. In this situation $\phi_{\widetilde\gamma(0)}$ is above $\phi_{T(\widetilde\gamma)(0)}$ relative to $\phi_{\widetilde\gamma(t)}$, so $\phi_{\widetilde\gamma(t-K)}$ is above $\phi_{T(\widetilde\gamma)(s-K)}$ relative to $\phi_{\widetilde\gamma(t)}$. 

This means that, if $\widetilde\gamma'_1$ and $\widetilde\gamma'_2$ lift $\widetilde\gamma_1$ and $\widetilde\gamma_2$ respectively and end at the same point $\widetilde z^*$, then the leaf containing the starting point of $\widetilde\gamma'_1$ is below the leaf containing the ending point of $\widetilde\gamma'_2$ (relatively to $\phi_{\widetilde z}$).

We say that a finite word $\mathbf e= (\varepsilon_i)_{0\le i < 2n}\in\{1,2\}^{2n}$ is a {\it palindromic word} of length $2n$ if it satisfies $\varepsilon_{n+j}=\varepsilon_{n-j-1},\, 0\le j < n$. Let us fix a base point $\widetilde z^*$ projecting on $z^*$.  To each palindromic word $\mathbf e$ of length $2n$, we associate the loop $\Gamma'_{\mathbf e}$ naturally defined by the closed path $\prod_{0\leq i< 2n} \gamma'_{\varepsilon _i}$ and the lift $\widetilde \gamma'_{\mathbf e}=\widetilde \gamma'_{\mathbf e}{}^{-}\widetilde \gamma'_{\mathbf e}{}^{+}$ of $\prod_{0\leq i< 2n} \gamma'_{\varepsilon _i}$, where $\widetilde \gamma'_{\mathbf e}{}^{-}$ is the lift of 
$\prod_{0\leq i< n} \gamma'_{\varepsilon _i}$ ending at $\widetilde z^*$ and $\widetilde \gamma'_{\mathbf e}{}^{+}$ the lift of 
$\prod_{n\leq i< 2n-1} \gamma'_{\varepsilon _i}$ starting at $\widetilde z^*$. The ending point of $\widetilde\gamma'_{\mathbf e}$ is the image of its starting point by a covering automorphism that we denote $T_{\mathbf e}$. We define the path $\widetilde\gamma'{}^2_{\mathbf e}=
\widetilde\gamma'_{\mathbf e}T_{\mathbf e}(\widetilde\gamma'_{\mathbf e})$ and the line $\widetilde\gamma'{}^{\infty}_{\mathbf e}=\prod_{k\in\Z}T^k_{\mathbf e}(\widetilde\gamma'_{\mathbf e})$, which is a lift of $\Gamma'_{\mathbf e
}$.

\begin{lemma}\label{lm:intersect-transversally}
If $\mathbf e\not=\mathbf{e'}$ are two palindromic words of the same length, then the paths $\widetilde \gamma'_{\mathbf e}$ and $\widetilde \gamma'_{\mathbf{ e'}}$ intersect $\widetilde{\mathcal F}$-transversally at $\widetilde z$.
\end{lemma}

\begin{proof}
If ${\bf e}\not={\bf e'}$, there exists $k\in\{0,\dots, n-1\}$ such that $\varepsilon_{n+j}=\varepsilon'_{n+j}$ if $0\leq j<k$ and $\varepsilon_{n+k}\not=\varepsilon'_{n+k}$. Let us suppose for example that $\varepsilon_{n+k}=1$ and $\varepsilon'_{n+k}=2$. The paths $\prod_{0\leq j<k}\gamma'_{\varepsilon_{n+j}}$ and $\prod_{0\leq j<k}\gamma'_{\varepsilon'_{n+j}}$ are equal, as are the lifts starting from $\widetilde z^*$. Let us write $\widetilde z$ for the ending point of the common lift. The leaf containing the ending point of the lift of $\prod_{0\leq j\leq k}\gamma'_{\varepsilon_{n+j}}$ starting from $\widetilde z^*$ is above (relative to $\phi_{\widetilde z}$ but also relative to $\phi_{\widetilde z^*}$) the leaf containing the ending point of the lift of $\prod_{0\leq j\leq k}\gamma'_{\varepsilon'_{n+j}}$ starting from $\widetilde z^*$. So, we have a similar result replacing $\prod_{0\leq j\leq k}\gamma'_{\varepsilon_{n+j}}$  with the extension $\widetilde \gamma_{\mathbf e}^{+}$ and $\prod_{0\leq j\leq k}\gamma'_{\varepsilon'_{n+j}}$  with the extension $\widetilde \gamma_{\mathbf e'}^{+}$. One proves similarly that the leaf containing the starting point of $\widetilde \gamma_{\mathbf e}^{-}$ is below \textcolor{black}{(relative to} $\phi_{\widetilde z^*}$) the leaf containing the starting point of $\widetilde \gamma_{\mathbf e'}^{-}$.

\end{proof}

By \textcolor{black}{ Lemma \ref{lm:associated_periodic_point}}, for every palindromic word $\mathbf e$ of length $2n$, there exists a fixed point $z_{\mathbf e}$ of \textcolor{black}{$f^{4nKq_0}$} such that $\Gamma'_{\mathbf e}$ is associated to  $z_{\mathbf e}$. 
%The next result immediately implies Theorem \ref{th:transverse_imply_periodic points}.

\begin{lemma}\label{lm:bounded}
There exists a constant $L>0$ such that, given a palindromic word $\mathbf e$ of length $2n$, there are at most $L n^2$ different palindromic words  $\mathbf{e'}$ of length $2n$ such that $\Gamma'_{\mathbf e}$ and $\Gamma'_{\mathbf e'}$ are equivalent. 
\end{lemma}

\begin{proof}

Let $\widetilde\gamma'_1$ and $\widetilde\gamma'_2$ be two respective lifts of  $\gamma'_1$ and $\gamma'_2$  to $\widetilde {\mathrm{dom}}(I)$. The group of covering automorphisms acts freely and properly. So there exists a constant $L'$ such that there are at most $L'$ automorphisms $S$ such that $\widetilde\gamma'_1\cap S(\widetilde\gamma'_1)\not=\emptyset$, at most $L'$ automorphisms $S$ such that $\widetilde\gamma'_2\cap S(\widetilde\gamma'_2)\not=\emptyset$ and at most $L'$ automorphisms $S$ such that $\widetilde\gamma'_1\cap S(\widetilde\gamma'_2)\not=\emptyset$. Of course, $L'$ is independent of the choices of $\widetilde\gamma'_1$ and $\widetilde\gamma'_2$. We deduce that for every palindromic word ${\bf e}$ of length $2n$, there are at most $8L'n^2$ automorphisms $S$ such that $\widetilde\gamma'_{\bf e}\cap S(\widetilde\gamma'{}^2_{\bf e})\not=\emptyset$. This implies that there are at most $8L'n^2$ automorphisms $S$ such that $\widetilde\gamma'_{\bf e}$ and $S(\widetilde\gamma'{}^2_{\bf e})$ have a $\widetilde{\mathcal{F}}$-transverse intersection. 

Suppose that $\mathbf e$ and $\mathbf e'$ are two palindromic words of length $2n$ such that $\Gamma'_{\mathbf e}$ and $\Gamma'_{\mathbf e'}$ are equivalent. There exists a covering automorphism $S_{\mathbf e'}$ such that $\widetilde\gamma'{}^{\infty}_{\mathbf e'}$ is equivalent to $S_{\mathbf e'}(\widetilde\gamma'{}^{\infty}_{\mathbf e})$ and such that $S_{\mathbf e'}\circ T_{\mathbf e}\circ S_{\mathbf e'}^{-1}=  T_{\mathbf e'}$.  Composing $S_{\mathbf e'}$ on the left by a power of $T_{\mathbf e'}$ if necessary, one can suppose that $\widetilde\gamma'_{\mathbf e'}$ is equivalent to a subpath of $S_{\mathbf e'}(\widetilde\gamma'{}^2_{\mathbf e})$. By Lemma \ref{lm:intersect-transversally}, one deduces that $\widetilde\gamma'_{\mathbf e}$ and $S_{\mathbf e'}(\widetilde\gamma'{}^2_{\mathbf e})$ intersect $\widetilde{\mathcal F}$-transversally. It remains to prove that  $S_{\mathbf e'}\not=S_{\mathbf e''}$ if $\mathbf e'\not =\mathbf e''$. But if this the case, then $\widetilde\Gamma'_{\mathbf e'}$ and $\widetilde\Gamma'_{\mathbf e''}$ are equivalent, which is impossible because this two paths intersect $\widetilde{\mathcal{F}}$-transversally at $\widetilde z$.
\end{proof}

\textcolor{black}{Since there exists $2^{n}$ different palindromic words of length $2n$, one concludes by Lemmas \ref{lm:associated_periodic_point} and \ref{lm:bounded} that $f^{4nKq_0}$ has at least $\frac{2^{n}}{L n^2}$ distinct fixed points, proving Proposition \ref{pr:transverse_loop_exponentialgrowth}.}

%The paths $\Gamma_ {\mathbf{e}}$ and $\Gamma_ {\mathbf{e'}}$ are not equivalent if $\mathbf{e}$ and $\mathbf{e'}$ are two periodic words corresponding to different orbits of the Bernoulli shift. So the associated   periodic orbits are distinct. The number of periodic points of $f^r$ of period $q\geq 2$ is at least equal to the number of periodic points of period $q$ of the Bernoulli shift.

\subsection{Topological entropy}

\bigskip

By the previous result, it is  natural to believe that the topological entropy is positive, in case $M$ is compact. The next result asserts that this is the case: 

\begin{theorem}\label{th:transverse_imply_entropy}
 Let $M$ be a compact surface, $\gamma_1, \gamma_2 :\R\to M$ be two admissible $\mathcal{F}$-positively recurrent transverse paths (possibly equal) with an $\mathcal{F}$-transverse intersection. Then the topological entropy of $f$ is positive.
\end{theorem}

\begin{remark} As we will see in the proof, Theorem  \ref{th:transverse_imply_entropy} will be stated even in case where $M$ is not compact by proving that its Alexandrov extension has positive entropy.  More precisely, write $\mathrm{dom}(I)_{\mathrm{alex}}$ for the Alexandrov compactification of $\mathrm{dom}(I)$ if it is not compact, and $f_{\mathrm{alex}}$ for the extension of $f\vert_{\mathrm{dom}(I)}$ that fixes the point at infinity (otherwise set $\mathrm{dom}(I)_{\mathrm{alex}}= \mathrm{dom}(I)$ and $f_{\mathrm{alex}}=f\vert_{\mathrm{dom}(I)}$ in what follows).  Of course, $f_{\mathrm{alex}}$  is a factor of $f$ and so $h(f)\ge h(f_{\mathrm{alex}})$ if $M$ is compact. 
\end{remark}

Theorem \ref{th:transverse_imply_entropy} will be the direct consequence of Lemma \ref{lm:recurrent_to_loop} and the following result:

\begin{proposition}\label{pr:transversecycle_imply_entropy}
 Let $\Gamma$ be a transverse loop with an $\mathcal{F}$-transverse self-intersection, and $\gamma$ its natural lift. Assume that there exists integers $K, r$ such that $\gamma\vert_{[0, K]}$ has an $\mathcal{F}$-transverse self-intersection, and such that $\gamma\vert_{[0, mK]}$ is admissible of order $mr$ for every $m\geq 1$. Then the topological entropy of $f_{\mathrm{alex}}$ is at least equal to $\log 2/(4r)$.
\end{proposition}

Before proving the proposition, we will need the following lemma:
\bigskip

\begin{lemma}\label{le:good_covering}
 There exists a covering $(V_z)_{z\in \mathrm{dom}(I)}$ of $\mathrm{dom}(I)$ satisfying the following properties:
 
 \smallskip 
 \noindent{\bf i)} \enskip\enskip $V_z$ is an open disk that contains $z$;

 \smallskip 
 \noindent{\bf ii)} \enskip\enskip for every $z_1$, $z_2$ in $\mathrm{dom}(I)$, for every integer $p\geq 1$ and for every $z\in V_{z_1}\cap f^{-p}(V_{z_2})$ there exists a transverse path joining $z_1$ to $z_2$ equivalent to a subpath of $I_{\mathcal F}^{p+2}(f^{-1}(z))$,  that is homotopic, with endpoints fixed, to the path $\alpha_1 I^p(z)\alpha_2^{-1}$, where $\alpha_1$ is a path in $V_{z_1}$ that joins $z_1$ to $z$ and $\alpha_2$ is a path in $V_{z_2}$ that joins $z_2$ to $f^p(z)$;
 
 \smallskip 
 \noindent{\bf iii)} \enskip\enskip in the previous assertion, if $p=1$, the homotopy class of the path that joins $z_1$ to $z_2$ does not depend on $z\in V_{z_1}\cap f^{-1}(V_{z_2})$.
 
 \end{lemma}
\begin{proof} One can construct
 an increasing sequence $(K_i)_{i\geq 1}$ of compact sets of $\mathrm{dom}(I)$ that cover $\mathrm{dom}(I)$ and such $K_{i+1}$ is a neighborhood of $K_i\cup f(K_{i})\cup f^{-1}(K_{i})$,  a distance on $\widetilde{\mathrm{dom}}(I)$\textcolor{black}{, denoted by $d$, } that is invariant under the action of the group of covering transformations,
and an equivariant family of leaves $(\phi^*_{\widetilde z})_{\widetilde z \in \widetilde{\mathrm{dom}}(I)}$, where $\phi^*_{\widetilde z}$ separates $\widetilde z$ and $\widetilde f(\widetilde z)$ and consequently is met by $\widetilde I_{\widetilde{ \mathcal F}}(\widetilde z)$ (equivariant means that $\phi^*_{T(\widetilde z)}=T(\phi^*_{\widetilde z})$ for every covering transformation $T$). Then one can construct an equivariant family of relatively compact open sets $(W_{\widetilde z})_{\widetilde z \in \widetilde{\mathrm{dom}}(I)}$, where  $W_{\widetilde z}$ contains $\widetilde z$,  projects onto an open disk of ${\mathrm{dom}(I)}$ and 
satisfies
$$\widetilde f^{-1}(W_{\widetilde z}) \subset R(\phi^*_{\widetilde f^{-1}(z)
}), \enskip \widetilde f(W_{\widetilde z}) \subset L(\phi^*_{\widetilde z}).$$ 
\medskip

\textcolor{black}{ Note that, given $\widetilde z, \phi^{*}_{\widetilde z}$ and $W_{\widetilde z}$ as above, if $\widetilde z'$ is sufficiently close to $\widetilde z$, then one could choose $\phi^{*}_{\widetilde z'}$ to be equal to $\phi^{*}_{\widetilde z}$ and also $W_{\widetilde z'}=W_{\widetilde z}$. Therefore we may assume that the family $(W_{\widetilde z})_{\widetilde z \in \widetilde{\mathrm{dom}}(I)}$ is locally finite.}

Set $K_i=\emptyset$ if $i\leq 0$. One knows that $f(\overline{K_i\setminus K_{i-1}})\subset \mathrm{int}(K_{i+1}\setminus K_{i-2})$. By a compactness argument, there exists a positive and decreasing sequence $(\eta_{i})_{i\geq 1}$  such that for every $\widetilde z\in\pi^{-1}(\overline{K_i\setminus K_{i-1}})$, the open ball $B(\widetilde z, \eta_{i})$ is included in $W_{\widetilde z}\cap\pi^{-1}(\mathrm{int}(K_{i+1}\setminus K_{i-2}))$ and its image $\widetilde f(B(\widetilde z, \eta_{i}))$ included in $\pi^{-1}(\mathrm{int}(K_{i+1}\setminus K_{i-2}))$, the ball $B(\widetilde z, \eta_{i})$ being defined with respect to $d$. Then, one considers a decreasing and positive sequence $(\eta'_{i})_{i\geq 1}$ satisfying $\eta'_i<\eta_{i+4}/4$ and such that for every $\widetilde z\in\pi^{-1}(\overline{K_i\setminus K_{i-1}})$, 
one has 
$$\widetilde f(B(\widetilde z, \eta'_i))\subset  B(\widetilde f(\widetilde z), \eta_{i+1}/2).$$
Finally, one constructs an equivariant family of open sets $(V_{\widetilde z})_{\widetilde z \in \widetilde{\mathrm{dom}}(I)}$, where  $V_{\widetilde z}$ contains $\widetilde z$,  is included in $B(\widetilde z, \eta'_i)$ if $ \widetilde z\in \pi^{-1}(K_i\setminus K_{i-1})$ and projects onto an open disk of ${\mathrm{dom}(I)}$. \textcolor{black}{ Note that $V_{\widetilde z}\subset W_{\widetilde z}$.}

By projection on $\mathrm{dom}(I)$, one gets a family $(V_{z})_{z \in {\mathrm{dom}(I)}}$ satisfying {\bf i)}. To prove that it satisfies {\bf ii)}, one must prove that if there exists a point $\widetilde z \in V_{\widetilde z_1}$ such that $\widetilde f^p(\widetilde z) \in V_{\widetilde z_2}$, then there exists a transverse path from $\widetilde z_1$ to $\widetilde z_2$ that is equivalent to a subpath of $\widetilde I^{p+2}_{\widetilde{\mathcal F}}(\widetilde f^{-1}(\widetilde z))$. \textcolor{black}{As $V_{\widetilde{z}_i}\subset W_{\widetilde{z}_i}, i\in \{1, 2\}$,  by the properties of the chosen family $(W_{\widetilde y})_{\widetilde y \in \widetilde{\mathrm{dom}}(I)}$,}
$$R(\phi_{\widetilde f^{-1}(\widetilde z )})\subset R(\phi^*_{\widetilde f^{-1}(\widetilde z_1)})\subset R(\phi_{\widetilde z_1})\subset R(\phi^*_{\widetilde z_1})\subset R(\phi_{\widetilde f(\widetilde z)})$$
and 
$$R(\phi_{\widetilde f^{p-1}(\widetilde z )})\subset R(\phi^*_{\widetilde f^{-1}(\widetilde z_2)})\subset R(\phi_{\widetilde z_2})\subset R(\phi^*_{\widetilde z_2})\subset R(\phi_{\widetilde f^{p+1}(\widetilde z)}).$$
One deduces that $\widetilde I^{p+2}_{\widetilde {\mathcal F}} (\widetilde f^{-1}(\widetilde z))$ meets $\phi_{\widetilde z_1}$ and $\phi_{\widetilde z_2}$. It remains to prove that $R(\phi_{\widetilde z_1})\subset R(\phi_{\widetilde z_2})$ to ensure that $\phi_{\widetilde z_1}$ is met before $\phi_{\widetilde z_2}$ and to prove the existence of a transverse path from $\widetilde z_1$ to $\widetilde z_2$. The case where $p\geq 2$ is easy because
$$ R(\phi_{\widetilde z_1})\subset R(\phi_{\widetilde f(\widetilde z))})\subset R(\phi_{\widetilde f^{p-1}(\widetilde z )})\subset R(\phi_{\widetilde z_2}).$$
To prove the result in the case where $p=1$ it is sufficient to prove that $V_{\widetilde z_2}$ is included in \textcolor{black}{$W_{\widetilde  f(\widetilde z_1)}$} because every point in \textcolor{black}{$W_{\widetilde f(\widetilde z_1)}$} belongs to $L(\phi^*_{\widetilde z_1})$. Moreover one will get {\bf iii)} because \textcolor{black}{$W_{\widetilde f(\widetilde z_1)}$} is disjoint from its images by the non trivial covering transformations. \textcolor{black}{Set $\eta_j=\eta_1$ and $\eta'_j=\eta'_1$ if $j\leq 0$, and let $i$ be such that $z_1\in K_i\setminus K_{i-1}$}. One knows that $\widetilde z\in\pi^{-1}(\mathrm{int}(K_{i+1}\setminus K_{i-2}))$, so $\widetilde f(\widetilde z)\in\pi^{-1}(\mathrm{int}(K_{i+2}\setminus K_{i-3}))$ and $\widetilde z_2 \in\pi^{-1}(\mathrm{int}(K_{i+3}\setminus K_{i-4}))$. One also knows that \textcolor{black}{$\widetilde f(\widetilde z_1)\in\pi^{-1}(\mathrm{int}(K_{i+1}\setminus K_{i-2}))$}. Using the fact that $\eta'_{i-3}<\eta_{i+1}/4$,  that $d(\widetilde f(\widetilde z), \widetilde f(\widetilde z_1))<\eta_{i+1}/2$, and that $d(\widetilde f(\widetilde z), \widetilde z_2)<\eta'_{i-3}$, one deduces that
$$V_{\widetilde z_2}\subset B(\widetilde z_2,\eta'_{i-3})\subset  \textcolor{black}{B(\widetilde f(\widetilde z_1),\eta_{i+1})\subset W_{\widetilde f(\widetilde z_1)}.}$$
\end{proof}

{\it Proof of Proposition \ref{pr:transversecycle_imply_entropy}.}  \enskip\enskip We keep the notations of Proposition \ref{pr:transverse_loop_exponentialgrowth}. Since, by Lemma \ref{lm:intersect-transversally} applied to $n=1$, the paths $(\gamma_1')^2$ and $(\gamma_2')^2$ intersect $\mathcal F$-transversally, one deduces that $\gamma'{}_1^2$ has a leaf on its right and a leaf on its left, which implies that $\gamma'_1$  satisfies the same property. One proves similarly that $\gamma'_2$ has a leaf on its right and a leaf on its left.  By Lemma \ref{le: neighborhood of infinity}, there exists a compact set $K$ such that for every $n\geq 1$ and every \textcolor{black}{$z\in\mathrm{dom}(I)\setminus\bigcup _{0\leq k<n}f^{-k}(K)$}, none of the paths $\gamma'_1$ and $\gamma'_2$ is equivalent to a subpath of $I_{\mathcal F}^n(z)$. We can find a compact set $K'$ larger than $K$ such that for all $z\in \mathrm{dom}(I)\setminus K'$, the trajectory $I(z)$ is disjoint from $\gamma_1'$ and $\gamma_2'$, as is every open set $V_{z'}$ given by Lemma \ref{le:good_covering} that contains $z$. Adding $\infty$ to $\mathrm{dom}(I)\setminus K'$, one gets a neighborhood $V_{\infty}$ of $\infty$ in $\mathrm{dom}(I)_{\mathrm{alex}}$.   
Define $V_{\infty,p}=\bigcap_{\vert k\vert\le p} f^{-k}(V_{\infty})$ and consider the covering ${\mathcal V}_p$ of $\mathrm{dom}(I)_{\mathrm{alex}}$ that consists of $V_{\infty,p}$ added to the covering $(V_z)_{z\in\mathrm{dom}(I)}$ of $\mathrm{dom}(I)$ given by Lemma \ref{le:good_covering}. Write $\Gamma'$ for the loop naturally defined by the closed path $(\gamma_1')^2(\gamma_2')^2$ and $\gamma'$ for its natural lift. Recall that $M(\Gamma')$ has been  defined in Proposition \ref{pr: finite_homotopy-classes} . Proposition  \ref{pr:transversecycle_imply_entropy} is an immediate consequence of the following:

\begin{lemma}\label{le: entropy covering}
The entropy of $f_{\mathrm{alex}}$ relative to the covering ${\mathcal V}_p$ is at least equal to $\log 2/(4r)-\log M(\Gamma')/2p$.
\end{lemma}

\begin{proof} 

As seen in the proof of Theorem \ref{th:transverse_imply_periodic points}, to every palindromic word $\mathbf{e}=(\varepsilon_i)_{0\leq i<2n}$ of length $2n$ is associated a fixed point $z_{\mathbf e}$ of $f^{4nr}$ and an associate $\mathcal F$-transverse loop defined by $\prod_{0\leq j<2n} \gamma'_{\varepsilon_{i}}$.  
Moreover, by Lemma \ref{lm:bounded}, there exists $L>0$, independent of $n$, such that there are at least  $2^n/Ln^2$ different equivalent classes among the associated loops.  
We will prove that every open set of the covering $\bigvee_{0\leq k<  4nr} f^{-k}({\mathcal V}^p)$ contains at most $Ln^2M(\Gamma')^{  2nr/p}$ points $z_{\mathbf e}$. We deduce that every finite sub-covering of $\bigvee_{0\leq k<  4nr} f^{-k}({\mathcal V}^p)$ has at least $2^{n}/Ln^2M(\Gamma')^{2nr/p}$ open sets and so
$$h(f_{\mathrm{alex}}, {\mathcal V}_p)\geq   \lim_{n\to+\infty} {1\over 4nr} \log (2^{n}/Ln^2M(\Gamma')^{2nr/p})=\log 2/(4r)-\log M(\Gamma')/2p.$$

Let us consider an element $W=\bigcap_{0\leq j< 4nr}f^{-j}(V^j)$ of ${\mathcal V}_p$. We suppose that it contains at least one point $z_{\mathbf e}$. Denote $J_{\infty}$ the set of $j\in\{0, \dots ,4nr-1\}$ such that there exists $j'\in\{0, \dots ,4nr\}$ satisfying $\vert j-j'\vert<p$ and $V^{j'}=V_{\infty,p}$ and denote $J_{<\infty}$ the complement of $J_{\infty}$. Note that $f^j(z_{\mathbf e})\in V_{\infty}$ if $j\in J_{\infty}$. By Lemma \ref{le: neighborhood of infinity}, one knows that the orbit of $z_{\mathbf e}$ cannot be contained in $V_{\infty}$. So $J_{<\infty}\not=\emptyset$.

\medskip
Let us begin with the case where $0\in J_{<\infty}$ and write $J_{<\infty}=\{j_0,\dots, j_{l^*}\}$, where $j_0=0<j_1\dots< j_{l^*}$, and add $j_{l^*+1}=4nr$. Every open set $V^{j_l}$ can be written $V^{j_l}=V_{z_{j_l}}$. For every $l$, choose a path $\delta_l$ in $V_{z_{j_l}}$ from $z_{j_l}$ to $f^{j_l}(z_{\mathbf e})$. By Lemma \ref{le:good_covering}, there exists a $\mathcal F$-transverse path $\beta_l$ from $z_{j_l}$ to $z_{j_{l+1}}$ equivalent to a subpath of \textcolor{black}{$I^{j_{l+1}-j_l+2}_{\mathcal F}(f^{j_l-1}(z_{\mathbf e}))$} and homotopic to $\delta_l I^{j_{l+1}-j_j} (f^{j_l}(z_{\mathbf e}))\delta^{-1}_{l+1}$. In the case where $j_{l+1}-j_l=1$, the homotopy class of $\beta_i$ (with endpoints $z_{i_l}$ and $z_{i_{l+1}}$ fixed) is uniquely determined. Let us explain now why there are at most $M(\Gamma')$ possible homotopy classes if $j_{l+1}-j_l\geq 2$. Note first that $j_{l+1}-j_l\geq 2p$ in that case, and that all points $f^j(z_{\mathbf e})$, $j_{l}-1\leq j\leq j_{l+1}+1$ belong to $V_{\infty}$. This implies that neither $\gamma_1'$ nor $\gamma_2'$ are equivalent to subpaths of \textcolor{black}{$I^{j_{l+1}-j_l+2}_{\mathcal F}(f^{j_l-1}(z_{\mathbf e}))$}. \textcolor{black}{Note that the \textcolor{black}{latter} is equivalent to a subpath of \textcolor{black}{$I^{4nr}_{\mathcal F}(z_{\mathbf e})$, which is equivalent to }$\prod_{0\leq j<2n} \gamma'_{\varepsilon_{i}}$. We remark that, if $\sigma$ is any transverse path that is equivalent to a subpath of $\prod_{0\leq j<2n} \gamma'_{\varepsilon_{i}}$, but that does not contain a subpath that is equivalent to either $\gamma_1'$ or $\gamma_2'$, then $\sigma$ must be equivalent to  a subpath of one the six possible following paths:}
$$\gamma'_1, \,\gamma'_2, \,\gamma'_1\gamma'_1, \,\gamma'_1\gamma'_2,\,\gamma'_2\gamma'_1,\,\gamma'_2\gamma'_2$$ and therefore \textcolor{black}{$\sigma$ (and thus $I^{j_{l+1}-j_l+2}_{\mathcal F}(f^{j_l-1}(z_{\mathbf e}))$)} must equivalent to a subpath of $\gamma'$. Furthermore \textcolor{black}{$\beta_i=\delta_l I^{j_{l+1}-j_j} (f^{j_l}(z_{\mathbf e}))\delta^{-1}_{l+1}$} is disjoint from \textcolor{black}{$\Gamma'$} by definition of $V_{\infty}$ because it is the case for $I^{j_{l+1}-j_j} (f^{j_l}(z_{\mathbf e}))$ and for the disks $V_{z_{j_l}}$ and $V_{z_{j_{l+1}}}$. One can apply Proposition  \ref{pr: finite_homotopy-classes} \textcolor{black}{ and obtain that $\beta_i$ must belong to one of the at most $M(\Gamma')$ different homotopy classes of paths connecting $z_{i_l}$ and $z_{i_{l+1}}$, as claimed before.}

The path $\prod_{0\leq l\leq l^*} \beta_l$ is a closed path based at $z_0$.  Noting that there exist at most $4nr/2p=2nr/p$ integers $l$ such that $j_{l+1}-j_l\not=1$, we deduce that there exist at most $M(\Gamma')^{2nr/p}$ homotopy classes (with fixed base point) possible. The loop defined by $\prod_{0\leq l\leq l^*} \beta_l$ is freely homotopic to $\Gamma'_{\mathbf e}$. So there exist at most $M(\Gamma')^{2nr/p}$ free homotopy classes defined by the loops $\Gamma'_{\mathbf e}$ such that $z_{\bf e}\in W$. \textcolor{black}{By Lemma \ref{lm:bounded}, to} prove that there is at most  $Ln^2M(\Gamma')^{2nr/p}$ points $z_{\bf e}$ in $W$, it is sufficient to prove the following stronger result: there exist at most $M(\Gamma')^{2nr/p}$ classes defined by the loops $\Gamma'_{\mathbf e}$ \textcolor{black}{that are equivalent as transverse paths and such} that $z_{\bf e}\in W$. Suppose that $z_{\bf e}$ and $z_{\bf e'}$ belong to $W$ and that the paths $\beta_l$, $0\leq l\leq l^*$, constructed with $z_{\bf e}$ and $z_{\bf e'}$ are all homotopic. Fix a lift $V_{\widetilde z_0}$ of $V_{z_0}$ and note $\widetilde z_{\bf e}$ and $\widetilde z_{\bf e'}$ the respective lifts of $z_{\bf e}$ and $z_{\bf e'}$ that belongs to $V_{\widetilde z_0}$. The whole $\widetilde{\mathcal F}$-transverse trajectories of $\widetilde z_{\bf e}$ and $\widetilde z_{\bf e'}$
are invariant by the same non trivial covering automorphism $T$, which is naturally defined by the lift $\prod_{0\leq l\leq l^*} \widetilde\beta_l$ of $\prod_{0\leq l\leq l^*} \beta_l$. Moreover, $\widetilde\beta_0$ is $\widetilde{\mathcal F}$-equivalent to a subpath of $\widetilde I_{\widetilde{\mathcal F}}(\widetilde z_{\bf e})$ and to a subpath of $\widetilde I_{\widetilde{\mathcal F}}(\widetilde z_{\bf e'})$. So these two lines meet a common leaf $\widetilde \phi$. One deduces that they meet $T^k(\widetilde\phi)$, for every $k\in\Z$ and consequently that they meet the same leaves. So there are equivalent.

\medskip
The case where $0\in J_{\infty}$ can be reduced to the case where $0\in J_{<\infty}$ after a cyclic permutation on $\{0,\dots, 4nr-1\}$ because the points $z_{\bf e}$ are all fixed by $f^{4nr}$. 
  \end{proof}

\bigskip

As a direct application of Proposition \ref{pr:transversecycle_imply_entropy} we can obtain the following result, which is connected to the study of the minimal entropy of pure braids in $\S^2$. There are sharper results with a larger lower bound for the entropy (see \cite{So}), but they use very different techniques. 

\begin{theorem}\label{Thm:entropy_lower_bounds}
 Let $f$ be an orientation preserving homeomorphism on $\S^2$ and $I$  a maximal hereditary singular isotopy. Assume that there exists $ z\in \mathrm{dom}(I)\cap\mathrm{fix}(f)$ such that the loop naturally defined by the trajectory $I(z)$ is not homotopic in $\mathrm{dom}(I)$ to a multiple of a simple loop. Then the entropy of $f$ is at least equal to $\log(2)/(4)$. 
\end{theorem}

\begin{proof}
Let $\mathcal{F}$ be a foliation transverse to the isotopy. By hypothesis, the transverse loop $\Gamma$ associated to ${z}$ is not a multiple of a simple loop, so by Proposition \ref{pr:transverse_on_sphere}, it has an $\mathcal{F}$-transverse self-intersection. If $\gamma$ is the natural lift of $\Gamma$, then for all integers $K, \,\gamma\vert_{[0, K]}$ is admissible of order $K$. Furthermore, by Proposition \ref{pr:shorttransversalitypforloops}, $\gamma\vert_{[0,2]}$ has an $\mathcal{F}$-transverse self-intersection. The theorem then follows directly from Proposition \ref{pr:transversecycle_imply_entropy}.
\end{proof}

\subsection{Associated subshifts}
  
Let us give a natural application of  Corollary \ref{co: first induction transverse}, Corollary \ref{co: induction transverse} and the results of this section. We keep the assumptions and notations given at the beginning of the section.  Consider a transverse path $\gamma:[a,b]\to\mathrm{dom}(\mathcal F)$ with finitely many double points, none of them corresponding to an end of the path and no triple points (by a slight modification of the argument given in the proof of Corollary \ref{co: no intersection}  one can show that
every transverse path is equivalent to such a path). There exists real numbers $a<t_1<\dots <t_{2r}<b$ and a fixed point free involution $\sigma$ on $\{1,\dots, 2r\}$ such that $\gamma(t_i)=\gamma(t_{\sigma (i)})$, for every $i\in \{1,\dots, 2r\}$ and such that $\gamma$ is injective on  the complement of $\{t_1,\dots, t_{2r}\}$.
Set $t_0=a$ and $t_{2r+1}=b$ and define for every $i\in \{0,\dots, 2r\}$ the path $\gamma_i=\gamma\vert_{[t_i,t_{i+1}]}$. Consider the incidence matrix $P\in M_{2r+1}(\Z)$ (indexed by $\{0,\dots, 2r\}$) such that $P_{i,j}=1$  if $j=i+1$ or $j=\sigma(i+1)$ and $0$ otherwise (in particular if $i=2r$). Note that the first column and the last row only contain $0$. For every $P$-admissible word $(i_s)_{1\leq s\leq s_0}$, which means that $P_{i_s,i_{s+1}}=1$ if $s<s_0$, the path $\prod_{1\leq s\leq s_0}\gamma_{i_s}$ is transverse to $\mathcal F$. Note that every transverse path $\gamma:[a',b']\to\mathrm{dom}(\mathcal F)$ whose image is contained in the image of $\gamma$ is a subpath of such a path $\prod_{1\leq s\leq s_0}\gamma_{i_s}$.

Suppose now that $\gamma$ is admissible of order $n$. Can we decide when a path $\prod_{1\leq s\leq s_0}\gamma_{i_s}$ is admissible and what is its order? More precisely, do there exist other incidence matrices $P'$ smaller than $P$ (which means that $P'_{i,j}=0$ if $P_{i,j}=0$) such that $\prod_{1\leq s\leq s_0}\gamma_{i_s}$ is admissible if $(i_s)_{1\leq s\leq s_0}$ is $P'$-admissible?

Corollaries \ref{co: first induction transverse} and \ref{co: induction transverse} imply that the following three matrices satisfy this property:

\smallskip
\noindent- \enskip the matrix $P^{\mathrm{strong}}$, where $P^{\mathrm{strong}}_{i,j}=1$ if and only if $j=i+1$, or if $j=\sigma(i+1)$ and $\gamma_{i}\gamma_{i+1}$ and  $\gamma_{j-1}\gamma_j$ have an $\mathcal{F}$-transverse intersection at $\gamma(t_{i+1})=\gamma(t_j)$;  

\smallskip
\noindent- \enskip  the matrix $P^{\mathrm{left}}$, where $P^{\mathrm{left}}_{i,j}=1$  if and only if $j=i+1$, or if $j=\sigma(i+1)$ and $\gamma$ has an $\mathcal{F}$-transverse positive self-intersection at  $\gamma(t_{i+1})=\gamma(t_j)$; 

\smallskip
\noindent- \enskip  the matrix $P^{\mathrm{right}}$,  where $P^{\mathrm{right}}_{i,j}=1$ if and only if $j=i+1$, or if $j=\sigma(i+1)$ and $\gamma$ has an $\mathcal{F}$-transverse negative self-intersection at  $\gamma(t_{i+1})=\gamma(t_j)$.

More precisely, if $P'$ is one of the three previous matrices, then for every $P'$-admissible word $(i_s)_{1\leq s\leq s_0}$, the path  $\prod_{1\leq s\leq s_0}\gamma_{i_s}$ is admissible of order $kn$, where $k$ is the number of $s<s_0$ such that $i_{s+1}=\sigma(i_s+1)$. As explained in Proposition \ref{pr: self-intersection}, its order can be less. One can adapt the proof of  Theorem \ref{th:transverse_imply_entropy} to give a lower bound to the topological entropy of $f_{\mathrm{alex}}$. For example it is at least equal to $1/n$ times the logarithm of the spectral radius of $P'$  if the paths $\gamma_{i}$ have a leaf on their right and a leaf on their left, otherwise one has to replace these paths by finite admissible words.  One can adapt the proof of Theorem \ref{th:transverse_imply_periodic points} to show that for every $P'$-admissible word $(i_s)_{1\leq s\leq s_0}$ such that $i_1=i_{s_0}$, the loop naturally defined by $\prod_{1\leq s<s_0}\gamma_{i_s}$ is associated to a periodic orbit (except for
 some exceptional cases).

 Let us illustrate this procedure with four examples, where we start with an admissible path of order $1$:

\begin{figure}[ht!]
\includegraphics[height=45mm]{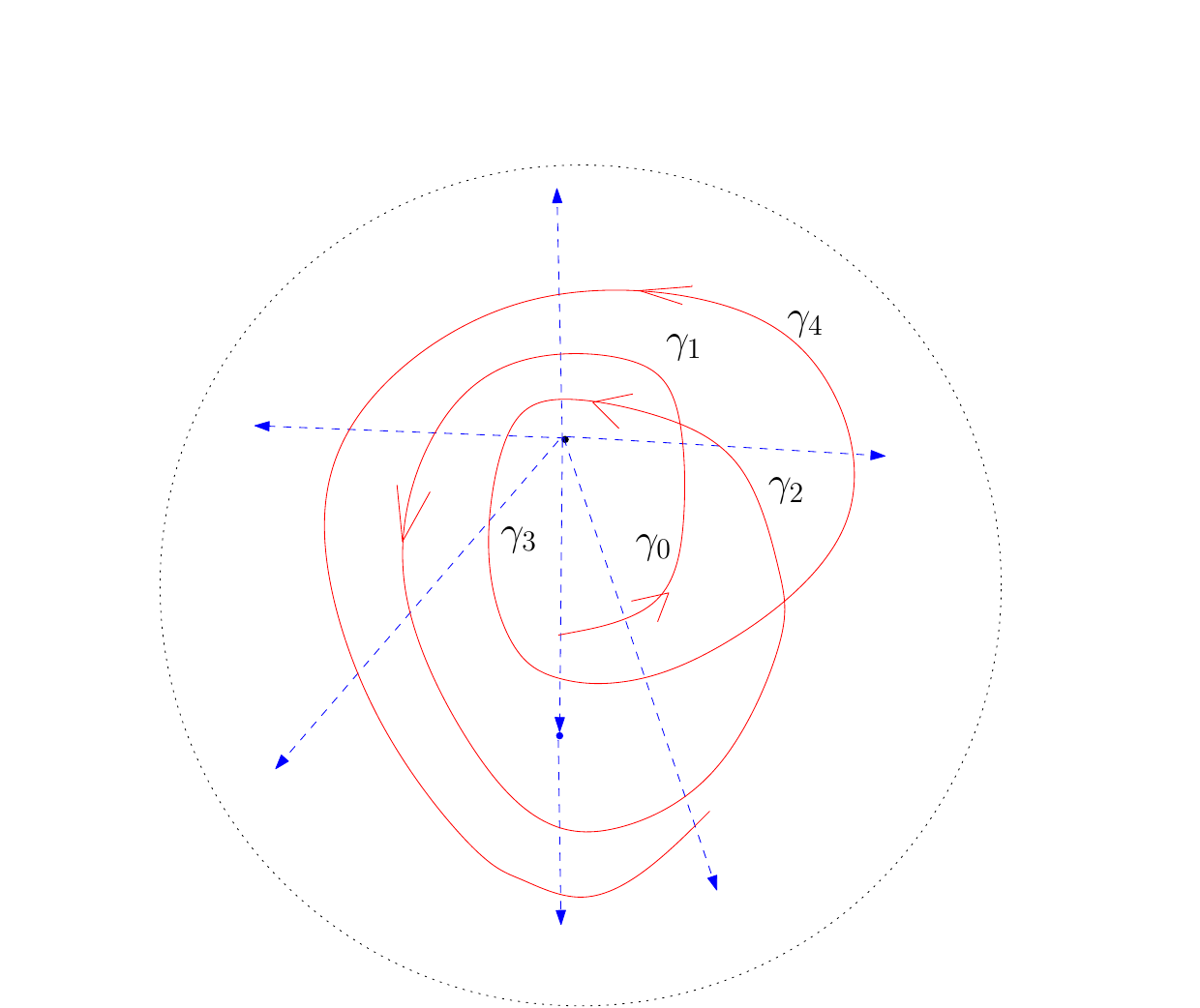}
\caption{\small Example 1 - Leafs of the foliation are represented as dashed lines, while transverse paths are solid.}
\label{figura_exemplo1}
\end{figure}

%FIRST EXAMPLE
For the first example, see Figure \ref{figura_exemplo1}, the admissibility matrices are 

$$ P^{\mathrm{strong}}_1 = \left( \begin{array}{ccccc}
0 & 1 & 0 & 0 & 0  \\
0 & 0 & 1 & 0 & 0 \\
0 & 0 & 0 & 1 & 0 \\
0 & 0 & 0 & 0 & 1 \\
0 & 0 & 0 & 0 & 0  \end{array} \right ),$$ 

$$
 P^{\mathrm{left}}_1= \left ( \begin{array}{ccccc}
0 & 1 & 0 & 1 & 0  \\
0 & 0 & 1 & 0 & 0 \\
0 & 0 & 0 & 1 & 0 \\
0 & 0 & 1 & 0 & 1 \\
0 & 0 & 0 & 0 & 0  \end{array} \right ), \, 
P^{\mathrm{right}}_1 = \left( \begin{array}{ccccc}
0 & 1 & 0 & 0 & 0  \\
0 & 0 & 1 & 0 & 1 \\
0 & 1 & 0 & 1 & 0 \\
0 & 0 & 0 & 0 & 1 \\
0 & 0 & 0 & 0 & 0  \end{array} \right ).
$$

The matrix $ P^{\mathrm{strong}}_1$ does not tell us anything, the only admissible paths are subpaths of $\gamma=\gamma_0\dots\gamma_4$. The only interesting informations got from $ P^{\mathrm{left}}_1$ and $ P^{\mathrm{right}}_1$ respectively are the facts that the loops naturally defined by $\gamma_2\gamma_3$ and $\gamma_1\gamma_2$ are linearly admissible of order $2$. Nevertheless the first loop has no leaf on its left while the second one has no leaf on its right. So, one cannot apply Proposition \ref{pr: realization} and deduce that the loops are equivalent to transverse loops associated to periodic points of period $2$. Note also that the spectral radius of $ P^{\mathrm{left}}_1$ and $ P^{\mathrm{right}}_1$ are equal to $1$. In this example, one cannot deduce neither the positivity of entropy, nor the existence of  periodic orbits.

%SECOND EXAMPLE

\begin{figure}[ht!]
\includegraphics[height=45mm]{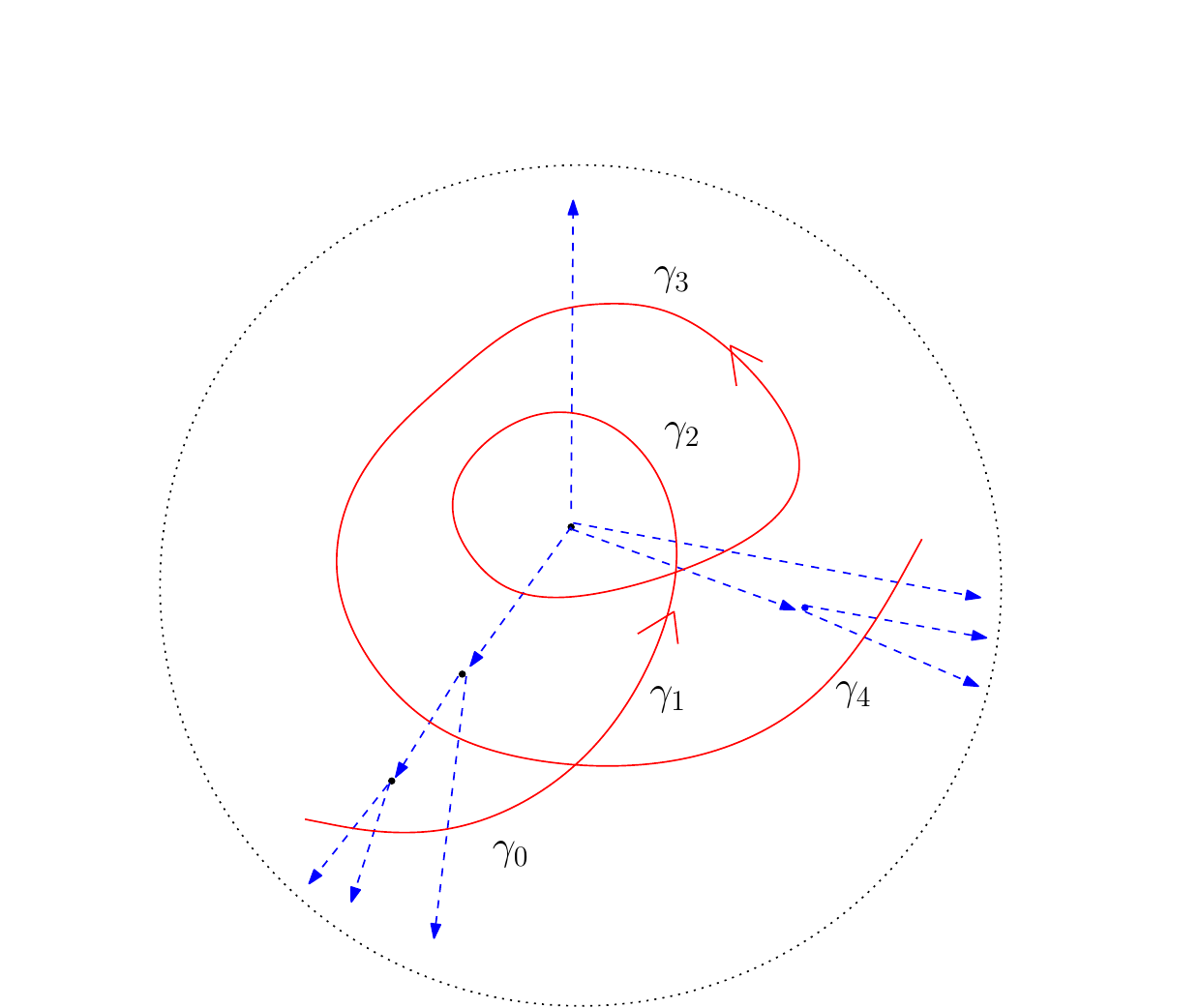}
\caption{\small Example 2 - Leafs of the foliation are represented as dashed lines, while transverse paths are solid.}
\label{figura_exemplo2}
\end{figure}

For the second example, see Figure \ref{figura_exemplo2}, the admissibility matrices are 

$$ P^{\mathrm{strong}}_2 = \left( \begin{array}{ccccc}
0 & 1 & 0 & 0 & 0  \\
0 & 0 & 1 & 0 & 0 \\
0 & 0 & 0 & 1 & 0 \\
0 & 0 & 0 & 0 & 1 \\
0 & 0 & 0 & 0 & 0  \end{array} \right ),$$ 

$$
 P^{\mathrm{left}}_2= \left ( \begin{array}{ccccc}
0 & 1 & 0 & 0 & 0  \\
0 & 0 & 1 & 0 & 0 \\
0 & 0 & 1 & 1 & 0 \\
0 & 1 &  0& 0 & 1 \\
0 & 0 & 0 & 0 & 0  \end{array} \right ), \, 
P^{\mathrm{right}}_2 = \left( \begin{array}{ccccc}
0 & 1 & 0 & 0 & 1  \\
0 & 0 & 1 & 1& 0 \\
0 & 0 & 0 & 1 & 0 \\
0 & 0 & 0 & 0 & 1 \\
0 & 0 & 0 & 0 & 0  \end{array} \right ).
$$

The matrix $ P^{\mathrm{strong}}_2$ does not tell us anything. The matrix $ P^{\mathrm{right}}_2$ is nilpotent and the only admissible paths are $\gamma_0\gamma_4,\, \gamma_0\gamma_1\gamma_3\gamma_4$ and $\gamma$, all of them admissible of order $1$ by Proposition \ref{pr: self-intersection}. The matrix $ P^{\mathrm{left}}_2$ is much more interesting: its spectral radius, the real root of the polynomial $X^3-X^2-1$, is larger than $1$. The loop naturally defined by $\gamma_2$  is linearly admissible of order $1$ but has no leaf on its left: one cannot deduce that it is equivalent to a transverse loop associated to a fixed point. If $p\geq 1$, the loop naturally defined by $\gamma_1\gamma_2^p\gamma_3$ is linearly admissible of order $p$ and has leaf on its right and a leaf on its left. More precisely, it has a transverse self-intersection, so one can apply Proposition \ref{pr: realization} and deduce that it is equivalent to a transverse loop associated to a periodic point of period $p$. In particular, the loop defined by $\gamma_1\gamma_2\gamma_3$ is equivalent to a transverse loop associated to a fixed point: one can apply Theorem \ref{Thm:entropy_lower_bounds} and deduce that the entropy of $f$ is at least $\log 2/4$.

%THIRD EXAMPLE

\begin{figure}[ht!]
\includegraphics[height=45mm]{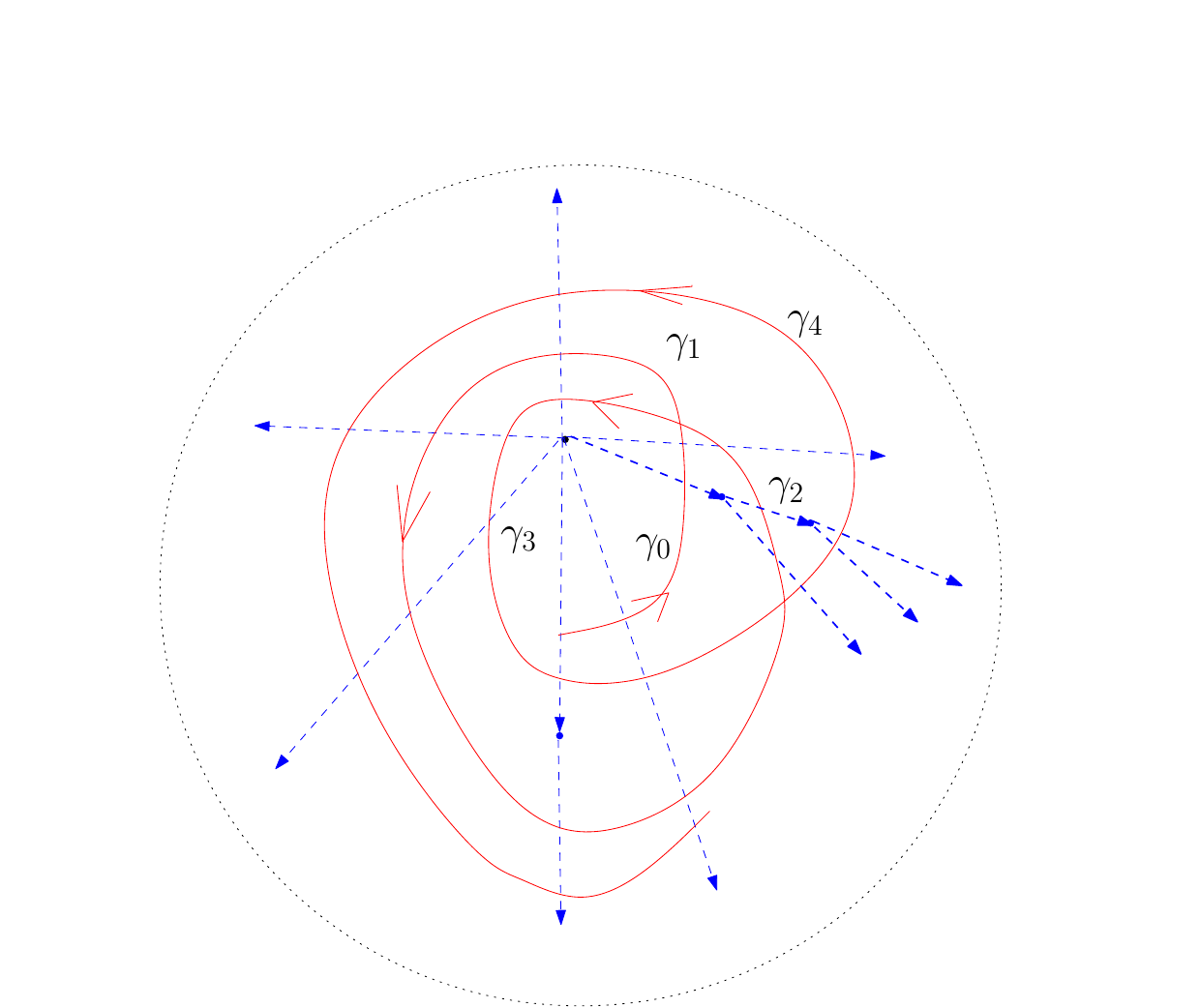}
\caption{\small Example 3 - Leafs of the foliation are represented as dashed lines, while transverse paths are solid.}
\label{figura_exemplo3}
\end{figure}
In third example, see figure \ref{figura_exemplo3}, the trajectory is the same as in the first example but the foliation is different. The admissibility matrices are 

$$ P^{\mathrm{strong}}_3 = \left( \begin{array}{ccccc}
0 & 1 & 0 & 1 & 0  \\
0 & 0 & 1 & 0 & 1 \\
0 & 1 & 0 & 1 & 0 \\
0 & 0 & 1 & 0 & 1 \\
0 & 0 & 0 & 0 & 0  \end{array} \right ),$$

$$
 P^{\mathrm{left}}_3= \left ( \begin{array}{ccccc}
0 & 1 & 0 & 1 & 0  \\
0 & 0 & 1 & 0 & 0 \\
0 & 0 & 0 & 1 & 0 \\
0 & 0 & 1 & 0 & 1 \\
0 & 0 & 0 & 0 & 0  \end{array} \right ), \, 
P^{\mathrm{right}}_3 = \left( \begin{array}{ccccc}
0 & 1 & 0 & 0 & 0  \\
0 & 0 & 1 & 0 & 1 \\
0 & 1 & 0 & 1 & 0 \\
0 & 0 & 0 & 0 & 1 \\
0 & 0 & 0 & 0 & 0  \end{array} \right ).
$$

The matrices $ P^{\mathrm{right}}_3$ and $ P^{\mathrm{left}}_3$ are the same as in the first example. Nevertheless, one can say more. Indeed the loops defined by $\gamma_2\gamma_3$ and $\gamma_1\gamma_2$, which  are linearly admissible of order $2$, now have a leaf on their left and a leaf on their right. They intersect $\mathcal{F}$-transversally and negatively the paths $\gamma_0\gamma_1$ and $\gamma_3\gamma_4$ respectively but they do not interest $\mathcal{F}$-transversally and positively a path drawn on $\gamma$.  So, one cannot apply \textcolor{black}{the second item of Proposition \ref{pr: realization}. However, by the first item of the same proposition,} if they are not equivalent to transverse loops associated to periodic points of period $2$, they are homotopic in the domain to a simple loop that does not meet its image by $f$. In particular, if $\Omega(f)=\S^2$, they must be equivalent to transverse loops associated to periodic points of period $2$. The matrix $ P^{\mathrm{strong}}_3$ is much more interesting: its spectral radius is equal to $\sqrt{2}$. Every path defined by a word of length $n$ in the alphabet $\{\gamma_1\gamma_2, \gamma_3\gamma_2\}$ is admissible of order $2n$ and intersect $\gamma$ transversally. The proofs of Theorem \ref{th:transverse_imply_entropy} and Theorem \ref{th:transverse_imply_periodic points} tell us that the topological entropy of $f$ is at least equal to $\log2/2$, and that the number of fixed point of $f^{2n}$ in the domain is at least equal to $e^n$ if $\Omega(f)=\S^2$.

\begin{figure}[ht!]
\includegraphics[height=45mm]{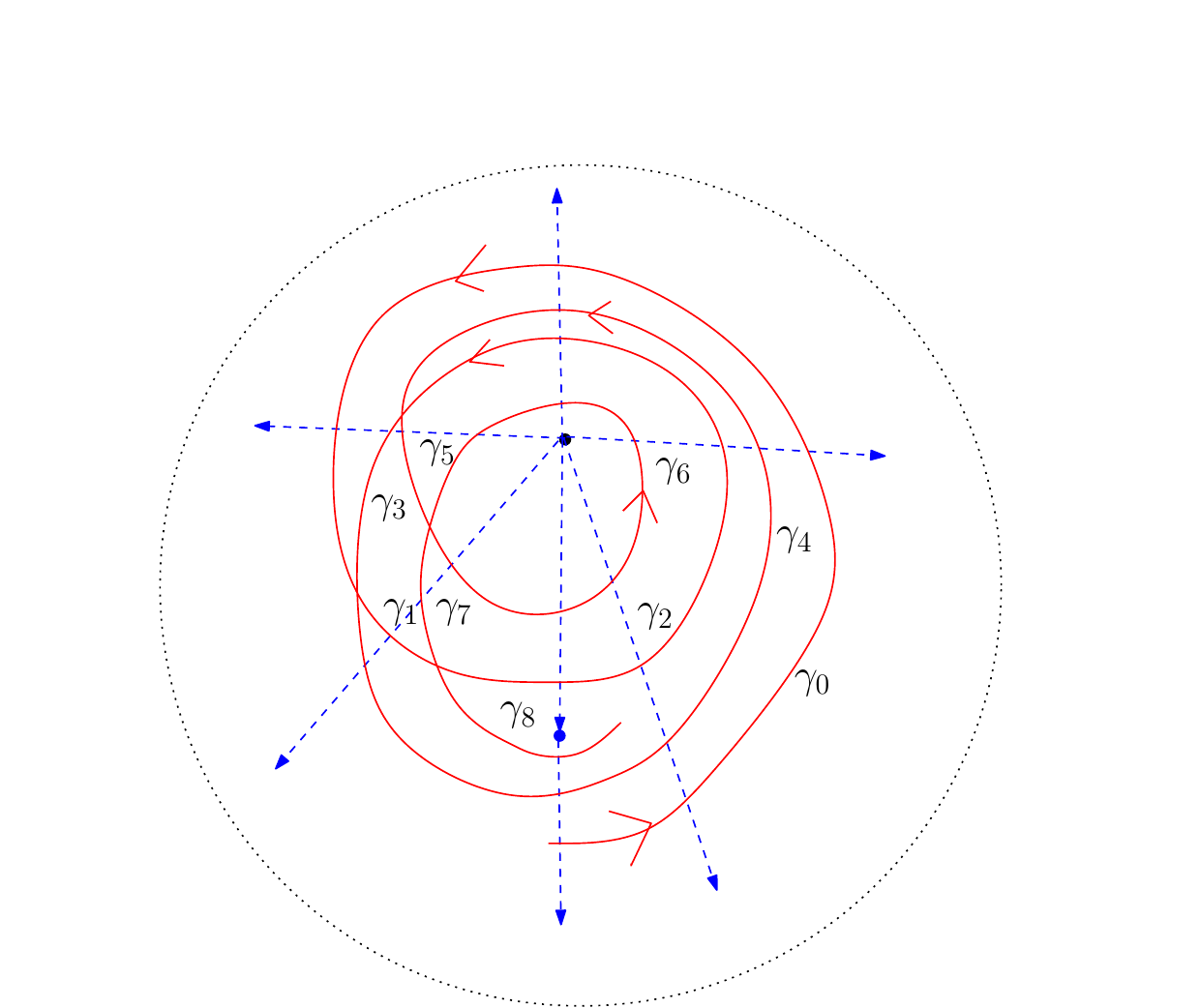}

\caption{\small Example 4 - Leafs of the foliation are represented as dashed lines, while transverse paths are solid.}
\label{figura_exemplo4}
\end{figure}

%FOURTH EXAMPLE
In the fourth example, see Figure \ref{figura_exemplo4}, the foliation is the same as in the first example but the trajectory is different. In particular, there are three points of  self-intersection of $\gamma$, and all are $\mathcal{F}$-transverse. The admissibility matrices are:

$$ P^{\mathrm{strong}}_1 = \left( \begin{array}{ccccccc}
0 & 1 & 0 & 0 & 0 & 0 & 0 \\
0 & 0 & 1 & 0 & 0 & 0 & 0 \\
0 & 0 & 0 & 1 & 0 & 0 & 0 \\
0 & 0 & 0 & 0 & 1 & 0 & 0 \\
0 & 0 & 0 & 0 & 0 & 1 & 0 \\
0 & 0 & 0 & 0 & 0 & 0 & 1 \\
0 & 0 & 0 & 0 & 0 & 0 & 0 \end{array} \right ),$$ 
$$
 P^{\mathrm{left}}_1 = \left (\begin{array}{ccccccc}
0 & 1 & 0 & 0 & 0 & 1 & 0 \\
0 & 0 & 1 & 1 & 0 & 0 & 0 \\
0 & 0 & 0 & 1 & 0 & 0 & 0 \\
0 & 0 & 0 & 0 & 1 & 0 & 0 \\
0 & 0 & 0 & 0 & 0 & 1 & 0 \\
0 & 0 & 0 & 0 & 1 & 0 & 1 \\
0 & 0 & 0 & 0 & 0 & 0 & 0 \end{array} \right), \, 
 P^{\mathrm{right}}_1 = \left( \begin{array}{ccccccc}
0 & 1 & 0 & 0 & 0 & 0 & 0 \\
0 & 0 & 1 & 0 & 0 & 0 & 0 \\
0 & 0 & 1 & 1 & 0 & 0 & 0 \\
0 & 0 & 0 & 0 & 1 & 0 & 1 \\
0 & 1 & 0 & 0 & 0 & 1 & 0 \\
0 & 0 & 0 & 0 & 0 & 0 & 1 \\
0 & 0 & 0 & 0 & 0 & 0 & 0 \end{array} \right).
$$

By inspection of $ P^{\mathrm{left}}_1$ one assures the existence of a single admissible loop $\gamma_4\gamma_5$, while inspection of $ P^{\mathrm{right}}_1$ we see that both loops $\gamma_2$ and $\gamma_2\gamma_3\gamma_4\gamma_1$ are admissible and that the entropy of $f$ must be positive.

\section{First applications}

In this section we give two applications for homeomorphisms of compact oriented surfaces. The first one is a new proof of Handel's result on transitive homeomorphisms of the sphere. The second application provides sufficient conditions for the existence of non-contractible periodic orbits, and has as a consequence a positive answer to a problem posed by P. Boyland for the annulus.
  
   \subsection{ Transitive maps of surfaces of genus $0$}
  In \cite{H1}, Handel prove that a transitive orientation preserving homeomorphism  $f$ of $\S^2$ with  at least three fixed points, but finitely many, 
 has infinitely many periodic orbits: more precisely the number of periodic points of period $n$ for some iterate of $f$ grows exponentially in $n$.  We will improve this result as follows\textcolor{black}{, with Theorem \ref{th: H_intro} of the introduction}:

\begin{theorem} Let $f:\S^2\to\S^2$ be an orientation preserving homeomorphism such that the complement of the fixed point set is not an annulus. If $f$ is topologically transitive then the number of periodic points of period $n$ for some iterate of $f$ grows exponentially in $n$. Moreover, the entropy of $f$ is positive.\end{theorem}

\begin{proof}
Recall that,  in our case, the transitivity implies the existence of a point $z$ whose $\omega$-limit and $\alpha$-limit sets are the whole sphere. One knows that every connected component of $\S^2\setminus\mathrm{fix}(f)$ is invariant(see Brown-Kister \cite{BK}). Since $f$ has a dense orbit, this complement must be connected. Moreover it cannot be a disk because $f$ has a dense orbit. Indeed the Brouwer Plane Translation Theorem implies that every fixed point free orientation preserving homeomorphism of the plane has only wandering points. One deduces that the fixed point set has at least three connected components. Choose three fixed points in different connected components and an isotopy $I'$ from identity to $f$ that fixes these three fixed points (this is always possible). The restriction of $I'$ to the complement of these three points is a hereditary singular isotopy. Using Theorem \ref{th: maximal} one can find a maximal hereditary singular isotopy $I$ larger than $I'$. Let $\mathcal{F}$ be a  foliation transverse to this isotopy. It has the same domain as $I$, and this domain is not an annulus because $I$ is larger than $I'$. The fact that $\omega(z)= \alpha(z)=\S^2$ implies that $I_{\mathcal{F}}^{\Z}
 (z)$ is  an admissible $\mathcal{F}$-bi-recurrent transverse path that contains as a subpath (up to equivalence) every admissible segment and consequently that  crosses all leaves of $\mathcal{F}$. Since $\mathrm{dom}(I)$ is not a topological annulus, this implies that $I_{\mathcal{F}}^{\Z}
 (z)$ has an $\mathcal{F}$-transverse self-intersection by Proposition \ref{pr:transverse_on_sphere}. The result follows from Theorems \ref{th:transverse_imply_periodic points} and \ref{th:transverse_imply_entropy}.
\end{proof}

\subsection{ Existence of non-contractible periodic orbits}

Let $f$ be a homeomorphism isotopic to identity on an oriented connected surface $M$ and $I'$ an identity isotopy of $f$. A periodic point $z\in M$ of period $q$ is said to have a {\it contractible} orbit if $I'^q(z)$ naturally defines a homotopically trivial loop, otherwise it is said to be {\it non-contractible}. In this subsection we examine some conditions that ensure the existence of non-contractible periodic orbits of arbitrarily high period. Through this subsection we assume that $\check M$ is the universal covering space of $M$ and write $\check \pi: \check M\to M$ for the covering projection.  Write $\check I'$ for the lifted identity isotopy and $\check f$ for the associated lift of $f$.   One can find  a maximal hereditary singular isotopy $I$ larger than $I'$. It can be lifted to an identity isotopy $\check I$ on $\check{\mathrm{dom}}(I)=\check\pi^{-1}(\mathrm{dom}(I))$. This isotopy is a maximal singular isotopy  of $\check f$ larger than $\check I'$. Let $\mathcal{F}$ be a  foliation transverse to $I$, its lift to $\check{\mathrm{dom}}(I)$, denoted by $\check{\mathcal{F}}$ is transverse to $\check I$.  

The main technical result is the following proposition:
\begin{proposition}\label{pr:technical_non-contractible}
Suppose that there exist an admissible \textcolor{black}{$\check{\mathcal{F}}$-bi-recurrent} path $\check\gamma$ for $\check f$, a leaf $\check \phi$ of $\check{\mathcal{F}}$ and three distinct covering automorphisms $T_i$, $1\leq i\leq 3$, such that $\check\gamma$ crosses each $T_i(\check \phi)$.  Then there exists $q>0$ and a non trivial covering automorphism $T=T_i\circ T_j^{-1}$ such that for all $r/s\in (0,1/q]$, the maps $\check f^{s}\circ T^{-r}$ and $\check f^{s}\circ T^{r}$ have fixed points. In particular, $f$ has non-contractible periodic points of arbitrarily large prime period.
\end{proposition}

\begin{proof} By assumptions, there are non trivial covering automorphisms. So $M$ is not simply connected and $\check M$ is a topological plane. For every loop $\check\Gamma$ in $\check M$, we will denote $\delta_{\check\Gamma}$ the dual function that vanishes on the unbounded connected component of $\check M\setminus\check\Gamma$. It is usually called the {\it winding number} of $\check\Gamma$.

\begin{sub-lemma}\label{sle:existence singular points}
If $\check\Gamma$ is a loop positively transverse to $\check{\mathcal F}$, the set of singular points $\check z$ of $\check{\mathcal{F}}$ such that $\delta_{\check\Gamma}(\check z)\not=0$ is a non empty compact subset $\Sigma_{\check\Gamma}$ of $\check M$. Furthermore $\Sigma_{\check\Gamma}=\Sigma_{\check\Gamma'}$ if $\check\Gamma$ and $\check\Gamma'$ are equivalent transverse loops.
\end{sub-lemma}

\begin{proof} The fact that $\Sigma_{\check\Gamma}=\Sigma_{\check\Gamma'}$ if $\check\Gamma$ and $\check\Gamma'$ are equivalent transverse loops is obvious as is the fact that $\Sigma_{\check\Gamma}$ compact. To prove that this set is not empty, let us consider a leaf $\check\phi$ that meets $\check\Gamma$. As recalled in the first section, at least one the two following assertions is true:

\smallskip
\noindent-\enskip\enskip the set $\alpha(\check \phi)$  is a non empty compact set and $\delta_{\check\Gamma}$ takes a constant positive value on it;

\smallskip
\noindent-\enskip\enskip the set $\omega(\check \phi)$  is a non empty compact set and $\delta_{\check\Gamma}$ takes a constant negative value on it.

\smallskip

Suppose for instance that we are in the first situation. If $\alpha(\check \phi)$ contains a singular point, we are done. If not, it is a closed leaf disjoint from $\check \Gamma$. More precisely, $\alpha(\check \phi)$ is contained in a bounded connected component of $\check M\setminus\check \Gamma$ where $\delta_{\check\Gamma}$ takes a constant positive value. This component contains the bounded component of the complement of $\alpha(\check \phi)$ and so contains a singular point. \end{proof}

Let us prove first that there exists an admissible loop $\check\Gamma$ that crosses each $T_i(\check \phi)$. Suppose first that $\check \gamma$ has a $\check{\mathcal{F}}$-transverse self-intersection and choose $a_1, b_1, t_1, a_2, b_2, t_2$ be such that $\check\gamma_1\vert_{[a_1,b_1]}$ intersects $\check{\mathcal{F}}$-transversally $\check\gamma_2\vert_{[a_2, b_2]}$ at $\check\gamma_1(t_1)=\check\gamma_2(t_2)$ and such that $\check\gamma_1\vert_{[a_1,b_1]}$ crosses each $T_i(\check \phi)$. The construction done in the proof of Lemma  \ref{lm:recurrent_to_loop} gives us such a loop $\check\Gamma$. Suppose now that  $\check \gamma$ has no $\check{\mathcal{F}}$-transverse self-intersection. By Proposition \ref{pr:transverse_on_sphere}, one knows that $\check\gamma$ is equivalent to the natural lift of a simple loop $\check\Gamma$ and this loop satisfies the desired property.

Let us prove now that one can find at least two distinct loops among the  $T_i^{-1}(\check\Gamma)$ that have a $\check{\mathcal{F}}$-transverse intersection. If not, by Proposition \ref{pr:transverse_on_surface}, one can find for every $i\in\{1,2,3\}$ a transverse loop $\check\Gamma'_i$ equivalent to $T_i^{-1}(\check\Gamma)$ such that the $\check\Gamma'_i$ are pairwise disjoint. The three functions  $ \delta_{\check\Gamma'_i}$ are decreasing on the leaf $\check\phi$. \textcolor{black}{For each $i$, either $ \delta_{\check\Gamma'_i}$ is not null in $\alpha(\check\phi)$ or $ \delta_{\check\Gamma'_i}$ is not null in $\omega(\check\phi)$, and therefore either there exists two different indices $i$ and $j$ such that for all points in $\alpha(\check\phi),\, \delta_{\check\Gamma'_i}\not=0$ and $\delta_{\check\Gamma'_j}\not=0$, or there exists two different indices $i$ and $j$ such that for all points in $\omega(\check\phi),\, \delta_{\check\Gamma'_i}\not=0$ and $\delta_{\check\Gamma'_j}\not=0$. In any case,  there exists} a point $\check z\in \check\phi$ and two different indices $i$ and $j$ such that $\delta_{\check\Gamma'_i}(\check z)\not=0$ and $\delta_{\check\Gamma'_j}(\check z)\not=0$. The fact that there exists a point where the two dual functions do not vanish tells us that one of the loops, let us say $\check\Gamma'_i$, is included in a bounded connected component of the complement of the other one $\check\Gamma'_j$, and that $\check\Gamma'_j$ is included in the unbounded connected component of the complement of  $\check\Gamma'_i$. One deduces that $\Sigma_{\check\Gamma'_i}\subset \Sigma_{\check\Gamma'_j}$. Setting $T=T_j\circ (T_{i})^{-1}$, one gets the inclusion $T(\Sigma_{\check\Gamma})\subset \Sigma_{\check\Gamma}$, where $\Sigma_{\check\Gamma}$ is a non empty compact set. We have found a contradiction because $T$ is a non trivial covering automorphism.

We have proved that there exist $i\not=j$ such that  $T_i^{-1}(\check\Gamma)$ and $T_j^{-1}(\check\Gamma)$ intersect $\check{\mathcal{F}}$-transversally. This implies that $\check\Gamma$ and $T(\check\Gamma)$ intersect $\check{\mathcal{F}}$-transversally, where $T=T_j\circ (T_{i})^{-1}$. Write $\check\gamma$ for the natural lift of $\check\Gamma$ and choose an integer $L$ sufficiently large, \textcolor{black}{so} that $\check\gamma\vert_{[0, L]}$ has a $\check{\mathcal{F}}$-transverse intersection with $T(\check\gamma)\vert_{[0,L]}$ at $\check\gamma(t)=T(\check\gamma)(s)$, with $s<t$. 
% Note that this implies $\gamma\vert_{[0, 2L]}$ has a transverse intersection with $T_1^{-1}(\gamma)\vert_{[0,2L]}$ at $\gamma(t)=T_1^{-1}(\gamma)(s+l)$.
The loop $\check\Gamma$ being admissible, there exists $q>0$ such that $\check\gamma\vert_{[-L,2L]}$ is admissible of order $q$. It follows from Corollaries \ref{co: induction transverse} and Proposition\ref{pr: order plane} that, for any $n>1$, the paths
$$\prod_{i=0}^{n-1}T^{i}\left( \check\gamma\vert_{[s-L,t+L]}\right), \,   \prod_{i=0}^{n-1}T^{-i}\left( \check\gamma\vert_{[t-L,s+l]}\right)$$
are admissible of order $nq$, and both have $\check{\mathcal{F}}$-transverse self-intersections. Therefore the paths $\check\gamma\vert_{[s-L,t+L]}$ and $\check \gamma\vert_{[t-L,s+l]}$ project onto closed paths of $M$ and the two loops naturally defined have $\mathcal{F}$-transverse self-intersection and are linearly admissible. So, one can deduce Proposition \ref{pr:technical_non-contractible}
from Proposition \ref{pr: realization}.
\end{proof}

Let us state a first application of Proposition \ref{pr:technical_non-contractible}. In \cite{T}  
conditions are given for a homeomorphism $f$, isotopic to the identity, of a compact surface $M$ to have only contractible periodic points. There it is shown, using Nielsen-Thurston theory, that for such $f$, under a suitable condition on the size of its fixed point  set,  there exists \textcolor{black}{an} uniform bound on the diameter of  the orbits of periodic points. The next theorem improves the main result of that note, by extending the uniform bound on the diameter of orbits from $\check f$ periodic points to $\check f$ recurrent points.  Note that the hypothesis that the fixed point set of $\check f$ project in a disk cannot be removed. There  exists an example of a $\mathcal{C}^{\infty}$ diffeomorphism $f$ of $\T^2$ preserving the Lebesgue measure and ergodic such that every periodic orbit of $f$ is contractible, and such that almost all points in the lift have orbits unbounded in every direction (see \cite{KT1}). \textcolor{black}{ The following is Theorem \ref{th:recurrent_on_the_lift_intro} of the introduction}:

\begin{theorem}\label{th:recurrent_on_the_lift} We suppose that $M$ is compact and furnished with a Riemannian structure. We endow the universal covering space $\check M$ with the lifted structure and denote by $d$ the induced distance. Let $f$ be a homeomorphism of $M$ isotopic to the identity and $\check f$ a lift to $\check M$ naturally defined by the isotopy. Assume that there exists an open topological disk $U\subset M$ such that the fixed point set of $\check f$ projects into $U$. Then;

\smallskip
\noindent-\enskip either there exists $K>0$ such that $d(\check f^n(\check z), \check z)\leq  K$, for all $n\geq 0$ and all \textcolor{black}{bi-recurrent} point $\check z$ of $\check f$;

\smallskip
\noindent-\enskip or there exists a nontrivial covering automorphism $T$ and $q>0$ such that, for all $r/s\in (-1/q,1/q)$, the map $\check f^{s}\circ T^{-r}$ has a fixed point. In particular, $f$ has non-contractible periodic points of arbitrarily large prime period.
\end{theorem}

{\it Proof.}
Let $I$ be a maximal hereditary singular isotopy larger than the given isotopy and $\mathcal F$ a foliation transverse to $I$. Denote $\check M$  the universal covering space of $M$ and $\check \pi: \check M\to M$ the covering projection. Write $\check I'$ for the lifted identity isotopy on $\check{\mathrm{dom}}(I)=\check\pi^{-1}(\mathrm{dom}(I))$ and $\check{\mathcal{F}}$ for the lifted foliation. The theorem follows directly from the next lemma and Proposition \ref{pr:technical_non-contractible}.

\begin{lemma}\label{lm:crosses_three} There exists $K>0$ such that, for all $\check z$ in $\check M$ and all $n\geq 0$, if $d(\check f^n(\check z),\check z)\ge K$, then there exists a leaf $\check \phi$ and three distinct covering automorphisms $T_i$, $1\leq i\leq 3$, such that $I_{\check{\mathcal F}}^n(\check z)$ crosses each $T_i(\check \phi)$.
\end{lemma}

\begin{proof}  One can find a neighborhood $V\subset U$ of $\mathrm{sing}(I)$ such that for every point $\check z\in \pi^{-1}(V)$, the points $\check z$ and $\check f(\check z)$ belong to the same connected component of $\pi^{-1}(U)$.  For reasons explained in the proof of Lemma \ref{le: neighborhood of infinity}, one knows that for  every $z\in M\setminus V$, there exists a small open disk $O_z\subset \mathrm{dom}(\mathcal{F})$ containing $z$ such that $I_{\mathcal F}^2(f^{-1}(z'))$ crosses $\phi_{z}$ if $z'\in O_z$. By compactness of $M\setminus V$, one can cover this set by a finite family $(O_{z_i})_{1\leq i\leq r}$. One constructs easily a partition $(X_{z_i})_{1\leq i\leq r}$ of $M\setminus V$ such that $X_{z_i}\subset O_{z_i}$. We have a unique partition  $(\check X_{\alpha})_{\alpha\in \mathcal{A}}$ of $\check M$ such that, either $\check X_{\alpha}$ is contained in a connected component of $\pi^{-1}(U)$ and projects onto $V$, or  there exists $i\in\{1,\dots,r\}$ such that $\check X_{\alpha}$ is contained in a connected component of $\pi^{-1}(O_{z_i})$ and projects onto $X_{z_i}$.  Write $\alpha(\check z)=\alpha$ if $\check z\in \check X_{\alpha}$. Let us define
$$K_0=\max_{\check z\in \check M} d( \check f( \check z), \check z)),\enskip K_1=\max_{\alpha\in  \mathcal{A}}\mathrm{diam}(\check X_{\alpha}) .$$
Fix $\check z\in\check M$, $n\geq 1$ and define a sequence $n_0<n_1<\dots<n_s$ in the following inductive way: 
$$n_0=0, \enskip n_{j+1}= 1+ \sup\{k\in\{n_j, \dots , n-1\}\,\vert\,\alpha(\check f^k(\check z))=\alpha(\check f^{n_j}(\check z))\}, \enskip n_s=n.$$
Note that  $d(\check f^{n_j}(\check z),\check f^{n_{j+1}}(\check z)) \leq K_0+K_1$, if $j<s$. \textcolor{black}{Note also that, if $\check X_{\alpha(\check f^{n_j}(\check z))}$ projects on $V$, then $f^{n_{j+1}-1}(z)$ also belongs to $V$ and by the choice of $V$ both $\check f^{n_{j+1}-1}(\check z)$ and $\check f^{n_{j+1}}(\check z)$ belong to the same connected component of $\pi^{-1}(U)$. As $\alpha(\check f^{n_j}(\check z))\not=\alpha(\check f^{n_{j+1}}(\check z))$, one gets that $\check X_{\alpha(\check f^{n_j+1}(\check z))}$ do not project on $V$.}
Fix $K>(6r+1)(K_0+K_1
)$. If $d(\check f^n(\check z),\check z)\ge K$, then $s\geq 6r+1$ and there exist at least $3r$ sets $\check X_{\alpha(\check f^{n_j}(\check z))}$, $0<j<s$, that do not project on $V$. This implies that there exist three points $f^{n_{j_l}}(\pi(z))$ that belong to the same $X_{z_i}$\textcolor{black}{, and therefore one finds that there exist a point $\check z_i\in \pi^{-1}(z_i)$ and two distinct nontrivial covering automorphisms $T_1, T_2$ such that the orbit of $\check z$ intersects the three distinct connected components of $\pi^{-1}(O_{z_i})$ that contain $\check z_i,T_1(\check z_i)$ and $T_2(\check z_i)$, respectively. By the choice of $O_{z_i}$, this implies that $I_{\check{\mathcal F}}^n(\check z)$ intersect $\phi_{\check z_i}, T_1(\phi_{\check z_i})$ and $T_2(\phi_{\check z_i})$}.
\end{proof}

Proposition \ref{pr:technical_non-contractible} is also fundamental in solving the following conjecture posed by P. Boyland (see, for instance, \cite{AT} where the conjecture is shown to be true generically for sufficiently smooth diffeomorphisms): Let $f$ be a homeomorphism of the closed annulus preserving a probability measure $\mu$ with full support, and let $\check f$ be a lift of $f$ to the universal covering space of the annulus. If the rotation set of $f$ is a non trivial segment and the rotation number of $\mu$ is null, is it true that $\mathrm{rot}(\mu)$ belongs to the interior of the rotation set?

We recall first Atkinson's Lemma on ergodic theory, that will be very useful in this paper \textcolor{black}{(see \cite{A}).}

\begin{proposition} \label{pr:atkinson}Let $(X,\mathcal{B},\mu)$ be a probability space, and let $T:X\to X$ be an ergodic automorphism. If $\varphi: X\to \R$ is an integrable map such that $\int\varphi\, d\mu=0$, then for every $B\in\mathcal{B}$ and every $\varepsilon >0$, one has
$$ \mu\left( \left\{x\in B, \exists n\geq 0, \enskip T^n(x)\in B\enskip\mathit{and}\enskip\left\vert \sum_{k=0}^{n-1}\varphi(T^k(x))\right\vert< \varepsilon\right\}\right)= \mu (B).$$
\end{proposition}

We have the following, \textcolor{black}{ Theorem \ref{th : annulus_intro} of the introduction}:

\begin{theorem}\label{th : annulus}
Let $f$ be a homeomorphism of $\A=\T^1\times[0,1]$ that is isotopic to the identity and  $\check f$ a lift to $\R\times[0,1]$. Suppose that $\mathrm{rot}(f)$ is a non trivial segment and that one of its endpoint $\rho$ is rational. Define
$${\mathcal M}_{\rho}=\left\{ \mu\in {\mathcal M}(f)\,,\, \mathrm{rot}(\mu)=\rho\right\}, \enskip {X}_{\rho}= 
\overline{\bigcup_{\mu\in {\mathcal M}_{\rho}} \mathrm{supp}(\mu)}.$$
Then every invariant measure supported on ${X}_{\rho}$ belongs to ${\mathcal M}_{\rho}$.
\end{theorem}    

\begin{proof}
Replacing $f$ by a power $f^q$ and $\check f$ by a lift $\check f^q\circ T^p$, one can assume that $\rho=0$ and $\mathrm{rot}(f)=[0,a]$, where $a>0$. The fact that $0$ is extremal implies that for every $\mu\in {\mathcal M}_{0}$, each ergodic measure $\mu'$ that appears in the ergodic decomposition of $\mu$ also belongs to ${\mathcal M}_{0}$. Atkinson's Lemma, with $T=f$ and $\varphi$  the map lifted by  $\check\varphi: z\mapsto \pi_1(\check f(\check z)-\check z)$, tells us that  $\mu'$-almost every point of $\A$ is lifted to  a recurrent point of $\check f$. The union of the supports of such ergodic measures being dense in $\mathrm{supp}(\mu)$, one deduces that the recurrent set of $\check f$ is dense in $\pi^{-1}(X_0)$. Writing $f=(f_1,f_2)$, one can extend $f$ to a homeomorphism of $\T^1\times\R$ such that $f(x,y)=(f_1(x,1), y)$ if $y\geq 1$ and $f(x,y)=(f_1(x,0), y)$ if $y\leq 0$ and still 
denote by $\check f$ the lift that extends the initial lift.  Let $I'$ be an identity isotopy of $f$ that is lifted to an identity isotopy \textcolor{black}{$\check I'$} of $\check f$. Let $I$ be a maximal hereditary singular isotopy larger than $I'$ and $\mathcal F$ a foliation transverse to $I$. Consider the lift  \textcolor{black}{$\check I$} of $I$ and the lifted foliation $\check{\mathcal{F}}$. If  there exists an invariant measure supported on ${X}_{0}$ whose rotation number is positive, there exists a recurrent point $z$ of rotation number strictly larger than $0$. Let us fix a lift $\check z$. As $\check{z}$ is not fixed by $\check f$, it belongs to the domain of $\check I$ and the path ${\mathcal I}_{\check{\mathcal F}}^{\Z}(\check z)$ meets infinitely many translates of $\phi_{\check z}$. But $\check z$ can be approximated by a recurrent point $\check z' $ \textcolor{black}{of $\check f$ because we have seen that the recurrent set of $\check f$ was dense in $\pi^{-1}(X_0)$. So we can suppose that ${\mathcal I}_{\check{\mathcal F}}^{\Z}(\check z')$ meets at least three translates of $\phi_{\check z}$}. The result now follows from Proposition \ref{pr:technical_non-contractible}. Indeed\textcolor{black}{, one finds some power $n$ of $T$ such that for any pair of integers $r,s$ with $s>0$ and such that $\abs{r/s}$ is sufficiently small, there exists a fixed point $\check z_{r,s}$ of $\check f^{r}\circ T^{-ns}$. The points $\check z_{r,s}$ project to periodic points $z_{r,s}$ in $\A$ such that the rotation number of $z_{r,s}$ is $r/ns$. In particular both $1/sn$ and $-1/sn$ belong to the rotation set of $\check f$ if $s$ is sufficiently large,} in contradiction with the fact that $0$ is an end of $\mathrm{rot}(f)$.
\end{proof}

 We deduce immediately the positive answer to  Boyland's question, \textcolor{black}{ Corollary \ref{cr:boylandannulus_intro} of the introduction}:

\begin{corollary}
Let $f$ be a homeomorphism of $\A$ that is isotopic to the identity and preserves a probability measure $\mu$ with full support. Let us fix a lift $\check f$. Suppose that $\mathrm{rot}(f)$ is a non trivial segment.  The rotation number $\mathrm{rot}(\mu)$ cannot be an endpoint of $\mathrm{rot}(f)$ if this endpoint  is rational.
\end{corollary}  

\begin{proof}
\textcolor{black}{ If $\mathrm{rot}(f)$ is an endpoint of $\mathrm{rot}(f)$ and this endpoint is a rational $\rho$, by Theorem \ref{th : annulus} we get that, if $X_{\rho}= \mathrm{supp}(\mu)$ is the whole annulus $\A$, then every invariant measure supported on $\A$ has  rotation vector $\rho$, which implies that $\mathrm{rot}(f)=\{\rho\}$.}   
\end{proof}

\section{ Entropy zero conservative  homeomorphisms of the sphere}

We will prove in this section the improvement of Franks-Handel's result about  area preserving diffeomorphisms of $\S^2$ with entropy zero, stated in the introduction \textcolor{black}{ as Theorem \ref{th: FH_intro}}. Let us begin by introducing an important notion due to Franks and Handel: let $f:\S^2\to\S^2$ be an orientation preserving homeomorphism, a point $z$ is {\it free disk recurrent }  if there exist an integer $n>1$ and a topological open disk $D$ containing $z$ and $f^n(z)$ such that $f(D)\cap D=\emptyset$. We will also need the notion of {\it heteroclinic point}, which means that its $\alpha$-limit and $\omega$-limit sets are included in connected subsets of $\mathrm{fix}(f)$.

\medskip

Let us state first some easy but useful facts. By definition if $f$ is a homeomorphism of a topological space $X$, a subset $Y$ is {\it free} if $f(Y)\cap Y=\emptyset$.

\begin{proposition} \label{pr: free disk} One has the following results:

\smallskip
\noindent{\bf i)}\enskip the set of free disk recurrent points is an invariant open set $\mathrm{fdrec}(f)$;

\smallskip
\noindent{\bf ii)}\enskip it contains every positively or negatively recurrent point outside $\mathrm{fix}(f)$;

\smallskip
\noindent{\bf iii)}\enskip every point in  $\S^2\setminus \mathrm{fdrec}(f)$ is heteroclinic;

\smallskip
\noindent{\bf iv)}\enskip every periodic connected component of $\mathrm{fdrec}(f)$ is fixed.

\end{proposition}

\begin{proof} If $D$ is a free disk  that contains $z$ and $f^n(z)$, it contains $z'$ and $f^n(z')$ if $z'$ is close to $z$. Moreover $f^k(D)$ is a free disk that contains $f^k(z)$ and $f^{k+n}(z)$, for every $k\in\Z$. So {\bf i)} is true.

For every $z\in\S^2\setminus \mathrm{fix}(f)$, one can choose a free disk $D$ that contains $z$. If $z$ is positively recurrent, there exists $n>1$ such that $f^n(z)\in D$.  If $z$ is negatively recurrent, there exist $n>1$ such that $f^{-n}(z)\in D$, which implies that $f^n(D)$ is a free disk that contains $z$ and $f^n(z)$. In both cases,  $z$ belongs to $\mathrm{fdrec}(f)$, which means that {\bf ii)} is true.

It is sufficient to prove {\bf iii)} for the $\omega$-limit set, the proof for the $\alpha$-limit set being similar. Let us prove first that $\omega(z)\subset \mathrm{fix}(f)$ if  $z\not\in \mathrm{fdrec}(f)$. Indeed, if $z'\in\omega(z)\setminus \mathrm{fix}(f)$, one can choose a free disk $D$ containing $z'$ and two integers $n'>n$ such that $f^n(z)$ and $f^{n'}(z)$ belong to $D$. It implies that $f^{-n}(D)$ is a free disk that contains $z$ and $f^{n'-n}(z)$. This contradicts the fact that $z\not\in\mathrm{fdrec}(f)$. To prove that $\omega(z)$ is included in a connected component of $ \mathrm{fix}(f)$, it is sufficient to prove that it is contained in a connected component of $O$, for every neighborhood $O$ of $ \mathrm{fix}(f)$. If $O$ is such a neighborhood,there exists a neighborhood $O'\subset O$ of $ \mathrm{fix}(f)$ such that for every $z\in O'\cap f^{-1}(O')$, the points $z$ and $f(z)$ belong to the same connected component of $O$. There exists $N$ such that $f^n(z)\in O'$ for every $n\geq N$. This implies that the $f^n(z)$, $n\geq N$, belong to the same connected component of $O$.

It remains to prove {\bf iv)}. If $W$ is a connected component of $\mathrm{fdrec}(f)$ of period $q>1$, it is not a connected component of $\S^2\setminus \mathrm{fix}(f)$ (see Brown-Kister \cite{BK}) and so one can find a  path $\alpha$ in $\S^2\setminus \mathrm{fix}(f)$ joining a point $z\in W$ to a point $z'\not\in W$.  Taking a subpath if necessary, one can suppose that $\gamma$ is included in $W$ but the endpoint $z'$ (which is in the frontier of $W$ and not fixed). Let us choose a path $\beta$ in $W$ joining $z$ to $f^q(z)$. It is a classical fact that there exists a simple path $\gamma$ joining $z'$ to $f^q(z')$ whose image is included  in the image of $\alpha^{-1}\beta f^q(\alpha)$. The point $z'$ is not periodic because it is neither in $\mathrm{fix}(f)$ nor in $\mathrm{fdrec}(f)$ and so the points $z'$, $f(z')$, $f^{q}(z')$ and $f^{q+1}(z')$ are distinct (recall that $q>1$). More precisely, \textcolor{black}{ since $\gamma\subset W$ and $W$ is free,} the path $\gamma$ is free and so one can find a free disk that contains it, which contradicts the fact that $z'$ is not in $\mathrm{fdrec}(f)$.
\end{proof}

Suppose now that the set of positively recurrent points is dense.  It is equivalent to say that $\Omega(f)=\S^2$ and in that case the set  of positively recurrent points is a dense $G_{\delta}$ set, as is the set of  bi-recurrent points (these conditions are satisfied in the particular case of an area preserving homeomorphism).  Note that, in this case, every connected component of $\mathrm{fdrec}(f)$ is periodic and so is fixed. Write $(W_{\alpha})_{\alpha\in \mathcal{A}_f}$ for the family of connected components of $\mathrm{fdrec}(f)$ and define $A_{\alpha}$ to be the interior in $\S^2\setminus \mathrm{fix}(f)$ of the closure of $W_{\alpha}$. 
Note that  $$A_{\alpha}= \S^2\setminus \overline{\bigcup_{\alpha'\in \mathcal{A}_f\setminus\{\alpha\}} A_{\alpha'}}\cup\mathrm{fix}(f)$$ 
because the recurrent points are dense in $\S^2$ and contained in $\mathrm{fdrec}(f)$ if not fixed.

We will prove the following result,  which implies \textcolor{black}{Theorem \ref{th: FH_intro} of the introduction, and } that extends Theorem 1.2 of \cite{FH}.

\begin{theorem}\label{th: FH} Let $f:\S^2\to\S^2$ be an orientation preserving homeomorphism such that $\Omega(f)=\S^2$ and $h(f)=0$. Then one has the following results:

\smallskip
\noindent{\bf i)}\enskip each $A_{\alpha}$ is an open annulus;

\smallskip
\noindent{\bf ii)}\enskip the sets $A_{\alpha}$ are the maximal fixed point free invariant open annuli;

\smallskip
\noindent{\bf iii)}\enskip every point that is not in  a $A_{\alpha}$  is heteroclinic;

\smallskip
\noindent{\bf iii)}\enskip let $C$ be a connected component of the frontier of $A_{\alpha}$ in $\S^2\setminus\mathrm{fix}(f)$, then the connected components of $\mathrm{fix}(f)$ that contain $\alpha(z)$ and $\omega(z)$ are independent of $z\in C$.

\end{theorem}

\bigskip
We will begin by stating a local version of this result, which means a version relative to a given maximal hereditary singular isotopy $I$. We denote $\widetilde{I}$ the lifted identity isotopy to the universal covering space $\widetilde{\mathrm {dom}}(I)$ of  ${\mathrm {dom}}(I)$ and $\widetilde{f}$ the induced lift of $f\vert_{ {\mathrm {dom}}(I)}$.
Say that a point $z\in {\mathrm {dom}}(I)$ is {\it free disk recurrent relative to $I$} or {\it $I$ free disk recurrent} if there exists an integer $n>0$ and a topological open disk $D\subset \mathrm{dom}(I)$ containing $z$ and $f^n(z)$, such that each lift to $\widetilde{\mathrm {dom}}(I)$ is disjoint from its image by $\widetilde{f}$ (we will say that $D$ is $I$-free). Let us state the local version of Proposition \ref{pr: free disk}.

 \bigskip
\begin{proposition}\label{pr: free disk local }  One has the following results:

\smallskip
\noindent{\bf i)}\enskip the set of $I$-free disk recurrent points is an invariant open set $\mathrm{fdrec}(I)$;

\smallskip
\noindent{\bf ii)}\enskip it contains every positively or negatively recurrent point in $\mathrm{dom}(I)$;

\smallskip
\noindent{\bf iii)}\enskip every point in  $\S^2\setminus \mathrm{fdrec}(I)$ is heteroclinic and its $\alpha$-limit and $\omega$-limit sets are included in connected subsets of $\mathrm{sing}(I)$;

\smallskip
\noindent{\bf iv)}\enskip every periodic connected component of $\mathrm{fdrec}(I)$ is fixed and lifted to fixed subsets of $\check f$.
\end{proposition}

\begin{proof} Replacing free disks by $I$-free disks, one proves the three first assertions exactly like in the global situation. Similarly, one can prove that every periodic connected component of $\mathrm{fdrec}(I)$ is fixed. Writing $\pi:\widetilde{\mathrm {dom}}(I)\to{\mathrm {dom}}(I)$ for the universal covering projection, it remains to prove that the connected components of $\pi^{-1}(W)$ are fixed by $\widetilde{f}$, if $W$ is a fixed connected component of $\mathrm{fdrec}(I)$. If they are not fixed, they are not connected components of  $ \widetilde{\mathrm {dom}}(I)$, which means that $W$ is not a connected component of $\mathrm {dom}(I)$. So one can find a simple path $\alpha$ joining a point $z\in W$ to a point $z'\in \partial W\cap \mathrm {dom}(I)$ and included in $W$ but the endpoint $z'$, and then construct a simple path $\gamma$ joining $z'$ to $f^2(z')$ included in $W$ but the two endpoints. It will lift to a $\widetilde{f}$- free simple path and so one can find a $I$-free disk that contains $\gamma$. This contradicts the fact that $z'$ is not in $\mathrm{fdrec}(I)$.
\end{proof}

Suppose now that $\Omega(f)=\S^2$.  Write $(W_{\beta})_{\beta\in \mathcal{B}_I}$ for the family of connected components of $\mathrm{fdrec}(I)$ and define $A_{\beta}$ to be the interior in $\mathrm{dom}(I)$ of the closure of $W_{\beta}$. One knows that the sets $W_{\beta}$, $A_{\beta}$ are fixed and lifted to fixed subsets of $\widetilde{f}$. Here again, one has  $$A_{\beta}=\mathrm{dom}(I)\setminus \overline{\bigcup_{\beta'\in \mathcal{B}_I\setminus\{\beta\}} A_{\beta'}}.$$ 

The local version of Theorem \ref{th: FH} is the following:

 \begin{theorem} \label{th: FH local} Let $f:\S^2\to\S^2$ be an orientation preserving homeomorphism such that $\Omega(f)=\S^2$ and $h(f)=0$, and $I$ a hereditary singular maximal isotopy. Then one has the following results:

\smallskip
\noindent{\bf i)}\enskip each $A_{\beta}$ is an open annulus;

\smallskip
\noindent{\bf ii)}\enskip the sets $A_{\beta}$ are the maximal invariant open annuli contained in $\mathrm{dom}(I)$;

\smallskip
\noindent{\bf iii)}\enskip every point that is not in  a $A_{\beta}$  is heteroclinic and its $\alpha$-limit and $\omega$-limit sets are included in connected subsets of $\mathrm{sing}(I)$;

\smallskip
\noindent{\bf iv)}\enskip let $C$ be connected component of the frontier of $A_{\beta}$ in $\mathrm{dom}(I)$, then the connected components of $\mathrm{sing}(I)$ that contain $\alpha(z)$ and $\omega(z)$ are independent of $z\in C$.
\end{theorem}

\bigskip
Let us explain first why the local theorem implies the global one. If $A$ is a topological annulus, an open set will be called {\it essential} if it contains an essential loop and {\it inessential} otherwise. A  closed set will be called {\it inessential} if there exists a connected component of its complement that is a neighborhood of the two ends in the case where $A$ is open, that meets the two boundary circle in the case where $A$ is closed, and that is a neighborhood of the unique end and meets the boundary circle in the remaining case. Otherwise, we will say that this set is {\it essential}.

\medskip

{\it Proof of  Theorem \ref{th: FH}, first part, proof of assertions (i), (ii), and (iii)}. \enskip Let us explain first why every fixed point free invariant \textcolor{black}{open} annulus $A$ is contained in an $A_{\alpha}$, $\alpha\in\mathcal{A}_f$. It is sufficient to prove that $\mathrm{fdrec}(f)\cap A$ is connected. Indeed $\mathrm{fdrec}(f)\cap A$ will be contained in a $W_{\alpha}$, $\alpha\in\mathcal{A}_f$, and consequently $A$ will be contained in $A_{\alpha}$. Let $W$ be a connected component of $ \mathrm{fdrec}(f)\cap A$. Applying Proposition \ref{pr: free disk}  to the end compactification of $A$, one knows that $W$ is fixed. If it is inessential, one gets an invariant open disk $D\subset A$ by filling $W$, which means adding the inessential components of its complement. \textcolor{black}{By Brouwer's plane translation Theorem, since the restriction of $f$ to $D$ has non wandering points, there must exist} a fixed point in this disk, which is impossible. So, every connected component of $ \mathrm{fdrec}(f)\cap A$ is essential. Suppose now that $ \mathrm{fdrec}(f)\cap A$ has at least two connected components. The complement in $A$ of the union of two such components has a unique compact connected component. It is located ``between" these components. This last set is invariant (by uniqueness) and contains points that are not free disk recurrent. But one knows that the $\alpha$-limit and $\omega$-limit sets of such points \textcolor{black}{contain fixed points and $A$ is fixed point free.} We have a contradiction.

Let us prove now that each $A_{\alpha}$, $\alpha\in{\mathcal A}_f$, is an annulus. It is sufficient to prove that it is contained in a fixed point free invariant annulus.
Let us consider a sequence $(z_i)_{i\geq 0}$  dense in $\mathrm{fix}(f)$, sequence which will be finite if there are finitely many fixed points. Let us fix $A_{\alpha}$. Let $I_0$ be a maximal hereditary singular isotopy whose singular set contains $z_0$, $z_1$, $z_2$. The set $W_{\alpha}$ is connected and included in $\mathrm{fdrec}(I_0)$ so it is contained in a connected component $W_{\beta_0}$, $\beta_0\in \mathcal B_{I_0}$. 
One deduces that $A_{\alpha}\subset A_{\beta_0}$.  If $ A_{\beta_0}$ is fixed point free, we stop the process.  If not, we consider the first $z_{k_1}$ that belongs to   $A_{\beta_0}$ and consider a maximal hereditary singular isotopy $I_1$ of $f\vert_{A_{\beta_0}}$ whose singular set contains $z_{k_1}$. Similarly, there exists $\beta_1\in \mathcal B_{I_1}$ such that $A_{\alpha}\subset A_{\beta_1}$. If $ A_{\beta_1}$ is fixed point free, we stop the process.  If not, we consider the first $z_{k_2}$ that belongs to  $ A_{\beta_1}$ and we continue. If the process stops,  the last annulus will be fixed point free. If the process does not stop, $A_{\alpha}$ is contained in the interior of $\bigcap_{i\geq 0} A_{\beta_i}$. The connected component $W$ of the interior of $\bigcap_{i\geq 0} A_{\beta_i}$ that contains $A_{\alpha}$ is invariant. Moreover, it is fixed point free because it is open and because the sequence $(z_i)_{i\geq 0}$  is dense in $\mathrm{fix}(f)$ and away from $W$. Let us prove that for $i$ large enough, $A_{\beta_{i+1}}$ is essential in  $A_{\beta_{i}}$ and that $W$ is an annulus. Let us suppose that $A_{\beta_{i+1}}$ is  inessential in  $A_{\beta_{i}}$ for infinitely many $i$.
Consider a simple loop $\Gamma$ in $W$. It bounds a disk (uniquely determined) included in \textcolor{black}{$A_{\beta_{i}}$}, every time $A_{\beta_{i+1}}$ is inessential in  $A_{\beta_{i}}$, which implies that it bounds a disk included in $\bigcap_{i\geq 0} A_{\beta_i}$, and so included in $W$. This means that $W$ is a disk, which contradicts the fact it is fixed point free. Suppose now that $A_{\beta_{i+1}}$ is essential in  $A_{\beta_{i}}$ for every $i\geq i_0$. \textcolor{black}{By the same reasoning, if $\Gamma\subset W$ is a simple loop such that $\Gamma$ is inessential in the $A_{\beta_i}$, $i\geq i_0$, then $\Gamma$ bounds} a disk in $W$. This implies that $W$ is an open annulus, that is essential in the $A_{\beta_i}$, $i\geq i_0$.

The assertion {\bf iii)} is obvious because every free disk recurrent point is contained in a $W_{\alpha}$ and so in \textcolor{black}{an} $A_{\alpha}$. We will postpone the proof of {\bf iv)} to the end of this section because we need a little bit more than what is stated in the local theorem.
 \hfill$\Box$

\bigskip

Before proving Theorem \ref{th: FH local}, we will state a result relative to a couple $(I, \mathcal F)$, where $\mathcal F$ is a foliation transverse to $I$. By Theorem \ref{th:transverse_imply_entropy} and the density of the set of recurrent points, one knows that two transverse trajectories never intersect $\mathcal{F}$-transversally. In particular, there is no transverse trajectory with $\mathcal{F}$-transverse self-intersection and by Proposition \ref{pr:transverse_on_sphere} every whole transverse trajectory of an $\mathcal{F}$-bi-recurrent point is equivalent to the natural lift of a transverse simple loop $\Gamma$. We denote by ${\mathcal G}_{I,\mathcal F}$ the set of such loops (well defined up to equivalence) and 
$\mathrm{rec}(f)_{\Gamma}$ the set of bi-recurrent points whose whole transverse trajectory is equivalent to the natural lift of $\Gamma$. \textcolor{black}{Consider a point $z\in\mathrm{dom}(I)$. For any given segment of $I_{\mathcal F}^{\Z}(z)$ there exists a neighborhood of $z$ such that this segment is equivalent to a subpath of $I_{\mathcal F}^{\Z}(z')$ if $z'$ belongs to this neighborhood. Suppose now that this segment meets a leaf more than once. The transverse simple loop $\Gamma$ associate to $z'$ does not depend on $z'$, if $z'$ is chosen bi-recurrent (remind that the set of bi-recurrent points is dense). Summarizing, we have stated that any segment of $I_{\mathcal F}^{\Z}(z)$ is equivalent to a subpath of the natural lift of a transverse simple loop, but this loop is uniquely defined (up to equivalence) if this segment meets a leaf more than once. Consequently, if $I_{\mathcal F}^{\Z}(z)$ meets a leaf more than once, it is equivalent to a subpath of the natural lift of a uniquely defined transverse simple loop}. One deduces that the set of points whose whole transverse trajectory meets a leaf more than once, admits a partition
$\bigsqcup_{\Gamma\in\mathcal{G}_{I,\mathcal F}} W_{\Gamma} $ in disjoint invariant open sets, where $z\in W_{\Gamma} $ if $I_{\mathcal F}^{\Z}(z)$ meets a leaf  at least twice and is a subpath of the natural lift of $\Gamma$. Define $A_{\Gamma}=\mathrm{int}(\overline {W_{\Gamma}})$. Note that  $$A_{\Gamma}=\mathrm{int}(\overline {\mathrm{rec}(f)_{\Gamma}})= \mathrm{dom}(I)\setminus \overline{\bigcup_{\Gamma'\in\mathcal{G}_{I,\mathcal F}\setminus\{\Gamma\}} A_{\Gamma'}}.$$ 
Recall that $U_{\Gamma}$ is the union of leaves that meet $\Gamma$.

\begin{proposition}\label{pr:FH_leafedversion} One has the following results: 
 
 \smallskip
 \noindent{\bf i)} the set $A_{\Gamma}$ is an essential open annulus of $U_{\Gamma}$;

 \smallskip
 \noindent{\bf ii)}  every point in $\mathrm{dom}(I)\setminus \bigcup_{\Gamma\in \mathcal {G}_{I,F}} A_{\Gamma}$ is heteroclinic and its $\alpha$-limit and $\omega$-limit sets are included in connected subsets of $\mathrm{sing}(I)$;
 
  \smallskip
 \noindent{\bf iii)}  let $C$ be a connected component of the frontier of $A_{\Gamma}$ in $\mathrm{dom}(I)$, then the connected components of $\mathrm{sing}(I)$ that contain $\alpha(z)$ and $\omega(z)$ are independent of $z \in C$.
 \end{proposition}
 
\subsection{Proof of Proposition \ref{pr:FH_leafedversion}}

This subsection is devoted entirely to the proof of Proposition \ref{pr:FH_leafedversion}.

The assertion {\bf ii)} is an immediate consequence of the following: if $z' \in\mathrm{dom}(I)$ belongs to the $\alpha$-limit or $\omega$-limit set of $z\in\mathrm{dom}(I)$, then the whole transverse trajectory of $z$ meets infinitely often the leaf $\phi_{z'}$ and so $z$ belongs to $\bigcup_{\Gamma\in \mathcal {G}_{I,\mathcal F}} W_{\Gamma}$.

\medskip
Let us prove {\bf i)}. One can always suppose that $\mathrm{dom}(I)$ is connected, otherwise one must replace $\mathrm{dom}(I)$ by its connected component that contains $\Gamma$ in what follows. Fix a lift $\widetilde\gamma$ of $\Gamma$ in $\widetilde{\mathrm{dom}}(I)$, write $T$ for the covering automorphism such that $\widetilde\gamma(t+1)=T(\widetilde\gamma(t))$, write $\widehat{\mathrm{dom}}(I)=\widetilde{\mathrm{dom}}(I)/T$ for the annular
covering space associated to $\Gamma$. Denote by $\widehat\pi: \widehat{\mathrm{dom}}(I)\to\mathrm{dom}(I)$ the covering projection, by $\widehat I$ the induced identity isotopy, by $\widehat f$ the induced lift of $f$, by $\widehat{\mathcal{F}}$ the induced foliation. The line $\widetilde\gamma$  projects onto the natural lift of a loop $\widehat\Gamma$. The union of leaves
 that meet $\widehat\Gamma$, denoted by $U_{\widehat\Gamma}$, is the annular component of $\widehat\pi^{-1}(U_{\Gamma})$. \textcolor{black}{We note that there cannot be an essential simple closed curve $\widehat \Gamma'$ contained in $U_{\widehat\Gamma}$ whose image by $\widehat f$ is disjoint from itself, otherwise the region bounded by $\widehat \pi(\widehat \Gamma')$ and $\widehat \pi(\widehat f(\widehat \Gamma'))$ would be wandering for any $f$. In particular, by the same reasoning as in Lemma \ref{le: two cases}, one gets that the $\alpha$ and $\omega$ limit of any given leaf of $\widehat{\mathcal{F}}$ that is contained in $U_{\widehat\Gamma}$ must be different ends of $\widehat{\mathrm{dom}}(I)$, and every leaf of $\widehat{\mathcal{F}}$ that does not intersect $U_{\widehat\Gamma}$ disconnects $\widehat{\mathrm{dom}}(I)$.} 
One gets a sphere $\widehat{\mathrm{dom}}(I)_{\mathrm {sph}}$
 by adding the end $N$ of $\widehat{\mathrm{dom}}(I)$ at the left of $\Gamma$ and the end $S$ at the right. The complement of $U_{\widehat \Gamma}$ has two connected components $l(\widehat\Gamma)\cup\{N\}$ and $r(\widehat\Gamma)\cup\{S\}$. Note that $\widehat\Gamma$ is the unique simple loop (up to equivalence) that is transverse to $\widehat{\mathcal{F}}$. Like in $\mathrm{dom}(I)$, transverse trajectories do not intersect $\widehat{\mathcal{F}}$-transversally. The set of points that lift a bi-recurrent point of $
 f$ is dense. If the trajectory of such a point $z$ meets a leaf at least twice, then $\widehat I^{\Z}_{ \widehat{\mathcal{F}}}(z)$ is the natural lift of $\widehat \Gamma$. Denote $\mathrm{rec}(f)_{\widehat \Gamma}$ the set of such points. Otherwise $\widehat I^{\Z}_{ \widehat{\mathcal{F}}}(z)$ meets either $l(\widehat\Gamma)$ or $r(\widehat\Gamma)$, the two situations being excluded, because $\widehat I^{\Z}_{ \widehat{\mathcal{F}}}(z)$ does not intersect $\widehat\Gamma$ $\widehat{\mathcal{F}}$-transversally. Denote $\mathrm{rec}(f)_N$ and $\mathrm{rec}(f)_S$ the set of points $z$ that lift a bi-recurrent point of $f$ and such that $\widehat I^{\Z}_{ \widehat{\mathcal{F}}}(z)$ meets $l(\widehat\Gamma)$ and $r(\widehat\Gamma)$ respectively. Note that the intersection of the complete transverse trajectory of $z\in\mathrm{rec}(f)_N$ and $U_{\widehat \Gamma}$, when not empty is equivalent to $\widehat\Gamma\vert_J$ where $J$ is an open interval of $\T^1$, and a  similar statement holds if $z\in\mathrm{rec}(f)_S$. In particular there exists $n\geq 0$ such that  $\widehat I^{\N}_{ \widehat{\mathcal{F}}}(\widehat f^n(z))\subset l(\widehat\Gamma)$ and $\widehat I^{-\N}_{ \widehat{\mathcal{F}}}(\widehat f^{-n}(z))\subset l(\widehat\Gamma)$.  Write $W_{\widehat\Gamma}$ for the set of points such that  $\widehat I^{\Z}_{ \widehat{\mathcal{F}}}(z)$ meets a leaf at least twice, write $W_N$ for the set of points $z\in
 \widehat{\mathrm{dom}}(I)$ such that  $\widehat I^{\Z}_{ \widehat{\mathcal{F}}}(z)$ meets $l(\widehat\Gamma)$, write $W_S$ for the set of points such that  $\widehat I^{\Z}_{ \widehat{\mathcal{F}}}(z)$ meets $r(\widehat\Gamma)$. We get three disjoint invariant open sets, that contain $\mathrm{rec}(f)_{\widehat \Gamma}$, $\mathrm{rec}(f)_N$, $\mathrm{rec}(f)_S$ respectively and whose union is dense. Note that the $\alpha$-limit and $\omega$-limit sets of a point $z\not\in W_{\widehat\Gamma}$ are reduced to one of the ends. These ends are both equal to $N$ if  $z\in\mathrm{rec}(f)_N$ and both equal to $S$ if  $z\in\mathrm{rec}(f)_S$. We will see later that they 
are both equal to $N$ if  $z\in \overline{W_N}$ and both equal to $S$ if  $z\in \overline{W_S}$. Note also that 
$$W_N=\bigcup_{k\in\Z} f^{-k}\left(l(\Gamma)\right), \enskip W_S=\bigcup_{k\in\Z} f^{-k}\left(r(\Gamma)\right).$$
Indeed, every leaf $\phi$ that is not in $U_{\widehat \Gamma}$ bounds a disk disjoint from $U_{\widehat \Gamma}$. So, if $\widehat I_{ \widehat{\mathcal{F}}}(z)$ meets $\phi$ and $\phi\subset l(\Gamma)$, then one of the point $z$ or $\widehat f(z)$ is in $l(\Gamma)$, and if $\phi\subset r(\Gamma)$, then one of the point $z$ or $\widehat f(z)$ is in $r(\Gamma)$.

 Observe that $W_{\widehat\Gamma}$ projects homeomorphically on $W_{\Gamma}$ and that $A_{\widehat\Gamma}=\mathrm{int}(\overline{W_{\widehat\Gamma}})=\widehat{\mathrm{dom}}(I)_{\mathrm {sph}}\setminus \overline{W_N}\cup \overline{W_S}$ projects homeomorphically on $A_{\Gamma}$. We want to prove that  $A_{\widehat\Gamma}$ is an annulus.

\begin{lemma}\label{lm:phi_S}There exists a leaf $\phi_S$ in $U_{\widehat \Gamma}$ that does not meet $\overline{W_N}$.
\end{lemma}

\begin{proof} 
Recall that the intersection of the whole transverse trajectory of $z\in\mathrm{rec}(f)_N$ and $U_{\widehat \Gamma}$, when not empty is equivalent to $\widehat\Gamma\vert_J$ where $J$ is an open interval of $\T^1$. Consider the set $\mathcal J$ of such intervals. The fact that there are no transverse intersection tells us that these intervals do not overlap: if two intervals intersect, one of them contains the other one. One deduces that there exists $t\in\T^1$ that does not belong to any $J$. Indeed, by a compactness argument, if $\T^1$ can be covered by the intervals of $\mathcal J$, there exists $r\geq 2$ such that it can be covered by $r$ such intervals but not less. By connectedness, at least two of the intervals intersect and one can lower the number $r$. Set $\phi_{S}=\phi_{\widehat\Gamma(t)}$. The set $\mathrm{rec}(f)_N$ being dense in $\overline{W_N}$, the leaf $\phi_S$ does not meet the whole transverse trajectories of points in $\overline{W_N}$. In particular, it does not meet $\overline{W_N}$.
\end{proof}

\begin{lemma} The set $O_S$ of points whose whole transverse trajectory meets $\phi_S$ is a connected essential open set.
\end{lemma}

\begin{proof}  Fix a lift $\widetilde\phi$ of $\phi_S$ in $\widetilde{\mathrm{dom}}(I)$. The set $\widetilde O$ of points whose trajectory meets $\widetilde\phi$ is equal to $\bigcup_{k\in \Z} \widetilde f^{-k}\left(\overline{ L(\widetilde\phi)\cap R(\widetilde f(\widetilde\phi))}\right)$, it its connected and simply connected.  So its projection $O_S$ is connected. Every lift of a point in $\mathrm{rec}(f)_{\widehat\Gamma}$ belongs to all the translates $T^k(\widetilde O)$, $k\in\Z$. So the union of the translates is connected, which means that $O_S$ is essential.\end{proof}

\begin{lemma}  The set $\overline{W_{N}}$ does not contains $S$ and for every $z\in \overline{W_{N}}$, one has $\alpha(z)=\omega(z)=\{N\}$.\end{lemma}

\begin{proof}   The set $\overline{W_{N}}$ is connected because it can be written
$$\overline{W_N}=\overline{\bigcup_{k\in\Z} f^{-k}\left(l(\Gamma)\cup\{N\}\right)}. $$
It does not contain $S$ because it is connected and does not intersect the essential open set $O_S$. Moreover, one knows that the $\alpha$-limit and $\omega$-limit sets of points in $\overline{W_{N}}$ are reduced to one of the ends. They both are equal to $N$, because $S\not\in\overline{W_{N}}$. \end{proof}

Similarly, there exists a leaf $\phi_N$ in $U_{\widehat \Gamma}$ that does not meet $\overline{W_S}$ and the set $O_N$ of points whose whole transverse trajectory meets $ \phi_N$ is a connected essential open set. Moreover, $N\not\in\overline{W_{S}}$ and for every $z\in \overline{W_{S}}$, one has $\alpha(z)=\omega(z)=\{S\}$. Consequently $\overline{W_{N}}$ and $\overline{W_{S}}$ do not intersect. Two points in $O_S\cap O_N$ are not separated neither by $\overline{W_{N}}$ nor by $\overline{W_{S}}$, because $O_S$ and $O_N$ are connected and disjoint from $\overline{W_{N}}$ and $\overline{W_{S}}$ respectively. So they are not separated by  $\overline{W_{S}}\cup\overline{W_{N}}$ because $\overline{W_{S}}\cap\overline{W_{N}}=\emptyset$. One 
deduces that $O_S\cap O_N$ is contained in a connected component $O$ of the complement of $\overline{W_S}\cup \overline{W_N}$, which is nothing but $A_{\widehat\Gamma}$. So we have $$W_{\widehat\Gamma}\subset O_S\cap O_N\subset  O\subset A_{\widehat\Gamma}\subset\overline{W_{\widehat\Gamma}}.$$We deduce that the sets appearing in the inclusions have the same closure and that $A_{\widehat\Gamma}$ is connected because  $O\subset A_{\widehat\Gamma}\subset\overline{O}$.  To conclude that $A_{\widehat\Gamma}$ is an essential annulus, it is sufficient to use the connectedness  of $\overline{W_N}$ and $\overline{W_S}$, they are the two connected components of the complement of $A_{\widehat\Gamma}$.

\medskip
It remains to prove {\bf iii)}. Note first that every leaf of $\mathcal F$ is met by a transverse simple loop and so is wandering. It implies that the $\alpha$-limit and $\omega$-limit sets of a leaf are included in two different connected components of $\mathrm{sing}(I)$. Let us fix $\Gamma\in\mathcal{G}_{I,{\mathcal F}}$. The complement of $A_{\Gamma}$ has two connected components. One of them contains all singularities at the left of $\Gamma$ and all leaves in $l(\Gamma)$, \textcolor{black}{ denote it by $L(A_{\Gamma})$}. One defines similarly $R(A_{\Gamma})$. Write $\Xi$ for the union of intervals $J\in\mathcal J$ defined in the proof of Lemma \ref{lm:phi_S}. A point $t\in\T^1$ belongs to $\Xi$ if and only if there exists $z\in\mathrm{rec}(f)\cap L(A_{\Gamma})$ whose whole transverse trajectory meets $\phi_{\Gamma(t)}$ or equivalently, if there exists $z\in L(A_{\Gamma})$ whose whole transverse trajectory meets $\phi_{\Gamma(t)}$. Note that if $C$ is a connected component of $\left(\partial A_{\Gamma}\setminus \mathrm{sing}(I)\right)\cap L(A_{\Gamma})$, then the set 
$$J_C=\{t\in \T^1, \enskip C\cap \phi_{\Gamma(t)}\not=\emptyset\}$$ is an interval contained in $\Xi$. Denote by $(t_-,t_+)$ the connected component of $\Xi$ that contains this interval. The assertion {\bf iii)} is an immediate consequence of the following:

\begin{lemma}The interval $J_C$ is equal to $(t_-,t_+)$. Moreover, for every $z\in C$, the connected components of $\mathrm{sing}(I)$ that contain $\alpha(z)$ and $\omega(z)$ coincide with the connected components of $\mathrm{sing}(I)$ that contain $\omega(\phi_{\Gamma(t_-)})$ and $\omega(\phi_{\Gamma(t_+)})$ respectively. 

\end{lemma}

\begin{proof} Fix $z\in C$.  Every point $f^k(z)$ belongs to a leaf $\phi_{\Gamma(t_k)}$, where $t_k\in\T^1$. By definition of $\Xi$, one knows that the whole transverse trajectory of $z$ never meets a leaf $\phi_{\Gamma(t)}$, $t\not\in \Xi$, and so $t\in  (t_-,t_+)$ if $\phi_{\Gamma(t)}$ meets this trajectory. In particular, the sequence $(t_k)_{k\in\Z}$ is an increasing sequence in $(t_-,t_+)$. We set 
$t'_-=\lim_{k\to -\infty} t_k$ and $t'_+=\lim_{k\to +\infty} t_k$. We write $F'_+$ for the connected component of $\mathrm{sing}(I)$ that contains $\omega(\phi_{\Gamma(t'_+)})$. We will prove first that the connected component of $\mathrm{sing}(I)$ that contains $\omega(z)$ is $F'_+$ and then that $t'_+=t_+$. We can do the same for the $\alpha$-limit set. One knows that  $\omega(z)$ is contained in $L(A_{\Gamma})\cap \mathrm{sing}(I)$. So, there exists a sequence $(z'_k)_{k\geq 0}$ such that $z'_k\in\phi^-(z_k)$ for every $k\geq 0$,  that ``converges to $F'_+$''  in the following sense: every neighborhood of $F'_+$ contains $z'_k$ for $k$ sufficiently large. Let us prove now that every neighborhood of $F'_+$ contains the segment $\gamma_k$ of $\phi_{\Gamma(t_k)}$ between $z'_k$ and $z_k$, for $k$ sufficiently large. If not, there exists a subsequence of $(\gamma_k)_{k\geq 0}$ that converges for the Hausdorff topology to a set that contains a point $z\not\in\mathrm{sing}(I)$. This point belongs to $l(\Gamma)$ and the leaf $\phi_{z}$ is met by a loop $\Gamma'\in\mathcal{G}_{I,{\mathcal F}}$. For convenience choose the loop passing through $z$, so that we know that $z_k$ belongs to $L(\Gamma')$, for infinitely many $k$. One deduces that  the connected component of $\mathrm{sing}(I)$ that contains $\omega(z)$ belongs to $L(\Gamma')$. But this implies that it also belongs to $L(A_{\Gamma'})$. This connected component being included in  the open disk $A_{\Gamma'}\cup L(A_{\Gamma'})$, every point $z_k$ belongs to this disk for $k$ large enough. This contradicts the fact that $z\in \partial A_{\Gamma}$, because $A_{\Gamma'}\cup L(A_{\Gamma'})$ is in the interior of $L(A_{\Gamma})$. It remains to prove that $t'_+=t_+$. If $t'_+<t_+$, then $\phi_{\Gamma(t'_+)}$ is met by a loop $\Gamma'\in\mathcal{G}_{I,{\mathcal F}}$ such that $A_{\Gamma'}\subset L(A_{\Gamma})$ and we prove similarly that for $k$ large enough $z_k$ belongs to the open disk $A_{\Gamma'}\cup L(A_{\Gamma'})$ getting the same contradiction. 
\end{proof}

\subsection{Proofs of Theorems \ref{th: FH local} and \ref{th: FH}}

\

\smallskip

{\it Proof of Theorem \ref{th: FH local}.} \enskip Note that if $\Gamma$ and $\Gamma'$ are two distinct elements of ${\mathcal G}_{I,\mathcal F}$, then $\Gamma$ is not freely homotopic to $\Gamma'$ \textcolor{black}{in $\mathrm{dom}(I)$}. Indeed,
 there exists a leaf $\phi\in U_{\Gamma}\setminus U_{\Gamma'}$. The two sets $\alpha(\phi)$ and $\omega(\phi)$ are separated by $\Gamma$ but not by $\Gamma'$ which implies that these two loops are not freely homotopic. Let us explain now why the families $(\mathrm{rec}(f)_{\Gamma})_{\Gamma\in\mathcal{G}_{I,\mathcal F}}$ and $(A_{\Gamma})_{\Gamma\in\mathcal{G}_{I,\mathcal F}}$ are independent of $\mathcal F$ (up to reindexation), they depend only on $I$. In particular, if ${\mathcal F}'$ is another foliation transverse to $I$, then every $\Gamma\in{\mathcal G}_{I,\mathcal F}$ is freely homotopic to a unique 
$\Gamma'\in{\mathcal G}_{I,{\mathcal F}'}$ and one has $\mathrm{rec}(f)_{\Gamma}=\mathrm{rec}(f)_{\Gamma'}$. Let $z$ be a recurrent point and $D\subset \mathrm{dom}(I)$ an open disk containing $z$. For every couple of points $(z',z'')$ in $D$, choose a path $\gamma_{z',z''}$ in $D$ joining  $z'$ to $z''$. Let $(n_k)_{k\geq 0}$ be an increasing sequence of integers such that $\lim_{k\to+\infty} f^{n_k}(z)=z$. For $k$ large enough, the path $I^{n_k}(z)\gamma_{f^{n_k}(z),z}$ defines a loop whose homotopy class is independent of the choices of $D$ and $\gamma_{f^{n_k}(z),z}$. If $z$ belongs to $\mathrm{rec}(f)_{\Gamma}$, this class is a multiple of the class of $\Gamma$. This means that the family of classes of loops $\Gamma \in{\mathcal G}_{I,\mathcal F}$ does not depend on $\mathcal F$. It implies that the family of sets  $\mathrm{rec}(f)_{\Gamma}$ does not depend on $\mathcal F$ either. We will denote $(A_{\kappa})_{\kappa\in \mathcal {K}_I}$ and $(\mathrm{rec}(f)_{\kappa})_{\kappa\in \mathcal {K}_I}$  our families indexed by homotopy classes. 
%The theorem will o deduce the re[Sentence removed]}The important fact in the proof of Theorem \ref{th: FH local} will be the fact that the families  $(A_{\kappa})_{\kappa\in \mathcal {K}_I}$ and $(A_{\beta})_{\beta\in \mathcal B_I}$ are the same.

\medskip

The fact that every invariant annulus contained in $\mathrm{dom}(I)$ is contained in an $A_{\beta}$, $\beta\in\mathcal B_I$, can be proven exactly like in the global case. So, to prove Theorem \ref{th: FH local}, particularly the fact that every $A_{\beta}$ is an annulus, it is sufficient to prove that it is equal to an $A_{\kappa}$, $\kappa\in \mathcal {K}_I$. Note that an $A_{\kappa}$ is an invariant annulus contained in $\mathrm{dom}(I)$ and so is contained in an $A_{\beta}$. If we prove that every $I$-free disk recurrent point is contained in an $A_{\kappa}$, we will deduce that each $A_{\beta}$ is a union of $A_{\kappa}$, which implies that it is equal to one $A_{\kappa}$ because it is connected. We will prove in fact that for every $I$-free disk recurrent point $z$, there exists a transverse foliation $\mathcal F$ such that $z$ belongs to a $W_{\Gamma}$, $\Gamma\in\mathcal{G}_{I,\mathcal F}$. Let us give the reason. In the construction of transverse foliations we have the following: if $X$ is a finite set included in an $I$-free disk $D$, one can construct a transverse foliation such that $X$ is included in a leaf (see Proposition \ref{pr:freedisktoleaf} at the next subsection).
Consequently, if $D$ contains two points $z$ and $f^n(z)$, $n>0$, one can construct a transverse foliation such that $z$ and $f^n(z)$ are on the same leaf, which implies that $z$ belongs to a $W_{\Gamma}$. \hfill{\qedsymbol}

\bigskip

{\it Proof of Theorem \ref{th: FH}, second part, proof of assertion {\bf iv)}.} \enskip Fix $\alpha_0\in{\mathcal A}_f$. The assertion {\bf iv)} is obviously true if the complement of $A_{\alpha_0}$ is the union of two fixed points. Let us prove it in case exactly one the connected components of the complement of $A_{\alpha_0}$ is a fixed point $z_0$. By assertion {\bf iii)} there exists at least one connected component $X_1\not=\{z_0\}$ of $\mathrm{fix}(f)$ that meets the frontier of $A_{\alpha_0}$. If $\{z_0\}$ and $X_1$ are the only connected components of $\mathrm{fix}(f)$, the result is also obviously true. If not, choose a third component $X_2$, then choose $z_1\in X_1\cap \partial (A_{\alpha_0})$ and $z_2\in X_2$ and finally a maximal hereditary singular isotopy $I$ whose singular set contains $z_0$, $z_1$ and $z_2$. We will prove that the connected component $A_{\beta_0}$, $\beta_0\in {\mathcal B}_I$, that contains $A_{\alpha_0}$ is reduced to $A_{\alpha_0}$. This will imply {\bf iv)}. Suppose that $A_{\beta_0}$ is not reduced to $A_{\alpha_0}$. In that case it contains other $A_{\alpha}$, $\alpha\in {\mathcal A}_f$, and the union of such sets is dense in $A_{\beta_0}$ and contain all the recurrent points. The two ends of $A_{\beta_0}$ are adjacent to $A_{\alpha_0}$ because $z_1\in\mathrm{sing}(I)$. It implies that $A_{\alpha_0}$ is the unique $A_{\alpha}$ that is essential in $A_{\beta_0}$. So, if $A_{\alpha}$ is included in $A_{\beta_0}$ and $\alpha\not=\alpha_0$, the union of $A_{\alpha}$ and of the connected component of its complement that are included in $A_{\beta_0}$  is an invariant open disk $D_{\alpha}\subset A_{\beta_0} $ disjoint from $A_{\alpha_0}$.
Let us consider a foliation $\mathcal F$ transverse to $I$ and the loop $\Gamma\in\mathcal{G}_{I,\mathcal F}$ such that $A_{\Gamma}=A_{\beta}$. We will work in the annular covering space, where $A_{\widehat\Gamma}$ is homeomorphic to  $A_{\Gamma}$ and will write $D_{\widehat\alpha}\subset A_{\widehat\Gamma}$ for the disk corresponding to $D_{\alpha}$ and $A_{\widehat \alpha_0}$ for the annulus corresponding to $A_{\alpha_0}$.

The fact that $\{z_0\}$ is a connected component of $\S^2\setminus A_{\alpha_0}$ and that the singular set of $I$ contains $z_0$ and two other points in different components of the fixed point set of $f$, implies that one of the sets $r(\widehat \Gamma)$ or $l(\widehat \Gamma)$ is empty and the other one is not. We will assume for instance that $r(\widehat \Gamma)=\emptyset$ and $l(\widehat \Gamma)\not=\emptyset$. We have seen in the proof of Proposition \ref{pr: realization} that there exists a compactification $\widehat{\mathrm{dom}}({I})_{\mathrm {ann}}$ obtained by blowing up the end $N$ at the left of $\widehat \Gamma$ by a circle $\widehat\Sigma_N$ such that $\widehat f$ extends to a homeomorphism $\widehat f_{\mathrm {ann}}$ that admits fixed points on the added circle  with a rotation number equal to zero for the lift $\widetilde f_{\mathrm {ann}}$ that extends $\widetilde f$. Note now that every recurrent point of $\widehat f$ that belongs to a $D_{\widehat\alpha}$ has a rotation number (for the lift $\widetilde f$) and that this number is a positive integer because $D_{\widehat\alpha}$ is fixed and included in $A_{\widehat\Gamma}$. So, every periodic orbit whose rotation number is not an integer belongs to $A_{\widehat\alpha_0}$.

There are different ways to get a contradiction. Let us begin by the following one. The closure of $A_{\widehat\beta_0}$ in $\widehat{\mathrm{dom}}({I})_{\mathrm {ann}}$ is an invariant essential closed set that contains $A_{\widehat\alpha_0}$ and meets $\Sigma_N$. In particular it contains fixed points of rotation number $0$ on $\Sigma_N$. Denote by $K$ the complement of $A_{\widehat\alpha_0}$ in the closure of $A_{\widehat\beta_0}$ in $\widehat{\mathrm{dom}}({I})_{\mathrm {ann}}$. It contains the fixed points located on $\Sigma_N$ and all the $D_{\widehat\alpha}$, which means that it contains fixed points of positive rotation number. It is an essential compact set, because $A_{\widehat\alpha_0}$ is an essential annulus which is a neighborhood of the end of $\widehat{\mathrm{dom}}({I})_{\mathrm {ann}}$. All points in $K$ being non wandering, one can  apply a result of S. Matsumoto \cite{Mm} saying that $K$ contains a periodic orbit of period $q$ and rotation number $p/q$ for every $p/q\in(0,1)$. But one knows that all such periodic points must belong to $A_{\widehat\alpha_0}$. We have a contradiction.

Let us give another explanation. We will need the following intersection property: every essential simple loop in $\widehat{\mathrm{dom}}(I)_{\mathrm {ann}}$ meets its image by $\widehat f_{\mathrm {ann}}$. The reason is very simple. Perturbing our loop, it is sufficient to prove that every essential simple loop in $\widehat{\mathrm{dom}}({I})$ meets its image by $\widehat f$.  Such a loop meets $A_{\widehat\beta_0}$ because the two ends of $\widehat{\mathrm{dom}}({I})$  are adjacent to $A_{\widehat\beta_0}$ and so contains a non wandering point (every point of $A_{\widehat\beta_0}$ is non wandering). This implies that the loop meets its image by $\widehat f$.

Using the fact that the entropy of $f^2$ is zero, one can consider the family of annuli $(A_{\alpha'})_{\alpha'\in {\mathcal A}(f^2)}$, and denote by $A_{\widehat\alpha'}$ the annulus of $A_{\widehat\beta_0}$ that corresponds to an annulus $A_{\alpha'}$ contained in $A_{\beta_0}$. Every periodic point $z$ of period $3$ and rotation number $1/3$ or $2/3$ belongs to an annulus $A_{\widehat\alpha'}$  and this annulus is $\widehat f^2$-invariant. It must be essential in $A_{\widehat\beta_0}$, otherwise the rotation number of $z$ should be a multiple of $1/2$. But if it is essential, it must be $\widehat f$-invariant, its $\widehat f$-period cannot be $2$. It is included in $A_{\widehat\alpha_0}$, otherwise it would be included in a non essential $A_{\widehat\alpha}$. Being given such an essential annulus, note that the set of periodic points of period $3$ and rotation $1/3$ or $2/3$ strictly above (which means on the same side as $\widehat\Sigma_N$) is compact. Indeed, the rotation number induced on the added circle is $0$. One deduces that there are finitely many annuli $A_{\widehat\alpha'}$ that contains periodic points of period $3$ and rotation number $1/3$ or $2/3$ above a given one and so there exists an highest essential annulus $A_{\widehat\alpha'_0}$ that contains periodic points of period $3$ and rotation number $1/3$ or $2/3$. If one adds the connected component of $\widehat{\mathrm{dom}}({I})_{\mathrm {ann}}\setminus  A_{\widehat\alpha'_0}$ containing $\widehat\Sigma_N$ to $A_{\widehat\alpha'_0}$, one gets an invariant semi-open annulus $A$ that contains $\widehat\Sigma_N$ and all disks $D_{\widehat\alpha}$. The restriction $\widehat f_{\mathrm {ann}}\vert_A$ satisfies the intersection property stated in Lemma \ref{le: intersection} because $A$ is essential in $\widehat{\mathrm{dom}}({I})_{\mathrm {ann}}$. The annulus $A$ contains a fixed point of rotation number $0$ and a fixed point of positive rotation number, so, by  Lemma \ref{le: intersection}, it contains at least one periodic orbit of period $3$ and rotation number $1/3$ and one periodic orbit of period $3$ and rotation number $2/3$. These two orbits must be included in $A_{\widehat\alpha'_0}$ by definition of this set. But $\widehat f\vert_{A_{\widehat\alpha'_0}}$ satisfies the intersection property because $A_{\widehat\alpha'_0}$ is essential in $\widehat{\mathrm{dom}}({I})_{\mathrm {ann}}$. So $A_{\widehat\alpha'_0}$ contains a periodic point of period $2$ and rotation number $1/2$, which is impossible.

\medskip
In the case where none of the connected components of the complement of $A_{\alpha_0}$ is a fixed point, one can crush one of these components to a point and used what has been done in the new sphere.\hfill{\qedsymbol}

Let us add some comments on the boundary of the annuli $A_{\alpha}$.

\medskip 

Let $f:\S^2\to\S^2$ be an orientation preserving homeomorphism such that $\Omega(f)=\S^2$ and $h(f)=0$. Suppose moreover than the fixed point set is totally disconnected. Every annulus $A_{\alpha}$, $\alpha\in\mathcal A_{f}$, admits accessible fixed points on its boundary. More precisely, if $X$ is a connected component of $\S^2\setminus A_{\alpha}$, there exists a simple path $\gamma$ joining a point $z\in A_{\alpha}$ to a point $z'\in\mathrm{sing}(f)\cap X$ and contained in $A_{\alpha}$ but the end $z'$. Indeed, one can always suppose that the other connected component of $\S^2\setminus A_{\alpha}$ is reduced to a point $z_0$ and that $f$ has least three fixed points (otherwise the result is obvious). What has been done in the previous proof tells us that there exists a maximal hereditary singular isotopy $I$, a transverse foliation $\mathcal F$ and $\Gamma\in\mathcal G_{I, \mathcal F}$ such that $A_{\alpha}=A_{\Gamma}$. There exists a leaf $\phi\subset U_{\Gamma}$ that is not met by any transverse trajectory that intersects $X$. This leaf (or the inverse of the leaf) joins $z_0$ to a fixed point $z\in X$ and is contained in $A_{\Gamma}$.

\medskip

Let $f:\S^2\to\S^2$ be an orientation preserving homeomorphism such that $\Omega(f)=\S^2$ and $h(f)=0$. Let $I$ be a maximal hereditary singular isotopy and $\mathcal F$ a transverse foliation. Every annulus $A_{\Gamma}$, $\Gamma\in \mathcal G_{I, \mathcal F}$, that meets $\phi$ is such that the connected components of $\mathrm{Fix}(I)$ that contains $\alpha(\phi)$ and $\omega(\phi)$ are separated by $A_{\Gamma}$. One deduces immediately that a point $z\in\mathrm{dom}(I)$ belongs to the frontier of at most two annuli 
$A_{\Gamma}$, $\Gamma\in \mathcal G_{I, \mathcal F}$. Of course this means that a point $z\in\mathrm{dom}(I)$ belongs to the frontier of at most two annuli 
$A_{\beta}$, $\beta\in \mathcal B_{I}$, but it also implies that a point $z\not\in \mathrm{fix}(f)$ belongs to the frontier of at most two annuli 
$A_{\alpha}$, $\alpha\in \mathcal A_{f}$. Indeed, suppose that $z\not\in \mathrm{fix}(f)$ belongs to the frontier of $A_{\alpha_i}$, $0\leq i\leq 2$. If $X_i$ is the connected component of $\S^2\setminus A_{\alpha_i}$ that does not contain $z$, then the three sets $X_i$ are disjoint. Choose a fixed point $z_i$ in each $X_i$ (such a fixed point exists because $X_i\cup A_{\alpha_i}$ is an invariant disk and $A_{\alpha_i}$ has no fixed points). Choose a maximal hereditary singular isotopy $I$ that fixes the $z_i$ and denote $A_{\beta_i}$, $\beta_i\in \mathcal B_I$, the annulus that contains $A_{\alpha_i}$. Note that the three annuli $A_{\beta_i}$ are distinct and that $z$ belongs to their frontier. We have a contradiction.  \hfill{\qedsymbol}

\medskip
\subsection{Transverse foliation and free disks}
We conclude this section by justifying a point used above in the proof of Theorem \ref{th: FH local}.

\begin{proposition}\label{pr:freedisktoleaf}  Let $f:M\to M$ be a homeomorphism isotopic to the identity on a surface $M$ and $I$ a maximal singular  isotopy. Let $X$ be a finite set contained in an $I$-free disk. Then, there exists a transverse foliation $\mathcal F$ such that $X$ is contained in a leaf of $\mathcal F$.

\end{proposition}

\begin{proof} The proof can be deduced immediately from the construction of transverse foliations, that we recall now (see \cite{Lec2}).
A brick decomposition ${\mathcal D}=(V,E,B)$ on a surface is 
given by a one dimensional
stratified set, the {\it skeleton}
$\Sigma({\mathcal D})$,
with a zero-dimensional submanifold $V$ such that any vertex $v\in V$ 
is locally the extremity of
exactly three edges $e\in E$. A {\it brick } $b\in B$ is the closure of a connected component  of 
the complement of $\Sigma({\mathcal {D}})$. Say that a brick decomposition ${\mathcal {D}}=(V,E,B)$ on $\mathrm{dom}(I)$ is $I$-free, if every brick is $I$-free, or equivalently, if its lifts to a brick decomposition $\widetilde{\mathcal {D}}=(\widetilde V,\widetilde E, \widetilde B)$ on the universal covering $\widetilde{\mathrm{dom}}(I)$, whose bricks are $\widetilde f$-free, where $\widetilde f$ is the lift associated to  $\widetilde I$. Say that ${\mathcal {D}}$ is minimal if there is no $I$-free brick decomposition whose skeleton is strictly included in the skeleton of ${\mathcal {D}}$. Such a decomposition always exists.

Write $G$ for the group of automorphisms of the universal covering space. Using the classical Franks' lemma on free disk chains \cite{F1}, one constructs a natural order $\leq$ on $\widetilde B$ that satisfies the following: 

\smallskip
\noindent - \enskip it is $G$-invariant;

\smallskip
\noindent - \enskip if $\widetilde  f(\widetilde \beta)$ meets $\beta'$, then $\beta'\leq \beta$;

\smallskip
\noindent - \enskip two adjacent bricks are comparable.

One can define an orientation on $\Sigma(\widetilde{\mathcal D})$ (inducing an orientation on $\Sigma({\mathcal D})$) such that the brick on the left of an edge $\widetilde e\in \widetilde E$ is smaller than the brick on the right. Moreover, every vertex $\widetilde v\in \widetilde V$ is the ending point of at least one oriented edge and the starting point of at least one oriented edge. In other words there is no sink and no source on the oriented skeleton. We have three possibilities for the bricks of $\widetilde B$:

\smallskip
\noindent - \enskip it can be a closed disk with a sink and a source on the boundary (seen from inside);

\smallskip
\noindent - \enskip it can be homeomorphic to $[0,+\infty[\times \R$ with a sink on the boundary and a source at infinity;

\smallskip
\noindent - \enskip it can be homeomorphic to $[0,+\infty[\times \R$ with a source on the boundary and a sink at infinity;

\smallskip
\noindent - \enskip it can be homeomorphic to $[0,1]\times \R$ with a sink and a source at infinity (in this case it can project onto a closed annulus).

Let us state now the fundamental result, easy to prove in the case where $G$ is abelian and much more difficult in the case it is not (see Proposition 3.2 of \cite{Lec2}): one can cover $\Sigma(\widetilde{\mathcal D})$ with a $G$-invariant family of Brouwer lines of $\widetilde f$, such that two  lines never intersect transversally in the following sense: if $\lambda$ and $\lambda'$ are two lines in this family, either they do not intersect, or one of the sets $R(\lambda)$, $R(\lambda')$ contains the other one.

Such family of lines inherits a natural order $\leq$, where
$$\lambda\leq \lambda' \Leftrightarrow R(\lambda)\subset R(\lambda').$$
 One can ``complete" this family to get a larger family, with the same properties, that possesses the topological properties of a lamination (in particular every line admits a compact and totally ordered neighborhood). Then one can arbitrarily foliate each brick $b\in B$ such that, when lifted to a foliation on a brick $\widetilde b\in \widetilde B$, every leaves goes from the source  to the sink. We obtain then, in a natural way, a decomposition of $\widetilde{\mathrm{dom}}(I)$ by a $G$-invariant family of Brouwer lines that do not intersect transversally, and that possesses  the topological structure of a plane foliation (it is a non Hausdorff one dimensional manifold).

It remains to blow up each vertex, by a desingularization process (see \cite{Lec1}) to obtained a $G$-invariant foliation by Brouwer lines.

 \end{proof}

\subsection{Zero entropy annulus homeomorphisms}

In this subsection we prove some results for a general open annulus homeomorphism whose extension to the  ends compactification has zero topological entropy.  A stronger version of the first result for diffeomorphisms was already proved in an unpublished paper of Handel \cite{H2}.

As noted in the introduction, given a homeomorphism of an open annulus $\T^1\times\R$ and a lift $\check f$ to $\R^2$, denote by $\pi: \R^2\to \T^1\times\R$ the covering projection, and by $\pi_1:\R^2 \to \R$ the projection in the first coordinate. For any point $z\in\T^1\times\R$ such that its $\omega$-limit set is not empty, we say that $z$ {\it has a rotation number $\mathrm{rot}(z)$}  if, for any compact set $K\subset\T^1\times\R$, any increasing sequence of integers $n_k$ such that $f^{n_k}(z)\in K$ and any $\check z\in\pi^{-1}(z)$, 
$$\lim_{k\to\infty}\frac{1}{n_k}\left( \pi_1(\check f^{n_k}(\check z)- \pi_1(\check z)\right)=\mathrm{rot}(z).$$ 
In general it is not expected that every point will have a rotation number, but if we assume that $f$ has zero entropy this must be the case, at least for recurrent points, as shown by the following theorem, which is a restatement of Theorem \ref{th:continuous_rotation_intro}

\begin{theorem}\label{th:rotation_number_exists}
Let $f$ be a homeomorphism of $\T^1\times \R$ isotopic to the identity, $\check f$ a lift of $f$ to the universal covering space, and let $f_{\mathrm{sphere}}$ be the natural extension of $f$ to the sphere obtained by compactifying each end with a point. If the topological entropy of $f_{\mathrm{sphere}}$ is zero, then each bi-recurrent point has a rotation number $\mathrm{rot}(z)$. Moreover, the function $z\mapsto \mathrm{rot}(z)$ is continuous on the set of bi-recurrent points.
% Furthermore, if $\rho(z)\in \Q$ and $z$ is not periodic, then $\rho$ is constant in a neighborhood of $z$.
\end{theorem}

\begin{proof}
For every compact set $K$ of $\T^1\times\R$ define the set $\mathrm{rot}_{\check f, K}(z)\subset \R\cup\{-\infty, \infty\}$ as following: $\rho$ belongs to $\mathrm{rot}_{\widetilde f, K}(z)$ if there exists an increasing sequence of integers $n_k$ such that $f^{n_k}(z)\in K$ and such that for any $\check z\in\pi^{-1}(z)$, one has
$$\lim_{k\to\infty}\frac{1}{n_k}\left( \pi_1(\check f^{n_k}(\check z)- \pi_1(\check z)\right)=\rho.$$ 
Writing $T:(x,y)\mapsto(x+1,y)$ for the fundamental covering automorphism, one immediately gets $\mathrm{rot}_{\check f\circ T^p, K}(z)=\mathrm{rot}_{\check f, K}(z)+p$ for every $p\in\Z$.  One can prove quite easily that $\mathrm{rot}_{\check f^q, K_q}(z)=q\mathrm{rot}_{\check f, K}(z)$, for every $q\geq 1$, where $K_q=\bigcup_{0\leq k<q} f^{-k}(K)$. Finally, note that $\mathrm{rot}_{\check f, K}(z)=\mathrm{rot}_{\check f, \overline O(f,z)\cap K}(z)$, where $O(f,z)$ is the $f$-orbit of $z$.

Now, recall why every positively recurrent point $z$ of $f$ is a positively recurrent point of $f^q$, for every $q\geq 1$. The set $R$ of integers $r$ such that there exists a subsequence of $(f^{nq}(z))_{n\geq 0}$ that converges to $f^r(z)$ or equivalently a subsequence of $(f^{nq-r}(z))_{n\geq 0}$ that converges to $z$, is non empty because $z$ is positively recurrent. Note that $R$ is stable by addition. Indeed if $r$ and $r'$ belong to $R$, one can approximate $f^{r+r'}(z)=f^r(f^{r'}(z))$ by a point $f^r(f^{n'q}(z))=f^{r+n'q}(z)$ as close as we want, and then approximate $f^{r+n'q}(z)= f^{n'q}(f^r(z))$ by a point $f^{n'q}(f^{nq}(z))=f^{(n+n')q}(z)$ as close as we want. One deduces that $qr$ belongs to $R$ if it is the case for $r$, which implies that $z$ is a positively recurrent point of $f^q$. Similarly, every bi-recurrent point $z$ of $f$ is a bi-recurrent point of $f^q$, for every $q\geq 1$.

For every couple $(p,q)$ of integers relatively prime ($q\geq 1$), one can choose an identity isotopy $I^*_{p,q}$ of $f^q$ that is lifted to an identity isotopy of $\check f^q\circ T^{-p}$, then a maximal hereditary singular isotopy $I_{p,q}$ such that $I'_{p,q}\preceq I_{p,q}$, and finally a singular foliation ${\mathcal F}_{p,q}$ transverse to $I_{p,q}$. The singular points of $I_{p,q}$ are periodic points of period $q$ and rotation number $p/q$. Let $z$ be a bi-recurrent point of $f$ that is not a singular point of $I_{p,q}$. By Theorem \ref{th:transverse_imply_entropy} and Proposition \ref{pr:transverse_on_sphere}, the whole trajectory $(I_{p,q})^{\Z}_{{\mathcal F}_{p,q}}(z)$ is the natural lift of a simple loop $\Gamma_{p,q}(z)$ (uniquely defined up to equivalence). In particular, one has $\Gamma_{p,q}(z)=\Gamma_{p,q}(f^q(z))$. Write $U_{\Gamma_{p,q}(z)}$ for the open annulus, union of leaves met by $\Gamma_{p,q}(z)$. Every leaf containing a point of $\overline{O(f^q,z)}$ is met by $(I_{p,q})^{\Z}_{{\mathcal F}_{p,q}}(z)$. It implies that $\overline{O(f^q,z)}\subset U_{\Gamma_{p,q}(z)}\cup\mathrm{sing}(I_{p,q})$. Note also that the function $z\mapsto\Gamma_{p,q}(z)$ is locally constant on the set of bi-recurrent points. Indeed if $z$ is a bi-recurrent point and $\gamma:[0,1]\to \mathrm{dom}(I_{p,q})$ a non simple transverse subpath of the natural lift of $\Gamma_{p,q}(z)$, then $\gamma$ is a subpath of the whole trajectory of $z'$, if $z'$ is a bi-recurrent point sufficiently close to $z$, which implies that $\Gamma_{p,q}(z')=\Gamma_{p,q}(z)$.  

One can lift the isotopy $I_{p,q}$ to a singular maximal isotopy $\check I_{p,q}$ of $\check f^q\circ T^{-p}$ and the foliation ${\mathcal F}_{p,q}$ to a foliation $\check{\mathcal F}_{p,q}$ transverse to $\check I_{p,q}$. Fix a lift $\check z\in\R^2$ of $z$. In the case where $\Gamma_{p,q}(z)$ is not essential, then $(\check I_{p,q})^{\Z}_{\check{\mathcal F}_{p,q}}(\check z)$ is the natural lift of a transverse simple loop $ \Gamma_{p,q}(\check z)$ that lifts $\Gamma_{p,q}(z)$. The $\widetilde f^q\circ T^{-p}$-orbit of $\check z$ stays in the annulus $U_{\Gamma_{p,q}(\check z)}$, union of leaves met by $\Gamma_{p,q}(\check z)$.   In the case where $\Gamma_{p,q}(z)$ is essential and the upper end of $\T^1\times\R$ is on the left of $\Gamma_{p,q}(z)$, then $(\check I_{p,q})^{\Z}_{\check{\mathcal F}_{p,q}}(\check z)$ is the natural lift of a transverse line
 $ \gamma_{p,q}(\check z)$ that lifts $\Gamma_{p,q}(z)$. The $\widetilde f^q\circ T^{-p}$-orbit of $\check z$ stays in the strip $U_{\gamma_{p,q}(\check z)}$, union of leaves met by $\gamma_{p,q}(\check z)$.  Fix a parameterization $\gamma_{p,q}(\check z):\R\to \mathrm{dom}(\check I_{p,q})$ such that $\gamma_{p,q}(\check z)(t+1)=T(\gamma_{p,q}(\check z)(t))$. For every $\check z'\in U_{\gamma_{p,q}(\check z)}$ there exist a unique real number,
 denoted by $\pi_{ \gamma_{p,q}(\check z)}(\check z')$ such that $\check z'$ and $ \gamma_{p,q}(\check z)(t)$ are on the same leaf. One gets a map $\pi_{ \gamma_{p,q}(\check z)}: U_{\gamma_{p,q}(\check z)}\to\R$, such that  the sequence $(\pi_{ \gamma_{p,q}(\check z)}(\check f^{q}\circ T^{-p})^k(\check z))_{k\in\Z}$ is increasing.  In the case where $\Gamma_{p,q}(z)$ is essential and the upper end of $\T^1\times\R$ is on the right of $\Gamma_{p,q}(z)$, one proves by the same argument that the sequence $(\pi_{ \gamma_{p,q}(\check z)}(\check f^{q}\circ T^{-p})^k(\check z))_{k\in\Z}$ is decreasing.  

Let $z$ be a periodic point that is not a singular point of $I_{p,q}$.  If $\Gamma_{p,q}(z)$ is not essential, then the rotation number of $z$ (defined for the lift $\widetilde f^q\circ T^{-p}$ of $f^q$) is equal to zero, which implies that $\mathrm{rot}(z)=p/q$. If $\Gamma_{p,q}(z)$ is essential and if the upper end of $\T^1\times\R$ is on the left of $\Gamma_{p,q}(z)$, the rotation number of $z$ (defined for the lift $\widetilde f^q\circ T^{-p}$ of $f^q$) is positive, which implies that $\mathrm{rot}(z)>p/q$.  If $\Gamma_{p,q}(z)$ is essential and if the upper end of $\T^1\times\R$ is on the right of $\Gamma_{p,q}(z)$, then $\mathrm{rot}(z)<p/q$.

Let us prove now that if $z$ is bi-recurrent, the periodic points that belong to the closure of $O(f,z)$ have the same rotation number. Otherwise, one can find a couple $(p,q)$ of integers relatively prime ($q\geq 1$), and two periodic points $z_1$, $z_2$ in the closure of $O(f^q,z)$ such that $\mathrm{rot}(z_1)<p/q<\mathrm{rot}(z_2)$. One deduces that $\Gamma_{p,q}(z_1)=\Gamma_{p,q}(z_2)=\Gamma_{p,q}(z)$, which is impossible because the upper end of $\T^1\times\R$ is on the right of $\Gamma_{p,q}(z_1)$ and on the left of $\Gamma_{p,q}(z_2)$.  

Let $K$ be a compact subset of $\T^1\times\R$  and $z$ a bi-recurrent point. Let us suppose first that the closure of $O(f,z)$ has no periodic points. For every couple $(p,q)$ of integers relatively prime ($q\geq 1$), the set $\overline{O(f,z)}\cap K_q$ is a compact subset of the annulus $U_{\Gamma_{p,q}(z)}$. In the case where $\Gamma_{p,q}(z)$ is not essential, one deduces that $\mathrm{rot}_{\widetilde f^q\circ T^{-p},K_q}(z)$ is reduced to $\{0\}$ and so $\mathrm{rot}_{\widetilde f,K}(z)$ is reduced to $\{p/q\}$. In the case where $\Gamma_{p,q}(z)$ is essential and the upper end of $\T^1\times\R$ is on the left of $\Gamma_{p,q}(z)$, there exits a real number $M$ such that for every point $\widetilde z'$ that lifts a point of $(\overline{O(f,z)}\cap K_q)\cup\{z\}$, one has $\vert \pi_1(\check z')-\pi_{ \gamma_{p,q}(\check z)}(\check z')\vert\leq M$. Using this property and the fact that the sequence $(\pi_{ \gamma_{p,q}(\check z)}(\check f^{q}\circ T^{-p})^k(\check z))_{k\in\Z}$ is increasing, one deduces that $\mathrm{rot}_{\widetilde f\circ T^{-p},K_q}(z)\subset[0,+\infty]$ and consequently that $\mathrm{rot}_{\widetilde f,K}(z)\subset [p/q,+\infty]$. Similarly, in the case where $\Gamma_{p,q}(z)$ is essential and the upper end of $\T^1\times\R$ is on the right of $\Gamma_{p,q}(z)$, one gets $\mathrm{rot}_{\widetilde f,K}(z)\subset [-\infty,p/q]$.

One immediately concludes that $\mathrm{rot}_{\widetilde f,K}(z)$ is reduced to a number in $\R\cup\{-\infty,\infty\}$ if not empty. Of course, this number is independent of $K$, we denote it $\mathrm{rot}(z)$. 
Suppose now that the closure of $O(f,z)$ contains periodic points. As said before, they have the same rotation number $p_0/q_0$. The argument above is still valid if $p/q\not=p_0/q_0$ and permit us to  concluded that $\mathrm{rot}_{\widetilde f,K}(z)$ is reduced to a number in $\R\cup\{-\infty,\infty\}$ independent of $K$. Of course the number is nothing but $p_0/q_0$. Note that in both situations $\mathrm{rot}(z)$ is uniquely defined by the following property:  

\smallskip
\noindent -\enskip if $p/q<\mathrm{rot}(z)$, then $\Gamma_{p,q}(z)$ is essential and the upper end of $\T^1\times\R$ is on the right of $\Gamma(z)$,

\smallskip
\noindent -\enskip if $p/q>\mathrm{rot}(z)$, then $\Gamma_{p,q}(z)$ is essential and the upper end of $\T^1\times\R$ is on the left of $\Gamma(z)$.

Using the fact that each function $z\mapsto\Gamma_{p,q}(z)$ is locally constant on the set of bi-recurrent points, one deduces immediately that the function $z\mapsto 
\mathrm{rot}(z)$ is continuous on the set of bi-recurrent points.

It remains to prove that $\mathrm{rot}(z)$ is finite. Of course one can suppose that the closure of $O(f,z)$ does not contain a periodic orbit (otherwise as said before $\mathrm{rot}(z)$ is rational). By assumption, $z$ is not periodic, so let us choose a free disk $D$ containing $z$. There exists an integer $s>0$ such that $f^s(z)\in D$ and an integer $r\in\Z$ such that $\check f^s(\check z)\in T^r(\check D)$, if $\check D\subset\R^2$ is a lift of $D$ and $\check z$ is the lift of $z$ contained in $\check D$.  Let us consider the singular isotopy $I_{k,1}$, where  $k>r/s$. As explained above the two points $z$ and $f^s(z)$ belong to the free disk $D$. So, by Proposition \ref{pr:freedisktoleaf} one can choose the foliation ${\mathcal F}_{k,1}$ such that $z$ and $f^s(z)$ belong to the same leaf $\phi$. This implies that $\Gamma_{k,1}(z)$ is essential and the upper end of $\T^1\times\R$ is on the right of $\Gamma_{k,1}(z)$. Consequently, one deduces that $\mathrm{rot}(z)\leq k$. One proves similarly that $\mathrm{rot}(z)\geq k'$ if $k'$ is an integer smaller than $r/s$.\end{proof}

%\begin{proposition}Let $f:\S^2\to\S^2$ be a homeomorphisms with zero topological entropy and at least $3$ fixed points. If $\Lambda=f(\Lambda)$ is closed and $f\mid_{\Lambda}$ is transitive, then $\Lambda$ contains at most one fixed point.
%\end{proposition}

An interesting consequence is:

\begin{proposition}
Let $f$ be an orientation preserving homeomorphism of $S^2$. Suppose that there exists a bi-recurrent point $z$ such that the closure of its orbit contains periodic points of minimal period $q_1<q_2$, where $q_1$ does not divide $q_2$. Then the entropy of $f$ is positive.
\end{proposition}

\begin{proof}
We suppose that the closure of the $f$-orbit $O(f,z)$ of $z$ contains a periodic point $z_1$ of period $q_1$ and a periodic point $z_2$ of period $q_2$. One can choose $z_1$ and $z_2$ in $\overline{O(f^{q_2},z)}$. Writing $r$ for the remainder of the Euclidean division of $q_2$ by $q_1$, one knows that $f^{q_2}(z_1)=f^r(z_1)$.  Since $q_2$ is not a multiple of $q_1$, it is larger than $2$ and $f^{q_2}$ must have at least three distinct fixed points. Choose a maximal hereditary identity isotopy $I$ of $f^{q_2}$ whose singular set contains $z_2$, $f^{r}(z_2)$ and at least a third fixed point of $f^{q_2}$, then consider a singular foliation $\mathcal F$ transverse to $I$.  Since $f$ has zero topological entropy, $f^{q_2}$ has zero topological entropy, and the path $I_{\mathcal F}^{\Z}(z_1)$ is equivalent to the natural lift of a simple transverse loop $\Gamma$. Using the fact that there exist at least there singular points, one can find two singular points $z'_2$ and $z''_2$ of $I$
 that are separated by $\Gamma$ and such that $\{z'_2,z''_2\}\not=\{z_2,f^r(z_2)\}$. The isotopy $I$ defines a natural lift of $f^q\vert_{S^2\setminus\{z'_2,z''_2\}}$ and for this lift, the rotation number of every point of the $f^{q_2}$-orbit of $z_1$ is a non vanishing number, while the rotation number of every singular point different from $z'_2$ and $z''_2$ is zero. As seen in the previous proposition, the points $z$ and $f^r(z)$ are bi-recurrent points of $f^{q_2}$. In the case where $z_2\not\in\{z'_2,z''_2\}$, the closure of $O(f^{q_2},z)$ contains two periodic points $z_1$ and $z_2$  with different rotation numbers. In the case where $f^r(z_2)\not\in\{z'_2,z''_2\}$, the closure of $O(f^{q_2},f^r(z))$ contains two periodic points $f^r(z_1)=f^{q_2}(z_1)$ and $f^r(z_2)$  with different rotation numbers. We have seen in the proof of the previous proposition that in both cases, the entropy of $f$ is positive.. 

\end{proof}

\section{ Applications to torus homeomorphims}

\bigskip

 In this section an element of $\Z^2$ will be called an {\it  integer} and an element of $\Q^2$ a {\it rational.} If $K$ is a convex compact subset of $\R^2$, a {\it supporting line} is an affine line that meets $K$ but does not separate two  points of $K$, a {\it vertex} is a point that belongs to infinitely many supporting lines.

Let us begin by stating the main results of this section, that are nothing but  \textcolor{black}{Theorems \ref{th:impossible_rotation_set_intro}, \ref{th:bounded_deviation_intro} and \ref{th:Llibre_MacKay_intro}} from the introduction.

\begin{theorem}\label{th:impossible_rotation_set}
 Let $f$ be a homeomorphism of $\T^2$ that is isotopic to the identity and $\check f$ a lift of $f$ to  $\R^2$. The frontier of $\mathrm{rot}(\check f)$ does not contain a segment with irrational slope that contains a rational point in its interior.
\end{theorem}

\begin{theorem}\label{th:bounded_deviation}
 Let $f$ be a homeomorphism of $\T^2$ that is isotopic to the identity and $\check f$ a lift of $f$ to  $\R^2$. If $\mathrm{rot}(\check f)$ has a non empty interior, then there exists $L\geq 0$ such that for every $z\in\R^2$ and every $n\geq 1$, one has $d(\check f^n(z)-z, n\mathrm{rot}(\check f))\leq L$.
\end{theorem}

\begin{theorem}\label{th:Llibre_MacKay}
Let $f$ be a homeomorphism of $\T^2$ that is isotopic to the identity and $\check f$ a lift of $f$ to  $\R^2$. If $\mathrm{rot}(\check f)$ has a non empty interior, then the topological entropy of
$f$ is positive.
\end{theorem}

Recall that Theorem \ref{th:Llibre_MacKay} has been known for a long time and is due to Llibre and MacKay, see \cite{LlM} and that Theorem \ref{th:bounded_deviation} was known for homeomorphisms in the special case of a polygon with rational vertices, see Davalos  \cite{D2}, and for $\mathcal{C}^{1+\epsilon}$ diffeomorphisms, see Addas-Zanata \cite{AZ}.

\bigskip

Let us state first some consequences of these results.
Let $f$ be a homeomorphism of $\T^2$ that is isotopic to the identity and $\check f$ a lift of $f$ to  $\R^2$. We suppose that $\mathrm{rot}(\check f)$ has non empty interior. For every non trivial linear form $\psi $ on $\R^2$, define 
$$\alpha(\psi)=\max\{ \psi(\mathrm{rot}(\mu))\,,\, \mu\in {\mathcal M}(f)\}.$$ 
The affine line of equation $\psi(z)=\alpha(\psi)$ is a supporting line of $\mathrm{rot}(\check f)$. Set 
$$ {\mathcal M}_{\psi}=\left\{ \mu\in {\mathcal M}(f)\,,\, \psi(\mathrm{rot}(\mu))=\alpha(\psi)\right\}, \enskip {X}_{\psi}= 
\overline{\bigcup_{\mu\in {\mathcal M}_{\psi}} \mathrm{supp}(\mu)}.$$

\bigskip
As already noted in \cite{AZ}, we can deduce from Theorem \ref{th:bounded_deviation} and Proposition \ref {pr:atkinson} (Atkinson's Lemma) the following result, \textcolor{black}{ Proposition \ref{pr: rotation number set_intro} of the introduction}.

\begin{proposition} \label{pr: rotation number set} Every measure $\mu$ supported on ${X}_{\psi}$ belongs to  ${\mathcal M}_{\psi}$. Moreover,  if $z$ lifts a point of $X_{\psi}$, then for every  $n\geq 1$, one has $\vert\psi(\check f^n(z))-\psi(z)-n\alpha(\psi)\vert \leq L\Vert \psi\Vert $, where 
$L$ is the constant given by  Theorem \ref{th:bounded_deviation}.

\end{proposition}

\begin{proof}  We will prove the second statement, it obviously implies the first one.  \textcolor{black}{Note first that the ergodic components of a measure $\mu\in{\mathcal M}_{\psi}$ also belong to ${\mathcal M}_{\psi}$. Furthermore, the set of points $A'$ having a lift $z$ satisfying that $\vert\psi(\check f^n(z))-\psi(z)-n\alpha(\psi)\vert\leq L\Vert \psi\Vert $ for every $n\geq 1$ is a closed set. It is then sufficient to prove that for every ergodic measure $\mu\in{\mathcal M}_{\psi}$ there exists a set $A\subset A'$ of full measure.} As seen before, \textcolor{black}{since $\mu\in{\mathcal M}_{\psi}$, the function lifted by $\psi\circ \check f-\psi-\alpha(\psi)$ has null mean, and we can apply Atkinson's lemma to obtain that} there exists a set $A$ of full measure  such that, for every point $z$ lifting a point of $A$, there exists a subsequence $(n_l)_{l\in\N}$ such that  $$\lim_{l\to+\infty} \psi(\check f^{n_l}(z))-\psi(z)-n_l\alpha(\psi)=0.$$By Theorem \ref{th:bounded_deviation}, one knows that for every $z\in \R^2$ and every $n\geq 1$, one has $\psi(\check f^n(z))-\psi(z)-n\alpha(\psi)\leq L\Vert \psi\Vert .$ It remains to prove that
$\psi(\check f^n(z))-\psi(z)-n\alpha(\psi)\geq -L\Vert \psi\Vert $ if $z$ lifts a point of $A$.
If $n_l$ is greater than $n$ one can write
$$\eqalign{&\psi(\check f^n(z))-\psi(z)-n\alpha(\psi)\cr= &\left(\psi(\check f^{n_l}(z))-\psi(z)-n_l\alpha(\psi)\right) - \left(\psi(\check f^{n_l}(z))-\psi(\check f^n(z))-(n_l-n)\alpha(\psi) \right)\cr
\geq &\psi(\check f^{n_l}(z))-\psi(z)-n_l\alpha(\psi)  - L\Vert \psi\Vert \cr}.$$
Letting $l$ tend to $+\infty$, one gets our inequality. \end{proof}

Let us state two corollaries. The first one, \textcolor{black}{ Corollary \ref{co: Boyland_intro} of the introduction, as} already noted in \cite{AZ}, follows immediately from the previous proposition.

\bigskip
\begin{corollary}\label{co: Boyland}
Let $f$ be a homeomorphism of $\T^2$ that is isotopic to the identity, preserving a measure $\mu$ of full support, and $\check f$ a lift of $f$ to  $\R^2$. Assume that $\mathrm{rot}(\check f)$ has a non empty interior. Then $\mathrm{rot}(\mu)$ belongs to the interior of $\mathrm{rot}(\check f)$.
\end{corollary}

P. Boyland had conjectured that, for a given $f$ and $\check{f}$ in the hypotheses of Corollary \ref{co: Boyland}, if  $\mathrm{rot}(\mu)$ was an integer then it belonged to the interior of $\mathrm{rot}(\check f)$. The previous result shows that the conjecture is true, and that the hypothesis on the rationality of the rotation vector of $\mu$ is superfluous.

The second corollary shows that, for points in the lift of the support of measures with rotation vector in a vertex, the displacement from the corresponding rigid rotation is uniformly bounded.

\begin{corollary}\label{co:bounded_displacement}  Let $\rho$ be a vertex of $\mathrm{rot}(\check f)$,  and set 
$$ {\mathcal M}_{\rho}=\left\{ \mu\in {\mathcal M}(f)\,,\, \mathrm{rot}(\mu)=\rho\right\}, \enskip {X}_{\rho}= 
\overline{\bigcup_{\mu\in {\mathcal M}_{\rho}} \mathrm{supp}(\mu)}.$$
There exists a constant $L_{\rho}$ such that if $z$ lifts a point of $X_{\rho}$, then for every  $n\geq 1$, one has $d(\check f^n(z))-z-n\rho) \leq L_{\rho}$.

\end{corollary}
 
 \begin{proof}  One can find two forms $\psi$ and $\psi'$, linearly independent such that $\rho$ belongs to the supporting lines defined by these forms. Note that $X_{\rho}=X_{\psi}\cap X_{\psi'}$ and apply Proposition \ref
{pr: rotation number set} .
\end{proof}

We remark that the conclusion from Corollary \ref{co:bounded_displacement} does not hold if instead of requiring that $\rho$ is a vertex of  $\mathrm{rot}(\check f)$ we assume that $\rho$ is an extremal point of $\mathrm{rot}(\check f)$, see Boyland-de Carvalho-Hall \cite{BCH}.  

\bigskip
Write $\partial \left(   \mathrm{rot}(\check f)\right) $ for the frontier of  $\mathrm{rot}(\check f)$. Let us define now
$$ {\mathcal M}_{\partial}=\left\{ \mu\in {\mathcal M}(f)\,,\, \mathrm{rot}(\mu)\in \partial \left(   \mathrm{rot}(\check f) \right)\right\},\enskip {X}_{\partial}= 
\overline{\bigcup_{\mu\in {\mathcal M}_{\partial}} \mathrm{supp}(\mu)} =\overline{\bigcup_{\psi\not=0} X_{\psi}}.$$
Similarly, we have:

\begin{proposition} Every ergodic measure $\mu$ supported on ${X}_{\partial}$ belongs to  ${\mathcal M}_{\partial}$. Moreover, if $z$ lifts a point of $X_{\partial}$, then for every  $n\geq 1$, one has  $d\left(\check f^n(z)-z, n\,\partial \left(   \mathrm{rot}(\check f)\right) \right)\leq L$, where 
$L$ is the constant given by Theorem \ref{th:bounded_deviation}.

\end{proposition}

\begin{proof}Here again, it is sufficient to prove the second statement. To do so, let us choose a non trivial linear form $\psi$ and let us prove that  for every $n\geq 1$, and for every point $z$ lifting a point of $X_{\psi}$, one has $$d\left(\check f^n(z)-z, n\,\partial \left(   \mathrm{rot}(\check f)\right) \right)\leq L.$$  The fact that, \textcolor{black}{by Proposition \ref{pr: rotation number set},} $$\vert\psi(\check f^n(z))-\psi(z)-n\beta(\psi)\vert \leq L\Vert\psi\Vert $$
 implies that $d(\check f^n(z)-z, \Delta)\leq L$ where $\Delta$ is the affine line of equation $\psi(z)=n\alpha(z)$. So, if $f^n(z)-z$ does not belong to $ n \,\mathrm{rot}(\check f)$, one has $$d\left(\check f^n(z)-z, n\,\partial \left(   \mathrm{rot}(\check f)\right) \right)=d(\check f^n(z)-z, n \,\mathrm{rot}(\check f))\leq L,$$ and if $\check f^n(z)-z$  belongs to $ n\, \mathrm{rot}(\check f)$, one has $$d\left(\check f^n(z)-z, n\,\partial \left(   \mathrm{rot}(\check f)\right) \right)\leq d(\check f^n(z)-z, \Delta)\leq L.$$
\end{proof}

Another application is a classification result about {\it Hamiltonian homeomorphisms}.  In our setting, a Hamiltonian homeomorphism is a torus homeomorphism preserving a probability measure $\mu$ which has a lift $\check f$ (the {\it Hamiltonian lift}) such that that $ \mathrm{rot}(\mu)=(0,0).$ An illustrative example is given by the time one map of a time dependent Hamiltonian flow, $1$ periodic in time, and its natural lift.

We will need the following result, which can be found in \cite{KT2}:

\begin{proposition}\label{pr:rotation_recurrent}
Let $f$ be a homeomorphism of $\T^2$ isotopic to the identity and $\check f$ a lift of $f$. If $(0,0)$ is a vertex of $ \mathrm{rot}(\check f)$ then, for any measure $\mu\in \mathcal{M}(f)$ such that $ \mathrm{rot}(\mu)=(0,0)$, almost every point lifts to a recurrent point of $\check f$ .
\end{proposition}

We have:

\begin{theorem}\label{th:hamiltonian_torus}
Let $f$ be a Hamiltonian homeomorphism of $\T^2$ such that its fixed point set is contained in a topological disk, and let $\check f$ be its Hamiltonian lift. Then one of the following three conditions holds:
%\begin{itemize}

\smallskip
\noindent - \enskip  The set $ \mathrm{rot}(\check f)$ \textcolor{black}{does not have} empty interior: in that case the origin lies in its interior.

\smallskip
\noindent - \enskip  The set $ \mathrm{rot}(\check f)$ is a non trivial segment: in that case $ \mathrm{rot}(\check f) $ generates a line with rational slope, the origin is not an end of $ \mathrm{rot}(\check f)$, furthermore, there exists an invariant essential open annulus in $\T^2$.

\smallskip
\noindent - \enskip The set $ \mathrm{rot}(\check f)$ is reduced to the origin: in that case, there exists  $K>0$ such that, for every $z\in \R^2$ and every $k\in\Z$,  one has \textcolor{black}{$\Vert \check{f}^k(z)-z\Vert \leq K.$}
%\end{itemize}

\end{theorem}

\begin{proof}
Suppose first that $ \mathrm{rot}(\check f)$ is reduced to the origin. The origin being a vertex, one knows by Proposition \ref{pr:rotation_recurrent} that the recurrent set of $\check f$ is dense in $\R^2$. So the assertion comes from Theorem \ref{th:recurrent_on_the_lift}.

Suppose now that $ \mathrm{rot}(\check f)$ is a non trivial segment. If the origin was an end of $ \mathrm{rot}(\check f)$ its would be a vertex and we would have a contradiction, still from from Proposition \ref{pr:rotation_recurrent} and Theorem \ref{th:recurrent_on_the_lift}. The fact that $ \mathrm{rot}(\check f) $ generates a line with rational slope is a consequence of Theorem \ref{th:impossible_rotation_set}. The existence of an essential open annulus which is left invariant by the dynamics whenever $ \mathrm{rot}(\check f)$ is a non trivial segment that generates a line with rational slope  is the main result of \cite{GKT}.

The case where $ \mathrm{rot}(\check f)$ has non empty interior is nothing but Corollary \ref{co: Boyland}.\end{proof}

Here again, as in Theorem \ref{th:recurrent_on_the_lift}, the requirement that the fixed point set is contained in a topological disk cannot be removed. As a consequence, we obtain the following boundedness result for area preserving homeomorphisms of the torus with restriction on its rotational behaviour\textcolor{black}{, Corollary \ref{co:hamiltonian_bounded_intro} of the introduction}:

\begin{corollary}\label{co:hamiltonian_bounded}
Let $f$ be a Hamiltonian homeomorphism of $\T^2$ such that all its periodic points are contractible, and such that its fixed point set is contained in a  topological disk. Then there exists $K>0$ such that if $\check f$ is the Hamiltonian lift of $f$, then for every $z\in\R^2$ and every $k\in\Z$, one has $\Vert \check f^k(z)-z\Vert\leq K$.
\end{corollary}

\begin{proof}
\textcolor{black}{By Theorem \ref{th:hamiltonian_torus}, $f$ must belong to one of the three described  possibilities. If $f$ is a homeomorphism of $\T^2$ such that all its periodic points are contractible, then by the main result of \cite{F2} the rotation set of any lift of $f$ must have empty interior (see also Remark \ref{rm:Franks} later in the paper), and so the first possibility in Theorem \ref{th:hamiltonian_torus} is excluded. Furthermore, it was shown in \cite{F3} that, if $g$ is an area preserving homeomorphism  with lift $\check g$ and the rotation set of $\check g$ is a line segment, then for every point in $\mathrm{rot}(\check g)$ with bi-rational coordinates there exists a periodic point for $f$ with the same rotation vector. Since $f$ has no periodic points that are not contractible, the second possibility is also excluded.}
\end{proof}

As a consequence we obtain the Proposition \ref{pr:Ginzburg_intro}:

\begin{proposition}\label{pr: generic}
 Let $\mathrm{Ham}_{\infty}(\T^2)$ be the set of Hamiltonian $C^{\infty}$ diffeomorphisms of $\T^2$ endowed with the Whitney $C^{\infty}$- topology. There exists a residual subset $\mathcal A$ of $\mathrm{Ham}_{\infty}(\T^2)$ such that $f$ has non-contractible periodic points if $f\in\mathcal A$.
\end{proposition}

\begin{proof}
We will prove that $f$ has non contractible periodic points if  the following properties are satisfied:
\begin{itemize}
\item{if $f^q(z)=z$, then $1$ is not an eigenvalue of \textcolor{black}{$Df^q(z)$};}
%\smallksip
%\noindent-\enskip \textcolor{green}
\item{if $z$ is an elliptic periodic point of period $q$ (which means that the eigenvalues of $Df^q(z)$ are on the unit circle), then $z$ is Moser stable (which means that $z$ is surrounded by $f^q$-invariant curves arbitrarily close to $z$);}
%\smallksip
%\noindent-\enskip \textcolor{green}{if $z$ is an elliptic periodic point of period $q$ (which means that the eigenvalues of $Df^q(z)$ are on the unit circle), then $z$ is Moser stable (which means that $z$ is surrounded by $f^q$-invariant curves arbitrarily close to $z$);}
\item{\textcolor{black}{if $z$, $z'$ are hyperbolic periodic points of period $q$, $q'$ respectively (which means that the eigenvalues of $Df^q(z)$ and $Df^{q'}(z')$ are real), then the stable and unstable manifolds of $z$ and $z'$ \textcolor{black}{are either disjoint or they} intersect transversally. }}
%\smallksip
%\noindent-\enskip \textcolor{green}{if $z$ is a hyperbolic periodic point of period $q$ (which means that the eigenvalues of $Df^q(z)$ are real), then the stable and unstable manifold of $z$ intersect transversally. }
\end{itemize}

The first property implies that the fixed point set of $f$ is finite and so included in a topological disk. By Corollary \ref{co:hamiltonian_bounded}, to get our result it remains to prove that there is no $K>0$ such that if $\check f$ is the Hamiltonian lift of $f$, then for every $z\in\R^2$ and every $k\in\Z$, one has $\Vert \check f^k(z)-z\Vert\leq K$. If such $K$ exists, choose a bounded open set $W$ containing the fundamental domain $[0,1]^2$. The set  $\bigcup_{k\in\Z} \check f^k(W)$ is an invariant bounded open set.  One finds an invariant bounded open disk $V$ containing $[0,1]^2$ by looking at the complement of the unbounded component of the complement of $W$. Let us show first that $\partial V$ has no periodic points. Since $V$ is bounded, we may take a sufficiently large integer $L$ such that, if $\hat \T=\R^2/(L\Z)^2$ is the torus that finitely covers $\T^2$, $\hat f$ is the induced homeomorphism and $\hat \pi:\R^2\to\hat\T^2$ is the projection, then $\hat \pi(\overline{V})$ is contained in a topological disk.  \textcolor{black}{The diffeomorphism $\hat f$ satisfies the following properties: }

\begin{itemize}
\item \textcolor{black}{if $\hat f^q(z)=z$, then $1$ is not an eigenvalue of $D\hat f^q(z)$;}
%\smallksip
%\noindent-\enskip \textcolor{green}
\item\textcolor{black}{every elliptic periodic point of $\hat f$ is Moser stable;}
%\smallksip
%\noindent-\enskip \textcolor{green}{if $z$ is an elliptic periodic point of period $q$ (which means that the eigenvalues of $Df^q(z)$ are on the unit circle), then $z$ is Moser stable (which means that $z$ is surrounded by $f^q$-invariant curves arbitrarily close to $z$);}
\item\textcolor{black}{ the stable and unstable manifolds of hyperbolic periodic points of $\hat f$ \textcolor{black}{are either disjoint or they} intersect transversally. }
%\smallksip
%\noindent-\enskip \textcolor{green}{if $z$ is a hyperbolic periodic point of period $q$ (which means that the eigenvalues of $Df^q(z)$ are real), then the stable and unstable manifold of $z$ intersect transversally. }
\end{itemize}

\textcolor{black}{By a theorem of J. Mather (see \cite{Mt}), one knows that the prime-end rotation number of $\hat\pi(V)$ is irrational. The main result from \cite{KLN} shows that the  frontier $\partial \hat\pi(V)$ has no periodic point because the prime-end rotation number of $\hat\pi(V)$ is irrational.} This implies that $\partial V$ has no periodic points. 

In fact it is not necessary to use \cite{KLN}. Indeed, working directly with $V$ and $\check f$, the boundedness condition implies that the stable and unstable manifolds of every hyperbolic periodic point $z$ are bounded. Mather's arguments implies that, under our generic conditions, the branches of 
$z$ have all the same closure. By a result of Pixton (\cite{Pi}), every stable branch of $z$ intersect every unstable branch and one can find surrounding curves arbitrarily close to $z$ contained in the union of the stable and unstable manifolds. By Mather's argument again, one knows that such a point $z$ cannot be contained in $\partial V$. Moreover there is no elliptic periodic point on $\partial V$. 

The fact that $V$ contains $[0,1]^2$ implies that $\bigcup_{p\in\Z^2} (\partial V+p)$ is connected.  Moreover, the interior of $\bigcup_{p\in\Z^2} (\partial V+p)$ is empty.  This set projects onto a compact subset of $\T^2$ whose interior is empty, which is totally essential (the connected components of its complement are open disks) and which does not contain periodic points of $f$. This contradicts a result of A. Koropecki (see \cite {Ko}) \textcolor{black}{that states the following: if $K$ is an invariant closed connected subset of a homeomorphism defined on a closed orientable surface and having no wandering points, and if $K$ has no periodic point, then either $M$ is a torus and $K$ coincides with $M$, or $K$ is a decreasing sequence of compact annuli.}\end{proof}

Before proving our three theorems, let us state some introductory results. In what follows (Proposition \ref{pr:Bounded_Leaves} and  Proposition \ref{pr: rational rotation number}) $f$ is a  homeomorphism of $\T^2$ that is isotopic to the identity and $\check f$ a lift of $f$ to  $\R^2$. We consider an identity isotopy $I'$ of $f$ that is lifted to an identity isotopy $\check I'$ of $\check f$. We consider a maximal hereditary singular isotopy $I$ larger than $I'$ and its lift $\check I$ to $\R^2$. We consider a foliation $\mathcal F$ transverse to $I$ an its lift $\check{\mathcal F}$ to $\R^2$.

\begin{proposition}\label{pr:Bounded_Leaves}
 If $(0,0)$  belongs to the interior of $ \mathrm{rot}(\check f)$ or to the interior of a segment with irrational slope included in $ \partial\left(\mathrm{rot}(\check f)\right)$, then the leaves of $\check{\mathcal F}$ are uniformly bounded.
\end{proposition}

\begin{proof}
\enskip Suppose first that $(0,0)$ belongs to the interior of $ \mathrm{rot}(\check f)$. One can find finitely many extremal points $\rho_i$  of $ \mathrm{rot}(\check f)$, $1\leq i\leq r$, that linearly generate the plane and positive numbers $t_i$, $1\leq i\leq r$, such that:
$$\sum_{1\leq i\leq r} t_i=1,\enskip \sum_{1\leq i\leq r} t_i\rho_i=(0,0).$$
Each $\rho_i$ is the rotation number of an ergodic measure $\mu_i\in{\mathcal M}(f)$. Applying Poincar\'e Recurrence Theorem and Birkhoff Ergodic Theorem, one can find a  positively recurrent point $z_i$ of $f$ having $\rho_i$ as a rotation number. Fix a lift $\check z_i$ of $z_i$ and a  small neighborhood $\check W_i$ of $\check z_i$ that trivializes $\check{\mathcal  F}$. One can find a subsequence $(\check f^{n_l}(z_i))_{l\geq 0}$ of $\check f^{n}(z_i)_{n\geq 1}$ and a sequence $(p_{i,l})_{l\geq 0}$ of integers such that $\check f^{n_l}(\check z_i)\in \check W_i+p_{i,l}$ and such that $\lim_{l\to +\infty} p_{i,l}/n_l=\rho_i$. One deduces that the transverse homological space $\mathrm{THS}({\mathcal F})$ contains $p_{i,l}$. If $l$ is large enough, the $p_{i,l}$ generate the plane and $(0,0)$ is contained in the interior of the polygonal defined by these points. By Proposition \ref{pr: THS}, we deduce that the leaves of $\check {\mathcal F}$ are uniformly bounded.

Suppose now that $(0,0)$ belongs to the interior of a segment with irrational slope included in $ \partial\left(\mathrm{rot}(\check f)\right)$. If this segment $[\rho_1,\rho_2]$ is chosen maximal, then $\rho_1$ and $\rho_2$ are extremal points of $ \mathrm{rot}(\check f)$ and respectively equal to the rotation number of ergodic measures $\mu_1$ and $\mu_2$ in ${\mathcal M}(f)$.  Let $W_i\subset\T^2$ be a trivializing box of $\mathcal F$ such that $\mu_i(W_i)\not=0$  and $\check W_i\subset \R^2$ a lift of $W_i$.  The first return map $\Phi_i:W_i\to W_i, z\mapsto f^{\tau_i}(z)$ (where $\tau_i: W_i\to\N$) is defined $\mu_i$-almost everywhere on $W_i$ as the displacement function $\xi_i:  W_i\to\Z^2$, where $\check f^{\tau_i(z)}(\check z)\in \check W_i+\xi_i(z)$, if $\check z$ is the lift of $z$ that belongs to $\check W_i$. Let $\psi:\R^2\to\R$ be a non trivial linear form that vanishes on our segment. Using  Birkhoff Ergodic Theorem, one knows that $\mu_i$-almost every point $z$ has a rotation number $\rho_i$, and so
 $$\lim_{n\to+\infty} {\sum_{k=0} ^{n-1} \xi_i (\Phi_i ^k(z))\over \sum_{k=0} ^{n-1} \tau_i(\Phi_i^k(z))} =\rho_i.$$ By Kac's theorem, one knows that $\tau_i$ is $\mu_i$-integrable and satisfies $\int_{W_i} \tau_i \,d\mu_i =\mu_i(\bigcup_{k\in\Z} f^k(W_i))\in\textcolor{black}{(0,1]}$. One can note that $\xi_i/\tau_i$ is bounded, which implies that $\xi_i$ is $\mu_i$-integrable. Consequently, one has 
$$\lim_{n\to+\infty} {\sum_{k=0} ^{n-1}\xi_i (\Phi_i ^k(z))\over \sum_{k=0} ^{n-1} \tau_i(\Phi_i^k(z))} = {\int_{W_i} \xi_i \,d\mu_i\over\int_{W_i} \tau_i \,d\mu_i},$$
which implies that $$\int_{W_i} \xi_i \,d\mu_i=\left(\int_{W_i} \tau_i \,d\mu_i\right)\rho_i\not=0$$
and $$\int_{W_i} \psi\circ\xi_i \,d\mu_i=\psi\left(\int_{W_i} \xi_i \,d\mu_i\right)=0.$$ Note that   $ \psi\circ\xi_i (z)\not=0$ if $ \xi_i ( z)\not=0$, because  $\xi_i (z)$ is an integer and the kernel of $\psi$ is generated by a segment with irrational slope.  We deduce that there exists $ z_1$, $z'_1$ in $ W_1$ such that  $$\psi\circ\xi_1 ( z_1)<0<\psi\circ\xi_1( z'_1).$$ Consequently, one can find $z''_1\in W_1$,  $ z''_2\in  W_2$ and integers $n_1$, $n_2$ such that $(0,0)$ is in the interior of the quadrilateral determined by 
$$\xi_1 ( z_1),\enskip \xi_1( z'_1),\enskip {\sum_{k=0} ^{n_1-1} \xi_1(\Phi_1^k(z''_1))\over \sum_{k=0} ^{n_1-1} \tau_1(\Phi_1^k( z''_1))}, \enskip {\sum_{k=0} ^{n_2-1} \xi_2(\Phi_2 ^k( z_2))\over \sum_{k=0} ^{n_2-1} \tau_2(\Phi_2^k(z''_2))},$$
because the last two points may be chosen arbitrarily close to $\rho_1$ and $\rho_2$.
The set $\mathrm{THS}({\mathcal F})$ containing  the integers
$$\xi_1 ( z_1),\enskip \xi_1( z'_1),\enskip \sum_{k=0} ^{n_1-1} \xi_1(\Phi_1 ^k( z''_1)), \sum_{k=0} ^{n_2-1} \xi_2(\Phi_2 ^k( z_2)),$$
one can apply Proposition \ref{pr: THS} to conclude that the leaves of $\check{\mathcal F}$ are uniformly bounded.
\end{proof}
% \hfill$\Box$

\begin{remark}\label{rm:Franks}
 As a corollary, one deduces that $\mathcal F$ is singular and that $\check f$ is not fixed point free. Applying this to $\check f^q-p$, for every rational $p/q\in\mathrm{int}( \mathrm{rot}(\check f))$, one deduces that there exists a point $z\in\R^2$ such that $\check f^q(z)=z+q$. This result was already well known, due to Franks \cite{F2}.
\end{remark}

\bigskip
\begin{proposition} \label{pr: rational rotation number} We suppose that the leaves of $\check{\mathcal F}$ are uniformly bounded. If there exists an admissible transverse path $\check\gamma:[a,b]\to\mathrm{dom}(\check{\mathcal F})$ of order $q$ and an integer $p\in\Z^2$ such that $\check\gamma$ and $\check\gamma+p$ intersect $\check{\mathcal{F}}$-transversally at $\phi_{\check\gamma(t)}=\phi_{(\check\gamma+p)(s)}$, where $s<t$, then $p/q$ belongs to $\mathrm{rot}(\check f)$.
\end{proposition}

\begin{proof}
\enskip By Corollary \ref{co: induction transverse} %\textcolor{green}{SENTENCE REMOVED}
%and Proposition \ref{pr: order plane},
one deduces that for every $k\geq 2$  the path
$$\check \gamma\vert_{[a,t]}\left(\prod_{0<i<k-1} (\check \gamma+ip)\vert_{[s,t]}\right) (\check \gamma+(k-1)p)\vert_{[s,b]}$$ 
is admissible of order $kq$.  This implies that there exists a point $\check z_k\in \phi_{\check \gamma(a)}$ such that $\check f^{kq}(\check z_k)\in \phi_{(\check \gamma+(k-1)p)(b)}= \phi_{ \check \gamma(b)}+(k-1)p$. The fact that the leaves of $\check {\mathcal F}$ are uniformly bounded tells us that there exists $K$ such that for every $k\geq 1$, one has $\Vert \check f^{kq}(\check z_k)-\check z_k-(k-1)p\Vert\leq K$. Denote $z_k$ the projection of $\check z_k$ in $\T^2$. Choose a measure $\mu$ that is the limit of a subsequence of
$\left({1\over kq} \sum_{i=0}^{kq-1}\delta_{f^{i}(z_{k})}\right)_{k\geq 2}$ for the weak$^*$ topology. It is an invariant measure of $f$ of rotation number $p/q$ for $\check f$. \end{proof}
%\hfill$\Box$

\bigskip

Let us state the following improved version of Atkinson's Lemma:

\begin{proposition} \label{pr: better_atkinson}
 Let $(X,\mathcal{B},\mu)$ be a probability space and $T:X\to X$ an ergodic automorphism. If $\varphi: X\to \R$ is an integrable map such that $\int\varphi\, d\mu=0$, then for every $B\in\mathcal{B}$  and every $\varepsilon >0$, one has
$$ \mu\left( \left\{x\in B, \enskip \exists n\geq 0, \enskip T^n(x)\in B\enskip\mathit{and}\enskip0\leq\sum_{k=0}^{n-1}\varphi(T^k(x))< \varepsilon\right\}\right)= \mu (B).$$
\end{proposition}

\begin{proof}
\enskip Let us consider $B\in{\mathcal B}$ and set
$$A=B\setminus\left\{x\in B, \enskip \exists n\geq 0, \enskip T^n(x)\in B\enskip\mathit{and}\enskip0\leq\sum_{k=0}^{n-1}\varphi(T^k(x))< \varepsilon\right\}\}.$$
Atkinson's result directly implies that  there exists a set $A'\subset A$ with $\mu(A')=\mu(A)$ such that, for every $x\in A'$, there exists a subsequence $(n_l)_{l\in\N}$ such that $T^{n_l}(x)\in A$ and $\lim_{l\to\infty} \sum_{k=0} ^{n_l-1} \varphi(T^k(x))=0$. Assume, for a contradiction, that $\mu(A)>0$. There exists some $x\in A'$  and $n_0>0$ such that $y=T^{n_0}(x) \in A$ and $ a=\sum_{k=0}^{n_0-1}\varphi(T^{k}(x)) \in(-\varepsilon, \varepsilon)$, and since $x\in A$ we know that $a<0$. Since $x \in A'$ there exists some $n_{1}>n_0$ such that $ T^{n_{1}}(x)\in A$ and $ a < \sum_{k=0}^{n_1-1}\varphi(T^{k}(x))< \varepsilon +a$. This implies that $T^{n_{1}-n_0}(y) \in A$ and that $0<\sum_{k=0}^{n_l-n_0-1}\varphi(T^{k}(y))<\varepsilon$, which is a contradiction since $y \in A$ proving the claim.
\end{proof}

\noindent {\it Proof of Theorem \ref{th:impossible_rotation_set}.}\enskip We will give a proof by contradiction. Replacing $f$ by $f^q$ and $\check f$ by $\check f^q-p$, where $q\in\N$ and $p\in\Z^2$, we can suppose that the frontier of $\mathrm{rot}(\check f)$ contains a segment $[\rho_0,\rho_1]$ with irrational slope, that $(0,0)$ is in its interior and that $\rho_0$ and $\rho_1$ are extremal points of $\mathrm{rot}(\check f)$. We can suppose moreover than for every $\rho\in\mathrm{rot}(\check f)$, one has
$\langle \rho_0^{\perp},\rho\rangle\leq 0\leq \langle \rho_1^{\perp},\rho\rangle$.
We consider two ergodic measures $\mu_0$ and $\mu_1$ in ${\mathcal M}(f)$ whose rotation vectors are $\rho_0$ and $\rho_1$ respectively. We know that there exists a point $z_0\in\R^2$ such that $\mathrm{rot}(z_0)=\rho_0$ and that projects onto a bi-recurrent point of $f$. 
 By Proposition \ref{pr: better_atkinson}, we have a stronger result:

\begin{lemma}\label{le: precise_atkinson}
There exists a point $z_0$, \textcolor{black}{projecting to a bi-recurrent point and with $\mathrm{rot}(z_0)=\rho_0$}, such that for every $\varepsilon\in \{-1,1\}$ one can find a sequence $(p_l,q_l)_{l\geq 0}$ in $\Z^2\times \N$ satisfying:
$$\lim_{l\to +\infty} q_l=+\infty, \enskip  \lim_{l\to +\infty}\check  f^{q_l}(z_0)-z_0-p_l=0, \enskip \textcolor{black}{\lim_{l\to +\infty} \langle \rho_0^{\perp},p_l\rangle =0,} \enskip \varepsilon\langle \rho_0^{\perp},p_l\rangle>0$$ and a sequence $(p'_l,q'_l)_{l\geq 0}$ in $\Z^2\times \N$ satisfying:
$$\lim_{l\to +\infty} q'_l=+\infty,   \enskip  \lim_{l\to +\infty}\check  f^{-q'_l}(z_0)-z_0-p'_l=0,\enskip \textcolor{black}{\lim_{l\to +\infty} \langle \rho_0^{\perp},p'_l\rangle =0,} \enskip\varepsilon\langle \rho_0^{\perp},p'_l\rangle>0.$$

\end{lemma}

\begin{proof} Let $W_0$ be a small disk such that $\mu_0(W_0)\not=0$ and $\check W_0$ a lift of $W_0$. Define the maps $\tau_0$ and $\xi_0$ like in Proposition \ref{pr:Bounded_Leaves}. The measure $\mu_0$ being ergodic, one knows that $\int_{W_0} \tau_0\, d\mu_0=1$ and that $\int_{W_0} \xi_0\, d\mu_0 =\rho_0$. Let us define on $W_0$ the first return map $T: z\mapsto f^{\tau_0(z)}(z)$ and the function $\varphi :z\mapsto \varepsilon
\langle \rho_0^{\perp}, \xi_0(z)\rangle$. 

For each integer $i\geq 1$, let  $(B_{i,j})_{1\leq j\leq  k_i}$ be a covering of $W_0$ by open sets with diameter smaller than $1/ i$, and define
$$C_{i,j}=\left\{x\in B_{i,j}\cap W_0, \enskip\exists n\geq 0, \enskip T^n(x)\in B_{i,j}\cap W_0\enskip\mathrm{and}\enskip0\leq\sum_{k=0}^{n-1}\varphi(T^k(x))< 1/i\right\}$$
Set $C_i= \bigcup_{j=1}^{k_i} C_{i,j}$ and $C= \bigcap _{i\geq 1} C_i$. By Proposition \ref{pr: better_atkinson}, one knows that \textcolor{black}{ $\mu_0(C_i)=\mu_0(C)=\mu_0(W_0)$, and if $C'$ is the subset of the bi-recurrent points of $C$ with rotation vector $\rho_0$, then $\mu_0(C')=\mu_)(C)$.}. If $z_0$ belongs to $C'$, one can find an increasing integer sequence $(m_l)_{l\geq 0}$ such that 
$$ \lim_{l\to +\infty} T^{m_l}(z_0)=z_0, \enskip \lim_{l\to+\infty}\sum_{k=0}^{m_l-1} \varepsilon\langle \rho_0^{\perp},\xi_0(T^k(z_0)\rangle= 0, \enskip \sum_{k=0}^{m_l-1} \varepsilon\langle \rho_0^{\perp},\xi_0(T^k(z_0)\rangle\geq 0.
$$ 
Setting $p_l=\sum_{k=0}^{m_l-1} \xi_0(T^k(z_0))$ and $q_l=\sum_{k=0}^{m_l-1} \tau_0(T^k(z_0))$, one gets the first assertion of the lemma, with a large inequality instead of a strict one. Noting that $\lim_{l\to+\infty} \Vert p_l\Vert =+\infty$ and that the line generated by $\rho_0$ has irrational slope, one deduces that the inequality is strict. The second assertion  can be proved analogously. \end{proof}

%We consider an identity isotopy $I'$ of $f$ that is lifted to an identity isotopy $\check I'$ of $\check f$. We consider a maximal hereditary singular isotopy $I$ larger than $I'$ and its lift $\check I$ to $\R^2$. We consider a foliation $\mathcal F$ transverse to $I$ an its lift $\check{\mathcal F}$ to $\R^2$.
% \textcolor{green}{TEXT REMOVED.}
\textcolor{black}{We note that, by Proposition \ref{pr:Bounded_Leaves},}  the leaves of $\check{\mathcal F}$ are uniformly bounded. \textcolor{black}{Let us choose $z_0$ as in the previous lemma.} The fact that $\mathrm{rot}(z_0)=\rho_0$ tells us that the whole trajectory $\check I^{\Z}(z_0)$ is a proper path directed by $\rho_0$. The fact that the leaves of $\check{\mathcal F}$ are uniformly bounded and that every leaf met by  $\check I_{\check{\mathcal F}}^{\Z}(z_0)$  is also met by $\check I^{\Z}(z_0)$ implies that $\check I_{\check{\mathcal F}}^{\Z}(z_0)$ is a transverse proper path directed by $\rho_0$. We  parameterize $\check I_{\check{\mathcal F}}^{\Z}(z_0)$ in such a way that $\check I_{\check{\mathcal F}}^{\Z}(z_0) \vert_{[l,l+1]}=\check I_{\check{\mathcal F}}(\check f^l(z_0))$. We consider sequences $(p_l,q_l)_{l\geq 0}$ and $(p'_l,q'_l)_{l\geq 0}$ given by the previous lemma (the sign $\varepsilon$ has no importance at the beginning).

\bigskip
\begin{lemma} \label{le: recurrence}
 For every closed segment $[a,b]\subset\R$ and every positive real numbers $L$, $\varepsilon$, there exists $p\in \Z^2$ and a segment $[a',b']\subset \R$ satisfying $a'-b>L$ such that $\vert\langle \rho_0^{\perp},p\rangle\vert<\varepsilon$ and such that the paths $\check I_{\check{\mathcal F}}^{\Z}(z_0) \vert_{[a,b]}$ and $(\check I_{\check{\mathcal F}}^{\Z}(z_0) +p)\vert_{[a',b']}$ are equivalent. One has a similar result replacing the inequality $a'-b>L$ by $a-b'>L$.
\end{lemma}

\begin{proof}
\enskip Let us choose integers $q$ and $q'$ such that $[a,b]\subset (q, q')$. \textcolor{black}{As $z_0$ projects to a bi-recurrent point,  one} can find $l$, \textcolor{black}{using Lemma \ref{le: precise_atkinson},} sufficiently large,  such that $q_l>q'-q+L$ and such that $\check f^{q_l}(\check f^q(z_0))-p_l$ is so close to $\check f^q(z_0)$ that we can affirm that $\check I_{\check{\mathcal F}}^{\Z}(z_0) \vert_{[a,b]}$ is equivalent to a path $(\check I_{\check{\mathcal F}}^{\Z}(z_0) -p_l)\vert_{[a',b']}$, where $[a',b']\subset (q+q_l, q'+q_l)$. \textcolor{black}{ Note that $\vert\langle \rho_0^{\perp},p_l\rangle\vert<\varepsilon$ if $l$ is sufficiently large.} The version with the inequality $a-b'>L$ can be proven similarly by using the sequences $(p'_l)_{l\geq 0}$ and $(q'_l)_{l\geq 0}$.
\end{proof}
%\hfill$\Box$

\bigskip
\begin{lemma} \label{le: no_intersection}
There is no $p\in\Z^2\setminus\{0\}$ such that $\check I_{\check{\mathcal F}}^{\Z}(z_0) $ and $\check I_{\check{\mathcal F}}^{\Z}(z_0) +p$ intersect $\check{\mathcal{F}}$-transversally.
\end{lemma}

\begin{proof}
\enskip Write $\check I_{\check{\mathcal F}}^{\Z}(z_0) =\gamma_0$ for convenience. Suppose that $\gamma_0$ and $\gamma_0+p$ intersect $\check{\mathcal{F}}$-transversally at $\phi= \phi_{\gamma_0(t)}=\phi_{(\gamma_0 +p)(s)}$. The leaves being uniformly bounded, one knows that 
$\phi_{\gamma_0(t)}\not =\phi_{\gamma_0(t)}+p$ and so $t\not=s$. Replacing $p$ with $-p$ if necessary, one can suppose that $s<t$. By Proposition \ref{pr: rational rotation number}, there exists $q\geq 1$ such that $p/q\in \mathrm{rot}(\check f)$. Consequently, one has $ \langle \rho^{\perp}_0,p\rangle\leq 0$.  By assumption, the segment $[0,\rho_0]$ has irrational slope and $p\not=0$, so one deduces that $ \langle \rho_0^{\perp},p\rangle<0$. 

\textcolor{black}{ We know that there exists a sufficiently large $N$ such that $\gamma_0\vert_{[-N, N]}$ and $(\gamma_0+p)\vert_{[-N, N]}$ intersect $\check{\mathcal{F}}$-transversally at $\phi$. By  Lemma \ref{le: recurrence}, we can find some $p'\in\Z^2$ such that $\vert\langle \rho_0^{\perp},p'\rangle\vert$ is sufficiently small as to get $\langle \rho_0^{\perp},p+p'\rangle<0$, and such that there exists some $a',b'$ with $N<a'<b'$, where $(\gamma_0+p')\vert_{[a', b']}$ is equivalent to $\gamma_0\vert_{[-N, N]}$. This implies that $(\gamma_0+p+p')\vert_{[a', b']}$ is equivalent to $(\gamma_0+p)\vert_{[-N, N]}$, and so} $\gamma_0$ and $\gamma_0+p+p'$ intersect $\check{\mathcal{F}}$-transversally at $\phi_{\gamma_0(t)}=\phi_{(\gamma_0 +p+p')(s')}$ where $s'>t$. So, one knows that  $\gamma_0$ and $\gamma_0-p-p'$ intersect $\check{\mathcal{F}}$-transversally at $\phi_{\gamma_0 (s')}=\phi_{(\gamma_0 -p-p')(t)}$. We deduce \textcolor{black}{ as before, by Proposition \ref{pr: rational rotation number}, that for some $q'>1$ the vector $\frac{-p-p'}{q'}\in \mathrm{rot}(\check f)$,  a contradiction since   $\langle \rho_0^{\perp},-p-p'\rangle<0$.} 
\end{proof}
%\hfill$\Box$

\begin{lemma}
 The path $\check I_{\check{\mathcal F}}^{\Z}(z_0) $ is a line
\end{lemma}

\begin{proof}
\enskip Here again, write $\check I_{\check{\mathcal F}}^{\Z}(z_0) =\gamma_0$. If $\gamma_0$ is a not a line, by Proposition \ref{pr:transverse_on_plane} one knows that there exist two segments $[a_0,b_0]$ and $[a_1,b_1]$ such that $\gamma_0\vert_{[a_0,b_0]}$ and $\gamma_0\vert_{[a_1,b_1]}$ intersect $\check{\mathcal{F}}$-transversally. By Lemma  \ref{le: recurrence}, one deduces that there exist $p\in\Z^2\setminus\{0\}$ and a segment $[a'_1,b'_1]$ such that $\gamma_0\vert_{[a_0,b_0]}$ and $(\gamma_0+p)\vert_{[a'_1,b'_1]}$ intersect $\check{\mathcal{F}}$-transversally. This contradicts Lemma  \ref{le: no_intersection}. 
\end{proof}
% \hfill$\Box$

\bigskip

Similarly, we can find a point $z_1$ of rotation number $\rho_1$ that projects onto a recurrent point of $f$ and such that $\check I_{\check{\mathcal F}}^{\Z}(z_1) $ is a line directed by $\rho_1$ that does not meet transversally its integer translated. 

\bigskip
\begin{lemma}
 The line $\check I_{\check{\mathcal F}}^{\Z}(z_1) $ intersects $\check{\mathcal{F}}$-transversally one of the translates of $\check I_{\check{\mathcal F}}^{\Z}(z_0) $.
\end{lemma}

\begin{proof}
\enskip Let us prove by contradiction that $\gamma_1=\check I_{\check{\mathcal F}}^{\Z}(z_1)$ intersects $\check{\mathcal{F}}$-transversally one \textcolor{black}{of} the translates of $\gamma_0=\check I_{\check{\mathcal F}}^{\Z}(z_0)$. If not,  let us denote by $U_0$ the union of leaves that are met by $\gamma_0$. Its complement can be written $R(U_0)\sqcup L(U_0)$ where $R(U_0)=R(\gamma_0)\setminus U_0$ is the union of $r(\gamma_0)$ and of the set of singularities at the right of $\gamma_0$ and $L(U_0)=L(\gamma_0)\setminus U_0$ is the union of $l(\gamma_0)$ and of the set of singularities at the left of $\gamma_0$. If $\gamma_1$ and $\gamma_0$ do not intersect $\check{\mathcal{F}}$-transversally, then by \textcolor{black}{Corollary} \ref{co: no_intersection_two_sets}, one knows that either $\gamma_1\cap R(U_0)=\emptyset$ or $\gamma_1\cap L(U_0)=\emptyset$. \textcolor{black}{As $\gamma_0$ is directed by $\frac{\rho_0}{\Vert\rho_0\Vert}$ and $\gamma_1$ is directed by the opposite vector $\frac{\rho_1}{\Vert\rho_1\Vert}=-\frac{\rho_0}{\Vert\rho_0\Vert}$,  one knows that if $\gamma_1\cap R(U_0)=\emptyset$, then $R(\gamma_1)\cap R(U_0)=\emptyset$, and if $\gamma_1\cap L(U_0)=\emptyset$, then $L(\gamma_1)\cap L(U_0)=\emptyset$.}  Consequently, $\gamma_0$ and $\gamma_1$ cannot meet a common leaf. Indeed if $\phi$ is such a leaf, one knows by Proposition  \ref{pr:transverse_on_plane} that it is met once by $\gamma_0$ and $\gamma_1$. So, the $\alpha$-limit set of $\phi$ is contained in $L(U_0)\cap \overline {L(\gamma_1)}$ and the $\omega$-limit set is included in $R(U_0)\cap\overline {R(\gamma_1)}$, which is impossible. 

So, if the conclusion of our lemma is not true,  there exists a partition $\Z^2=A^-\sqcup A^+$, where 
$$\eqalign{p\in A^-&\Leftrightarrow (r(\gamma_0)\cup \textcolor{black}{U_0})+p\subset l(\gamma_1),\cr
p\in A^+&\Leftrightarrow (l(\gamma_0)\cup \textcolor{black}{U_0})+p\subset r(\gamma_1).\cr}$$
\textcolor{black}{Also, by Lemma \ref{le: no_intersection} and the fact that $\gamma_0$ is directed by $\rho_0$, one knows that $l(\gamma_0+p)\subset l(\gamma_0)$ if $0\le \langle\rho_0^{\perp},p\rangle$ and one deduces} this partition is a cut of the order on $\Z^2$ defined as follows
$$p\preceq p'\Leftrightarrow \langle\rho_0^{\perp},p\rangle\leq \langle\rho_0^{\perp},p'\rangle.$$
Let us fix a leaf $\phi$ that intersects $\gamma_1$.  By Lemma \ref{le: recurrence}, one knows that there exists $p_0\not=(0,0)$ such that $\phi$ intersects $\gamma_1+p_0$. One deduces that $A^-+p_0=A^-$ and $A^++p_0=A^+$, which of course is impossible.
\end{proof}
% \hfill$\Box$

\bigskip
\noindent {\it End of the proof of Theorem \ref{th:impossible_rotation_set}.}\enskip Replacing $\gamma_0$ by a translate if necessary, we can always suppose that $\gamma_0$ and $\gamma_1$ intersect $\check{\mathcal{F}}$-transversally at $\gamma_0(t_0)=\gamma_1(t_1)$ and we define $\gamma=\gamma_0\vert_{[-\infty,t_0]}\gamma_1\vert_{[t_1,+\infty]}$ which is an admissible transverse proper path. There exist two segments $[a_0,b_0]$ and $[a_1,b_1]$ containing $t_0$ and $t_1$ respectively in their interior, such that $\gamma_0\vert_{[a_0,b_0]}$ and $\gamma_1\vert_{[a_1,b_1]}$ intersect transversally at $\gamma_0(t_0)=\gamma_1(t_1)$. Using Lemma \ref{le: recurrence}, one can find $p_0$ and $p_1$ in $\Z^2$ distinct, and segments $[a'_0,b'_0]\subset (-\infty, t_0)$ and $[a'_1,b'_1]\subset (t_1,+\infty)$ such that $(\gamma_0+p_0)\vert_{[a'_0,b'_0]}$ is equivalent to $\gamma_0\vert_{[a_0,b_0]}$ and $(\gamma_1+p_1)\vert_{[a'_1,b'_1]}$ is equivalent to $\gamma_1\vert_{[a_1,b_1]}$. We deduce that there exists $t'_0\in (a'_0,b'_0)$ and $t'_1\in (a'_1+t_0-t_1,b'_1+t_0-t_1)$ such that $\gamma+p_0$ and $\gamma+p_1$ intersect $\check{\mathcal{F}}$-transversally at $\phi= \phi_{(\gamma+p_0)(t'_0)}=\phi_{(\gamma+p_1)(t'_1)}$. So $\gamma$ and $\gamma+p_1-p_0$ intersect $\check{\mathcal{F}}$-transversally at $\phi-p_0= \phi_{\gamma(t'_0)}=\phi_{(\gamma+p_1-p_0)(t'_1)}$. Observing that $t'_0<t'_1$, one deduces\textcolor{black}{, by Proposition \ref{pr: rational rotation number},}   that there exists $q\geq 1$ such that $(p_1-p_0)/q\in \mathrm{rot}(\check f)$ and \textcolor{black}{thus} $ \langle \rho_0^{\perp},p_1-p_0\rangle<0$. But Lemma \ref{le: precise_atkinson} tells us that  \textcolor{black}{$p_0$ and $p_1$ can be chosen  so that $ \langle \rho_0^{\perp},p_0\rangle<0$ and $ \langle \rho_0^{\perp},p_1\rangle>0$.}  We have found a contradiction. \hfill$\Box$

\bigskip
\noindent {\it Proof of Theorem \ref{th:bounded_deviation}.}\enskip  In the proof, we will use the sup norm $\Vert \enskip\Vert_{\infty}$ where  $\Vert (x_1,x_2)\Vert_{\infty}=\max(\vert x_1\vert, \vert x_2\vert)$ which will be more convenient that the Euclidean norm and will write $d_{\infty}(z,X)=\inf_{z'\in X} \Vert z-z'\Vert_{\infty}$. Replacing $f$ by $f^q$ and $\check f$ by $\check f^q-p$, where $q\in\N$ and $p\in\Z^2$, we can suppose that $(0,0)$ is in the interior of $\mathrm{rot}(\check f)$. Here again, we consider an identity isotopy $I'$ of $f$ that is lifted to an identity isotopy $\check I'$ of $\check f$. We consider a maximal hereditary singular isotopy $I$ larger than $I'$ and its lift $\check I$ to $\R^2$. We consider a foliation $\mathcal F$ transverse to $I$ an its lift $\check{\mathcal F}$ to $\R^2$.
 One knows by Proposition \ref{pr:Bounded_Leaves} that the leaves of $\check{\mathcal F}$ are uniformly bounded. \textcolor{black}{In the remainder of the proof we will usually work in the universal covering space of $\T^2$, with paths transversal to the lifted foliation $\check{\mathcal{F}}$}. The theorem is an immediate consequence of the following, where the {\it direction} $D(\gamma)$ of a path $\gamma: [a,b]\to \R^2$ is defined as $D(\gamma)=\gamma(b)-\gamma(a)$:

\begin{proposition} \label{pr: bounded_direction}
There exists a constant $L$ such that for every transverse admissible path $\gamma$  of order $n$, one has $d_{\infty}(D(\gamma), n\,\mathrm{rot}(f))\leq L$.
\end{proposition}

We will begin by proving:

\begin{lemma}
 There exist a transverse admissible path $\gamma^*:[0,3]\to \R^2$,  a real number $K^*$ and an integer $p^*\in\Z^2$ such that:

\smallskip
\noindent-\enskip every transverse path $\gamma$ whose diameter is larger than $K^*$ intersects $\check{\mathcal{F}}$-transversally an integer translate of $\gamma^*_{\vert [1,2]}$;

\smallskip

\noindent-\enskip$\gamma^*_{\vert [2,3]}$ and $\gamma^*_{\vert [0,1]}+p^*$ intersect $\check{\mathcal{F}}$-transversally.
\end{lemma}

\bigskip
\noindent{\it Proof.}
 \enskip Let us choose $N$ large enough such $(1/N,0)$ and $(0,1/N)$ belong to the interior of $\mathrm{rot}(\check f)$. As previously noted \textcolor{black}{ in Remark \ref{rm:Franks}}, there exists a point $z_0$ satisfying $\check f^N(z_0)=z_0+(1,0)$ and a point $z_1$ satisfying
$\check f^N(z_1)=z_1+(0,1)$. The transverse trajectories $\check I_{{\check{\mathcal F}}}^{\Z}(z_0) $ and $\check I_{{\check{\mathcal F}}}^{\Z}(z_1) $ are parameterized, such that  $\check I_{{\check{\mathcal F}}}^{\Z}(z_0)(t+1)=\check I_{{\check{\mathcal F}}}^{\Z}(z_0)(t)+(1,0)$ and $\check I_{{\check{\mathcal F}}}^{\Z}(z_1)(t+1)=\check I_{{\check{\mathcal F}}}^{\Z}(z_1)(t)+(0,1)$. 

\bigskip

\begin{sub-lemma}\label{sle: bound}\enskip There exists a real number $K$ such that if $\gamma$ is a transverse path that does not intersect $\check{\mathcal{F}}$-transversally $\check I_{{\check{\mathcal F}}}^{\Z}(z_0) $ , then either $\pi_2(\gamma(t))> -K$ or $\pi_2(\gamma(t))<K$ and if it does not intersect $\check{\mathcal{F}}$-transversally $\check I_{{\check{\mathcal F}}}^{\Z}(z_1) $, then either $\pi_1(\gamma(t))>-K$ or $\pi_1(\gamma(t))<K$. 

\end{sub-lemma}

\begin{proof} There exists $K_0>0$ such that the diameter of each leaf of $\mathcal F$ is bounded by $K_0$ and there exists $K'_0>0$ such that $\check I_{{\check{\mathcal F}}}^{\Z}(z_0)  \subset \R\times(-K'_0,K'_0)$. Setting $K=K_0+K'_0$, note that every leaf that intersects \textcolor{black}{$\R\times(-\infty, -K]$} belongs to $r(\check I_{{\check{\mathcal F}}}^{\Z}(z_0) )$ and every leaf that intersects $\R\times[K,+\infty)$ belongs to $l(\check I_{{\check{\mathcal F}}}^{\Z}(z_0))$. It remains to apply Corollary  \ref{co: no_intersection_two_sets}. We have a similar argument for $\check I_{{\check{\mathcal F}}}^{\Z}(z_1)$.
\end{proof}

\bigskip

Setting $K^*=2K+1$, one deduces immediately:

\begin{corollary}
 If $\gamma$ is a transverse path and if the diameter of $\pi_2\circ\gamma$ is larger than $K^*$, there exists $p_0\in \Z^2$ such that $\gamma$ intersects $\check{\mathcal{F}}$-transversally $\check I_{{\check{\mathcal F}}}^{\Z}(z_0) +p_0$ and if the diameter of $\pi_1\circ\gamma$ is larger than $K^*$, there exists $p_1\in \Z^2$ such that $\gamma$ intersects $\check{\mathcal{F}}$-transversally $\check I_{{\check{\mathcal F}}}^{\Z}(z_1) +p_1$. 
\end{corollary}

 In particular $\gamma_0=\check I_{{\check{\mathcal F}}}^{\Z}(z_0)$ intersects $\gamma_1=\check I_{{\check{\mathcal F}}}^{\Z}(z_1)$ $\check{\mathcal{F}}$-transversally at a point $\gamma_0(t_0)=\gamma_1(t_1)$. One can find an integer $r>0$ such that $\gamma_0\vert_{[t_0-r,t_0+r]}$ and $\gamma_1\vert_{[t_1-r,t_1+r]}$ intersect $\check{\mathcal{F}}$-transversally at $\gamma_0(t_0)=\gamma_1(t_1)$. Let $\gamma^*:[0,3]\to \R^2$ be a path such that

\smallskip
\noindent-\enskip \enskip  $\gamma^*_{[0,1]}$ is a reparameterization of $\gamma_0\vert_{[t_0-(4r+2),t_0-(2r+2)]}$;

\smallskip
\noindent-\enskip \enskip  $\gamma^*_{[1,2]}$ is a reparameterization of $\gamma_0\vert_{[t_0-(2r+2),t_0]}\gamma_1\vert_{[t_1,t_1+(2r+2)]}$;

\smallskip
\noindent-\enskip \enskip  $\gamma^*_{[2,3]}$ is a reparameterization of $\gamma_1\vert_{[t_1+(2r+2),t_1+(4r+2)]}$.

\medskip
Let us prove that $\gamma^*$ satisfies the proposition. Observe first that $\gamma^*$ is admissible of order $(8r+4)N$ by Corollary \ref{co: first induction transverse} and that the paths $\gamma^*\vert_{[0,2]}$ and $\gamma^*\vert_{[1,3]}$ are admissible of order $(6r+2)N$. Note first that $\gamma^*_{\vert [2,3]}$ and $\gamma^*_{\vert [0,1]}+(3r+2,3r+2)$ intersect $\check{\mathcal{F}}$-transversally. One can set $p^*=(3r+2,3r+2)$. Let $\gamma$ be a transverse path such that the diameter of $\pi_2\circ \gamma$ is larger than $K^*$. By Sub-lemma \ref{sle: bound} one knows that there exists  $p_0\in\R^2$ such that  $\gamma$ intersects $\check{\mathcal{F}}$-transversally $\gamma_0+p_0$. This means that there exist two real segments $J$ and $J_0$ such that

\smallskip
\noindent-\enskip \enskip  $\gamma\vert_{J}$ intersects $\check{\mathcal{F}}$-transversally $\gamma_0\vert_{J_0}+p_0$;

\smallskip
\noindent-\enskip \enskip  $\gamma\vert_{\mathrm{int} J}$ and $\gamma_0\vert_{\mathrm{int }J_0}+p_0$ are equivalent.

\medskip
If the length of $J_0$ is smaller  than $2r+1$, then $J_0$ is included in an interval $[t_0-(2r+2)+l_0, t_0+l_0]$. This implies that $\gamma$ intersects $\check{\mathcal{F}}$-transversally $\gamma^*\vert_{[1,2]}+p_0+(l_0,0)$. If the length is at least equal to $2r+1$, then $J_0$ contains an interval $[t_0-r+l_0, t_0+r+ l_0]$. This implies that $\gamma$ intersects $\check{\mathcal{F}}$-transversally $\gamma_1\vert_{[t_1-r, t_1+r]}+ p_0+(l_0,0)$ and so intersects $\check{\mathcal{F}}$-transversally $\gamma^*\vert_{[1,2]}+p_0+(l_0,-r)$. We get the same conclusion for a  transverse path such that the diameter of $\pi_1\circ \gamma$ is larger than $K^*$.\hfill$\Box$

\bigskip
\noindent{\it Proof of the Proposition  \ref{pr: bounded_direction}.} \enskip 
We denote by $K^{**}$ the diameter of $\gamma^*$ and by $K^{***}$ the diameter of $\mathrm{rot}(f)$.
Let $\gamma: [a,b]\to \R^2$ be a transverse path such that $\Vert D(\gamma)\Vert>2K^*$. One can find $c$, $d$ in $(a,b)$ with $c<d$ such that $\Vert D(\gamma\vert_{[a,c]})\Vert=\Vert D(\gamma\vert_{[d,b]})\Vert=K^*$.  \textcolor{black}{Note that, if $c$ is chosen to be minimal with this property, and $d$ is chosen to be maximal, then the diameter of both $\gamma\vert_{[a,c]}$ and $\gamma\vert_{[d,b]}$ are at most $2K^*$.} There exist $p$ and $p'$ in $\Z^2$ such that

\smallskip
\noindent-\enskip \enskip $\gamma\vert_{[a,c]}$ and $\gamma^*\vert_{[1,2]}+p$ intersect $\check{\mathcal{F}}$-transversally at $\gamma(t)=\gamma^*(s)+p$;

\smallskip
\noindent-\enskip \enskip $\gamma\vert_{[d,b]}$ and $\gamma^*\vert_{[1,2]}+p'$ intersect $\check{\mathcal{F}}$-transversally at $\gamma(t')=\gamma^*(s')+p'$.

\medskip
If $\gamma$ is admissible of order $n$, then the path
$$\gamma'=(\gamma^*\vert_{[0,s]}+p)\gamma\vert_{[t,t']}(\gamma^*\vert_{[s',3]}+p')$$ is admissible of order $n+(12r+4)N$ by Corollary \ref{co: first induction transverse} and one has $$\Vert D(\gamma')-D(\gamma)\Vert\leq \textcolor{black}{4K^*}+2K^{**}.$$Recall that $\gamma^*_{[2,3]}$ intersects $\check{\mathcal{F}}$-transversally  $\gamma^*_{[0,1]}+p^*$. One deduces that $(\gamma^*\vert_{[s',3]}+p')$ intersects $\check{\mathcal{F}}$-transversally $(\gamma^*\vert_{[0,s]}+p)+p''$, where $p''=p'-p+p^*$ and so that $\gamma'$ intersects $\check{\mathcal{F}}$-transversally $\gamma'+p''$. Proposition \ref{pr: rational rotation number} tells us that $p''/(n+(12r+4)N)$ belongs to $\mathrm{rot}(\check f)$, which implies $$d(p'', n\,\mathrm{rot}(\check f))\leq (12r+4)N K^{***}.$$ Observe now that \textcolor{black}{$\Vert p''-D(\gamma')\Vert \leq 2K^{**}$} and so $\Vert p''-D(\gamma)\Vert \leq \textcolor{black}{4K^*+4K^{**}}$. So, one gets 
$$d(D(\gamma), n\,\mathrm{rot}(\check f))\leq \textcolor{black}{4K^*+4K^{**}}+(12r+4)NK^{***}.$$\hfill$\Box$

\bigskip
\noindent {\it Proof of Theorem \ref{th:Llibre_MacKay}
.}\enskip Here again, using the fact that for every $q\geq 1$ and every $p\in\Z^2$, one has $\mathrm{rot}(\check f^q+p)=q\mathrm{rot} (\check f)+p$, it is easy to see that it is sufficient to prove the result in the case where $(0,0)$ belongs to the interior of $\mathrm{rot}(\check f)$.  Here again, we consider an identity isotopy $I'$ of $f$ that is lifted to an identity isotopy $\check I'$ of $\check f$. We consider a maximal hereditary singular isotopy $I$ larger than $I'$ and its lift $\check I$ to $\R^2$. We consider a foliation $\mathcal F$ transverse to $I$ an its lift $\check{\mathcal F}$ to $\R^2$. We know that the leaves of $\check{\mathcal F}$ are uniformly bounded. We can immediately deduce the theorem from what has been done in the previous proof and Theorem \ref{th:transverse_imply_entropy}. Indeed, we know that there are two transverse loops associated to periodic points that have a transverse intersection. We will give a proof that does not use Theorem \ref{th:transverse_imply_entropy} by exhibiting separated sets.
\medskip

Let us begin with the following lemma:

\bigskip
\begin{lemma} \label{bounded_distance}
 There exists a constant $K$ such that for every point \textcolor{black}{$z\in\mathrm{dom}(\check{\mathcal F})$} and any $z'$ for which $\phi_{z'}$ intersects $I_{\check{\mathcal F}}(z)$, one has $d(z, z')\leq K$.
\end{lemma}

\begin{proof}
 \enskip There exists $K'>0$ such that the diameter of each leaf of $\check{\mathcal F}$ is bounded by $K'$. Moreover, the set $\bigcup_{t\in[0,1], z\in [0,1]^2}I(z)$, being compact, is included in $[-K'',K''+1]^2$,for some $K''>0$. The leaves that $ \check I_{\check{\mathcal F}}(z)$ intersects, are also intersected by $\check I(z)$ (see the beginning of Section 3). One deduces that $K=K'+K''$ satisfies the conclusion of the lemma.
\end{proof}
% \hfill$\Box$

\bigskip
 We consider the paths $\gamma_0= \check I_{\check{\mathcal F}}(z_0)$ and $\gamma_1= \check I_{\check{\mathcal F}}^{\Z}(z_1)$  defined in the proof of Theorem \ref{th:bounded_deviation}. We keep the same notations and set $z^*=\gamma_0(t_0)=\gamma_1(t_1)$. Let us define
$$K'''=\max\left(\mathrm{diam}( \gamma_0\vert_{[t_0,t_0+r]}), \,\mathrm{diam}( \gamma_1\vert_{[t_1,t_1+r]})\right)$$ and choose an integer 
$m\geq 1$ such that $mr\geq K'''+2K_0+K+1$. Set 
$$\gamma'_0=\gamma_0\vert_{[t_0,t_0+mr]},\enskip \gamma'_1=\gamma_1\vert_{[t_1,t_1+mr]}.$$ Fix $n$ and for every $\mathbf{e}=(\varepsilon_1,\dots,\varepsilon_n)\in\{0,1\}^{n}$ define
$$\gamma'_{\mathbf{e}}= \prod_{1\leq i\leq n} (\gamma'_{\varepsilon_i}+p_{i-1}),$$
where the sequence $(p_i)_{0\leq i\leq n}$ satisfies $k_0=0$ and is defined inductively by the relation:
$$p_{i+1}= \begin{cases} p_i+(mr,0) & \mbox{if } \varepsilon_i=0, \\
p_i+(0, mr) & \mbox{if }\varepsilon_i=1. \end{cases}$$
The path $\gamma'_{\omega}$ is admissible of order $l=nmrN$. More precisely, there exists a point $z_{\mathbf{e}}\in \phi_{z^*}$ such that $\check f^{l}(z_{\mathbf{e}})\in \phi_{z^*}+k_n$, and such that $\gamma'_{\mathbf{e}}= \check I_{\check{\mathcal F}}^{l}(z_{\mathbf{e}})$.

\begin{lemma} If $\mathbf{e}$ and $\mathbf{e}'$ are two different sequences in $\{0,1\}^{n}$, there exists $j\in\{0, \dots, l-1\}$ such that $\Vert \check f^j (z_{\mathbf{e}})-  \check f^j (z_{\mathbf{e}'})\Vert\geq 1$.  \end{lemma}

\begin{proof} Consider  the integer $i^*$ such that $ \varepsilon_{i^*}\not= \varepsilon'_{i^*}$ and $ \varepsilon_{i}= \varepsilon'_{i}$ if $i<i^*$. The leaf $\phi_{z^*}+p_{i^*}$ is intersected by $\gamma'_{\mathbf{e}}$ but not by $\gamma'_{\mathbf{e}'}$.  More precisely $d(\phi_{z^*}+p_{i^*}, \gamma'_{\mathbf{e}'})\geq mr-K'-K_0$. Using Lemma  \ref{bounded_distance}
, one deduces that there exists $j\in\{0, \dots, l\}$ such that $d(\check f^j(z_{\omega}), \phi_{z^*}+p_{i^*})\leq K$. Moreover, one knows that $d(\check f^j(z_{\mathbf{e}'}), \gamma'_{\mathbf{e}'})\leq K_0$ because $\gamma'_{\mathbf{e}}$ intersects $\phi_{\check f^j(z_{\mathbf{e}'})}$. One deduces that  
$$\Vert \check f^j (z_{\mathbf{e}})-  \check f^j (z_{\mathbf{e}'})\Vert\geq mr-K'-2K_0-K\geq 1.$$\end{proof}

\bigskip
To finish the proof of the proposition, let us define on $\T^2$ the distance
$$d(Z,Z')=\inf_{\pi(z)=Z,\,\pi(z')=Z'}\Vert z-z'\Vert,$$where$$\eqalign{\pi: \R^2&\to\T^2, \cr z&\mapsto z+\Z^2\cr}$$ is the projection. Note that for every $Z\in \T^2$, one has
$$\pi^{-1}(B(Z,1/4))=\bigsqcup_{\pi(z)=Z} B(z, 1/4)$$ and every map $\pi\vert_{B(z,1/4)}$ is an isometry from $B(z,1/4)$ onto $B(Z, 1/4)$.

Fix $\varepsilon\in(0,1/4)$  such that for every $z$, $z'$ in $\R^2$, one has$$\Vert z-z'\Vert<\varepsilon\Rightarrow 
\Vert \check f(z)-\check f(z')\Vert<1/4.$$ One deduces that two points $Z$ and $Z'$ such that $d(f^j(Z), f^j(Z))<\varepsilon$, for every $j\in\{0, \dots l-1\}$ are lifted by points $z$, $z'$ such that $\Vert \check f^j(z)-\check f^j(z)\Vert <\varepsilon$, for every $j\in\{0, \dots l-1\}$.

Consequently, the points $z_{\mathbf{e}}$ project on a $(nmrN, \varepsilon)$-separated set of cardinality $2^n$.  One deduces that  $h(f)>\log 2/mrN$.
\hfill$\Box$

\bibliographystyle{alpha}
\renewcommand{\refname}{

  \end{document}